\documentclass[12pt]{amsart}
\usepackage{amsmath, amssymb, amsthm, latexsym}
\usepackage{amsmath}
\usepackage{amsthm}

\usepackage{epsfig}
\usepackage{amsfonts}
\usepackage{amssymb}
\usepackage{amsmath}   
\usepackage{amsthm}
\usepackage{latexsym}

\usepackage{fullpage}
\usepackage{amssymb}
\usepackage{latexsym}
\usepackage[center]{caption}
\usepackage{tikz}
\usetikzlibrary{matrix}
\usepackage{mathtools}
\usepackage{enumerate}
\usepackage{nameref}
\usepackage{varioref}
\usepackage{hyperref}
\usepackage{cleveref}
\usepackage{stmaryrd}
\usepackage{tikz-cd}

\newcounter{braid}
\newcounter{strands}
\pgfkeyssetvalue{/tikz/braid height}{1cm}
\pgfkeyssetvalue{/tikz/braid width}{1cm}
\pgfkeyssetvalue{/tikz/braid start}{(0,0)}
\pgfkeyssetvalue{/tikz/braid colour}{black}
\pgfkeys{/tikz/strands/.code={\setcounter{strands}{#1}}}
\makeatletter
\def\cross{%
  \@ifnextchar^{\message{Got sup}\cross@sup}{\cross@sub}}

\def\cross@sup^#1_#2{\render@cross{#2}{#1}}

\def\cross@sub_#1{\@ifnextchar^{\cross@@sub{#1}}{\render@cross{#1}{1}}}

\def\cross@@sub#1^#2{\render@cross{#1}{#2}}

\def\render@cross#1#2{
  \def\strand{#1}
  \def\crossing{#2}
  \pgfmathsetmacro{\cross@y}{-\value{braid}*\braid@h}
  \pgfmathtruncatemacro{\nextstrand}{#1+1}
  \foreach \thread in {1,...,\value{strands}}
  {
    \pgfmathsetmacro{\strand@x}{\thread * \braid@w}
    \ifnum\thread=\strand
    \pgfmathsetmacro{\over@x}{\strand * \braid@w + .5*(1 - \crossing) * \braid@w}
    \pgfmathsetmacro{\under@x}{\strand * \braid@w + .5*(1 + \crossing) * \braid@w}
    \draw[braid] \pgfkeysvalueof{/tikz/braid start} +(\under@x pt,\cross@y pt) to[out=-90,in=90] +(\over@x pt,\cross@y pt -\braid@h);
    \draw[braid] \pgfkeysvalueof{/tikz/braid start} +(\over@x pt,\cross@y pt) to[out=-90,in=90] +(\under@x pt,\cross@y pt -\braid@h);
    \else
    \ifnum\thread=\nextstrand
    \else
     \draw[braid] \pgfkeysvalueof{/tikz/braid start} ++(\strand@x pt,\cross@y pt) -- ++(0,-\braid@h);
    \fi
   \fi
  }
  \stepcounter{braid}
}

\tikzset{braid/.style={double=\pgfkeysvalueof{/tikz/braid colour},double distance=1pt,line width=2pt,white}}

\newcommand{\braid}[2][]{%
  \begingroup
  \pgfkeys{/tikz/strands=2}
  \tikzset{#1}
  \pgfkeysgetvalue{/tikz/braid width}{\braid@w}
  \pgfkeysgetvalue{/tikz/braid height}{\braid@h}
  \setcounter{braid}{0}
  \let\sigma=\cross
  #2
  \endgroup
}
\makeatother

\input xypic
\newtheorem{theorem}{Theorem}[section]
\newtheorem{proposition}[theorem]{Proposition}

\newtheorem{lemma}[theorem]{Lemma}

\newtheorem{corollary}[theorem]{Corollary}
\theoremstyle{definition}

\newtheorem{definition}[theorem]{Definition}

\newtheorem{assumption}[theorem]{Assumption}
\newtheorem{remark}[theorem]{Remark}
\newtheorem{example}[theorem]{Example}
\newtheorem{claim}[theorem]{Claim}

\def\Z{\mathbb{Z}}

\def\Pi{\mathbb{P}^{\infty}}

\def\Zpk{\mathbb{Z}/p^{k}}
\def\Zpk1{\mathbb{Z}/p^{k-1}}

\def\sl2{\widetilde{SL_{2}(\Z)}}

\DeclareMathOperator{\Gal}{Gal}

\DeclareMathOperator{\Spec}{Spec}
\DeclareMathOperator{\End}{End}

\DeclareMathOperator{\Sym}{Sym}
\DeclareMathOperator{\Fil}{Fil}

\DeclareMathOperator{\Hom}{Hom}

\DeclareFontFamily{U}{wncy}{}
    \DeclareFontShape{U}{wncy}{m}{n}{<->wncyr10}{}
    \DeclareSymbolFont{mcy}{U}{wncy}{m}{n}
    \DeclareMathSymbol{\Sh}{\mathord}{mcy}{"58} 
\author{Daniel J. Kriz}
\title{A New $p$-adic Maass-Shimura Operator and Supersingular Rankin-Selberg $p$-adic $L$-functions}

\begin{document}

\begin{abstract}We give a construction of a new $p$-adic Maass-Shimura operator defined on an affinoid subdomain of the preperfectoid $p$-adic universal cover $\mathcal{Y}$ of a modular curve $Y$. We define a new notion of $p$-adic modular forms as sections of a certain sheaf $\mathcal{O}_{\Delta}$ of ``nearly rigid functions" which transform under the action of subgroups of  the Galois group $\mathrm{Gal}(\mathcal{Y}/Y)$ by $\mathcal{O}_{\Delta}^{\times}$-valued weight characters. This extends Katz's notion of $p$-adic modular forms as functions on the Igusa tower $Y^{\mathrm{Ig}}$; indeed we may recover Katz's theory by restricting to a natural $\mathbb{Z}_p^{\times}$-covering $\mathcal{Y}^{\mathrm{Ig}}$ of $Y^{\mathrm{Ig}}$, viewing $\mathcal{Y}^{\mathrm{Ig}} \subset \mathcal{Y}$ as a sublocus. Our $p$-adic Maass-Shimura operator sends $p$-adic modular forms of weight $k$ to forms of weight $k + 2$. Its construction comes from a relative Hodge decomposition with coefficients in $\mathcal{O}_{\Delta}$ defined using Hodge-Tate and Hodge-de Rham periods arising from Scholze's Hodge-Tate period map and the relative $p$-adic de Rham comparison theorem. By studying the effect of powers of the $p$-adic Maass-Shimura operator on modular forms, we construct a continuous $p$-adic $L$-function which satisfies an ``approximate" interpolation property with respect to the the algebraic parts of central critical $L$-values of anticyclotomic Rankin-Selberg families on $GL_2 \times GL_1$ over imaginary quadratic fields $K/\mathbb{Q}$, including the ``supersingular" case where $p$ is not split in $K$. Finally we establish a new $p$-adic Waldspurger formula which, in the case of a newform, relates the formal logarithm of a Heegner point to a special value of the $p$-adic $L$-function.
\end{abstract}

\maketitle
\tableofcontents

\section{Introduction}

\subsection{Previous constructions and Katz's theory of $p$-adic modular forms on the ordinary locus}

Let us start by giving a brief account of Katz and Bertolini-Darmon-Prasanna/Liu-Zhang-Zhang's construction of $p$-adic $L$-functions over imaginary quadratic fields $K$ in which $p$ splits in $K$.  The splitting assumption of Katz allows one to make use of his theory of $p$-adic modular forms in order to construct his and Bertolini-Darmon-Prasanna/Liu-Zhang-Zhang's $p$-adic $L$-functions, now colloquially known as the \emph{Katz} and \emph{BDP/LZZ $p$-adic $L$-functions}, respectively. Namely, the $p$-adic $L$-functions over $K$ which Katz, Bertolini-Darmon-Prasanna and Liu-Zhang-Zhang, construct is a linear functional on the space of ($p$-adic) modular forms, which is obtained by evaluating $p$-adic differential operators applied to modular forms at ordinary CM points associated with $K$. This means the CM points belong to the \emph{ordinary locus} $Y^{\mathrm{ord}} \subset Y$, which is the affinoid subdomain of (the rigid analytification of) $Y$ obtained by removing all points which reduce to supersingular points on the special fiber (this latter locus being isomorphic to a finite union of rigid analytic open unit discs). Here, the ordinariness assumption is crucial in order to establish nice analytic properties of the $p$-adic $L$-function, namely that ($p$-adic) modular forms have local coordinates in neighborhoods of CM points with respect to which the differential operators alluded to above have a nice, clearly analytic expression. In Katz's setting, one views $p$-adic modular forms as functions on a pro\'{e}tale cover $Y^{\mathrm{Ig}} \rightarrow Y^{\mathrm{ord}}$ using an explicit trivialization of the Hodge bundle. Then on $Y^{\mathrm{Ig}}$, he defines a differential operator $\theta$ called the \emph{Atkin-Serre operator}, which sends $p$-adic modular forms of weight $k$ to forms of weight $k+2$, and the nice coordinates are provided by Serre-Tate coordinates. One can express 
$$\theta = (1+T)\frac{d}{dT}$$
in terms of the Serre-Tate coordinate $T$, and using this expression one can show easily that the family
$$\{\theta^j f\}_{j \in \mathbb{Z}_{\ge 0}}$$
for a given $p$-adic modular form $f$ gives rise to a ``nearly-analytic" function of $j$: after applying a certain Hecke operator known as \emph{$p$-stabilization} to $f$ (which corresponds to removing an Euler factor in the $p$-adic $L$-function), one can show that
$$\theta^jf^{(p)},$$
where $f^{(p)}$ denotes the $p$-stabilization, is an analytic function (valued in the space of $p$-adic modular forms) of $j \in \mathbb{Z}_p^{\times}$. 

One could also use coordinates provided by $q$-expansions, if one compactifies all modular curves under our consideration; we stick to the open modular curve in this article in order to avoid boundary issues occurring at cusps, which present bigger technical issues when defining the pro\'{e}tale topology later). 

The key property of $Y^{\mathrm{Ig}} \rightarrow Y^{\mathrm{ord}}$ which allows one to construct the differential operator $\theta$ is the existence of the \emph{unit root splitting} of the Hodge filtration on $Y^{\mathrm{Ig}}$. Namely, one can find sections of the relative de Rham cohomology $\mathcal{H}_{\mathrm{dR}}^1(\mathcal{A})|_{Y^{\mathrm{Ig}}}$ which are \emph{horizontal} with respect to the algebraic Gauss-Manin connection 
$$\nabla : \mathcal{H}_{\mathrm{dR}}^1(\mathcal{A}) \rightarrow \mathcal{H}_{\mathrm{dR}}^1(\mathcal{A}) \otimes_{\mathcal{O}_Y} \Omega_Y^1$$
(here a section $\alpha$ being horizontal means that $\nabla(\alpha) = 0$), and which are also eigenvectors for the canonical (Frobenius-linear) Frobenius endomorphism 
$$F : Y^{\mathrm{Ig}} \rightarrow Y^{\mathrm{Ig}}$$
over $W$. (The reason for the terminology ``unit root" is because one of the eigenvalues for the $F$ is a $p$-adic unit, since we restrict to a covering of the ordinary locus $Y^{\mathrm{ord}}$.) The unit root splitting is a functorial, $F$-equivariant splitting of the Hodge filtration, which allows one to then define the differential operator $\theta$. This uses a standard formalism of Katz which produces such a differential (weight-raising) operator, whenever a splitting of the Hodge filtration with nice properties (e.g. $\mathrm{Gal}(Y^{\mathrm{Ig}}/Y^{\mathrm{ord}})$-equivariance) exists.

Another key property of the unit root splitting is that for CM elliptic curves $A$, which by the theory of complex multiplication always have models over $\overline{\mathbb{Q}}$, it is induced by the splitting of $H_{\mathrm{dR}}^1(A)$ defined over $\overline{\mathbb{Q}}$ given by the eigendecomposition under the CM action. This CM splitting over $\overline{\mathbb{Q}}$ also gives rise to the real-analytic Hodge decomposition over $\mathbb{C}$ from classical Hodge theory, which in that setting gives rise to the real-analytic Maass-Shimura operator $\frak{d}$ sending nearly holomorphic modular forms of weight $k$ to nearly holomorphic forms of weight $k+2$. The consequence is that after normalizing by appropriate ``canonical" periods $\Omega_p$ and $\Omega_{\infty}$, one can show that given an algebraic modular form $w$ of weight $k$, the values of $\theta^jw(y)/\Omega_p^{k+2j}$ and $\frak{d}^jw(y)/\Omega_{\infty}^{k+2j}$ at CM points $y$ belong to $\overline{\mathbb{Q}}$ and \emph{coincide}. This observation of Katz is essential to establishing interpolation properties of the Katz and BDP/LZZ $p$-adic $L$-functions, i.e. to relate them to critical values of complex $L$-functions for fibers $\mathbb{V}_x$ where $x \in \Sigma$ are algebraic points in the interpolation (Panchishkin) range. This is because such critical $L$-values can be expressed as period integrals over the CM torus (or finite sums over orbits of CM points) of $\frak{d}^jw$, and hence by the above discussion these can be related to such $p$-adic period sums of $\theta^jw$ over CM points, which themselves give rise to the Katz and BDP/LZZ $p$-adic $L$-functions. 

Let us elaborate on Serre-Tate coordinates and Katz's notion of $p$-adic modular forms, and expound on the above discussion. To fix ideas, suppose that a modular curve $Y$ represents a fine moduli space (for example, if its topological fundamental group as an analytic space over $\mathbb{C}$ is \emph{neat} in the sense that it has no torsion), and so it admits a universal object
$$\pi : \mathcal{A} \rightarrow Y.$$
The \emph{Hodge bundle} is then defined as 
$$\omega := \pi_*\Omega_{\mathcal{A}/Y}^1$$
and weight $k$ modular forms can be identified with sections of $\omega^{\otimes k}$. Katz's theory of modular forms comes from viewing modular forms as functions on a certain covering of a $p$-adic rigid analytic subdomain of $Y$, by trivializing $\omega$ on that cover. Specifically, the rigid analytic subdomain is the \emph{ordinary locus $Y^{\mathrm{ord}}$ of the modular curve $Y$}, a $p$-adic rigid space obtained from $Y$ by removing all points which under the moduli interpretation of $Y$ correspond to supersingular elliptic curves, and the cover on which we trivialize $\omega$ is the \emph{Igusa tower} $Y^{\mathrm{Ig}} \rightarrow Y^{\mathrm{ord}}$, which is a $\mathbb{Z}_p^{\times}$-cover or in other words,
$$\mathrm{Gal}(Y^{\mathrm{Ig}}/Y^{\mathrm{ord}}) \cong \mathbb{Z}_p^{\times}.$$

To obtain the trivialization of $\omega$, Katz uses the simple structure of the $p$-divisible groups of ordinary elliptic curves, namely that they are isomorphic over $W = W(\overline{\mathbb{F}}_p)$ to 
$$\mu_{p^{\infty}} \times \mathbb{Q}_p/\mathbb{Z}_p.$$ 
By the Weil pairing (or Cartier duality), such a trivialization for a given $p$-divisible group $A[p^{\infty}]$ of an ordinary elliptic curve $A$ is determined by fixing an isomorphism 
$$A[p^{\infty}]^{\text{\'{e}t}} \cong \mathbb{Q}_p/\mathbb{Z}_p.$$
In fact, $Y^{\mathrm{ord}}$ is exactly the cover of $Y^{\mathrm{ord}}$ defined over $W$ parametrizing such trivializations 
$$\alpha : \mathbb{Q}_p/\mathbb{Z}_p \xrightarrow{\sim} A[p^{\infty}]^{\text{\'{e}t}},$$
or equivalently (by the previous discussion), trivializations
$$\alpha : \mu_{p^{\infty}} \times \mathbb{Q}_p/\mathbb{Z}_p \xrightarrow{\sim} A[p^{\infty}]$$
of the entire $p$-divisible group. 

Let $A_0/\overline{\mathbb{F}}_p$ be an elliptic curve corresponding to a closed geometric point $y_0$ on the special fiber $Y_0^{\mathrm{ord}} = Y^{\mathrm{ord}}\otimes_W \overline{\mathbb{F}}_p$, and let $A/W$ denote any lift of it (i.e. with $A\otimes_W \overline{\mathbb{F}}_p \cong A_0$), corresponding to a point $y$ on $Y^{\mathrm{ord}}$. Formally completing $Y^{\mathrm{Ig}}$ along $y_0$ hence gives the formal moduli space $\hat{D}(y_0)$ of deformations of $A_0$ (with some level structure, which we will suppress for brevity). Since there is a canonical isomorphism 
$$A[p^{\infty}]^{\text{\'{e}t}} \cong A_0[p^{\infty}](\overline{\mathbb{F}}_p),$$
then a choice of trivialization
$$\alpha_0 : \mathbb{Q}_p/\mathbb{Z}_p \xrightarrow{\sim} A_0[p^{\infty}](\overline{\mathbb{F}}_p)$$
fixes $A[p^{\infty}]^{\text{\'{e}t}}$ in the formal neighborhood $\tilde{D}(\tilde{y}_0)$ of $\tilde{y}_0 = (A_0,\alpha_0)$ in $Y^{\mathrm{Ig}}$. Hence $\tilde{D}(\tilde{y}_0)$ is parametrized exactly by the connected component $A[p^{\infty}]^0$ of $A[p^{\infty}]$, and so there is an (in fact, canonical) isomorphism
$$\tilde{D}(\tilde{y}_0) \cong \hat{\mathbb{G}}_m.$$
The canonical coordinate $T$ on the torus gives rise to the \emph{Serre-Tate coordinate}, also denoted by $T$, on $\tilde{D}(\tilde{y}_0)$, and on the associated residue disc $\tilde{D}(\tilde{y}_0)\otimes_W W[1/p]$ (viewed as the rigid analytic generic fiber).

Katz also uses the description of formal neighborhoods on $Y^{\mathrm{Ig}}$ around closed points of the special fiber as being canonically isomorphic to $\hat{\mathbb{G}}_m$ to construct a global section of the Hodge bundle on the Igusa tower, which is often called the \emph{canonical differential}; in terms of the Serre-Tate coordinate, the canonical differential is just given by $dT/T$. Using tensorial powers of the canonical differential, modular forms, viewed as sections of powers $\omega^{\otimes k}$ of the Hodge-bundle $\omega$ restricted to $Y^{\mathrm{ord}}$, can be identified as functions on $Y^{\mathrm{Ig}}$. Since the canonical differential transforms by 
$$a^*\omega_{\mathrm{can}}^{\mathrm{Katz}} = a\omega_{\mathrm{can}}^{\mathrm{Katz}}$$
for $a \in \mathbb{Z}_p^{\times} \cong \mathrm{Gal}(Y^{\mathrm{Ig}}/Y^{\mathrm{ord}})$, the we can even identify a modular form of weight $k$, i.e. a section of $w\in \omega^{\otimes k}(Y^{\mathrm{ord}})$, as a function $f$ of \emph{weight $k$} on $Y^{\mathrm{Ig}}$, via 
$$w|_{Y^{\mathrm{Ig}}} = f \cdot \omega_{\mathrm{can}}^{\mathrm{Katz},\otimes k},$$
where weight $k \in \mathbb{Z}$ means that $f$ transforms as 
\begin{equation}\label{eq:Katzweightk}a^*f = a^{-k}f
\end{equation}
for $a \in \mathbb{Z}_p^{\times} = \Gal(Y^{\mathrm{Ig}}/Y^{\mathrm{ord}})$. Katz also uses this viewpoint to generalize modular forms to \emph{$p$-adic modular forms of weight $k \in \mathbb{Z}_p^{\times}$}, which are functions on $Y^{\mathrm{Ig}}$ which have weight $k \in \mathbb{Z}_p^{\times}$ in the same way as defined above. 

\subsection{Outline of our theory of $p$-adic analysis on the supersingular locus and construction of $p$-adic $L$-functions}

The key question addressed by this article is that of developing a satisfactory theory of $p$-adic analysis of modular forms on the supersingular locus of modular curves, and subsequently to construct ``supersingular" $p$-adic $L$-functions for Rankin-Selberg families $\mathbb{V}$ associated to automorphic representations $(\pi_w)_K \times \chi^{-1}$ for anticyclotomic characters $\chi$ over an imaginary quadratic field $K/\mathbb{Q}$, where $\pi_w$ is the automorphic representation of $GL_2(\mathbb{A}_{\mathbb{Q}})$ attached to a normalized new eigenform $w$ (i.e. a newform or Eisenstein series), $(\pi_w)_K$ denotes its base change to an automorphic representation of $GL_2(\mathbb{A}_K)$, and $\chi$ varies through a family of anticyclotomic Hecke characters over $K$. Here, ``supersingular" means that we assume that $p$ is inert or ramified in $K$. This is analogous, outside the splitting assumption on $p$, to the ``ordinary" setting in which Katz and Bertolini-Darmon-Prasanna/Liu-Zhang-Zhang construct their one-variable $p$-adic $L$-functions. In fact our theory addresses the ordinary and supersingular settings uniformly by working on an affinoid subdomain $\mathcal{Y}_x \subset \mathcal{Y}$ of the $p$-adic universal cover $\mathcal{Y}$ (defined below), where $\mathcal{Y}_x$ contains a natural cover $\mathcal{Y}^{\mathrm{Ig}}$ of $Y^{\mathrm{Ig}}$, and so we can recover the aforementioned one-variable $p$-adic $L$-functions in the ordinary case. 

One motivation for the construction of supersingular Rankin-Selberg $p$-adic $L$-functions is to develop special value formulas in the same style as those of Katz and BDP, where in the former case a special value of the Katz $p$-adic $L$-function is related to the $p$-adic logarithm of elliptic units attached to $K$, and in the latter case the special value formula is a ``$p$-adic Waldspurger formula" involving the $p$-adic formal logarithm of a Heegner point attached to $K$ (when a Heegner hypothesis holds for $K$ and level $N$ of $w$). Indeed, we succeed in proving such a formula in the case $p\nmid N$ in Section \ref{Waldspurgersection} and either $p > 2$ or $p$ is not ramified in $K$, though in future work we expect to remove both $p\nmid N$ and extend to $p = 2$ and ramified in $K$ (as well as relax the Heegner hypothesis on $N$, which would necessitate considering more general quaternionic Shimura curves than modular curves). Such formulas apply in cases not accessible by the previous formulas of Katz and BDP/LZZ, and upon removing the $p\nmid N$ assumption would apply for example to CM elliptic curves over $\mathbb{Q}$ where one considers $p$ as ramified in the CM field $K$ and one studies Heegner points (or mock Heegner points) attached to $K$. 

To develop a satisfactory theory of $p$-adic analysis on the supersingular locus, namely a notion of $p$-adic modular forms on the supsersingular locus $Y^{\mathrm{ss}}$ which behaves well with respect to some differential operator $d$ (say, there is a notion of weight which is raised by $d$, and given a $p$-adic modular form $f$, $d^jf$ or some stabilization $d^jf$ gives rise to some $p$-adic analytically well-behaved family), there are several technical difficulties which must be overcome. One of which is that there is no obvious canonical differential with which to trivialize $\omega$ over a cover in order to view modular forms as function on the cover (in the same way as $\omega_{\mathrm{can}}^{\mathrm{Katz}}$ does so for $\omega$ on $Y^{\mathrm{Ig}}\rightarrow Y^{\mathrm{ord}}$). It is also a difficulty that there is no ``canonical line" in the $p$-divisible group of a supersingular curve as there is for $\mu_{p^{\infty}} \subset A[p^{\infty}]$ when $A$ is ordinary. Hence there is no natural splitting of the Hodge filtration with which to define a differential operator $d$ analogous to the Atkin-Serre operator in the ordinary setting, and even if one were to construct such an operator, there is no obvious analogue of the Serre-Tate coordinate under which to locally express $p$-adic modular forms $f$ and study the analytic properties of $d^jf$. 

To overcome these difficulties, we use two key ingredients which in some sense generalize the strategy of Katz. The lack of unit root splitting, whose construction comes from a horizontal basis for the Gauss-Manin connection defined as sections of the relative \'{e}tale cohomology $\mathcal{H}_{\text{\'{e}t}}^1(\mathcal{A})$ over $Y^{\mathrm{Ig}}$ which are eigenvectors of the canonical Frobenius. This unit root splitting gives a splitting of the Hodge filtratration
$$0 \rightarrow \omega|_{Y^{\mathrm{ord}}} \rightarrow \mathcal{H}_{\mathrm{dR}}^1(\mathcal{A})|_{Y^{\mathrm{ord}}} \rightarrow \omega^{-1}|_{Y^{\mathrm{ord}}} \rightarrow 0$$
as an exact sequence of $\mathcal{O}_{Y^{\mathrm{ord}}}$-modules, where $\mathcal{O}_Y$ denote the rigid-analytic structure sheaf on $Y$. Note that unlike in the complex-analytic setting, we do not have to extend the sheaf of rigid functions (the analogue of holomorphic functions) to a large sheaf (of ``real analytic functions") in order to obtain the Hodge decomposition, as long as we restrict to $Y^{\mathrm{ord}} \subset Y$. 

As no unit root basis of exists outside of $Y^{\mathrm{ord}}$, we instead consider the moduli space of all horizontal bases of \'{e}tale cohomology. This moduli space is representable by the \emph{$p$-adic universal cover $\mathcal{Y}$} (which we define more explicitly in the next paragraph), and with universal object being given by $(\mathcal{A},\alpha_{\infty})$ where $\alpha_{\infty}$ is the universal full $p^{\infty}$-level structure. We then use a relative $p$-adic de Rham comparison theorem to view $\alpha_{\infty}$ as a universal horizontal basis for relative de Rham cohomology; unlike in the ordinary case, this comparison involves extending the structure sheaf to a certain period sheaf $\mathcal{O}\mathbb{B}_{\mathrm{dR},Y}^+$ (where this is reall a sheaf on the pro\'{e}tale site $Y_{\text{pro\'{e}t}}$) first constructed by Scholze in \cite{Scholze2}. From this horizontal ``framing" of the relative de-Rham cohomology $\mathcal{H}_{\mathrm{dR}}^1(\mathcal{A})$, we get a ``Hodge-de Rham period" measuring the position of the Hodge filtration and the ``Hodge-Tate period" measuring the position of the Hodge-Tate filtration, as considered by Scholze in loc. cit., and use these periods to construct a relative Hodge decomposition which we use as a substitute for the unit root splitting. This splitting is in fact defined over an ``intermediate period sheaf"
$$\mathcal{O}_{\Delta,\mathcal{Y}} := \mathcal{O}\mathbb{B}_{\mathrm{dR},\mathcal{Y}}^+/(t),$$ 
equipped with natural connection
$$\nabla : \mathcal{O}_{\Delta,\mathcal{Y}}  \rightarrow \mathcal{O}_{\Delta,\mathcal{Y}} \otimes_{\mathcal{O}_{\mathcal{Y}}} \Omega_{\mathcal{Y}}^1$$
which is $\mathbb{B}_{\mathrm{dR},\mathcal{Y}}^+/(t)$-linear, induced by the natural connection
$$\nabla : \mathcal{O}\mathbb{B}_{\mathrm{dR},\mathcal{Y}}^+ \rightarrow \mathcal{O}\mathbb{B}_{\mathrm{dR},\mathcal{Y}}^+ \otimes_{\mathcal{O}_{\mathcal{Y}}} \Omega_{\mathcal{Y}}^1$$
which is $\mathbb{B}_{\mathrm{dR},\mathcal{Y}}^+$-linear. Moreover, there is a natural map
$$\mathcal{O}_{\mathcal{Y}} \subset \mathcal{O}\mathbb{B}_{\mathrm{dR},\mathcal{Y}}^+ \overset{\mod t}{\twoheadrightarrow} \mathcal{O}_{\Delta,\mathcal{Y}}$$
which is in fact an inclusion compatible with connections, and such that its composition with the natural map 
$$\theta : \mathcal{O}_{\Delta,\mathcal{Y}} \twoheadrightarrow \hat{\mathcal{O}}_{\mathcal{Y}}$$
where $\hat{\mathcal{O}}_{\mathcal{Y}}$ is the $p$-adically completed structure sheaf on $\mathcal{Y}$ is the natural map
$$\mathcal{O}_{\mathcal{Y}} \rightarrow \hat{\mathcal{O}}_{\mathcal{Y}}.$$
Here $\theta$ is induced by the natural relative analogue 
$$\theta : \mathcal{O}\mathbb{B}_{\mathrm{dR},\mathcal{Y}}^+ \twoheadrightarrow \hat{\mathcal{O}}_{\mathcal{Y}}$$ of Fontaine's map $\theta : B_{\mathrm{dR}}^+ \twoheadrightarrow \mathbb{C}_p$. Here, $t$ is a global analogue of Fontaine's ``$2\pi i$" and is a global section of a period sheaf $\mathbb{B}_{\mathrm{dR},\mathcal{Y}}^+$ on $\mathcal{Y}$, which is itself a relative version of Fontaine's ring of periods $B_{\mathrm{dR}}^+$. We call $\mathcal{O}_{\Delta,\mathcal{Y}}$ ``intermediate" because it, in the sense above, lies in between $\mathcal{O}_{\mathcal{Y}}$ and $\mathcal{O}\mathbb{B}_{\mathrm{dR},\mathcal{Y}}^+$. In analogy with having to extend from holomorphic to real analytic functions on the complex universal cover $\mathcal{H}$ in order to define the complex-analytic Hodge decomposition, we view $\mathcal{O}_{\Delta,\mathcal{Y}}$ as a sheaf of ``$p$-adic nearly holomorphic (or rigid) functions on the $p$-adic universal cover $\mathcal{Y}$". 

The aforementioned moduli space parametrizing elliptic curves with full $p^{\infty}$-level structure is represented by a $GL_2(\mathbb{Z}_p)$-profinite-\'{e}tale cover $\mathcal{Y}$ of $Y$ (viewing the latter as an adic space over $\mathrm{Spa}(\mathbb{Q}_p,\mathbb{Z}_p)$) called the (preperfectoid) $p$-adic universal cover 
$$\mathcal{Y} = \varprojlim_i Y(p^i),$$
as considered by Scholze in \cite{Scholze2} and Scholze-Weinstein in \cite{ScholzeWeinstein}. Here $Y(p^i)$ is the modular curve obtained by adding full $p^i$-level structure to the moduli space represented by $Y$, and $\mathcal{Y}$ is an adic space over $\mathrm{Spa}(\mathbb{Q}_p,\mathbb{Z}_p)$ which is an object in the pro\'{e}tale site $Y_{\text{pro\'{e}t}}$. Here, the full universal $p^{\infty}$-level structure $\alpha_{\infty}$ is just a trivialization of the Tate module of $\mathcal{A}$
$$\alpha_{\infty} : \hat{\mathbb{Z}}_{p,\mathcal{Y}}^{\oplus 2} \xrightarrow{\sim} T_p\mathcal{A}|_{\mathcal{Y}},$$
here $\hat{\mathbb{Z}}_{p,\mathcal{Y}}$ is the ``constant sheaf" on $\mathcal{Y}$ associated with $\mathbb{Z}_p$, except that sections are continuous functions into $\mathbb{Z}_p$ where the latter has the $p$-adic (and not discrete) topology. Now let $\mathcal{O}_Y$ denote the pro\'{e}tale structure sheaf on $Y_{\text{pro\'{e}t}}$. Using the Hodge-de Rham comparison theorem of Scholze (\cite{Scholze}), we then have a natural inclusion
$$\mathcal{H}_{\mathrm{dR}}^1(\mathcal{A})\otimes_{\mathcal{O}_Y} \mathcal{O}\mathbb{B}_{\mathrm{dR},Y}^+ \overset{\iota_{\mathrm{dR}}}{\subset} \mathcal{H}_{\text{\'{e}t}}^1(\mathcal{A})\otimes_{\hat{\mathbb{Z}}_{p,Y}} \mathcal{O}\mathbb{B}_{\mathrm{dR},Y}^+$$
on $Y_{\text{pro\'{e}t}}$ compatible with filtrations (on the left, the convolution of the Hodge filtration on $\mathcal{H}_{\mathrm{dR}}^1(\mathcal{A})$ and the natural filtration on $\mathcal{O}\mathbb{B}_{\mathrm{dR},Y}^+$, and on the right is just the filtration on $\mathcal{O}\mathbb{B}_{\mathrm{dR},Y}^+$) and connections (on the left, the convolution of the Gauss-Manin connection on $\mathcal{H}_{\mathrm{dR}}^1(\mathcal{A})$ and the natural connection on $\mathcal{O}\mathbb{B}_{\mathrm{dR},Y}^+$ via the Leibniz rule, and on the right is just the connection on $\mathcal{O}\mathbb{B}_{\mathrm{dR},Y}^+$). Pulling back to $\mathcal{Y}$, we then have 
\begin{equation}\label{eq:comparison}\mathcal{H}_{\mathrm{dR}}^1(\mathcal{A})\otimes_{\mathcal{O}_Y} \mathcal{O}\mathbb{B}_{\mathrm{dR},Y}^+|_{\mathcal{Y}} \overset{\iota_{\mathrm{dR}}}{\hookrightarrow} \mathcal{H}_{\text{\'{e}t}}^1(\mathcal{A})\otimes_{\hat{\mathbb{Z}}_{p,Y}} \mathcal{O}\mathbb{B}_{\mathrm{dR},Y}^+|_{\mathcal{Y}} \underset{\sim}{\xrightarrow{\alpha_{\infty}^{-1}}}(\mathcal{O}\mathbb{B}_{\mathrm{dR},\mathcal{Y}}^+\cdot t^{-1})^{\oplus 2}
\end{equation}
where the last isomorphism uses the universal $p^{\infty}$-level structure $\alpha_{\infty}$ and the isomorphism 
$$\mathcal{H}_{\text{\'{e}t}}^1(\mathcal{A}) \cong T_pA(-1)$$
given by the Weil pairing. We also use the fact that there is a natural isomorphism
$$\hat{\mathbb{Z}}_{p,\mathcal{Y}}(-1)  = \hat{\mathbb{Z}}_{p,\mathcal{Y}}\cdot t^{-1},$$
as $t$ is a period for the cyclotomic character. 

We note that there is a natural sublocus $\mathcal{Y}^{\mathrm{Ig}} \subset \mathcal{Y}$ which parametrizes unit root splittings on ordinary elliptic curves, viewed as arithmetic $p^{\infty}$ level structures
$$\alpha : \mathbb{Z}_p(1) \oplus \mathbb{Z}_p \xrightarrow{\sim} T_pA,$$
\emph{together} with a trivialization
$$\mathbb{Z}_p(1) \cong \mathbb{Z}_p,$$
and it is clear that these two data are equivalent to a full $p^{\infty}$-level structure
$$\alpha : \mathbb{Z}_p \oplus \mathbb{Z}_p \xrightarrow{\sim} T_pA.$$
Thus, it is a natural $\mathbb{Z}_p^{\times}$-cover of the Igusa tower $Y^{\mathrm{Ig}}$, and a $\mathbb{Z}_p^{\times} \times \mathbb{Z}_p^{\times}$-cover of $Y^{\mathrm{ord}}$. In fact, one can show that our theory restricts to Katz's on $\mathcal{Y}^{\mathrm{Ig}}$, that $q_{\mathrm{dR}}$-expansions recover Serre-Tate expansions, and that the restriction of our canonical differential $\omega_{\mathrm{can}}|_{\mathcal{Y}^{\mathrm{Ig}}}$ coincides with $\omega_{\mathrm{can}}^{\mathrm{Katz}}$. 

Using (\ref{eq:comparison}), one sees that $\alpha_{\infty,1} := \alpha_{\infty}|_{\hat{\mathbb{Z}}_{p,\mathcal{Y}}\oplus \{0\}}$ and $\alpha_{\infty,2} := \alpha_{\infty}|_{\{0\}\oplus \hat{\mathbb{Z}}_{p,\mathcal{Y}}}$ are horizontal sections for the connection. Moreover, upon making the identification (via the Weil pairing)
$$T_p\mathcal{A}\cong \Hom(\mathcal{A}[p^{\infty}],\mu_{p^{\infty}}),$$
we get a natural map
$$HT_{\mathcal{A}} : T_p\mathcal{A}\otimes_{\hat{\mathbb{Z}}_{p,Y}} \mathcal{O}_Y \rightarrow \omega, \hspace{1cm} \alpha \mapsto \alpha^*\frac{dT}{T}$$
where $dT/T$ is the canonical invariant differential on $\mu_{p^{\infty}}$. (It is also sometimes customary to denote $HT_{\mathcal{A}} = d\log$, as $d\log T = dT/T$.) We then define the \emph{fake Hasse invariant} as
$$\frak{s} := HT_{\mathcal{A}}(\alpha_{\infty,2}),$$
and in fact we have that on the restriction to $\mathcal{Y}^{\mathrm{Ig}}$,
$$\frak{s}|_{\mathcal{Y}^{\mathrm{Ig}}} = \omega_{\mathrm{can}}^{\mathrm{Katz}}|_{\mathcal{Y}^{\mathrm{Ig}}}.$$
Consider the affinoid subdomain
$$\mathcal{Y}_x = \{\frak{s} \neq 0\} \subset \mathcal{Y}.$$
We note that $\frak{s} \in \omega(\mathcal{Y}_x)$ is a non-vanishing global section, i.e. a generator. Then let $\frak{s}^{-1} \in \omega^{-1}(\mathcal{Y}_x)$ the generator corresponding to $\frak{s}$ under Poincar\'{e} duality. The trivialization 
$$\omega|_{\mathcal{Y}^{\mathrm{Ig}}} \cong \mathcal{O}_{\mathcal{Y}^{\mathrm{Ig}}}$$
induced by $\frak{s}|_\mathcal{Y}^{\mathrm{Ig}} = \omega_{\mathrm{can}}^{\mathrm{Katz}}$, along with the universal unit root splitting on $\mathcal{Y}^{\mathrm{Ig}}$ given by $\alpha_{\infty}|_{\mathcal{Y}^{\mathrm{Ig}}}$, gives rise to a $p$-adic differential operator (the Atkin-Serre operator) $\theta_{\mathrm{AS}} : \mathcal{O}_{\mathcal{Y}^{\mathrm{Ig}}} \rightarrow \mathcal{O}_{\mathcal{Y}^{\mathrm{Ig}}}$ with nice $p$-adic analytic properties, seen using Serre-Tate coordinates. The key to these nice $p$-adic properties is the identity
$$\sigma(\omega_{\mathrm{can}}^{\mathrm{Katz},\otimes 2}) = d\log T$$
where $T$ is the Serre-Tate coordinate, and 
\begin{equation}\label{eq:KS1}\sigma : \omega^{\otimes 2} \xrightarrow{\sim} \Omega_Y^1
\end{equation}
is the Kodaira-Spencer isomorphism. 

By the above discussion, $\frak{s}$ seems like a natural candidate to extend Katz's idea of viewing $p$-adic modular forms (sections of $\omega$) as functions to the (non-Galois) covering $\mathcal{Y}_x \rightarrow Y$. However, unlike in Katz's situation on $\mathcal{Y}^{\mathrm{Ig}}$, in our situation the splitting of (a lift of) the Hodge filtration which we define and use will require extending coefficients from $\mathcal{O}_{\mathcal{Y}_x}$ to a larger sheaf $\mathcal{O}_{\Delta,\mathcal{Y}_x}$ (which can be viewed as ``the sheaf of nearly rigid functions", in analogy to extending the sheaf of holomorphic functions in order to define the Hodge decomposition in the complex analytic situation), and with respect to this splitting $\frak{s}$ will \emph{not} be the most convenient choice for trivializing $\omega \otimes_{\mathcal{O}_Y}\mathcal{O}_{\Delta,\mathcal{Y}_x}$. Instead, we will trivialize using the generator
$$\omega_{\mathrm{can}} := \frac{\frak{s}}{y_{\mathrm{dR}}} \in (\omega \otimes_{\mathcal{O}_Y} \mathcal{O}_{\Delta,\mathcal{Y}_x})(\mathcal{Y}_x)$$
where $y_{\mathrm{dR}} \in \mathcal{O}_{\Delta,\mathcal{Y}_x}(\mathcal{Y}_x)^{\times}$ is a certain $p$-adic period associated with $\frak{s}$. Hence this induces a trivialization
$$\omega\otimes_{\mathcal{O}_Y}\mathcal{O}_{\Delta,\mathcal{Y}_x} \cong \mathcal{O}_{\Delta,\mathcal{Y}_x}.$$
One can show that $y_{\mathrm{dR}} = 1$ on the sublocus $\mathcal{Y}^{\mathrm{Ig}} \subset \mathcal{Y}$, and so we have
\begin{equation}\label{eq:Katzcoincide}\omega_{\mathrm{can}}|_{\mathcal{Y}^{\mathrm{Ig}}} = \omega_{\mathrm{can}}^{\mathrm{Katz}}|_{\mathcal{Y}^{\mathrm{Ig}}}.
\end{equation}
In analogy with (\ref{eq:KS1}), we also have 
\begin{equation}\label{eq:KS2}\sigma(\omega_{\mathrm{can}}^{\otimes, 2}) = dz_{\mathrm{dR}}
\end{equation}
where $z_{\mathrm{dR}} = \mathbf{z}_{\mathrm{dR}} \pmod{t}$ for the de Rham fundamental period $\mathbf{z}_{\mathrm{dR}}$, which we describe in more detail below. We note that the analogy between (\ref{eq:KS1}) and (\ref{eq:KS2}), along with (\ref{eq:Katzcoincide}) suggests that $z_{\mathrm{dR}}$ provides the correct analogue of $\log T$. It is this observation which later leads to our notion of the $q_{\mathrm{dR}} = \exp(z_{\mathrm{dR}}-\theta(z_{\mathrm{dR}}))$-coordinate as an analogue (and extension) for the Serre-Tate coordinate $T$, and $q_{\mathrm{dR}}$-expansions as analogues (and extensions) of Serre-Tate $T$-expansions.

We can also use $\omega_{\mathrm{can}}$ to generalize Katz's notion of $p$-adic modular forms. Let $\mathcal{U} \subset \mathcal{Y}_x$ be a subadic space, let $\lambda : \mathcal{Y} \rightarrow Y$ denote the natural projection, and let $\lambda(\mathcal{U}) = U$. Then letting 
$$\Gamma = \mathrm{Gal}(\mathcal{U}/U) \subset \mathrm{Gal}(\mathcal{Y}/Y) = GL_2(\mathbb{Z}_p),$$
we have a natural map
$$\omega^{\otimes k}|_U(U) \overset{\lambda^*}{\hookrightarrow} \omega^{\otimes k}|_{\mathcal{U}}(\mathcal{U}) \hookrightarrow (\omega \otimes_{\mathcal{O}_Y} \mathcal{O}_{\Delta,\mathcal{Y}_x}|_{\mathcal{U}})^{\otimes k}(\mathcal{U}) \underset{\sim}{\xrightarrow{\omega_{\mathrm{can}}^{\otimes k}}} \mathcal{O}_{\Delta,\mathcal{Y}_x}|_{\mathcal{U}}(\mathcal{U}).$$
In fact, the image under this map consists of sections $f \in \mathcal{O}_{\Delta,\mathcal{Y}_x}|_{\mathcal{U}}(\mathcal{U})$ such that 
\begin{equation}\label{eq:transformation}\left(\begin{array}{ccc} a & b \\
c & d\\
\end{array}\right)^*f = (bc-ad)^{-k}(cz_{\mathrm{dR}}+a)^kf
\end{equation}
for any $\left(\begin{array}{ccc} a & b \\
c & d\\
\end{array}\right) \in \Gamma.$ In this situation, we say that $f$ has \emph{weight $k$ for $\Gamma$ on $\mathcal{U}$}. 
We note that when $\mathcal{U} = \mathcal{Y}^{\mathrm{Ig}}$ and so $U = Y^{\mathrm{ord}}$, and 
$$\Gamma = (\mathbb{Z}_p^{\times})^{\oplus 2} \subset GL_2(\mathbb{Z}_p)$$ 
the diagonal subgroup, and then (\ref{eq:transformation}) becomes
\begin{equation}\label{eq:transformationKatz}\left(\begin{array}{ccc} a & 0 \\
0 & d\\
\end{array}\right)^*f = (-d)^{-k}f.
\end{equation}
In particular, $f$ descends to a section in $\mathcal{O}_Y(Y^{\mathrm{Ig}})$ and we recover Katz's notion (\ref{eq:Katzweightk}) of a $p$-adic modular form of weight $k$. Our main interest, which is defining a satisfactory notion of $p$-adic modular form on the supersingular locus, will involve the case $\mathcal{U} = \mathcal{Y}^{\mathrm{ss}}, U = Y^{\mathrm{ss}}$ and $\Gamma = GL_2(\mathbb{Z}_p)$. 

Let us now elaborate on the construction of our splitting of (a lift of) the Hodge filtration alluded to above, which is crucial to the construction of the $p$-adic Maass-Shimura operator and its algebraicity properties. Unlike in Katz's theory, outside of $\mathcal{Y}^{\mathrm{Ig}}$, $\alpha_{\infty,1}$ and $\alpha_{\infty,2}$ do not generate either the Hodge or Hodge-Tate filtrations, and instead we are led to considering natural periods 
$$\mathbf{z}_{\mathrm{dR}}, \mathbf{z} \in \mathcal{O}\mathbb{B}_{\mathrm{dR},Y}^+(\mathcal{Y}_x),$$
where the \emph{Hodge-de Rham period} $\mathbf{z}_{\mathrm{dR}} \in \mathcal{O}\mathbb{B}_{\mathrm{dR},Y}^+(\mathcal{Y}_x)$ measures the position of the Hodge filtration
$$\omega|_{\mathcal{Y}_x} = \frak{s}\cdot \mathcal{O}_{\mathcal{Y}_x} \subset \mathcal{H}_{\mathrm{dR}}^1(\mathcal{A}) \otimes_{\mathcal{O}_Y} \mathcal{O}\mathbb{B}_{\mathrm{dR},\mathcal{Y}_x}^+ \overset{\iota_{\mathrm{dR}}}{\hookrightarrow} \mathcal{H}_{\text{\'{e}t}}^1(\mathcal{A}) \otimes_{\hat{\mathbb{Z}}_{p,Y}} \mathcal{O}\mathbb{B}_{\mathrm{dR},\mathcal{Y}_x}^+ \underset{\sim}{\xrightarrow{\alpha_{\infty}^{-1}}} (\mathcal{O}\mathbb{B}_{\mathrm{dR},\mathcal{Y}_x}^+\cdot t^{-1})^{\oplus 2},$$
and the \emph{Hodge-Tate period} $\mathbf{z} \in \mathcal{O}_Y(\mathcal{Y}_x) \subset \mathcal{O}\mathbb{B}_{\mathrm{dR},Y}^+(\mathcal{Y}_x)$ measures the position of the Hodge-Tate filtration
$$\omega^{-1}|_{\mathcal{Y}_x} = \frak{s}^{-1}\cdot \mathcal{O}_{\mathcal{Y}_x} \subset \mathcal{H}_{\text{\'{e}t}}^1(\mathcal{A}) \otimes_{\hat{\mathbb{Z}}_{p,Y}} \mathcal{O}\mathbb{B}_{\mathrm{dR},\mathcal{Y}_x}^+ \underset{\sim}{\xrightarrow{\alpha_{\infty}^{-1}}} (\mathcal{O}\mathbb{B}_{\mathrm{dR},\mathcal{Y}_x}^+\cdot t^{-1})^{\oplus 2}.$$
Using these periods, and recalling our notation $\mathcal{O}_{\Delta,\mathcal{Y}_x} = \mathcal{O}\mathbb{B}_{\mathrm{dR},\mathcal{Y}_x}^+/(t)$, one can construct a Hodge decomposition
\begin{equation}\label{eq:decomposition1}T_p\mathcal{A} \otimes_{\hat{\mathbb{Z}}_{p,Y}} \mathcal{O}_{\Delta,\mathcal{Y}_x} \xrightarrow{\sim} (\omega \otimes_{\mathcal{O}_Y}\mathcal{O}_{\Delta,\mathcal{Y}_x}) \oplus (\omega^{-1}\otimes_{\mathcal{O}_Y}\mathcal{O}_{\Delta,\mathcal{Y}_x}\cdot t)
\end{equation}
where the projection onto the first factor is given by $HT_{\mathcal{A}}$ (i.e. the inclusion of the first factor is a section of $HT_{\mathcal{A}}$), and so this gives a splitting of the Hodge-Tate filtration. Unfortunately, the kernel of this splitting is \emph{not horizontal} with respect to the Gauss-Manin connection, which is a crucial property of the unit root splitting in Katz's theory. Indeed, restricting to $\mathcal{Y}^{\mathrm{Ig}}$ to recover Katz's theory, we that $\mathbf{z} = \infty$ on $\mathcal{Y}^{\mathrm{Ig}}$ and so in particular is constant and hence annihilated by $\nabla$. However, in general we have
\begin{equation}\label{eq:Hodgeline}\omega^{-1} \otimes_{\mathcal{O}_Y} \mathcal{O}_{\Delta,\mathcal{Y}_x}\cdot t = \langle \alpha_{\infty,1} - \frac{1}{\mathbf{z}}\alpha_{\infty,2}\rangle \mathcal{O}_{\Delta,\mathcal{Y}_x},
\end{equation}
and the fact that $\nabla(\mathbf{z}) \neq 0$ on $\mathcal{Y}_x$ means that for any section $w \in (\Omega_Y^1 \otimes_{\mathcal{O}_Y} \mathcal{O}_{\Delta,\mathcal{Y}_x}^1)(\mathcal{Y}_x)$ and $f \in \mathcal{O}_{\Delta,\mathcal{Y}_x}(\mathcal{Y}_x)$, 
\begin{align*}\nabla_w(\omega_{\mathrm{can}}^{-1}\cdot f) &= \nabla_w\left(\langle \alpha_{\infty,1} - \frac{1}{\mathbf{z}}\alpha_{\infty,2}\rangle f\right)\\
& = \nabla_w\left(\langle \alpha_{\infty,1} - \frac{1}{\mathbf{z}}\alpha_{\infty,2}\rangle\right)f + \langle \alpha_{\infty,1} - \frac{1}{\mathbf{z}}\alpha_{\infty,2}\rangle\nabla_w(f) \\
&= -\nabla\left(\frac{1}{\mathbf{z}}\right)\alpha_{\infty,2} + \langle \alpha_{\infty,1} - \frac{1}{\mathbf{z}}\alpha_{\infty,2}\rangle\nabla_w(f) \not\in \langle \alpha_{\infty,1} - \frac{1}{\mathbf{z}}\alpha_{\infty,2}\rangle \mathcal{O}_{\Delta,\mathcal{Y}_x} \\
&= \omega^{-1} \otimes_{\mathcal{O}_Y} \mathcal{O}_{\Delta,\mathcal{Y}_x}\cdot t,
\end{align*}
and so (\ref{eq:Hodgeline}) (and hence the Hodge decomposition (\ref{eq:decomposition1})) is not horizontal with respect to $\nabla$. To remedy this, we instead replace (\ref{eq:decomposition1}) with another splitting
\begin{equation}\label{eq:decomposition2}T_p\mathcal{A} \otimes_{\hat{\mathbb{Z}}_{p,Y}} \mathcal{O}_{\Delta,\mathcal{Y}_x} \xrightarrow{\sim} (\omega \otimes_{\mathcal{O}_Y}\mathcal{O}_{\Delta,\mathcal{Y}_x}) \oplus \mathcal{L},
\end{equation}
where $\mathcal{L}$ is a free $\mathcal{O}_{\Delta,\mathcal{Y}_x}$-module of rank 1. Now the projection onto the first factor is \emph{not} given by $HT_{\mathcal{A}}$, but instead its kernel is \emph{horizontal} in the sense that
$$\nabla_w(\mathcal{L}) \subset \mathcal{L}$$
for any section $w$ of $\Omega_Y^1\otimes_{\mathcal{O}_Y}\mathcal{O}_{\Delta,\mathcal{Y}_x}$. Moreover, (\ref{eq:decomposition2}) recovers the usual Hodge-Tate decomposition upon applying the natural map $\theta : \mathcal{O}_{\Delta,\mathcal{Y}_x} \twoheadrightarrow \hat{\mathcal{O}}_{\mathcal{Y}_x}$ where $\hat{\mathcal{O}}_{\mathcal{Y}_x}$ denotes the $p$-adically completed structure sheaf (and $\theta$ is analogous to Fontaine's universal cover $\theta : B_{\mathrm{dR}}^+ \twoheadrightarrow \mathbb{C}_p$)
\begin{equation}\label{eq:decomposition3}T_p\mathcal{A} \otimes_{\hat{\mathbb{Z}}_{p,Y}} \hat{\mathcal{O}}_{\mathcal{Y}_x} \xrightarrow{\sim} (\omega \otimes_{\mathcal{O}_Y}\hat{\mathcal{O}}_{\mathcal{Y}_x}) \oplus (\omega^{-1}\otimes_{\mathcal{O}_Y}\hat{\mathcal{O}}_{\mathcal{Y}_x}(1)).
\end{equation}

Now we can define a $p$-adic Maass-Shimura operator $d$ with respect to the splitting (\ref{eq:decomposition2}). Since (\ref{eq:decomposition3}) recovers the relative Hodge-Tate decomposition, it is induced at CM points by the algebraic CM splitting, and so as in Katz's theory one can show (\emph{using} the horizontalness of (\ref{eq:decomposition2})) that for an algebraic modular form $w \in \omega^{\otimes k}(Y)$, writing
$$w|_{\mathcal{Y}_x} = f\cdot \omega_{\mathrm{can}}^{\otimes k},\; f \in \mathcal{O}_{\Delta,\mathcal{Y}_x}(\mathcal{Y}_x), \hspace{1cm} F\cdot (2\pi i dz)^{\otimes k},\; F \in \mathcal{O}^{\mathrm{hol}}(\mathcal{H}^+),$$
where $\mathcal{O}^{\mathrm{hol}}$ denotes the sheaf of (complex) holomorphic function and $\mathcal{H}^+ \rightarrow Y$ the complex universal cover (i.e. the complex upper half-plane), we have that the value
$$(\theta \circ d^j)f(y)/\Omega_p(y)^{k+2j}$$
at a CM point $y \in \mathcal{Y}_x$ is an algebraic number for an appropriate $p$-adic period $\Omega_p(y)$ (depending on $y$), and in fact is equal (in $\overline{\mathbb{Q}}$) to the algebraic number
$$\frak{d}^jF(y)/\Omega_{\infty}(y)^{k+2j}$$
at the same CM point $y \in \mathcal{Y}_x$ for an appropriate complex period $\Omega_{\infty}(y)$ (only depending on the image of $y$ under the natural projection $\mathcal{Y} \rightarrow Y$):
\begin{equation}\label{eq:algebraicnumber}(\theta \circ d^j)f(y)/\Omega_p(y)^{k+2j} = \frak{d}^jF(y)/\Omega_{\infty}(y)^{k+2j}.
\end{equation}
The key fact for proving this algebraicity is that the fiber $\mathcal{O}_{\Delta,\mathcal{Y}_x}(y)$ contains a unique copy of $\overline{\mathbb{Q}}_p$ by Hensel's lemma, and so the the composition $\overline{\mathbb{Q}}_p \subset \mathcal{O}_{\Delta,\mathcal{Y}_x}(y) \overset{\theta}{\twoheadrightarrow} \mathbb{C}_p$ is the natural inclusion; then since the specialization $\mathbf{z}(y) \in \overline{\mathbb{Q}}_p$, we have $\theta(\mathbf{z}(y)) = \mathbf{z}(y)$, and so $(\theta \circ d^j)f(y) = \theta(d^jf(y)) = d^jf(y)$, and this latter value is equal (after normalizing by periods) to $\frak{d}^jf(y)$ since both (\ref{eq:decomposition2}) and the complex analytic Hodge decomposition are both induced by the CM splitting at $y$. 

It is the algebraicity of $\theta \circ d^j$ at CM points, and moreover the fact that it is equal in value to complex Maass-Shimura derivatives, which makes it applicable to questions regarding interpolation of critical $L$-values and hence construction of $p$-adic $L$-functions. Ultimately, for the construction of the latter, it is necessary to understand the analytic behavior of $(\theta \circ d^j)f$ around CM points $y$, and here the framework for understanding such analytic properties is provided by $q_{\mathrm{dR}}$-expansions of modular forms, given by a $q_{\mathrm{dR}}$-expansion map 
\begin{equation}\label{eq:qdRexp}\omega^{\otimes k}|_{\mathcal{Y}_x} \underset{\sim}{\xrightarrow{\omega_{\mathrm{can}}^{\otimes k}}} y_{\mathrm{dR}}^k\mathcal{O}_{\mathcal{Y}_x} \overset{q_{\mathrm{dR}}-\mathrm{exp}}{\hookrightarrow} \hat{\mathcal{O}}_{\mathcal{Y}_x}\llbracket q_{\mathrm{dR}}-1\rrbracket \subset \mathcal{O}_{\Delta,\mathcal{Y}_x} \underset{\sim}{\xrightarrow{\omega_{\mathrm{can}}^{\otimes k}}} \omega^{\otimes k} \otimes_{\mathcal{O}_Y}\mathcal{O}_{\Delta,\mathcal{Y}_x}.
\end{equation}
A key fact is that on the supersingular locus $\mathcal{Y}^{\mathrm{ss}} \subset \mathcal{Y}_x$, (\ref{eq:qdRexp}) \emph{coincides} with the natural inclusion
\begin{equation}\label{eq:naturalinclusion}\omega|_{\mathcal{Y}^{\mathrm{ss}}} \hookrightarrow \omega\otimes_{\mathcal{O}_Y} \mathcal{O}\mathbb{B}_{\mathrm{dR},\mathcal{Y}^{\mathrm{ss}}}^+ \xrightarrow{\mod t} \omega\otimes_{\mathcal{O}_Y} \mathcal{O}_{\Delta,\mathcal{Y}^{\mathrm{ss}}},
\end{equation}
which is induced by the natural inclusion $\mathcal{O}_{\mathcal{Y}} \subset \mathcal{O}\mathbb{B}_{\mathrm{dR},\mathcal{Y}}^+ \xrightarrow{\mod t} \mathcal{O}_{\Delta,\mathcal{Y}}$. In fact, recalling that $\hat{\mathcal{O}}_Y$ denotes the $p$-adic completion of the structure sheaf $\mathcal{O}_Y$, we have a natural inclusion
$$\hat{\mathcal{O}}_{\mathcal{Y}}\llbracket q_{\mathrm{dR}}-1\rrbracket \subset  \mathcal{O}_{\Delta,\mathcal{Y}}$$
which is compatible with the natural connections on each sheaf, and which is in fact an \emph{equality} on $\mathcal{Y}^{\mathrm{ss}}$:
$$\hat{\mathcal{O}}_{\mathcal{Y}^{\mathrm{ss}}}\llbracket q_{\mathrm{dR}}-1\rrbracket = \mathcal{O}_{\Delta,\mathcal{Y}^{\mathrm{ss}}}.$$
Hence we see that, at least on the supersingular locus $\mathcal{Y}^{\mathrm{ss}}$, $q_{\mathrm{dR}}$ provides the correct coordinate when viewing a rigid modular form 
$$w \in \omega^{\otimes k}(\mathcal{Y}^{\mathrm{ss}}) \subset (\omega \otimes_{\mathcal{O}_Y} \mathcal{O}_{\Delta,\mathcal{Y}^{\mathrm{ss}}})^{\otimes k}(\mathcal{Y}^{\mathrm{ss}}) \underset{\sim}{\xrightarrow{\omega_{\mathrm{can}}^{\otimes k}}} \mathcal{O}_{\Delta,\mathcal{Y}^{\mathrm{ss}}}(\mathcal{Y}^{\mathrm{ss}})$$
as a ``nearly rigid function". The coordinate $q_{\mathrm{dR}} \in \mathcal{O}_{\Delta,\mathcal{Y}_x}(\mathcal{Y}_x)$ plays the role analogous to that of the Serre-Tate coordinate, and in fact the $q_{\mathrm{dR}}$-expansion of a modular form recovers the Serre-Tate expansion upon restricting to $\mathcal{Y}^{\mathrm{Ig}}$. 

In fact, one can write down an explicit formula for $\theta \circ d^j$ in terms of $q_{\mathrm{dR}}$-coordinates
\begin{equation}\label{eq:explicitformula}\theta \circ d^j = \sum_{i = 0}^j\binom{j}{i}\binom{j+k-1}{i}i!\left(-\frac{\theta(y_{\mathrm{dR}})}{\mathbf{z}}\right)^i\theta\circ \left(\frac{q_{\mathrm{dR}}d}{dq_{\mathrm{dR}}}\right)^{j-i}.
\end{equation}
On $\mathcal{Y}^{\mathrm{Ig}}$, as we noted before, $\mathbf{z} = \infty$ and so we have
$$(\theta \circ d^j)|_{\mathcal{Y}^{\mathrm{Ig}}} = \theta \circ \left(\frac{q_{\mathrm{dR}}d}{dq_{\mathrm{dR}}}\right)^j|_{\mathcal{Y}^{\mathrm{Ig}}} = \theta\circ \theta_{\mathrm{AS}}^j = \theta_{\mathrm{AS}}^j$$
where the last equality follows from the fact that 
$$\frac{dq_{\mathrm{dR}}}{q_{\mathrm{dR}}}|_{\mathcal{Y}^{\mathrm{Ig}}} = dz_{\mathrm{dR}}|_{\mathcal{Y}^{\mathrm{Ig}}} = dT|_{\mathcal{Y}^{\mathrm{Ig}}}$$
and that
$$\mathcal{O}_{\mathcal{Y}_x} \subset \mathcal{O}_{\Delta,\mathcal{Y}_x} \overset{\theta}{\twoheadrightarrow} \hat{\mathcal{O}}_{\mathcal{Y}_x}$$
is the natural completion map.
Hence, again restricting to $\mathcal{Y}^{\mathrm{Ig}} \subset \mathcal{Y}_x$, we recover Katz's theory.

In order to construct the $p$-adic $L$-function, we consider the image of a modular form $w \in \omega^{\otimes k}(Y)$ under the $q_{\mathrm{dR}}$-expansion map (\ref{eq:qdRexp}), and study the growth of the coefficients of its $q_{\mathrm{dR}}$-expansion around supersingular CM points $y$:
$$\omega^{\otimes k}|_{\mathcal{Y}^{\mathrm{ss}}} \xrightarrow{q_{\mathrm{dR}}-\mathrm{exp}} \hat{\mathcal{O}}_{\mathcal{Y}^{\mathrm{ss}}}\llbracket q_{\mathrm{dR}}-1\rrbracket \xrightarrow{\mathrm{stalk\;at\;}y} \hat{\mathcal{O}}_{\mathcal{Y}^{\mathrm{ss}},y}\llbracket q_{\mathrm{dR}}-1\rrbracket = \hat{\mathcal{O}}_{\mathcal{Y}^{\mathrm{ss}},y}\llbracket q_{\mathrm{dR}}^{1/p^b}-1\rrbracket$$
for any $b \in \mathbb{Q}$, where the last equality is just a formal change of variables. By the remarks above involving (\ref{eq:qdRexp}) and (\ref{eq:naturalinclusion}), since $y \in \mathcal{Y}^{\mathrm{ss}}$, we see that the above map 
coincides with the natural map 
$$\omega^{\otimes k}|_{\mathcal{Y}^{\mathrm{ss}}} \hookrightarrow \omega^{\otimes k} \otimes_{\mathcal{O}_Y} \mathcal{O}_{\Delta,\mathcal{Y}^{\mathrm{ss}}} \xrightarrow{\mathrm{stalk\;at\;}y} \omega^{\otimes k} \otimes_{\mathcal{O}_Y} \mathcal{O}_{\Delta,\mathcal{Y}^{\mathrm{ss}},y}$$
and hence is compatible with the natural connections on all sheaves; in particular, this compatibility shows that the formula (\ref{eq:explicitformula}) gives the action of the $p$-adic Maass-Shimura operator $\theta^j \circ d^j$. One of the main results of Section \ref{overconvergencesection} is that for appropriate $b \in \mathbb{Q}$ (depending on $y$, but we later show that $b$ is the same for all $y$ in the same CM orbit), one in fact has that the above map factors through
\begin{equation}\label{eq:goodqdRexp}\omega^{\otimes k}|_{\mathcal{Y}^{\mathrm{ss}}} \rightarrow \hat{\mathcal{O}}_{\mathcal{Y}^{\mathrm{ss}},y}^+\llbracket q_{\mathrm{dR}}^{1/p^b}-1\rrbracket [1/p]
\end{equation}
where $\hat{\mathcal{O}}_Y^+ \subset \hat{\mathcal{O}}_Y$ denotes the $p$-adically completed integral structure structure sheaf. Given $w \in \omega^{\otimes k}(Y)$, we can construct the $p$-adic $L$-function associated with $w$ by considering sums of the images $w(q_{\mathrm{dR}}^{1/p^b})_y$ of $w$ under (\ref{eq:goodqdRexp}) for various orbits of CM points $y$ on $\mathcal{Y}^{\mathrm{ss}}$, and then applying the Maass-Shimura operators $p^{bj}\theta \circ d^j$ (normalized by $p^{bj}$) using the formula (\ref{eq:explicitformula}). By the formula, we see that as long as $p$-adic valuations $|y_{\mathrm{dR}}(y)|, |\mathbf{z}(y)|$ of the specializations $y_{\mathrm{dR}}(y), \mathbf{z}(y)$ of the $p$-adic periods $y_{\mathrm{dR}}, \mathbf{z}$ satisfy certain bounds, then images in the stalks 
$$p^{bj}(\theta \circ d^j)w(q_{\mathrm{dR}}^{1/p^b})_y$$
``converge" to some $p$-adic continuous function in $j \in \mathbb{Z}/(p-1) \times \mathbb{Z}_p$, in a sense we now make more precise. Define the \emph{$p$-stabilization} by 
$$w^{\flat}(q_{\mathrm{dR}}^{1/p^b})_y = w(q_{\mathrm{dR}}^{1/p^b})_y - \frac{1}{p}\sum_{j = 0}^{p-1}w(\zeta_p^jq_{\mathrm{dR}}^{1/p^b})_y.$$
One can show directly from (\ref{eq:explicitformula}) that
$$p^{bj}(\theta \circ d^j)w^{\flat}(q_{\mathrm{dR}}^{1/p^b})_y$$
is a $p$-adic continuous function of $j \in \mathbb{Z}/(p-1) \times \mathbb{Z}_p$. Then for any $j_0 \in \mathbb{Z}_{\ge 0}$ we have
\begin{equation}\label{eq:convergence}\lim_{m \rightarrow \infty}p^{b(j_0+p^m(p-1))}(\theta \circ d^{j_0 + p^m(p-1)})w(q_{\mathrm{dR}}^{1/p^b})_y = p^{bj_0}(\theta \circ d^{j_0})w^{\flat}(q_{\mathrm{dR}}^{1/p^b})_y.
\end{equation}
It is by summing $p^{bj}(\theta \circ d^j)w^{\flat}(q_{\mathrm{dR}}^{1/p^b})_y$ against anticyclotomic Hecke characters $\chi$ evaluated at $y$ over appropriate orbits of the CM point $y$ (associated with an order of an imaginary quadratic field $K$) that we arrive at the construction of our $p$-adic continuous $L$-function. The expression (\ref{eq:convergence}) along with the algebraicity theorem (\ref{eq:algebraicnumber}) gives rise to an ``approximate interpolation property" which our $p$-adic $L$-function satisfies, namely that values of the $p$-adic $L$-function in a certain range are determined by \emph{limits} of algebraic normalizations of central critical $L$-values associated with the Rankin-Selberg family $(w,\chi)$. 

In fact, we expect in forthcoming work to make the above approximate interpolation property into an actual interpolation formula, see Remark (\ref{interpolationremark}). 

We end this outline with a few remarks on how we obtain the $p$-adic Waldspurger formula in Section \ref{Waldspurgersection} when $k = 2$. A key property of the $p$-adic Maass-Shimura operator $d^j$ is that it sends $p$-adic modular forms of weight $k$ in the sense of (\ref{eq:transformation}) to modular forms of weight $k+2j$. Hence the limit
$$\lim_{m \rightarrow \infty}p^{bp^m(p-1)}(\theta \circ d^{p^m(p-1)})w^{\flat}(q_{\mathrm{dR}})_y$$
converges to a $p$-adic modular form of weight 0 on some small affinoid neighborhood of $y$, for some subgroup $\Gamma \subset GL_2(\mathbb{Z}_p)$. Let $K_p$ denote the $p$-adic completion of $K$ (with respect to a fixed embedding $\overline{\mathbb{Q}} \hookrightarrow \overline{\mathbb{Q}}_p$) In fact, one can show that for on some affinoid $\mathcal{U} \supset \mathcal{Y}^{\mathrm{Ig}} \sqcup \mathcal{C}$, where $\mathcal{C} \subset \mathcal{Y}$ is a locus of CM points associated with $K$ such that $\mathrm{Gal}(\mathcal{C}/C) \cong \mathcal{O}_{K_p}^{\times} \subset GL_2(\mathbb{Z}_p)$ (induced by some embedding $K_p \hookrightarrow M_2(\mathbb{Z}_p)$; the subadic space $\mathcal{C}$ itself does not depend on this choice of embedding), the limit
\begin{equation}\label{eq:convergenceprimitive}G := \lim_{m \rightarrow \infty}p^{bp^m(p-1)}(\theta \circ d^j)w^{\flat}(q_{\mathrm{dR}})|_{\mathcal{U}}
\end{equation}
converges to a $p$-adic modular form of weight 0 on $\mathcal{U}$ for some $\Gamma$ with $(\mathbb{Z}_p^{\times})^{\oplus 2} \subset \Gamma \subset GL_2(\mathbb{Z}_p)$. In particular, by restriction it induces a rigid function on $\mathcal{Y}^{\mathrm{Ig}} \sqcup \mathcal{U}$, which is of weight 0 for $(\mathbb{Z}_p^{\times})^{\oplus 2}$ on $\mathcal{Y}^{\mathrm{Ig}}$ and of weight 0 for $\Gamma$ on $\mathcal{U}$. This means that $G$ descends to a section $G$ on an affinoid open $U \subset Y'$ for some finite cover $Y' \rightarrow Y$; here we use the fact that while $\mathcal{Y}^{\mathrm{Ig}} \subset \mathcal{Y}$ is not affinoid open, its image on any finite cover is isomorphic to a copy of the ordinary locus $Y^{\mathrm{ord}} \subset Y$, which being the complement of a finite union of residue discs, is an (admissible) affinoid open. In particular, $G$ is rigid on $U$, and one can show using Coleman's theory of integration that on $U \cap Y'^{\mathrm{ord}}$, it is equal to the formal logarithm $\log_{w^{\flat}}|_{U \cap Y'^{\mathrm{ord}}}$ for some $p$-stabilization of the newform $w$. (Here $p$-stabilization denotes the image of $w$ under some explicit Hecke operator at $p$.) Then the rigidity of $G$ on $U$ implies that $dG$ is a rigid 1-form on $U$, and so by the theory of Coleman integration the rigid primitive $G$ on $\mathcal{U}$ is unique up constant, which implies
\begin{equation}\label{eq:primitive2}G = \log_{w^{\flat}}|_U.
\end{equation}
Since the relevant special value of our $p$-adic $L$-function corresponds to evaluating (\ref{eq:convergenceprimitive}) on an orbit of the CM point $y$, one sees that we arrive at our $p$-adic Waldspurger formula by evaluating (\ref{eq:primitive2}) at an appropriate Heegner point. 

\subsection{Main Results}
We now finally state our main results. We adopt the notation of Section \ref{padicLfunctionsection}, and the reader should refer to there for precise definitions and assumptions.

We let $p$ denote a prime which is inert or ramified in $K$, and let $w$ be a new eigenform (i.e. a newform or Eisenstein series) of weight $k$ for $\Gamma_1(N)$ and nebentype $\epsilon$, where $p\nmid N \ge 3$. Let $A$ be a fixed elliptic curve with CM by an order $\mathcal{O}_c \subset \mathcal{O}_K$ of conductor $p \nmid c$ also with $(c,Nd_K) = 1$, let $\alpha : \mathbb{Z}_p^{\oplus 2} \xrightarrow{\sim} T_pA$ be a choice of full $p^{\infty}$-level structure as in Choice \ref{goodchoice}, let $y = (A,\alpha) \in \mathcal{C}(\overline{K}_p,\mathcal{O}_{\overline{K}_p})$ and let $\Omega_p(y)$ and $\Omega_{\infty}(y)$ be the associated periods as in Definition \ref{perioddefinition}, and also let
$$w|_{\mathcal{H}^+} = F\cdot (2\pi idz)^{\otimes k}.$$
Then for Hecke characters $\chi \in \Sigma$, in the notation of Section \ref{constructionsection}, we have that the values $L(F,\chi^{-1},0)$ are central critical. Given an algebraic Hecke character, let $\check{\chi}$ denote its $p$-adic avatar, and let $\mathbb{N}_K : \mathbb{A}_K^{\times} \rightarrow \mathbb{C}^{\times}$ denote the norm character, which has infinity type $(1,1)$. We collect our results into one Main Theorem. 
 
\begin{theorem}[Theorems \ref{thm:maintheorem}, \ref{approximateinterpolationproperty}, \ref{padicWaldspurgerformula}]We have a $p$-adic continuous function
$$\mathcal{L}_{p,\alpha}(w,\cdot) : \overline{\check{\Sigma}}_+ \rightarrow \mathbb{C}_p$$
such that $\mathcal{L}_{p,\alpha}(w,\cdot)$ satisfies the following: Suppose $\{\chi_j\} \subset \Sigma_+$ is a sequence of algebraic Hecke characters where $\chi_j$ has infinity type $(k+j,-j)$, such that $\check{\chi}_j \rightarrow \check{\chi} \in \overline{\check{\Sigma}}_+\setminus \check{\Sigma}_+$. Then when $d_K$ is odd, we have the ``approximate" interpolation property
\begin{equation}\begin{split}&\mathcal{L}_{p,\alpha}(w,\check{\chi})^2 \\
&= \lim_{\check{\chi}_j \rightarrow \check{\chi}}\left(\frac{1}{p^b\theta(\Omega_p)(A,t,\alpha)^2}\right)^{k+2j} i_{\infty}^{-1}\left(\frac{C(w,\chi_j,c)\sigma(w,\chi_j)^{-1}}{\Omega_{\infty}(A,t)^{2(k+2j)}}\cdot L(F,\chi_j^{-1},0)\right)
\end{split}
\end{equation}
where $b$ is as defined in (\ref{integralcorrection2}) for $y = (A,\alpha)$, $C(w,\chi_j,c)$, $\sigma(w,\chi_j)$, and the central $L$-values $L(F,\chi_j^{-1},0) = L((\pi_w)_K \times \chi_j^{-1},1/2)$ are as in Theorem \ref{BDPalgebraicity}. 

Finally, when $k = 2$, we have our ``$p$-adic Waldspurger formula": Suppose that either $p$ is inert or $p > 2$. For any character $\chi : \mathcal{C}\ell(\mathcal{O}_c) \rightarrow \overline{\mathbb{Q}}_p^{\times}$, we have
$$\mathcal{L}_{p,\alpha}(w,\check{\mathbb{N}}_K\chi) = \begin{cases} \frac{1}{p^2(p^2-1)}\log_{w^{\flat}}P_K(\chi) & \text{$p$ is inert in $K$}\\
\frac{1}{p^3(p-1)} \log_{w^{\flat}}P_K(\chi) & \text{$p$ is ramified in $K$}
\end{cases}.$$
where $P_K(\chi)$ is the Heegner point as defined in Section \ref{Waldspurgersection}.
\end{theorem} 

In Proposition \ref{choiceofalpha}, we show that a different choice of $\alpha$ as in Choice \ref{goodchoice} amounts to multiplying $\mathcal{L}_{p,\alpha}(w,\cdot)$ by a $\mathcal{O}_{K_p}^{\times}$-valued character. 

Again, we stress that in forthcoming work we remove assumptions such as $p\nmid N$, the Heegner hypothesis, and prove analogues of the above theorem in greater generality. Moreover, we replace the ``approximate" interpolation property in the above theory with an exact interpolation property. This expectation is based on the following heuristic argument: $p$-adic interpolation of $L$-values is usually achieved upon removing the Euler factors at $p$ from the relevant $L$-values. In the case where $p$ is inert or ramified in $K$, the Euler factor at $p$ of $L(F,\chi^{-1},0)$ is \emph{constant}:
$$1-a_p(F)\chi^{-1}(p) + \epsilon(p)p^{k-1}\chi^{-1}(p^2) = 1-a_p(F)\epsilon^{-1}(p)p^{-k}\ + \epsilon^{-1}(p)p^{-k-1},$$
where $a_p(F)$ denotes the $p^{\mathrm{th}}$ fourier coefficient of $F$ (i.e. the Hecke eigenvalue at $p$ of $w$), for \emph{any} $\chi \in \Sigma$. Hence, removing the Euler factor from the interpolated $L$-values $L(F,\chi^{-1},0)$ amounts to multiplying by a constant (independent of $\chi \in \Sigma$), and so any $p$-adic continuous function which interpolates the $L$-values with Euler factor at $p$ removed $L^{(p)}(F,\chi^{-1},0)$ must also be equal to a constant times the limit of the values $L(F,\chi^{-1},0)$, which implies that by the approximate interpolation property that the $p$-adic continuous function interpolating the $L^{(p)}(F,\chi^{-1},0)$ should be equal to $\mathcal{L}_{p,\alpha}(w,\cdot)$ times a constant. This implies that $\mathcal{L}_{p,\alpha}(w,\cdot)$ is already continuous itself. We elaborate more on how we might actually prove this continuity in Remark \ref{interpolationremark}. We also expect to remove the assumption ``$d_K$ odd" assumption in the above approximate interpolation formula; we only make it here in order to use Theorem \ref{BDPalgebraicity}. 

\noindent\textbf{Acknowledgments.}\hspace{.5cm}This is a modified version of the author's Ph.D. thesis. The author thanks his advisers Shou-Wu Zhang and Christopher Skinner for helpful discussions and support throughout this project. This work was partially supported by the National Science Foundation under grant DGE 1148900.

\section{Preliminaries for the construction of the $p$-adic Maass-Shimura operator}

\subsection{The preperfectoid infinite level modular curve}\label{sec:review}
Henceforth, fix an algebraic closure $\overline{\mathbb{Q}}$ of $\mathbb{Q}$, and view all number fields (finite extensions of $\mathbb{Q}$) as embedded in $\overline{\mathbb{Q}}$. Henceforth, fix embeddings
\begin{equation}\label{fixedembedding}i_p : \overline{\mathbb{Q}} \hookrightarrow \overline{\mathbb{Q}}_p, \hspace{1cm} i_{\infty} : \overline{\mathbb{Q}} \hookrightarrow \mathbb{C}
\end{equation}
Let $\mathbb{C}_p$ denote the $p$-adic completion of $\overline{\mathbb{Q}}_p$. For every field $L \subset \overline{\mathbb{Q}}_p$, let $L_p$ denote the $p$-adic completion of $L$ in $\overline{\mathbb{Q}}_p$. Let $|\cdot|$ denote the unique $p$-adic valuation on $\mathbb{C}_p$ normalized with $|p| = 1/p$. 

Let $p$ be a prime. Fix a positive integer $N \ge 3$ which is coprime with $p$. Let $Y_1(N)$ denote the modular curve associated with $\Gamma_1(N)$. Let 
$$\mathcal{Y} := \varprojlim_n Y(\Gamma_1(N)\cap \Gamma(p^n))$$
denote the infinite level modular curve with $\Gamma_1(N) \cap \Gamma(p^{\infty})$ level structure, which can be viewed as an adic space defined over $\mathrm{Spa}(\mathbb{Q}_p,\mathbb{Z}_p)$. It is a \emph{preperfectoid space}, in the sense that upon base changing to a perfectoid field containing $\mathbb{Q}_p$ and $p$-adically completing the structure sheaf, one obtains the infinite level perfectoid modular curve constructed by Scholze in \cite{Scholze2}. One can view $\mathcal{Y}$ naturally as an adic space over $\mathrm{Spa}(\mathbb{Q}_p,\mathbb{Z}_p)$ (\cite{ScholzeWeinstein}) which is an object (in fact a profinite \'{e}tale object)  of the pro\'{e}tale site $Y_{\text{pro\'{e}t}}$ as defined in \cite{Scholze}. Given any extension $K/\mathbb{Q}_p$, we can form the base change $Y_K := Y\times_{\mathrm{Spa}(\mathbb{Q}_p,\mathbb{Z}_p)} \mathrm{Spa}(K,\mathcal{O}_K^+)$. Henceforth, given any separable extension $K/\mathbb{Q}_p$, we can define an object also denoted by $Y_K \in Y_{\text{pro\'{e}t}}$ via 
$$Y_K = \varprojlim_{\mathbb{Q}_p \subset L \subset K, \; \text{$L/\mathbb{Q}_p$ finite}}Y_{L_i}.$$
If $K \supset \mathbb{Q}_p$ is a perfectoid field, $Y_K$ is a perfectoid object in $Y_{\text{pro\'{e}t}}$. There is a natural equivalence of sites $Y_{\text{pro\'{e}t},K} \cong Y_{\text{pro\'{e}t}}/Y_K$, where the latter denotes $Y_{\text{pro\'{e}t}}$ localized at $Y_K$. 

It is a somewhat subtle point that we work on $\mathcal{Y}$ instead of its associated ``strong completion'' (in the sense of \cite{ScholzeWeinstein}) $\hat{\mathcal{Y}}$, which is a true perfectoid space after base change to a perfectoid field (and also a perfectoid space over $\mathrm{Spa}(\mathbb{Q}_p,\mathbb{Z}_p)$ in the sense of Kedlaya-Liu, \cite{KedlayaLiu}). The reason for this is that we need to work with differentials in $\Omega_{\mathcal{Y}}^1$ in order to define our differential operators (which arise from extensions of the Gauss-Manin connection $\nabla : \mathcal{H}_{\mathrm{dR}}^1(\mathcal{A}) \rightarrow \mathcal{H}_{\mathrm{dR}}(\mathcal{A}) \otimes_{\mathcal{O}_Y} \Omega_{\mathcal{Y}}^1$), and it is a fact that $\Omega_{\hat{\mathcal{Y}}}^1 = 0$, and more generally that all perfectoid spaces $X$ have $\Omega_X^i = 0$ for all $i \ge 1$ (for example, see \cite[Section 4.4]{Chojecki}). 

The adic space $\mathcal{Y}$ has the following moduli-theoretic interpretation: a $(\mathbb{C}_p,\mathcal{O}_{\mathbb{C}_p})$-point on $\mathcal{Y}$ corresponds to a triple $(A,t,\alpha)$ consisting of an elliptic curve $A/\mathbb{C}_p$, a $\Gamma_1(N)$ level structure $t \in A[N]$, and a trivialization $\alpha : \mathbb{Z}_p^{\oplus 2} \xrightarrow{\sim} T_pA$ of its Tate module. (This trivialization is equivalent to a $\Gamma(p^{\infty})$ Drinfeld level structure on $A$.) Then $\mathcal{Y}$ represents a (fine) moduli space with universal object given by $\mathcal{A}_{\infty} := \mathcal{A}\times_{Y_1(N)}\mathcal{Y}$, where $\mathcal{A} \rightarrow Y_1(N)$ is the universal elliptic curve with $\Gamma_1(N)$-level structure. We also will write
$$\mathcal{A}_{\infty} = (\mathcal{A},\alpha_{\infty})$$
where $\alpha_{\infty}$ denotes the universal $\Gamma(p^{\infty})$ level structure, which is equivalent to a trivialization 
$$\alpha_{\infty} : \hat{\mathbb{Z}}_{p,\mathcal{Y}}^{\oplus 2} \xrightarrow{\sim} T_p\mathcal{A}$$
of the universal Tate module $T_p\mathcal{A}$ (where $\hat{\mathbb{Z}}_{p,\mathcal{Y}}$ is the ``$p$-adic constant sheaf corresponding to $\mathbb{Z}_p$, to be defined in Section \ref{periodsheaves}). From now on, we will often suppress the $\Gamma_1(N)$-level structure $t$ in our notation for simplicity (and also because it doesn't affect any of the following).

\subsection{The pro\'{e}tale site, structure sheaves and period sheaves}\label{periodsheaves}
Let $Y_{\text{\'{e}t}}$ denote the \emph{small \'{e}tale site} of the adic space $Y$ over $\mathrm{Spa}(\mathbb{Q}_p,\mathbb{Z}_p)$. Hence, objects of $Y_{\text{\'{e}t}}$ are \'{e}tale maps $U \rightarrow Y$, and coverings are \'{e}tale coverings. The category pro-$Y_{\text{\'{e}t}}$ consists of functors $F : I \rightarrow Y_{\text{\'{e}t}}$ from small cofiltered index categories $I$ in which morphisms are given by
$$\Hom_{\mathrm{pro}-Y_{\text{\'{e}t}}}(F,G) := \varprojlim_J\varinjlim_I\Hom_{Y_{\text{\'{e}t}}}(F(i),G(j)).$$

Throughout this article, it will be necessary to work on the \emph{pro\'{e}tale site} $Y_{\text{pro\'{e}t}}$ as defined in \cite{Scholze}, whose objects form a full subcategory of pro-$Y_{\text{\'{e}t}}$. Hence, objects of $Y_{\text{pro\'{e}t}}$ are pro-\'{e}tale maps $U \rightarrow Y$, which means that $U$ can be written as $\varprojlim_i U_i$, where $U_i$ are objects of pro-$Y_{\text{\'{e}t}}$ such that each $U_i \rightarrow Y$ is an \'{e}tale map in pro-$Y_{\text{\'{e}t}}$ in the sense of loc. cit., and $U_i \rightarrow U_j$ is a \emph{finite} \'{e}tale map in pro-$Y_{\text{\'{e}t}}$ in the sense of loc. cit. for all large $i > j$. Coverings are given by pro\'{e}tale coverings. Given a $U \in Y_{\text{pro\'{e}t}}$, a presentation $U = \varprojlim_i U_i$ as above is called a \emph{pro\'{e}tale presentation}. It is clear that there is hence a natural projection of sites $\nu : Y_{\text{pro\'{e}t}} \rightarrow Y_{\text{\'{e}t}}$. Given any object $U \in Y_{\text{pro\'{e}t}}$, let $Y_{\text{pro\'{e}t}}$ denote the localized site, consisting of objects $V \in Y_{\text{pro\'{e}t}}$ with a map $V \rightarrow U$. It is shown in \cite{Scholze} that if $Y_K$ denotes the base change of $Y$ to $\mathrm{Spa}(K,\mathcal{O}_K)$, then there is a natural equivalence of sites $Y_{K,\text{pro\'{e}t}} \cong Y_{\text{pro\'{e}t}}/Y_K$. 

Let $K$ denote a perfectoid field over $\mathbb{Q}_p$. An \emph{affinoid perfectoid} is an object $U \in Y_{\text{pro\'{e}t}}/Y_K$ with a pro\'{e}tale presentation $\varprojlim_i U_i$ where $U_i = \mathrm{Spa}(R_i,R_i^+)$ is an affinoid such that if $R^+$ denotes the $p$-adic completion of $\varinjlim_i R_i^+$ and $R = R^+[1/p]$, then $(R,R^+)$ is a perfectoid $(K,\mathcal{O}_K$)-algebra. A perfectoid object is one with an open covering by affinoid perfectoids. We recall the key fact that affinoid perfectoids form a basis of $Y_{\text{pro\'{e}t}}$ (\cite{Scholze}). 

Given any sheaf $\mathcal{F}$ on $Y_{\text{pro\'{e}t}}$ and any object $U \in Y_{\text{pro\'{e}t}}$, we let $\mathcal{F}_U := \mathcal{F}|_U$ denote the induced sheaf on $Y_{\text{pro\'{e}t}}/U$. If the sheaf is written as $\mathcal{F}_Y$, we will often write $\mathcal{F}_U := \mathcal{F}_{Y}|_U$ for brevity. We let $\mathcal{O}_Y$ denote the structure sheaf on $Y_{\text{pro\'{e}t}}$ (which is defined to be the $\nu^*\mathcal{O}$, where $\mathcal{O}$ is the structure sheaf on $Y_{\text{\'{e}t}}$); similarly, let $\mathcal{O}_Y^+$ denote the interal structure sheaf. Let $\hat{\mathcal{O}}_Y^+ := \varprojlim_n \mathcal{O}_Y^+/p^n$ denote the integral $p$-adically completed structure sheaf, and let $\hat{\mathcal{O}}_Y := \hat{\mathcal{O}}_Y^+[1/p]$ denote the $p$-adically completed structure sheaf.

As they will be omnipresent throughout our discussion, we will recall construction of the period sheaves and $\mathcal{O}\mathbb{B}_{\mathrm{dR},Y}^+$ and $\mathcal{O}\mathbb{B}_{\mathrm{dR},Y}$ on $Y_{\text{pro\'{e}t}}$, defined as in \cite[Section 6]{Scholze}, as well as some of their key properties. Let $K$ denote any perfectoid field containing $\mathbb{Q}_p$, and let $Y_K$ denote the base change $Y\times_{\mathrm{Spa}(\mathbb{Q}_p),\mathrm{Spa}(\mathbb{Z}_p)} \mathrm{Spa}(K,\mathcal{O}_K)$. Since $Y_{K,\text{pro\'{e}t}} \cong Y_{\text{pro\'{e}t}}/Y_K$, and affinoid perfectoids form a basis of $Y_{K,\text{pro\'{e}t}}$, then it suffices to define the sections of the above period sheaves on affinoid perfectoid in $Y_{K,\text{pro\'{e}t}}$ for any perfectoid field $K$ containing $\mathbb{Q}_p$. Recall the tilted structure sheaf 
$$\hat{\mathcal{O}}_{Y_K^{\flat}}^+ := \varprojlim_{x \mapsto x^p}\mathcal{O}_{Y_K}^+/p = \varprojlim_{x \mapsto x^p}\hat{\mathcal{O}}_{Y_K}^+$$
and the period sheaves
$$\mathbb{A}_{\mathrm{inf},Y_K} := W(\hat{\mathcal{O}}_{Y_K^{\flat}}^+) \hspace{1cm} \mathbb{B}_{\mathrm{inf},Y_K} := \mathbb{A}_{\mathrm{inf}}[p^{-1}] \hspace{1cm} \mathcal{O}\mathbb{B}_{\mathrm{inf},Y_K} := \mathcal{O}_{Y_K}\otimes_{W(\kappa)}\mathbb{B}_{\mathrm{inf},Y_K}.$$
which are initially defined naturally on affinoid perfectoids (see \cite[Section 6]{Scholze}).  We have the natural projection $\theta : \mathbb{B}_{\mathrm{inf},Y_K} \twoheadrightarrow \hat{\mathcal{O}}_{Y_K}$; $\ker\theta$ is locally generated by some section $\xi$. We let $$\mathbb{B}_{\mathrm{dR},Y_K}^+ = \varprojlim_i \mathbb{B}_{\mathrm{inf},Y_K}/(\ker\theta)^i,\hspace{1cm} \mathbb{B}_{\mathrm{dR},Y_K} := \mathbb{B}_{\mathrm{dR},Y_K}^+[1/t]$$
where $t$ is any generator of $\ker(\theta)\mathbb{B}_{\mathrm{dR},Y_K}^+$. The projection $\theta : \mathbb{B}_{\mathrm{inf},Y_K} \twoheadrightarrow \hat{\mathcal{O}}_{Y_K}$ extends $\mathcal{O}_{Y_K}$-linearly to a projection $\theta : \mathcal{O}\mathbb{B}_{\mathrm{inf},Y_K} \twoheadrightarrow \hat{\mathcal{O}}_{Y_K}$. Then 
$$\mathcal{O}\mathbb{B}_{\mathrm{dR},Y_K}^+ := \varprojlim_n\mathcal{O}\mathbb{B}_{\mathrm{inf},Y_K}/(\ker\theta)^n \hspace{1cm} \mathcal{O}\mathbb{B}_{\mathrm{dR},Y_K} := \mathcal{O}\mathbb{B}_{\mathrm{dR},Y_K}^+[t^{-1}].$$
It is a fact that the composition
\begin{equation}\label{inclusion}\mathcal{O}_{Y_K} \rightarrow \mathcal{O}\mathbb{B}_{\mathrm{dR},Y_K}^+ \rightarrow \mathcal{O}\mathbb{B}_{\mathrm{dR},Y_K}^+/\ker\theta = \hat{\mathcal{O}}_{Y_K}
\end{equation}
is the natural inclusion $\mathcal{O}_{Y_K} \rightarrow \hat{\mathcal{O}}_{Y_K}$.
We recall that $\mathcal{O}\mathbb{B}_{\mathrm{dR},Y_K}^+$ is equipped with a canonical integrable $\mathbb{B}_{\mathrm{dR},Y}^+$-linear connection 
\begin{equation}\label{connection}\nabla : \mathcal{O}\mathbb{B}_{\mathrm{dR},Y_K}^+ \rightarrow \mathcal{O}\mathbb{B}_{\mathrm{dR},Y}^+\otimes_{\mathcal{O}_{Y_K}}\Omega_{Y_K}^1
\end{equation}
extending the trivial connection $d : \mathcal{O}_{Y_K} \rightarrow \Omega_{Y_K}^1$. Since the above connection is $\mathbb{B}_{\mathrm{dR},Y_K}^+$-linear, it extends to a connection 
$$\nabla : \mathcal{O}\mathbb{B}_{\mathrm{dR},Y_K} \rightarrow \mathcal{O}\mathbb{B}_{\mathrm{dR},Y_K}\otimes_{\mathcal{O}_{Y_K}}\Omega_{Y_K}^1.$$
The sheaf $\mathcal{O}\mathbb{B}_{\mathrm{dR},Y_K}$ has a natural filtration given by $\Fil^i\mathcal{O}\mathbb{B}_{\mathrm{dR},Y_K}^+ = (\ker\theta)^i\mathcal{O}\mathbb{B}_{\mathrm{dR},Y_K}^{+}$, and $\mathcal{O}\mathbb{B}_{\mathrm{dR},Y_K}$ has a filtration given by $\Fil^i\mathcal{O}\mathbb{B}_{\mathrm{dR},Y_K} = \sum_{j + j' = i, j, j' \in \mathbb{Z}}t^j\Fil^{j'}\mathcal{O}\mathbb{B}_{\mathrm{dR},Y_K}^+$.
By the previous discussion, we hence get all the analogous sheaves and maps on $Y_{\text{pro\'{e}t}}$.   

Henceforth, given an adic space $X$ with pro\'{e}tale site $X_{\text{pro\'{e}t}}$, let
$$\hat{\mathbb{Z}}_{p,X} = \varprojlim_n (\mathbb{Z}/p^n)_X$$
be an inverse limit of the usual constant sheaves $(\mathbb{Z}/p^n)_X$ on $X_{\text{pro\'{e}t}}$ (which are equal to those constant sheaves pulled back from the small \'{e}tale site $X_{\text{\'{e}t}}$); its sections correspond to continuous maps from $X_{\text{pro\'{e}t}}$ into $\mathbb{Z}_p$ where $\mathbb{Z}_p$ is given the \emph{$p$-adic topology} (and \emph{not} the discrete topology). We also define the \emph{Tate twist} $\hat{\mathbb{Z}}_{p,X}(1)$ by
$$\hat{\mathbb{Z}}_{p,X}(1) = \varprojlim_{n} (\mu_{p^n})_X$$
where $(\mu_{p^n})_X$ is the usual constant sheaf on $X_{\text{pro\'{e}t}}$. For any non-negative integer $n$, we let $\hat{\mathbb{Z}}_{p,X}(n) = \hat{\mathbb{Z}}_{p,X}(1)^{\otimes n}$ (where $\otimes$ is taken over $\hat{\mathbb{Z}}_{p,X}$), and we let $\hat{\mathbb{Z}}_{p,X}(-n) = \Hom_{\hat{\mathbb{Z}}_{p,X}}(\hat{\mathbb{Z}}_{p,X}(n),\hat{\mathbb{Z}}_{p,X})$. Given any $\hat{\mathbb{Z}}_{p,X}$-module $\mathcal{F}$ and $n \in \mathbb{Z}$, we let $\mathcal{F}(n) := \mathcal{F} \otimes_{\hat{\mathbb{Z}}_{p,X}}\hat{\mathbb{Z}}_{p,X}(n)$.

\subsection{Relative \'{e}tale cohomology and the Weil pairing}Henceforth, we will make free use of the canonical principal polarization
$$\mathcal{A} \cong \check{\mathcal{A}}$$
where $\check{\mathcal{A}}$ denotes the dual of an abelian variety $\mathcal{A}$. 

We have that
$$\mathcal{H}_{\text{\'{e}t}}^1(\mathcal{A}) := R^1\pi_*\hat{\mathbb{Z}}_{p,\mathcal{A}}$$
 is a $\hat{\mathbb{Z}}_{p,\mathcal{Y}}$-local system of rank 2 which gives the \emph{relative \'{e}tale cohomology} of the family $\pi : \mathcal{A} \rightarrow Y$. We denote 
$$T_p\mathcal{A} := \Hom_{\hat{\mathbb{Z}}_{p,Y}}(\mathcal{H}_{\text{\'{e}t}}^1(\mathcal{A}),\hat{\mathbb{Z}}_{p,Y})$$
which gives the \emph{relative Tate module} of the same family. Recall the Weil pairings given by, using the principal polarization $\mathcal{A} \cong \check{\mathcal{A}}$ on $\mathcal{A}$
\begin{equation}\label{truncatedWeil}\langle \cdot, \cdot\rangle_n : \mathcal{A}[p^n] \times \mathcal{A}[p^n] \rightarrow \mu_{p^n}
\end{equation}
which, taking the inverse limit over $n$, form a nondegenerate pairing
\begin{equation}\label{Weil}\langle \cdot,\cdot\rangle : T_p\mathcal{A} \times T_p\mathcal{A} \rightarrow \hat{\mathbb{Z}}_{p,Y}(1)
\end{equation} 
as an inverse limit of constant sheaves on $X_{\text{pro\'{e}t}}$. This induces a canonical isomorphism 
\begin{equation}\label{dualisomorphism}T_p\mathcal{A} \cong \mathcal{H}_{\text{\'{e}t}}^1(\mathcal{A})(1).
\end{equation}

\subsection{The $GL_2(\mathbb{Q}_p)$-action on $\mathcal{Y}$} \label{action}
\begin{definition}Henceforth, we denote the canonical homogeneous coordinates on $\mathbb{P}^1 = \mathbb{P}(\mathbb{Q}_p^2)$ by $x, y \in \mathcal{O}_{\mathbb{P}^1}(1)$. Henceforth, let 
$$\mathbb{P}_x^1 = \{x \neq 0\}, \hspace{1cm} \mathbb{P}_y^1 = \{y \neq 0\}$$
be the standard affine cover of $\mathbb{P}^1$. 
\end{definition}
Recall the right $GL_2(\mathbb{Q}_p)$-action on $\mathcal{Y}$, which acts on $(\mathbb{C}_p,\mathcal{O}_{\mathbb{C}_p})$-points $(A,\alpha)$ in the following way (as recalled in \cite[Section 2.2]{ChojeckiHansenJohansson}). Recall $GL_2(\mathbb{Q}_p)$ has a left action on $\mathbb{C}_p^{\oplus 2}$ in the standard way (viewing elements of the latter as column vectors), and thus if we denote the contragredient by $g^{\vee} := g^{-1}\det(g)$, we get a \emph{right} action $L \cdot g = g^{\vee}(L)$ on $\mathbb{Z}_p^{\oplus 2}$. Fix $n \in \mathbb{Z}$ such that $p^ng \in M_2(\mathbb{Z}_p)$ but $p^{n-1}g\not\in M_2(\mathbb{Z}_p)$. Then for all $m \in \mathbb{Z}_{> 0}$ sufficiently large, the kernel of $p^ng^{\vee} \pmod{p^m}$ stabilizes to some subgroup $H$ of $A[p^m]$. Then we put $(A,\alpha)\cdot g := (A,\alpha')$, where $\alpha'$ is defined as the composition
$$\mathbb{Z}_p^{\oplus 2} \xrightarrow{p^ng}\mathbb{Q}_p^{\oplus 2} \xrightarrow{\alpha} V_pA \xrightarrow{(\check{\phi}_*)^{-1}} V_p(A/H),$$
where $p^ng$ acts in the usual way on the left of $\mathbb{Z}_p^{\oplus 2}$ (viewing elements of the latter space as column vectors), where $\phi : A \rightarrow A/H$ is natural isogeny given by projection and $\phi_*$ is the induced map $T_pA \rightarrow T_p(A/H)$ (extended by linearity to $V_pA \rightarrow V_p(A/H)$). Note that if $g = \left(\begin{array}{cc} a & b \\ c & d\\ \end{array}\right) \in GL_2(\mathbb{Z}_p)$, then $(A,\alpha)\cdot g = (A,\alpha')$ where $\alpha'(e_1) = a\alpha(e_1) + c\alpha(e_2)$ and $\alpha'(e_2) = b\alpha(e_1) + d\alpha(e_2)$. 

In fact, one can check that the following diagram is commutative:
\begin{equation}
\begin{tikzcd}[column sep =small]\label{isogenydiagram}
     & \mathbb{Z}_p^{\oplus 2} \arrow{r}{\alpha} \arrow{d}{p^ng^{\vee}} & T_pA \arrow{d}{\phi_*}&\\
     & \mathbb{Z}_p^{\oplus 2} \arrow{r}{\alpha'} & T_p(A/H)&
\end{tikzcd}
\end{equation}
where in the arrow on the left $p^ng^{\vee}$ acts on $\mathbb{Z}_p^{\oplus 2}$ via left multiplication (viewing the latter as column vectors). 

One also has the standard \emph{right} $GL_2(\mathbb{Q}_p)$-action on $\mathbb{P}^1$, given by an element $g$ acting on lines $\mathcal{L}$ (viewed as being spanned by column vectors) via $\mathcal{L} \cdot g = g^{\vee}\mathcal{L}$. 

Finally, we remark that given a geometric point $y$ of $Y$, the automorphism group of the fiber of $y$ under the natural profinite \'{e}tale map $\mathcal{Y} \rightarrow Y$ is $GL_2(\mathbb{Z}_p)$, since $GL_2(\mathbb{Z}_p)$ acts simply transitively on this geometric fiber by changing the basis $\alpha : \mathbb{Z}_p^{\oplus 2} \xrightarrow{\sim} T_pA$. 

\subsection{The Hodge-Tate period and the Hodge-Tate period map}\label{HTsection}In this section, we define the Hodge-Tate period and Hodge-Tate period map $\pi_{\mathrm{HT}} : \mathcal{Y} \rightarrow \mathbb{P}^1$ following \cite{Scholze2} and \cite{CaraianiScholze}. However, as those latter works define $\pi_{\mathrm{HT}}$ on the \emph{perfectoid space} $\hat{\mathcal{Y}}$  associated with $\mathcal{Y}$ (i.e. the strong $p$-adic completion of $\mathcal{Y}$) rather than on $\mathcal{Y}$ itself, it will be necessary for us later to work on the \emph{preperfectoid} space $\mathcal{Y} \in Y_{\text{pro\'{e}t}}$ (since we will need to consider $\Omega_{\mathcal{Y}}^1$), we must define $\pi_{\mathrm{HT}}$ on $\mathcal{Y}$ as a map of adic spaces over $\mathrm{Spa}(\mathbb{Q}_p,\mathbb{Z}_p)$. One can also recover our definition of $\pi_{\mathrm{HT}}$ from the more general construction of period morphisms on (preperfectoid) moduli of $p$-divisible groups given by Scholze-Weinstein \cite{ScholzeWeinstein}. 

Recall the \emph{Hodge bundle} defined by 
$$\omega_{\mathcal{A}} := \pi_*\Omega_{\mathcal{A}/Y}^1.$$
Now, we define the \emph{Hodge-Tate map}
\begin{equation}\label{HTmap}HT_{\mathcal{A}} : T_p\mathcal{A} \rightarrow \omega_{\mathcal{A}}
\end{equation}
to be the following: the truncated Weil pairings (\ref{truncatedWeil}) are compatible in the sense that the following diagram commutes
\begin{equation}
\begin{tikzcd}[column sep =small]\label{isogenydiagram}
   \mathcal{A}[p^n]  & \times &\mathcal{A}[p^n] \arrow[hookrightarrow]{d}{}  \arrow{rr}{\langle \cdot, \cdot \rangle_n}&  &\mu_{p^n} \arrow[hookrightarrow]{d}&\\
   \mathcal{A}[p^{n+1}]  \arrow{u}{p}& \times & \mathcal{A}[p^{n+1}] \arrow{rr}{\langle \cdot, \cdot \rangle_{n+1}} & &\mu_{p^{n+1}},
\end{tikzcd}
\end{equation}
and hence they can be put together (taking an inverse limit along the left first vertical arrows, and a direct limit along the middle and right vertical arrows) to get a pairing
\begin{equation}\label{infinitepairing}T_p\mathcal{A} \times \mathcal{A}[p^{\infty}] \rightarrow \mu_{p^{\infty}}.
\end{equation}
Hence, we have a natural identification $T_p\mathcal{A} \cong \Hom_{p-\mathrm{div}}(\mathcal{A}[p^{\infty}],\mu_{p^{\infty}})$ (where the $\Hom$ is taken in the category of $p$-divisible groups over $\mathbb{Z}_p$). Using this identification, we define
\begin{equation}\label{HTmapdefinition}HT_{\mathcal{A}}(\alpha) = \alpha^*\frac{dT}{T} \in \pi_*\Omega_{\mathcal{A}[p^{\infty}]/Y}^1  = \pi_*\Omega_{\mathcal{A}/Y}^1 = \omega_{\mathcal{A}}.
\end{equation}
By nondegeneracy of the Weil pairing, we know that $HT_{\mathcal{A}}$ is not the zero map. Let $\mathcal{H}^{0,1} \subset T_p\mathcal{A}\otimes_{\hat{\mathbb{Z}}_{p,Y}}\mathcal{O}_Y$ denote its kernel. 

Pulling back \ref{HTmapdefinition} along the cover $\mathcal{Y} \rightarrow Y$, we get a section $HT_{\mathcal{A}}(\alpha_{\infty,2}) \in \omega_{\mathcal{A}}(\mathcal{Y})$. Let 
$$\mathcal{Y}_x = \{HT_{\mathcal{A}}(\alpha_{\infty,2}) \neq 0\} \subset \mathcal{Y}, \hspace{1cm} \mathcal{Y}_y := \{HT_{\mathcal{A}}(\alpha_{\infty,1}) \neq 0\}.$$
 Denote $\mathcal{O}_{\mathcal{Y}} := \mathcal{O}_{Y|\mathcal{Y}}$. We now define the \emph{Hodge-Tate period} $\mathbf{z} \in \mathcal{O}_{\mathcal{Y}}(\mathcal{Y}_y)$ via the relation
\begin{equation}\label{relation}HT_{\mathcal{A}}(\alpha_{2,\infty}) = \mathbf{z}\cdot HT_{\mathcal{A}}(\alpha_{1,\infty}).
\end{equation}
Since $HT_{\mathcal{A}}$ is not the zero map, we have that $\mathcal{Y}_x \cup \mathcal{Y}_y = \mathcal{Y}$. Using the trivialization $\alpha_{\infty}^{-1} : T_p\mathcal{A}|_{\mathcal{Y}} \xrightarrow{\sim} \hat{\mathbb{Z}}_{p,\mathcal{Y}}^{\oplus 2}$, we now see that the line 
\begin{equation}\label{line}\mathcal{H}^{0,1}|_{\mathcal{Y}} \subset T_p\mathcal{A}\otimes_{\hat{\mathbb{Z}}_p,\mathcal{Y}} \mathcal{O}_{\mathcal{Y}} \underset{\sim}{\xrightarrow{\alpha_{\infty}^{-1}}} \mathcal{O}_{\mathcal{Y}}^{\oplus 2}
\end{equation}
is the sub-$\mathcal{O}_{\mathcal{Y}}$-module obtained by gluing together the sheaves $\mathcal{H}^{0,1}|_{\mathcal{Y}_x} := (\alpha_{\infty,1} - \frac{1}{\mathbf{z}}\alpha_{\infty,2})\mathcal{O}_{\mathcal{Y}}$ and $\mathcal{H}^{0,1}|_{\mathcal{Y}_y} := (\mathbf{z}\cdot \alpha_{\infty,2} - \alpha_{\infty,2})\mathcal{O}_{\mathcal{Y}}$ along $\mathcal{Y}_x \cap \mathcal{Y}_y$ (on which affinoid subdomain the reader will easily check that the two sheaves coincide, by the definition of $\mathbf{z}$). The inclusion \ref{line} defines a map of adic spaces over $\mathrm{Spa}(\mathbb{Q}_p,\mathbb{Z}_p)$
$$\pi_{\mathrm{HT}} : \mathcal{Y} \rightarrow \mathbb{P}^1$$
which we call the \emph{Hodge-Tate period map}. It is straightforward to check that $\pi_{\mathrm{HT}}$ is equivariant with respect to the $GL_2(\mathbb{Q}_p)$-actions defined in Section \ref{action}. Furthermore, since $\mathcal{Y}$ has the same underlying topological space as its associated perfectoid space (once base-changed to any perfectoid field), the map $\pi_{\mathrm{HT}}$ is \'{e}tale on the supersingular locus
$$\mathcal{Y}^{\mathrm{ss}} := \varprojlim_n Y(\Gamma_1(N)\cap\Gamma(p^n))^{\mathrm{ss}}$$
and maps to the Drinfeld upper half-plane $\Omega = \mathbb{P}^1(\mathbb{C}_p) \setminus \mathbb{P}^1(\mathbb{Q}_p)$. (For the \'{e}taleness of $\pi_{\mathrm{HT}} : \mathcal{Y}^{\mathrm{ss}} \rightarrow \Omega \subset \mathbb{P}^1$, see also \cite[Remark 7.1.2 and Theorem 7.2.3]{ScholzeWeinstein}.) The stratification $\mathcal{Y} = \mathcal{Y}^{\mathrm{ord}}\sqcup \mathcal{Y}^{\mathrm{ss}}$ is in fact given by 
$$\mathcal{Y}^{\mathrm{ss}} = \pi_{\mathrm{HT}}^{-1}(\Omega), \hspace{1cm} \mathcal{Y}^{\mathrm{ord}} = \pi_{\mathrm{HT}}^{-1}(\mathbb{P}^1(\mathbb{Q}_p)).$$

Finally, we note that pulling back the coordinate $z = -x/y$ on $\mathbb{P}^1$, we also obtain the Hodge-Tate period via
\begin{equation}\label{zdefinition}\mathbf{z} = \pi_{\mathrm{HT}}^*z
\end{equation}
on $\mathcal{Y}$. Moreover,
$$HT_{\mathcal{A}}(\alpha_{\infty,1}) = \pi_{\mathrm{HT}}^*x, \hspace{1cm} HT_{\mathcal{A}}(\alpha_{\infty,2}) = \pi_{\mathrm{HT}}^*(-y).$$ 

By (\ref{isogenydiagram}), $\mathbf{z}$ satisfies the transformation property
\begin{equation}\label{ztransformationprop}
\begin{split}
\left(\begin{array}{ccc}a & b\\
c & d\\
\end{array}\right)^* \mathbf{z} = \frac{d\mathbf{z} + b}{c\mathbf{z} + a}
\end{split}
\end{equation}
for any $\left(\begin{array}{ccc}a & b\\
c & d\\
\end{array}\right) \in GL_2(\mathbb{Q}_p)$.

\subsection{The relative Hodge-Tate filtration}

By definition, $\mathcal{H}^{0,1}$ is the kernel of the map
\begin{equation}\label{HTmap1}\mathrm{HT}_{\mathcal{A}} : T_p\mathcal{A} \otimes_{\hat{\mathbb{Z}}_{p,Y}} \mathcal{O}_{\mathcal{Y}} \rightarrow \omega_{\mathcal{A}}\otimes_{\mathcal{O}_Y}\mathcal{O}_{\mathcal{Y}}.
\end{equation}
and extending by linearity to $\otimes_{\mathcal{O}_{\mathcal{Y}}} \mathcal{O}\mathbb{B}_{\mathrm{dR},\mathcal{Y}}^+$, then $\mathcal{H}^{0,1}\otimes_{\mathcal{O}_{\mathcal{Y}}}\mathcal{O}\mathbb{B}_{\mathrm{dR},\mathcal{Y}}^+$ is the kernel of the map
\begin{equation}\label{HTmap2}\mathrm{HT}_{\mathcal{A}} : T_p\mathcal{A} \otimes_{\hat{\mathbb{Z}}_{p,Y}} \mathcal{O}\mathbb{B}_{\mathrm{dR},\mathcal{Y}}^+ \rightarrow \omega_{\mathcal{A}}\otimes_{\mathcal{O}_Y}\mathcal{O}\mathbb{B}_{\mathrm{dR},\mathcal{Y}}^+
\end{equation}
which reduces modulo $\ker\theta$ (i.e. after extending by linearity to $\otimes_{\mathcal{O}\mathbb{B}_{\mathrm{dR},\mathcal{Y}}^+,\theta}\hat{\mathcal{O}}_{\mathcal{Y}}$) to the map
\begin{equation}\label{HTmap3}\mathrm{HT}_{\mathcal{A}} : T_p\mathcal{A} \otimes_{\hat{\mathbb{Z}}_{p,Y}} \hat{\mathcal{O}}_{\mathcal{Y}} \rightarrow \omega_{\mathcal{A}}\otimes_{\mathcal{O}_Y}\hat{\mathcal{O}}_{\mathcal{Y}}.
\end{equation}
Since the composition $\mathcal{O}_{\mathcal{Y}} \subset \mathcal{O}\mathbb{B}_{\mathrm{dR},\mathcal{Y}}^+ \xrightarrow{\theta} \hat{\mathcal{O}}_{\mathcal{Y}}$ is the natural $p$-adic completion map $\mathcal{O}_{\mathcal{Y}} \rightarrow \hat{\mathcal{O}}_{\mathcal{Y}}$, the map (\ref{HTmap3}) is just the $\otimes_{\mathcal{O}_Y}\hat{\mathcal{O}}_Y$-linear extension of (\ref{HTmap1}), and its kernel is given by $\mathcal{H}^{0,1}\otimes_{\mathcal{O}_{Y}}\hat{\mathcal{O}}_{\mathcal{Y}}$; in particular $$\mathcal{H}^{0,1}\otimes_{\mathcal{O}_{Y}}\hat{\mathcal{O}}_{\mathcal{Y}}|_{\mathcal{Y}_x} = (\alpha_{\infty,1} - \frac{1}{\hat{\mathbf{z}}}\cdot\alpha_{\infty,2})\hat{\mathcal{O}}_{\mathcal{Y}}$$
and 
$$\mathcal{H}^{0,1}\otimes_{\mathcal{O}_{Y}} \hat{\mathcal{O}}_{\mathcal{Y}}|_{\mathcal{Y}_y} = (\hat{\mathbf{z}}\cdot\alpha_{\infty,1} - \alpha_{\infty,2})\hat{\mathcal{O}}_{\mathcal{Y}}|_{\mathcal{Y}_y},$$ 
where $\hat{\mathbf{z}} \in \hat{\mathcal{O}}_{\mathcal{Y}}(\mathcal{Y})$ is the image of $\mathbf{z} \in \mathcal{O}_{\mathcal{Y}}(\mathcal{Y})$ under the natural map $\mathcal{O}_{\mathcal{Y}}(\mathcal{Y}) \rightarrow \hat{\mathcal{O}}_{\mathcal{Y}}(\mathcal{Y})$. By construction, $\hat{\mathbf{z}}$ is the (reciprocal of the) \emph{fundamental period} as defined in \cite[Definition 1.2]{ChojeckiHansenJohansson} (since it is also defined by the relation (\ref{relation}) for the extended Hodge-Tate map (\ref{HTmap3})). When no confusion arises, we will also often conflate $\hat{\mathbf{z}}$ and $\mathbf{z}$, as they are the same section under the natural $p$-adic completion map $\mathcal{O} \rightarrow \hat{\mathcal{O}}$. 

Caraiani and Scholze \cite{CaraianiScholze} define the \emph{relative Hodge filtration}
\begin{equation}\label{relativeHTfiltration1}0 \rightarrow \omega_{\mathcal{A}}^{-1}\otimes_{\mathcal{O}_Y}\hat{\mathcal{O}}_Y \rightarrow \mathcal{H}_{\text{\'{e}t}}^1(\mathcal{A})\otimes_{\hat{\mathbb{Z}}_{p,Y}} \hat{\mathcal{O}}_Y \rightarrow \omega_{\mathcal{A}}\otimes_{\mathcal{O}_Y} \hat{\mathcal{O}}_Y(-1) \rightarrow 0,
\end{equation}
or after twisting by $(1)$ and applying the isomorphism $\mathcal{H}_{\text{\'{e}t}}^1(\mathcal{A})(1) \cong T_p\mathcal{A}$ given by (\ref{dualisomorphism}), equivalently as
\begin{equation}\label{relativeHTfiltration2}0 \rightarrow \omega_{\mathcal{A}}^{-1}\otimes_{\mathcal{O}_Y}\hat{\mathcal{O}}_Y(1) \rightarrow T_p\mathcal{A} \otimes_{\hat{\mathbb{Z}}_{p,Y}} \hat{\mathcal{O}}_Y \rightarrow \omega_{\mathcal{A}}\otimes_{\mathcal{O}_Y} \hat{\mathcal{O}}_Y \rightarrow 0.
\end{equation}

Taking the specialization of (\ref{relativeHTfiltration2}) at a geometric point $Y = (A) \in Y(\overline{\mathbb{Q}}_p,\mathcal{O}_{\overline{\mathbb{Q}}_p})$, we obtain the usual Hodge-Tate exact sequence
\begin{equation}\label{absHTexactsequence}0 \rightarrow \mathrm{Lie}(A)(1) \xrightarrow{(HT_A)^{\vee}(1)} T_pA\otimes_{\mathbb{Z}_p}\mathbb{C}_p \xrightarrow{HT_A} \Omega_{A/\mathbb{C}_p}^1 \rightarrow 0.
\end{equation}

We have the following result on the integral Hodge-Tate complex.

\begin{theorem}[\cite{FarguesGenestierLafforgue}, Theorem II.1.1]\label{integraltheorem}Suppose $p > 2$. Taking an integral model $A/\mathcal{O}_{\mathbb{C}_p}$, we have a sequence
$$0 \rightarrow \mathrm{Lie}(A)(1) \xrightarrow{(HT_A)^{\vee}(1)} T_pA\otimes_{\mathbb{Z}_p}\mathcal{O}_{\mathbb{C}_p} \xrightarrow{HT_A} \Omega_{A/\mathcal{O}_{\mathbb{C}_p}}^1 \rightarrow 0$$
which is in general \emph{not} exact, but which is a complex with cohomology groups killed by $p^{1/(p-1)}$.
\end{theorem}

\begin{proposition}\label{HTarrows}The penultimate arrow in (\ref{relativeHTfiltration1}) is given by $HT_{\mathcal{A}}$ from (\ref{HTmap2}), and its kernel is given by the image of
\begin{equation}\label{HTinclusion}(HT_{\mathcal{A}})^{\vee}(1) : \omega_{\mathcal{A}}^{-1}\otimes_{\mathcal{O}_Y}\hat{\mathcal{O}}_Y(1) \hookrightarrow T_p\mathcal{A}\otimes_{\hat{\mathbb{Z}}_{p,Y}}\hat{\mathcal{O}}_Y.
\end{equation}
In particular, since $\mathcal{H}^{0,1}\otimes_{\mathcal{O}_{Y}}\hat{\mathcal{O}}_{Y}$ is the kernel of (\ref{HTmap2}), we have 
\begin{equation}\label{HTinclusion2}(HT_{\mathcal{A}})^{\vee}(1)(\omega_{\mathcal{A}}^{-1}\otimes_{\mathcal{O}_{Y}}\hat{\mathcal{O}}_{Y}(1))= \mathcal{H}^{0,1}\otimes_{\mathcal{O}_{Y}}\hat{\mathcal{O}}_{Y}
\end{equation}
and
\begin{equation}\label{HTinclusion3}(HT_{\mathcal{A}})^{\vee}(1)(\omega_{\mathcal{A}}^{-1}\otimes_{\mathcal{O}_{Y}}\mathcal{O}\mathbb{B}_{\mathrm{dR}Y}^+(1)) = \mathcal{H}^{0,1}\otimes_{\mathcal{O}_Y}\mathcal{O}\mathbb{B}_{\mathrm{dR},Y}^+.
\end{equation}
\end{proposition}

\begin{proof}By (\ref{absHTexactsequence}), the penultimate arrow in (\ref{relativeHTfiltration1}), which is a map of locally finite free $\mathcal{O}_Y$-modules, specializes to $HT_A$ on geometric fibers. Hence by Nakayama's lemma, it specializes to $HT_{\mathcal{A}}$ on geometric stalks, and so the map must be $HT_{\mathcal{A}}$ since $Y_{\text{pro\'{e}t}}$ has enough points given by profinite covers of geometric points. Since the kernel of $HT_A$ in (\ref{absHTexactsequence}) is given by the image of $(HT_A)^{\vee}(1)$, the same stalk-wise argument works to show that the kernel of (\ref{HTmap2}) is given by the image of $(HT_{\mathcal{A}})^{\vee}(1)$. This immediately implies (\ref{HTinclusion2}). For the final claim, we again can check on stalks at profinite covers of geometric points: since $\mathcal{O}\mathbb{B}_{\mathrm{dR},Y}^+/(\ker\theta) \cong \hat{\mathcal{O}}_Y$, by (\ref{HTinclusion2}) and Nakayama's lemma (since the image of $\ker\theta$ in the stalk is the maximal ideal of the stalk of $\mathcal{O}\mathbb{B}_{\mathrm{dR},Y}^+$) we have (\ref{HTinclusion3}). 

\end{proof}

\subsection{The fake Hasse invariant}\label{split} Henceforth, we define the \emph{fake Hasse invariant} to be
\begin{equation}\label{fakeHasseinvariant}\frak{s} := HT_{\mathcal{A}}(\alpha_{\infty,2}) \in \Gamma(\mathcal{Y},\omega_{\mathcal{A}}),
\end{equation}
which is a nonvanishing section in $\omega_{\mathcal{A}}(\mathcal{Y}_y)$, and so trivializes $\omega_{\mathcal{A}}$ on $\mathcal{Y}_y$. It is equivalently obtained by pulling back the global section $y$ of $\mathcal{O}_{\mathbb{P}^1}(1)$ along $\pi_{\mathrm{HT}} : \mathcal{Y} \rightarrow \mathbb{P}^1$.

We will briefly address a formal model of $\mathcal{Y}$ corresponding to the compactification $X$ of $Y$. As per the results of \cite{Scholze2}, there is a pro\'{e}tale cover $\mathcal{X} \rightarrow X$ extending $\mathcal{Y} \rightarrow Y$, which also makes $\mathcal{X}$ a preperfectoid space over $\mathrm{Spa}(\mathbb{Q}_p,\mathbb{Z}_p)$ (the fact that the latter map is indeed defined over $\mathrm{Spa}(\mathbb{Q}_p,\mathbb{Z}_p)$ is explained in the ensuing remarks after \cite[Theorem 2.8]{ChojeckiHansenJohansson}). Let $\frak{X}$ be the formal model corresponding to $\mathcal{X}$ by applying \cite[Lemma II.1.1]{Scholze2} to the proper adic space $\mathcal{X}$. Then there is a line bundle $\omega_{\mathcal{A}}^+$ on $\frak{X}$ whose generic fiber is $\omega_{\mathcal{A}}$. In fact, by the same lemma in loc. cit., we have that
$$\frak{s} \in \Gamma(\frak{X},\omega_{\mathcal{A}}^+)$$
and in fact trivializes $\omega_{\mathcal{A}}^+$ on the adic-theoretic closure of $\mathcal{Y}_y$ in $\frak{X}$.

\subsection{Relative de Rham cohomology and the Hodge-de Rham filtration}
Recall the Poincar\'{e} sequence (of $\pi^{-1}\mathcal{O}_{Y}$-modules)
\begin{equation}\label{deRham}0 \rightarrow \mathcal{O}_{\mathcal{A}} \xrightarrow{d} \Omega_{\mathcal{A}/Y}^1 \rightarrow 0
\end{equation}
where $d : \mathcal{O}_{\mathcal{A}} \rightarrow \Omega_{\mathcal{A}/Y}^1$ is the tensorial exterior differential. Taking the first higher direct image under $\pi : \mathcal{A} \rightarrow Y$ of this (\ref{deRham}), we get the usual de Rham bundle $$\mathcal{H}_{\mathrm{dR}}^1(\mathcal{A}) := R^1\pi_{\mathrm{dR},*}\mathcal{O}_{\mathcal{A}} = R^1\pi_*(0 \rightarrow \mathcal{O}_{\mathcal{A}} \rightarrow \Omega_{\mathcal{A}/Y}^1 \rightarrow 0)$$
on $Y_{\text{pro\'{e}t}}$, which is a rank 2 vector bundle equipped with the usual \emph{Hodge-de Rham} filtration 
\begin{equation}\label{relativeHdRsequence}0 \rightarrow \omega_{\mathcal{A}} \rightarrow \mathcal{H}_{\mathrm{dR}}^1(\mathcal{A}/Y) \rightarrow \omega_{\mathcal{A}}^{-1} \rightarrow 0.
\end{equation}
and Gauss-Manin connection 
\begin{equation}\label{GMconnection}\nabla : \mathcal{H}_{\mathrm{dR}}^1(\mathcal{A}) \rightarrow \mathcal{H}_{\mathrm{dR}}^1(\mathcal{A})\otimes_{\mathcal{O}_Y} \Omega_Y^1
\end{equation}
which satisfies Griffiths transversality with respect to the Hodge-de Rham filtration.

We extend the filtration (\ref{relativeHdRsequence}) to $\mathcal{H}_{\mathrm{dR}}^1(\mathcal{A})\otimes_{\mathcal{O}_Y}\mathcal{O}\mathbb{B}_{\mathrm{dR},Y}^+$ by taking the convolution of the (decreasing) Hodge-de Rham filtration with the natural (decreasing) filtration on $\mathcal{O}\mathbb{B}_{\mathrm{dR},Y}^+$ as defined in Section \ref{periodsheaves}. That is, we define
$$\Fil^i(\mathcal{H}_{\mathrm{dR}}^1(\mathcal{A})\otimes_{\mathcal{O}_Y}\mathcal{O}\mathbb{B}_{\mathrm{dR},Y}^+) = \sum_{j+j' = i}\Fil^j\mathcal{H}_{\mathrm{dR}}^1(\mathcal{A})\otimes_{\mathcal{O}_Y}\Fil^{j'}\mathcal{O}\mathbb{B}_{\mathrm{dR},Y}^+.$$

 Since the natural connection on $\mathcal{O}\mathbb{B}_{\mathrm{dR},Y}^+$ extends $d : \mathcal{O}_Y \rightarrow \Omega_Y^1$, we can extend (\ref{GMconnection}) to $\mathcal{H}_{\mathrm{dR}}^1(\mathcal{A})$ by the Leibniz rule to get a connection
$$\nabla : \mathcal{H}_{\mathrm{dR}}^1(\mathcal{A})\otimes_{\mathcal{O}_{Y}}\mathcal{O}\mathbb{B}_{\mathrm{dR},Y}^+ \rightarrow \mathcal{H}_{\mathrm{dR}}^1(\mathcal{A})\otimes_{\mathcal{O}_{Y}}\mathcal{O}\mathbb{B}_{\mathrm{dR},Y}^+
\otimes_{\mathcal{O}_{Y}}\Omega_{Y}^1$$
of filtered locally free $\mathcal{O}\mathbb{B}_{\mathrm{dR},Y}^+$-modules with integrable connection satisfying Griffiths transversality with respect to the convolution filtration on $\mathcal{H}_{\mathrm{dR}}^1(\mathcal{A})\otimes_{\mathcal{O}_Y}\mathcal{O}\mathbb{B}_{\mathrm{dR},Y}^+$. Here, recall that given a locally free $\mathcal{O}$-module with decreasing filtration $\mathcal{M}$ and a connection $\nabla : \mathcal{M} \rightarrow \mathcal{M} \otimes_{\mathcal{O}} \Omega^1$, for $\nabla$ to satisfy Griffiths transversality means that
$$\nabla(\Fil^i\mathcal{M}) \subset \Fil^{i-1}\mathcal{M} \otimes_{\mathcal{O}} \Omega^1.$$

\subsection{Relative $p$-adic de Rham comparison theorem}
Scholze's relative de Rham comparison theorem (\cite[Theorem 8.8]{Scholze}) says that we have an isomorphism
\begin{equation}\label{comparison}T_p\mathcal{A}\otimes_{\hat{\mathbb{Z}}_{p,Y}}\mathcal{O}\mathbb{B}_{\mathrm{dR},Y} = \mathcal{H}_{\text{\'{e}t}}^1(\mathcal{A})\otimes_{\hat{\mathbb{Z}}_{p,Y}}\mathcal{O}\mathbb{B}_{\mathrm{dR},Y} \cong \mathcal{H}_{\mathrm{dR}}^1(\mathcal{A})\otimes_{\mathcal{O}_Y}\mathcal{O}\mathbb{B}_{\mathrm{dR},Y}
\end{equation}
of sheaves on $Y_{\text{pro\'{e}t}}$, compatible with filtrations and connections.
The compatibility with the connections implies that the \emph{horizontal} sections (i.e. sections $f$ with $\nabla(f) = 0$) of the Gauss-Manin connection \begin{equation}\begin{split}\nabla : \mathcal{H}_{\mathrm{dR}}^1(\mathcal{A}) \otimes_{\mathcal{O}_{Y}} \mathcal{O}\mathbb{B}_{\mathrm{dR},Y} \cong T_p\mathcal{A}\otimes_{\hat{\mathbb{Z}}_{p,Y}}\mathcal{O}\mathbb{B}_{\mathrm{dR},Y} &\rightarrow  T_p\mathcal{A}\otimes_{\hat{\mathbb{Z}}_{p,Y}}\mathcal{O}\mathbb{B}_{\mathrm{dR},Y}\otimes_{\mathcal{O}_{Y}}\Omega_{Y}^1\\
&\cong \mathcal{H}_{\mathrm{dR}}^1(\mathcal{A}) \otimes_{\mathcal{O}_{Y}} \mathcal{O}\mathbb{B}_{\mathrm{dR},Y}\otimes_{\mathcal{O}_{Y}}\Omega_{Y}^1
\end{split}
\end{equation}
are given by $T_p\mathcal{A}\otimes_{\hat{\mathbb{Z}}_{p,Y}}\mathbb{B}_{\mathrm{dR},Y}$.

\subsection{$\mathcal{O}\mathbb{B}_{\mathrm{dR}}^+$-modules and $\mathbb{B}_{\mathrm{dR}}^+$-local systems}
Recall that by \cite[Theorem 7.2]{Scholze}, there is an equivalence between the category of $\mathbb{B}_{\mathrm{dR}}^+$-local systems and $\mathcal{O}\mathbb{B}_{\mathrm{dR}}^+$-modules.  By \cite[Proposition 2.2.3]{CaraianiScholze} (see also \cite[Proposition 7.9]{Scholze}), we have that the $\mathbb{B}_{\mathrm{dR}}^+$-local systems
$$\mathbb{M} := (\mathcal{H}_{\text{\'{e}t}}^1(\mathcal{A})\otimes_{\hat{\mathbb{Z}}_{p,Y}}\mathcal{O}\mathbb{B}_{\mathrm{dR},Y}^+)^{\nabla = 0} = \mathcal{H}_{\text{\'{e}t}}^1(\mathcal{A})\otimes_{\hat{\mathbb{Z}}_{p,Y}}\mathbb{B}_{\mathrm{dR},Y}^+$$
and 
$$\mathbb{M}_0 := (\mathcal{H}_{\mathrm{dR}}^1(\mathcal{A})\otimes_{\mathcal{O}_Y}\mathcal{O}\mathbb{B}_{\mathrm{dR},Y}^+)^{\nabla = 0}$$
satisfy 
\begin{equation}\label{latticeinclusion}\mathbb{M}_0 \subset \mathbb{M}
\end{equation}
and thus, after tensoring with $\otimes_{\mathbb{B}_{\mathrm{dR},Y}}^+\mathcal{O}\mathbb{B}_{\mathrm{dR},Y}^+$, this induces an inclusion of $\mathcal{O}\mathbb{B}_{\mathrm{dR},Y}^+$-modules which we denote by $\iota_{\mathrm{dR}}$:
\begin{equation}\label{OBdRlocalsystem}\mathcal{H}_{\mathrm{dR}}^1(\mathcal{A})\otimes_{\mathcal{O}_Y}\mathcal{O}\mathbb{B}_{\mathrm{dR},Y}^+= \mathbb{M}_0\otimes_{\mathbb{B}_{\mathrm{dR},Y}^+} \mathcal{O}\mathbb{B}_{\mathrm{dR},Y}^+ \overset{\iota_{\mathrm{dR}}}{\subset} \mathbb{M}\otimes_{\mathbb{B}_{\mathrm{dR},Y}^+} \mathcal{O}\mathbb{B}_{\mathrm{dR},Y}^+ = \mathcal{H}_{\text{\'{e}t}}^1(\mathcal{A})\otimes_{\hat{\mathbb{Z}}_{p,Y}}\mathcal{O}\mathbb{B}_{\mathrm{dR},Y}^+.
\end{equation}
Using the isomorphism $\mathcal{H}_{\text{\'{e}t}}^1(\mathcal{A})\cong T_p\mathcal{A}(-1)$ furnished by (\ref{dualisomorphism}), this induces an inclusion 
\begin{equation}\label{OBdRlocalsystem2}\begin{split}\mathcal{H}_{\mathrm{dR}}^1(\mathcal{A})\otimes_{\mathcal{O}_Y}\mathcal{O}\mathbb{B}_{\mathrm{dR},Y}^+ (1) &= \mathbb{M}_0\otimes_{\mathbb{B}_{\mathrm{dR},Y}^+} \mathcal{O}\mathbb{B}_{\mathrm{dR},Y}^+(1) \\
&\overset{\iota_{\mathrm{dR}}}{\subset} \mathbb{M}\otimes_{\mathbb{B}_{\mathrm{dR},Y}^+} \mathcal{O}\mathbb{B}_{\mathrm{dR},Y}^+(1) = T_p\mathcal{A}\otimes_{\hat{\mathbb{Z}}_{p,Y}}\mathcal{O}\mathbb{B}_{\mathrm{dR},Y}^+,
\end{split}
\end{equation}
which, upon evaluating on $\mathcal{Y}\rightarrow Y$ and using $\hat{\mathbb{Z}}_{p,\mathcal{Y}}(1) \cong \hat{\mathbb{Z}}_{p,\mathcal{Y}}\cdot t$, induces
\begin{equation}\label{OBdRlocalsystem2}\begin{split}\mathcal{H}_{\mathrm{dR}}^1(\mathcal{A})\otimes_{\mathcal{O}_Y}\mathcal{O}\mathbb{B}_{\mathrm{dR},\mathcal{Y}}^+ \cdot t &= \mathbb{M}_0\otimes_{\mathbb{B}_{\mathrm{dR},Y}^+} \mathcal{O}\mathbb{B}_{\mathrm{dR},\mathcal{Y}}^+\cdot t \\
&\overset{\iota_{\mathrm{dR}}}{\subset} \mathbb{M}\otimes_{\mathbb{B}_{\mathrm{dR},Y}^+} \mathcal{O}\mathbb{B}_{\mathrm{dR},\mathcal{Y}}^+\cdot t = T_p\mathcal{A}\otimes_{\hat{\mathbb{Z}}_{p,Y}}\mathcal{O}\mathbb{B}_{\mathrm{dR},\mathcal{Y}}^+ \xrightarrow{\alpha_{\infty}^{-1}}\mathcal{O}\mathbb{B}_{\mathrm{dR},\mathcal{Y}}^{+,\oplus 2}.
\end{split}
\end{equation}

\section{Construction of the $p$-adic Maass-Shimura operator}

\subsection{A pro\'{e}tale local description of $\mathcal{O}\mathbb{B}_{\mathrm{dR}}^{(+)}$} Recall that for an affinoid $\mathcal{U} \subset \mathcal{Y}^{\mathrm{ss}}$ we have an \'{e}tale morphism of adic spaces $\pi_{\mathrm{HT}} : \mathcal{U} \rightarrow \Omega^2 \subset \mathbb{P}^1$. Using the $p$-adic exponential, this adic-locally defines an \'{e}tale morphism 
$$\mathcal{U} \rightarrow \mathbb{T} = \mathrm{Spa}(\mathbb{Q}_p\langle T,T^{-1}\rangle,\mathbb{Z}_p\langle T,T^{-1}\rangle),$$
and taking finite \'{e}tale covers 
$$\mathbb{T}^{1/p^m} = \mathrm{Spa}(\mathbb{Q}_p\langle T^{1/p^m},T^{-1/p^m}\rangle,\mathbb{Z}_p\langle T^{1/p^m},T^{-1/p^m}\rangle)$$
of $\mathbb{T}$ to extract $p^{\mathrm{th}}$-power roots of the $p$-adic exponential, we can extend to this to \'{e}tale morphisms $\mathcal{U}\rightarrow \mathbb{T}$ defined on larger \'{e}tale neighborhoods. In this section, we will deal with the more general situation where we have an arbitrary adic space $\mathcal{U}$ over $\mathrm{Spa}(\mathbb{Q}_p,\mathbb{Z}_p)$ with a fixed \'{e}tale-locally defined \'{e}tale morphism $\pi : \mathcal{U} \rightarrow \mathbb{T}$, and specialize in Section \ref{q'coordinates} to the situation described above.

Let $K/\mathbb{Q}_p$ denote a perfectoid field, with ring of integers $\mathcal{O}_K$. 
\begin{equation}\label{infinitytorus}\tilde{\mathbb{T}} = \varprojlim_m \mathbb{T}^{1/p^m} \in \mathbb{T}_{\text{pro\'{e}t}}.
\end{equation}
Letting $\tilde{\mathbb{T}}_K$ denote the base change to $\mathrm{Spa}(K,\mathcal{O}_K)$, this is an affinoid perfectoid object of $\mathbb{T}_{K,\text{pro\'{e}t}}$. Using our given \'{e}tale-locally defined \'{e}tale map $\mathcal{U} \rightarrow \mathbb{T}$, let 
$$\tilde{\mathcal{U}} = \mathcal{U}\times_{\mathbb{T}}\tilde{\mathbb{T}} \in \mathcal{U}_{\text{pro\'{e}t}}.$$
Let $\tilde{\mathcal{U}}_K \in \mathcal{U}_{K,\text{pro\'{e}t}}$ denote the base change to $\mathrm{Spa}(K,\mathcal{O}_K)$, so that $\tilde{\mathcal{U}}_K$ is an affinoid perfectoid object in $\tilde{\mathcal{U}}_{K,\text{pro\'{e}t}}$. The following proposition and its proof are copied almost verbatim from \cite{Scholze}.

\begin{proposition}[\cite{Scholze}, Proposition 6.10]\label{localdescription} We have a natural isomorphism of sheaves on the localized site $\mathcal{U}_{\text{pro\'{e}t}}/\tilde{\mathcal{U}}$
$$\mathbb{B}_{\mathrm{dR},\tilde{\mathcal{U}}}^+\llbracket X\rrbracket \xrightarrow{\sim} \mathcal{O}\mathbb{B}_{\mathrm{dR},\tilde{\mathcal{U}}}^+$$
sending $X \mapsto T\otimes 1 - 1\otimes[T^{\flat}]$. In particular, we have
$$\mathbb{B}_{\mathrm{dR},\tilde{\mathcal{U}}}^+\llbracket X\rrbracket [t^{-1}] \xrightarrow{\sim} \mathcal{O}\mathbb{B}_{\mathrm{dR},\tilde{\mathcal{U}}}$$
\end{proposition}
\begin{proof}As in loc. cit., the key ingredient of the argument is:
\begin{claim}\label{claim} There is a \emph{unique} morphism $$\mathcal{O}_{\tilde{\mathcal{U}}} \rightarrow \mathbb{B}_{\mathrm{dR},\tilde{\mathcal{U}}}^+\llbracket X\rrbracket$$ 
sending $T \mapsto [T^{\flat}] + X$ and such that the resulting map
$$\mathcal{O}_{\tilde{\mathcal{U}}} \rightarrow \mathbb{B}_{\mathrm{dR},\tilde{\mathcal{U}}}^+\llbracket X\rrbracket/\ker \theta = \hat{\mathcal{O}}_{\tilde{\mathcal{U}}}$$
is the natural inclusion. 
\end{claim}
Once we have shown this claim, we get a natural map
$$(\mathcal{O}_{\mathcal{U}}\otimes_{W(\kappa)}W(\hat{\mathcal{O}}_{\mathcal{U}^{\flat}}^+))|_{\tilde{\mathcal{U}}} \rightarrow \mathbb{B}_{\mathrm{dR},\mathcal{U}}^+|_{\tilde{\mathcal{U}}}\llbracket X\rrbracket$$
which induces the inverse of the map in the statement upon passing to the $(\ker\theta)$-adic completion. 

In order to show the claim, we recall the following lemmas from \cite{Scholze}.
\begin{lemma}[\cite{Scholze}, Lemma 6.11]\label{lemma1} Let $(R,R^+)$ be a perfectoid affinoid $(k,\mathcal{O}_k)$-algebra, and let $S$ be a finitely generated $\mathcal{O}_K$-algebra. Then any morphism
$$f : S \rightarrow \mathbb{B}_{\mathrm{dR}}^+(R,R^+)\llbracket X\rrbracket$$
such that $\theta(f(S)) \subset R^+$ extends to the $p$-adic completion of $S$.
\end{lemma}

\begin{lemma}[cf. \cite{Scholze}, Lemma 6.12]\label{lemma2} Let $V = \mathrm{Spa}(R,R^+)$ be an affinoid adic space of finite type over $\mathrm{Spa}(W(\kappa)[p^{-1}],W(\kappa))$ with an \'{e}tale map $V \rightarrow \mathbb{T}$. Then there exists a finitely generated $W(\kappa)[T]$-algebra $R_0^+$, such that $R_0 = R_0^+[p^{-1}]$ is \'{e}tale over 
$$W(\kappa)[p^{-1}][T]$$
and $R^+$ is the $p$--adic completion of $R_0^+$.
\end{lemma}

\begin{proof}[Proof of Claim \ref{claim}]We again follow the argument in loc. cit. Note that we have a map
\begin{equation}\label{map}W(\kappa)[p^{-1}][T] \rightarrow \mathbb{B}_{\mathrm{dR},\mathcal{U}}^+|_{\tilde{\mathcal{U}}}\llbracket X\rrbracket
\end{equation}
sending $T$ to $[T^{\flat}] + X$. Now take any affinoid perfectoid $V \in \mathcal{U}_{\text{pro\'{e}t}}/\tilde{\mathcal{U}}_K$ writing it as a union of inverse limits of affinoid $V_i \in \mathcal{U}_{\text{pro\'{e}t}}$. Then $\mathcal{O}_{\mathcal{U}}(V) = \varinjlim_i\mathcal{O}_{\mathcal{U}}(V_i)$. Applying Lemma \ref{lemma2} to $V_i = \mathrm{Spa}(R_i,R_i^+) \rightarrow \mathcal{U} \rightarrow \mathbb{T}$, we have algebras $R_{i0}^+$ whose generic fibers $R_{i0}$ are \'{e}tale over $W(\kappa)[p^{-1}][T]$. By Hensel's lemma, the map (\ref{map}) \emph{uniquely} lifts to a map $R_{i0} \rightarrow \mathbb{B}_{\mathrm{dR},\mathcal{U}}^+(V)\llbracket X\rrbracket$ and in particular we get a unique lifting of $R_{i0}^+$. This map in turn extends by Lemma \ref{lemma1} to the $p$-adic completion $R_i^+ = \mathcal{O}_{\mathcal{U}}^+(V_i)$, and hence we get a unique lifting $\mathcal{O}_{\mathcal{U}}(V_i) \rightarrow \mathbb{B}_{\mathrm{dR},\mathcal{U}}^+\llbracket X\rrbracket$. Taking the direct limit, we get a unique lifting $\mathcal{O}_{\mathcal{U}}(V) \rightarrow \mathbb{B}_{\mathrm{dR},\mathcal{U}}^+\llbracket X\rrbracket$. Since $V$ was arbitrary and the affinoid perfectoids form a basis of the pro\'{e}tale site, and invoking the fact that $\mathcal{O}_{\mathcal{U}}$ is a sheaf on $\mathcal{U}_{\text{pro\'{e}t}}$ (and so satisfies the gluing axiom), we have proven the claim.
\end{proof}

\end{proof}

\subsection{The de Rham fundamental periods}\label{deRhamperiod}

Let
\begin{equation}\label{generator}t := \log[\epsilon] \in \Fil^1\mathbb{B}_{\mathrm{dR}}^+(\mathcal{Y})
\end{equation}
and
$$\epsilon := (\langle \alpha_{\infty,2},\alpha_{\infty,1}\rangle_0, \langle \alpha_{\infty,2},\alpha_{\infty,1}\rangle_1,\ldots) \in \hat{\mathbb{Z}}_p(1)(\mathcal{Y})$$
where $\langle \cdot, \cdot \rangle$ is the Weil pairing \ref{Weil} (viewed as a pairing $T_p\mathcal{A} \times T_p\mathcal{A} \rightarrow \hat{\mathbb{Z}}_{p,\mathcal{Y}}(1)$ via the universal principal polarization $\mathcal{A}\cong \check{\mathcal{A}}$), and $\langle \cdot, \cdot \rangle_n$ denotes its image under the projection $\hat{\mathbb{Z}}_{p,\mathcal{Y}}(1) \rightarrow \mu_{p^n}$. Hence $t$ is a generator (i.e. a non-vanishing global section) of $\Fil^1\mathbb{B}_{\mathrm{dR}}^+(\mathcal{Y})$ over $\mathbb{B}_{\mathrm{dR}}^+(\mathcal{Y})$.

Note also that $t$ gives a natural isomorphism
\begin{equation}\label{Tatetwist}\hat{\mathbb{Z}}_{p,\mathcal{Y}}(1) \cong \hat{\mathbb{Z}}_{p,\mathcal{Y}}\cdot t \hspace{.5cm} :
\hspace{.5cm} \gamma \mapsto \log[\gamma]
\end{equation}
of $\hat{\mathbb{Z}}_{p,\mathcal{Y}}$-modules on $\mathcal{Y}_{\text{pro\'{e}t}}$, and thus by the isomorphism (\ref{dualisomorphism}) induced by the Weil pairing, we have a natural isomorphism
\begin{equation}\label{etaleTateisomorphism}\mathcal{H}_{\text{\'{e}t}}^1(\mathcal{A}) \cong T_p\mathcal{A}(-1) \cong T_p\mathcal{A}\cdot t^{-1} \underset{\alpha_{\infty}^{-1}}{\xrightarrow{\sim}} \left(\hat{\mathbb{Z}}_{p,\mathcal{Y}}\cdot t^{-1}\right)^{\oplus 2}.
\end{equation} 

Using (\ref{Tatetwist}) and (\ref{latticeinclusion}), we write the global section of $\omega_{\mathcal{A}}$ given by the fake Hasse invariant $$\frak{s} := HT_{\mathcal{A}}(\alpha_2) \in \mathcal{H}_{\mathrm{dR}}^1(\mathcal{A})\otimes_{\mathcal{O}_{\mathcal{Y}}}\mathcal{O}\mathbb{B}_{\mathrm{dR}}^+(\mathcal{Y})\subset \mathcal{H}_{\text{\'{e}t}}^1(\mathcal{A})\otimes_{\hat{\mathbb{Z}}_{p,\mathcal{Y}}}\mathcal{O}\mathbb{B}_{\mathrm{dR},\mathcal{Y}}^+(\mathcal{Y}) {\xrightarrow{\sim}}(\mathcal{O}\mathbb{B}_{\mathrm{dR},\mathcal{Y}}^+\cdot t^{-1})^{\oplus 2}$$
as 
\begin{equation}\label{x'y'definition}\frak{s} =  \mathbf{x}_{\mathrm{dR}}\alpha_{\infty,1}/t + \mathbf{y}_{\mathrm{dR}}\alpha_{\infty,2}/t
\end{equation}
where $\mathbf{x}_{\mathrm{dR}},\mathbf{y}_{\mathrm{dR}} \in \mathcal{O}\mathbb{B}_{\mathrm{dR}}^+(\mathcal{Y})$. 

\begin{definition}We call $\mathbf{x}_{\mathrm{dR}},\mathbf{y}_{\mathrm{dR}} \in \mathcal{O}\mathbb{B}_{\mathrm{dR}}^+(\mathcal{Y})$ the \emph{de Rham fundamental periods}.
\end{definition}

\begin{remark}By the proof of Proposition 2.2.5 of \cite{CaraianiScholze}, we see that the map 
$$\omega_{\mathcal{A}}\otimes_{\mathcal{O}_Y}\mathcal{O}\mathbb{B}_{\mathrm{dR},\mathcal{Y}}^+ \rightarrow \mathcal{H}_{\mathrm{dR}}^1(\mathcal{A})\otimes_{\mathcal{O}_Y}\mathcal{O}\mathbb{B}_{\mathrm{dR},\mathcal{Y}}^+ \rightarrow \mathcal{H}_{\text{\'{e}t}}^1(\mathcal{A})\otimes_{\hat{\mathbb{Z}}_{p,Y}}\hat{\mathcal{O}}_{\mathcal{Y}} \xrightarrow{HT_{\mathcal{A}}} \omega_{\mathcal{A}}(-1)\otimes_{\mathcal{O}_Y} \mathcal{O}\mathbb{B}_{\mathrm{dR},Y}^+$$ 
obtained from reducing the right hand side of (\ref{OBdRlocalsystem}) modulo $\ker\theta$ is just the 0 map. Hence $\mathbf{x}_{\mathrm{dR}},\mathbf{y}_{\mathrm{dR}} \in \ker\theta \subset\mathcal{O}\mathbb{B}_{\mathrm{dR}}^+(\mathcal{Y})$.
\end{remark}

\subsection{$GL_2(\mathbb{Q}_p)$-transformation properties of the de Rham fundamental periods}
For any $\gamma \in GL_2(\mathbb{Q}_p)$, let 
$$p^n\gamma = \left(\begin{array}{ccc} a & b \\
c & d\\
\end{array}\right)$$
where $n \in \mathbb{Z}$ is such that $p^n\gamma \in M_2(\mathbb{Z}_p)$ but $p^{n-1} \not\in M_2(\mathbb{Z}_p)$. Retain the same notation as in Section \ref{action}, so that we have
\begin{align*}&\alpha_{\infty,1}\cdot \left(\begin{array}{ccc} a & b \\
c & d\\
\end{array}\right) = a(\check{\phi}_*)^{-1}(\alpha_{\infty,1}) + c(\check{\phi}_*)^{-1}(\alpha_{\infty,2})\\
& \alpha_{\infty,2}\cdot \left(\begin{array}{ccc} a & b \\
c & d\\
\end{array}\right) = b(\check{\phi}_*)^{-1}(\alpha_{\infty,1}) + d(\check{\phi}_*)^{-1}(\alpha_{\infty,2}).
\end{align*}
Note that when $\gamma \in GL_2(\mathbb{Z}_p)$, then $(\check{\phi}_*)^{-1}$ is just the identity. From the calculations
\begin{equation}\begin{split}\label{stransformationprop}\gamma^*\frak{s} = \left(\begin{array}{ccc} a & b \\
c & d\\
\end{array}\right)^*HT_{\mathcal{A}}(\alpha_{\infty,2}) &= HT_{\mathcal{A}}(b(\check{\phi}_*)^{-1}(\alpha_{\infty,1}) + d(\check{\phi}_*)^{-1}(\alpha_{\infty,2})) \\
&= (b\mathbf{z}^{-1} + d)HT_{\mathcal{A}}((\check{\phi}_*)^{-1}(\alpha_{\infty,2}))\\
&= (b\mathbf{z}^{-1} + d)(\check{\phi}^{-1})^*\frak{s}
\end{split}
\end{equation}
and
\begin{equation}\begin{split}&\gamma^*t= \left(\begin{array}{ccc} a & b \\
c & d\\
\end{array}\right)^*\log[(\langle \alpha_{\infty,2},\alpha_{\infty,1}\rangle_0,\ldots] \\
&= \log[(\langle a(\check{\phi}_*)^{-1}(\alpha_{\infty,1}) + c(\check{\phi}_*)^{-1}(\alpha_{\infty,2}),b(\check{\phi}_*)^{-1}(\alpha_{\infty,1}) + d(\check{\phi}_*)^{-1}(\alpha_{\infty,2})\rangle_0,\ldots)] \\
&= (bc-ad)\log[(\langle (\check{\phi}_*)^{-1}(\alpha_{\infty,2}),(\check{\phi}_*)^{-1}(\alpha_{\infty,1})\rangle_0,\ldots)] \\
&= (bc-ad)\log[(\langle \alpha_{\infty,2},\alpha_{\infty,1}\rangle_0,\ldots)] \\
&= (bc-ad)t,
\end{split}
\end{equation}
where the last equality follows from the functoriality of the Weil pairing, as follows. Given an isogeny $\phi : A \rightarrow A'$ with dual isogeny $\check{\phi} : \check{A'} \rightarrow \check{A}$, and the Weil pairings $\langle\cdot,\cdot\rangle$ and $\langle\cdot,\cdot\rangle'$ associated with $A$ and $A'$, we have
$$\langle \check{\phi}_*(x),y \rangle = \langle x, \phi_*(y) \rangle'.$$
Now by the construction of our $\phi$ (using our definition of $p^n\gamma$), we see that either $(\check{\phi}_*)^{-1}(\alpha_{\infty,1}) = \phi_*(\alpha_{\infty,1})$ or $(\check{\phi}_*)^{-1}(\alpha_{\infty,2}) = \phi_*(\alpha_{\infty,2})$, and hence 
\begin{align*}\langle (\check{\phi}_*)^{-1}(\alpha_{\infty,i}),(\check{\phi}_*)^{-1}(\alpha_{\infty,3-i})\rangle &= \langle \phi_*(\alpha_{\infty,i}),(\check{\phi}_*)^{-1}(\alpha_{\infty,3-i})\rangle \\
&= \langle \alpha_{\infty,i},\check{\phi}_*(\check{\phi}_*)^{-1}(\alpha_{\infty,3-i})\rangle \\
&= \langle \alpha_{\infty,i},\alpha_{\infty,3-i}\rangle
\end{align*}
for $i = 1$ or 2. 

We have:
\begin{proposition}Let $\gamma \in GL_2(\mathbb{Q}_p)$ and
$$p^n\gamma = \left(\begin{array}{ccc} a & b \\
c & d\\
\end{array}\right)$$
where $n \in \mathbb{Z}$ is such that $p^n\gamma \in M_2(\mathbb{Z}_p)$ but $p^{n-1} \not\in M_2(\mathbb{Z}_p)$.Then $\mathbf{x}_{\mathrm{dR}}/t$ and $\mathbf{y}_{\mathrm{dR}}/t$ satisfy the transformation laws
\begin{equation}\label{x'transformationprop}\gamma^*(\mathbf{x}_{\mathrm{dR}}/t) = (d\mathbf{x}_{\mathrm{dR}}-b\mathbf{y}_{\mathrm{dR}})(b\mathbf{z}^{-1}+d)(bc-ad)^{-1} = (b{\mathbf{z}_{\mathrm{dR}}}^{-1} + d)(b\mathbf{z}^{-1} + d)(bc-ad)^{-1}(\mathbf{x}_{\mathrm{dR}}/t),
\end{equation}
and
\begin{equation}\label{y'transformationprop}\gamma^*(\mathbf{y}_{\mathrm{dR}}/t) = (-c\mathbf{x}_{\mathrm{dR}} + a\mathbf{y}_{\mathrm{dR}})(b\mathbf{z}^{-1} + d)(bc-ad)^{-1} = (c{\mathbf{z}_{\mathrm{dR}}} + a)(b\mathbf{z}^{-1} + d)(bc-ad)^{-1}(\mathbf{y}_{\mathrm{dR}}/t),
\end{equation}
\end{proposition}

\begin{proof}Let 
$$X := \gamma^*\mathbf{x}_{\mathrm{dR}}, \hspace{2cm} Y := \gamma^*\mathbf{y}_{\mathrm{dR}}.$$
Then by definition, 
$$\frak{s} = HT_{\mathcal{A}}(\alpha_{\infty,2}) = \mathbf{x}_{\mathrm{dR}}\alpha_{\infty,1} t^{-1} + \mathbf{y}_{\mathrm{dR}}\alpha_{\infty,2} t^{-1}
$$
and so
\begin{align*}&\gamma^*\frak{s} \\
&= HT_{\mathcal{A}}\left(b\frac{(\check{\phi}_*)^{-1}(\alpha_{\infty,1})}{t} + d\frac{(\check{\phi}_*)^{-1}(\alpha_{\infty,2})}{t}\right) \\
&= (b\mathbf{z}^{-1}+d)\left(\mathbf{x}_{\mathrm{dR}}\frac{(\check{\phi}_*)^{-1}(\alpha_{\infty,1})}{t}+ \mathbf{y}_{\mathrm{dR}}\frac{(\check{\phi}_*)^{-1}(\alpha_{\infty,2})}{t}\right).
\end{align*}
From (\ref{stransformationprop}), we have
\begin{align*}&(b\mathbf{z}^{-1} + d)\left(\mathbf{x}_{\mathrm{dR}}\frac{(\check{\phi}_*)^{-1}(\alpha_{\infty,1})}{t} + \mathbf{y}_{\mathrm{dR}}\frac{(\check{\phi}_*)^{-1}(\alpha_{\infty,2})}{t}\right) = \gamma^*\frak{s} \\
&= X\left(a\frac{(\check{\phi}_*)^{-1}(\alpha_{\infty,1})}{t} + c\frac{(\check{\phi}_*)^{-1}(\alpha_{\infty,2})}{t}\right) (bc-ad)^{-1} \\
&+ Y\left(b\frac{(\check{\phi}_*)^{-1}(\alpha_{\infty,1})}{t} + d\frac{(\check{\phi}_*)^{-1}(\alpha_{\infty,2})}{t}\right)  (bc-ad)^{-1}\\
&= (aX + bY)(bc-ad)^{-1}\frac{(\check{\phi}_*)^{-1}(\alpha_{\infty,1})}{t} + (cX + dY)(bc-ad)^{-1}\frac{(\check{\phi}_*)^{-1}(\alpha_{\infty,2})}{t}
\end{align*}
which, by equating the coefficients of $\frac{(\check{\phi}_*)^{-1}(\alpha_{\infty,1})}{t^{-1}}$ and $\frac{(\check{\phi}_*)^{-1}(\alpha_{\infty,2})}{t^{-1}}$, implies that 
$$X = -(d\mathbf{x}_{\mathrm{dR}} - b\mathbf{y}_{\mathrm{dR}})(b\mathbf{z}^{-1} + d), \hspace{2cm} Y = -(-c\mathbf{x}_{\mathrm{dR}} + a\mathbf{y}_{\mathrm{dR}})(b\mathbf{z}^{-1} + d)$$
as desired.
\end{proof}

\subsection{The $p$-adic Legendre relation}

\begin{proposition}\label{compositionproposition}
The composition
$$HT_{\mathcal{A}}\circ \iota_{\mathrm{dR}} : \omega_{\mathcal{A}}\otimes_{\mathcal{O}_{Y}}\mathcal{O}\mathbb{B}_{\mathrm{dR},\mathcal{Y}}^+\cdot t \rightarrow \omega_{\mathcal{A}}\otimes_{\mathcal{O}_{Y}}\mathcal{O}\mathbb{B}_{\mathrm{dR},\mathcal{Y}}^+.$$
is the natural inclusion.
\end{proposition}

\begin{proof}We are considering the composition
\begin{equation}\label{composition}\omega_{\mathcal{A}} \otimes_{\mathcal{O}_{\mathcal{Y}}} \mathcal{O}\mathbb{B}_{\mathrm{dR},\mathcal{Y}}^+\cdot t \subset \mathcal{H}_{\mathrm{dR}}^1(\mathcal{A})\otimes_{\mathcal{O}_Y}\mathcal{O}\mathbb{B}_{\mathrm{dR},\mathcal{Y}}^+\cdot t \overset{\iota_{\mathrm{dR}}}{\hookrightarrow} T_p\mathcal{A} \otimes_{\hat{\mathbb{Z}}_{p,Y}}\mathcal{O}\mathbb{B}_{\mathrm{dR},\mathcal{Y}}^+ \overset{HT_{\mathcal{A}}}{\twoheadrightarrow} \omega_{\mathcal{A}}\otimes_{\mathcal{O}_{Y}}\mathcal{O}\mathbb{B}_{\mathrm{dR},\mathcal{Y}}^+.
\end{equation}
Since $Y_{\text{pro\'{e}t}}/\mathcal{Y}$ has enough points given by profinite covers of geometric points, it suffices to check that the above map is the natural inclusion on stalks at profinite covers of geometric points. For all sheaves appearing in (\ref{composition}), the stalk at a profinite cover of a geometric point is just a profinite direct product of stalks at the base geometric point, so we are reduced to checking that (\ref{composition}) is the natural inclusion on stalks at geometric points. By Proposition \ref{localdescription}, the stalk of $\mathcal{O}\mathbb{B}_{\mathrm{dR},\mathcal{Y}}^+$ (on $Y_{\text{pro\'{e}t}}/\mathcal{Y}$) at a geometric point is isomorphic to $B_{\mathrm{dR}}^+\llbracket X\rrbracket $ for some indeterminate $X$, where $B_{\mathrm{dR}}^+$ is Fontaine's de Rham period ring. Taking the stalk (in the pro\'{e}tale site) of (\ref{composition}) at a geometric point $y = (A,\alpha) \in \mathcal{Y}(\overline{\mathbb{Q}}_p)$ (which by Proposition \ref{localdescription} also gives the data of a choice of embedding $\overline{\mathbb{Q}}_p \hookrightarrow B_{\mathrm{dR}}^+\llbracket X \rrbracket$ which modulo $X$ is equal to the natural embedding $\overline{\mathbb{Q}}_p \subset B_{\mathrm{dR}}^+$), we get
\begin{equation}\label{stalkcomposition}\begin{split}\omega_{\mathcal{A},y} \otimes_{\mathcal{O}_{Y,y}} B_{\mathrm{dR}}^+\llbracket X\rrbracket \cdot t \subset \mathcal{H}_{\mathrm{dR}}^1(\mathcal{A})_y\otimes_{\mathcal{O}_{Y,y}}\mathbb{B}_{\mathrm{dR}}^+\llbracket X\rrbracket \cdot t &\overset{\iota_{\mathrm{dR},y}}{\hookrightarrow} T_pA\otimes_{\mathbb{Z}_p}B_{\mathrm{dR}}^+\llbracket X\rrbracket \\
&\overset{HT_A}{\twoheadrightarrow} \omega_{\mathcal{A},y}\otimes_{\mathcal{O}_{Y,y}} B_{\mathrm{dR}}^+\llbracket X\rrbracket.
\end{split}
\end{equation}
Taking stalks at $y = (A,t)$ of the inclusion (\ref{latticeinclusion}), we get that there is a $B_{\mathrm{dR}}^+$-lattice 
\begin{equation}\label{stalk}M_0 \subset \mathcal{H}_{\mathrm{dR}}^1(\mathcal{A})_y\otimes_{\mathcal{O}_{Y,y}} B_{\mathrm{dR}}^+\llbracket X\rrbracket
\end{equation}
of rank 2 such that 
\begin{equation}\label{naturalinclusion}M_0\otimes_{B_{\mathrm{dR}}^+,i}B_{\mathrm{dR}}^+\llbracket X\rrbracket = \mathcal{H}_{\mathrm{dR}}^1(\mathcal{A})_y\otimes_{\mathcal{O}_{Y,y}} B_{\mathrm{dR}}^+\llbracket X\rrbracket,
\end{equation}
where $i : B_{\mathrm{dR}}^+\subset B_{\mathrm{dR}}^+\llbracket X\rrbracket$ is the natural inclusion, and also by (\ref{latticeinclusion}) we have
\begin{equation}\label{fiberinclusion}M_0 \subset T_pA\otimes_{\mathbb{Z}_p}B_{\mathrm{dR}}^+.
\end{equation}
Colmez has shown (\cite{Colmez}) that the $p$-adic de Rham comparison theorem
\begin{equation}\label{Colmezinclusion}H_{\mathrm{dR}}^1(A/\overline{\mathbb{Q}}_p)\otimes_{\overline{\mathbb{Q}}_p} B_{\mathrm{dR}}^+\cdot t \overset{\iota_0}{\hookrightarrow} T_pA\otimes_{\mathbb{Z}_p} B_{\mathrm{dR}}^+
\end{equation}
is induced by the inclusion 
\begin{align*}H_{\mathrm{dR}}^1(A/\overline{\mathbb{Q}}_p)\otimes_{\overline{\mathbb{Q}}_p} B_{\mathrm{dR}}^+ \subset \Hom_{\mathrm{B}_{\mathrm{dR}}^+}(T_pA\otimes_{\mathbb{Z}_p}B_{\mathrm{dR}}^+,B_{\mathrm{dR}}^+)& \cong H_{\text{\'{e}t}}^1(A,\mathbb{Z}_p) \otimes_{\mathbb{Z}_p} B_{\mathrm{dR}}^+ \\
&\cong T_pA\otimes_{\mathbb{Z}_p}B_{\mathrm{dR}}^+\cdot t^{-1}
\end{align*}
(where the last isomorphism is given by the Weil pairing (\ref{dualisomorphism})) induced by the non-degenerate $B_{\mathrm{dR}}^+$-bilinear ``$p$-adic Poincar\'{e} pairing''
\begin{equation}\label{padicperiodpairing}\langle\cdot,\cdot\rangle_p : (H_{\mathrm{dR}}^1(A/\overline{\mathbb{Q}}_p)\otimes_{\overline{\mathbb{Q}}_p}B_{\mathrm{dR}}^+\cdot t)\times(T_pA\otimes_{\mathbb{Z}_p}
B_{\mathrm{dR}}^+) \rightarrow B_{\mathrm{dR}}^+\cdot t
\end{equation}
defined as
$$\langle\omega,\gamma\rangle_p = \lim_{n \rightarrow \infty}p^n\int_{\gamma_n}\omega$$
where $\gamma = (\gamma_n) \in T_pA\otimes_{\mathbb{Z}_p}B_{\mathrm{dR}}^+$, and here by definition
$$\int_{\gamma_n}\omega := \int_{x_n}^{y_n}\omega$$
where $(x_n), (y_n)$ are any sequences ``born\'{e}es'' (see \cite[]{Colmez}) of elements of $A(B_{\mathrm{dR}}^+)$ such that $\theta(y_n - x_n) = \gamma_n$ for all $n$.

We claim that $\iota_{\mathrm{dR},y}$ in (\ref{stalkcomposition}) is given by the $\otimes_{B_{\mathrm{dR}}^+,i}B_{\mathrm{dR}}^+\llbracket X\rrbracket$-linear extension of the $B_{\mathrm{dR}}^+$-linear map $\iota_0$ in (\ref{Colmezinclusion}) restricted to $\Omega_{A/\overline{\mathbb{Q}}_p}^1\otimes_{\overline{\mathbb{Q}}_p}B_{\mathrm{dR}}^+\cdot t$. The claim is verified easily after noting the following observations: The natural inclusion $\mathbb{Z}_p \subset B_{\mathrm{dR}}^+\llbracket X\rrbracket$ factors through $\mathbb{Z}_p \subset B_{\mathrm{dR}}^+$, and so 
$$T_pA\otimes_{\mathbb{Z}_p}B_{\mathrm{dR}}^+\llbracket X\rrbracket = T_pA\otimes_{\mathbb{Z}_p}B_{\mathrm{dR}}^+\otimes_{B_{\mathrm{dR}}^+,i}B_{\mathrm{dR}}^+\llbracket X\rrbracket,$$
and (\ref{fiberinclusion}) is given by $\iota_0$ (since the relative comparison theorem specializes to the absolute comparison theorem; in fact, by the description of $\mathrm{gr}^0\mathbb{M}_0$ in \cite[p.45]{Scholze}, taking the stalk at $y$ of $\mathrm{gr}^0\mathbb{M}_0$ on $\mathcal{Y}$, we have that $M_0 = H_{\mathrm{dR}}^1(A/\overline{\mathbb{Q}}_p)\otimes_{\overline{\mathbb{Q}}_p}B_{\mathrm{dR}}^+$). Now the claim follows from (\ref{naturalinclusion}). 

Using the description of $\iota_{\mathrm{dR},y}$ given by our claim, we can now verify that (\ref{stalkcomposition}) is the natural inclusion. Let $w\in \omega_{\mathcal{A},y}\otimes_{\mathcal{O}_{Y,y}}B_{\mathrm{dR}}^+\llbracket X\rrbracket\cdot t$ and $\iota_{\mathrm{dR},y}(w) = x\alpha_1 + y\alpha_2$. Then 
$$HT_A(\iota_{\mathrm{dR},y}(w)) = x\alpha_1^*\frac{dT}{T} + y\alpha_2^*\frac{dT}{T}.$$
Since the Hodge-Tate exact sequence (\ref{absHTexactsequence}) is self-dual under the Weil pairing (\ref{Weil}) (see \cite[Proposition 4.10]{Scholze3}), we have the following ``pullback formula'' : for any $\gamma = (\gamma_n) \in \omega_{\mathcal{A},y}^{-1}\otimes_{\mathcal{O}_{Y,y}}B_{\mathrm{dR}}^+\llbracket X\rrbracket\cdot t \subset  T_pA \otimes_{\mathbb{Z}_p}B_{\mathrm{dR}}^+\llbracket X\rrbracket$, and any $\beta = (\beta_n) \in T_pA \cong \Hom_{\mathbb{Z}_p}(T_pA,\mathbb{Z}_p\cdot t)$ (using (\ref{dualisomorphism})), we have
$$\left\langle \beta^*\frac{dT}{T},\gamma\right\rangle_p = \lim_{n\rightarrow \infty}p^n\int_{\gamma_n}\beta^*\frac{dT}{T} = \lim_{n\rightarrow \infty}p^n\int_{\beta_n(\gamma_n)}\frac{dT}{T} = \langle \beta,\gamma\rangle$$
where the last equality holds since 
$$t = \lim_{n \rightarrow\infty}p^n\int_{\alpha_{1,n}}\alpha_2^*\frac{dT}{T} = \log[(\langle\alpha_{2,1},\alpha_{1,1}\rangle_1,\langle\alpha_{2,2},\alpha_{1,2}\rangle_2,\ldots\rangle)]$$
by the definition of $t$ (\ref{generator}). 

Now by direct calculation, we have
\begin{equation}\begin{split}\langle HT_A(\iota_{\mathrm{dR},y}(w)),\gamma\rangle_p = \left\langle x\alpha_1^*\frac{dT}{T} + y\alpha_2^*\frac{dT}{T},\gamma\right\rangle_p 
= \langle x\alpha_1 + y\alpha_2,\gamma\rangle &= \langle \iota_{\mathrm{dR},y}(w),\gamma\rangle \\
&= \langle w,\gamma\rangle_p
\end{split}
\end{equation}
where the last equality again follows from the self-duality of (\ref{absHTexactsequence}) under the Weil pairing (\ref{Weil}). Hence by the nondegeneracy of (\ref{padicperiodpairing}) we have $HT_A(\iota_{\mathrm{dR}},y(w)) = w$, as desired.
\end{proof}

We have the following ``$p$-adic Legendre'' relation of periods in $\mathcal{O}\mathbb{B}_{\mathrm{dR}}^+(\mathcal{Y}_x)$ associated with the fake Hasse invariant $\frak{s}$.

\begin{corollary}\label{Legendretheorem}We have the following identity of sections of $\mathcal{O}\mathbb{B}_{\mathrm{dR},Y}^+(\mathcal{Y}_x)$:
\begin{equation}\label{periodsrelation}\mathbf{x}_{\mathrm{dR}}/\mathbf{z} + \mathbf{y}_{\mathrm{dR}} = t
\end{equation}
where $\mathbf{z}$ is the Hodge-Tate period and $\mathbf{x}_{\mathrm{dR}}, \mathbf{y}_{\mathrm{dR}}, t \in \mathcal{O}\mathbb{B}_{\mathrm{dR}}^+(\mathcal{Y})$ are defined in Section \ref{deRhamperiod}.
\end{corollary}

\begin{proof}Recall that by definition,
$$\iota_{\mathrm{dR}}(t\cdot\frak{s}) = \mathbf{x}_{\mathrm{dR}}\alpha_{\infty,1} + \mathbf{y}_{\mathrm{dR}}\alpha_{\infty,2}.$$
Applying Proposition \ref{compositionproposition} to $\frak{s}\cdot t \in \omega_{\mathcal{A}}(\mathcal{Y}_x)\cdot t = \omega_{\mathcal{A}}\cdot t(\mathcal{Y}_x)$ (the last equality following since $t \in \mathbb{B}_{\mathrm{dR}}^+(\mathcal{Y})$ is a global section), we see that
$$t\cdot\frak{s} = HT_{\mathcal{A}}(\iota_{\mathrm{dR}}(t\cdot\frak{s})) = \mathbf{x}_{\mathrm{dR}}/\mathbf{z}\cdot \frak{s} + \mathbf{y}_{\mathrm{dR}}\cdot\frak{s} = (\mathbf{x}_{\mathrm{dR}}/\mathbf{z} + \mathbf{y}_{\mathrm{dR}})\cdot \frak{s}$$
and the desired identity follows. 
\end{proof}

\begin{corollary}\label{WeilPoincarecomparison}For any section $w$ of $\omega_{\mathcal{A}}|_{\mathcal{Y}}$ and section $\gamma$ of $T_p\mathcal{A}|_{\mathcal{Y}}$, we have
$$\langle w,\gamma\rangle_p = \langle \iota_{\mathrm{dR}}(w),\gamma \rangle.$$
\end{corollary}

\begin{proof}By Proposition \ref{compositionproposition}, we have $HT_{\mathcal{A}}\circ\iota_{\mathrm{dR}}(w) = w$, and so letting $\iota_{\mathrm{dR}}(w) = x\frac{\alpha_{\infty,1}}{t} + y\frac{\alpha_{\infty,2}}{t}$, we have
\begin{equation}\begin{split}&\langle w,\gamma\rangle_p \\
&= \langle HT_{\mathcal{A}}(\iota_{\mathrm{dR}}(w)),\gamma\rangle_p = \left\langle HT_{\mathcal{A}}\left(x\frac{\alpha_{\infty,1}}{t}+y\frac{\alpha_{\infty,2}}{t}\right),\gamma\right\rangle_p \\
&= x\lim_{n \rightarrow \infty}p^n\int_{\gamma_n}\left(\frac{\alpha_{\infty,1}}{t}\right)^*\frac{dT}{T} + y\lim_{n \rightarrow \infty}p^n\int_{\gamma_n}\left(\frac{\alpha_{\infty,2}}{t}\right)^*\frac{dT}{T}\\
&= x\left\langle \frac{\alpha_{\infty,1}}{t},\gamma\right\rangle + y\left\langle \frac{\alpha_{\infty,2}}{t},\gamma\right\rangle= \left\langle x\frac{\alpha_{\infty,1}}{t} + y\frac{\alpha_{\infty,2}}{t},\gamma\right\rangle = \langle \iota_{\mathrm{dR}}(w),\gamma\rangle.
\end{split}
\end{equation}
\end{proof}

\begin{corollary}\label{quotientsection}The map 
$$(HT_{\mathcal{A}})^{\vee} : \omega_{\mathcal{A}}^{-1}\otimes_{\mathcal{O}_Y}\mathcal{O}\mathbb{B}_{\mathrm{dR},\mathcal{Y}}^+ \hookrightarrow T_p\mathcal{A}\otimes_{\hat{\mathbb{Z}}_{p,Y}}\mathcal{O}\mathbb{B}_{\mathrm{dR},\mathcal{Y}}^+\cdot t^{-1}$$
is a section for the natural projection
$$T_p\mathcal{A}\otimes_{\hat{\mathbb{Z}}_{p,Y}}\mathcal{O}\mathbb{B}_{\mathrm{dR},\mathcal{Y}}^+\cdot t^{-1} \twoheadrightarrow (T_p\mathcal{A}\otimes_{\hat{\mathbb{Z}}_{p,Y}}\mathcal{O}\mathbb{B}_{\mathrm{dR},\mathcal{Y}}^+\cdot t^{-1})/\iota_{\mathrm{dR}}(\omega_{\mathcal{A}}\otimes_{\mathcal{O}_Y}\mathcal{O}\mathbb{B}_{\mathrm{dR},\mathcal{Y}}^+).$$
\end{corollary}

\begin{proof}First note that we have
\begin{align*}&(T_p\mathcal{A}\otimes_{\hat{\mathbb{Z}}_{p,Y}}\mathcal{O}\mathbb{B}_{\mathrm{dR},\mathcal{Y}}^+\cdot t^{-1})/\iota_{\mathrm{dR}}(\omega_{\mathcal{A}}\otimes_{\mathcal{O}_Y}\mathcal{O}\mathbb{B}_{\mathrm{dR},\mathcal{Y}}^+) \\
&\underset{\sim}{\xrightarrow{\alpha_{\infty}^{-1}}}\mathcal{O}\mathbb{B}_{\mathrm{dR},\mathcal{Y}}^{+,\oplus 2}/\left((\mathbf{x}_{\mathrm{dR}}\cdot\alpha_{\infty,1}/t + \mathbf{y}_{\mathrm{dR}}\cdot\alpha_{\infty,2}/t)\mathcal{O}\mathbb{B}_{\mathrm{dR},\mathcal{Y}}^+\right).
\end{align*}
By the $p$-adic Legendre relation (\ref{periodsrelation}), the right-hand side isomorphic to the sub-$\mathcal{O}\mathbb{B}_{\mathrm{dR},\mathcal{Y}}^+$-module 
$$(\alpha_{\infty,1} - 1/\mathbf{z}\cdot\alpha_{\infty,2})\mathcal{O}\mathbb{B}_{\mathrm{dR},\mathcal{Y}}^+$$
of $\mathcal{O}\mathbb{B}_{\mathrm{dR},\mathcal{Y}}^{+,\oplus 2} \underset{\sim}{\xrightarrow{\alpha_{\infty}}} T_p\mathcal{A}\otimes_{\hat{\mathbb{Z}}_{p,Y}}\mathcal{O}\mathbb{B}_{\mathrm{dR},\mathcal{Y}}^+$ on $\mathcal{Y}_x$, and is the sub-$\mathcal{O}\mathbb{B}_{\mathrm{dR},\mathcal{Y}}^+$-module 
$$(\mathbf{z}\cdot\alpha_{\infty,1} - \alpha_{\infty,2})\mathcal{O}\mathbb{B}_{\mathrm{dR},\mathcal{Y}}^+$$
on $\mathcal{Y}_x$. Hence it is isomorphic to $\mathcal{H}^{0,1}\otimes_{\mathcal{O}_Y}\mathcal{O}\mathbb{B}_{\mathrm{dR},\mathcal{Y}}^+$, by the definition of $\mathcal{H}^{0,1}$ (see Section \ref{HTsection}) and the Hodge-Tate period $\mathbf{z}$, and now the Corollary follows from Proposition (\ref{HTarrows}). 
\end{proof}

\subsection{Relation to classical (Serre-Tate) theory on the ordinary locus}\label{ordinarysection}
Recall the ordinary locus $Y^{\mathrm{ord}} \subset Y$, with universal ordinary elliptic curve $\mathcal{A}^{\mathrm{ord}} \rightarrow Y^{\mathrm{ord}}$; here, $\mathcal{A}^{\mathrm{ord}}$ can be thought of as the restriction $\mathcal{A}|_{Y^{\mathrm{ord}}}$ of the universal elliptic curve $\mathcal{A} \rightarrow Y$. In particular, since $Y$ has a formal model $\hat{Y}/\mathbb{Z}_p$ with a moduli interpretation, the universal elliptic curve $\mathcal{A} \rightarrow Y$ and hence $\mathcal{A}^{\mathrm{ord}} \rightarrow Y^{\mathrm{ord}}$ spread out to formal schemes $\hat{\mathcal{A}} \rightarrow \hat{Y}$ and $\hat{\mathcal{A}}^{\mathrm{ord}} \rightarrow \hat{Y}^{\mathrm{ord}}$ over $\mathbb{Z}_p$; let $\hat{\mathcal{A}}_0^{\mathrm{ord}}, \hat{Y}_0^{\mathrm{ord}}$ denote the special fibers. Now the adic reduction map $\mathcal{A}^{\mathrm{ord}} \rightarrow \hat{\mathcal{A}}_0^{\mathrm{ord}}$ allows us to make the identifications
$$\mathcal{A}^{\mathrm{ord}}[p^n](\overline{\mathbb{F}}_p) = \hat{\mathcal{A}}_0^{\mathrm{ord}}[p^n](\overline{\mathbb{F}}_p)$$
and so
$$(T_p\mathcal{A}^{\mathrm{ord}})^{\text{\'{e}t}} = T_p\hat{\mathcal{A}}_0^{\mathrm{ord}}(\overline{\mathbb{F}}_p).$$

The Igusa tower $Y^{\mathrm{Ig}} \rightarrow Y^{\mathrm{ord}}$ is a $\mathbb{Z}_p^{\times}$-torsor which represents the moduli space classifying isomorphism classes of triples $(A,t,\mu_{p^{\infty}} \hookrightarrow A[p^{\infty}])$, where $A$ is an ordinary elliptic curve, $t \in A[N]$ is a generator, and $\mu_{p^{\infty}} \hookrightarrow A[p^{\infty}]$ is an embedding of $p$-divisible groups. By (\ref{infinitepairing}), an equivalent way to define $Y^{\mathrm{Ig}}$ is as the solution to the moduli problem classifying isomorphism classes of triples $(A,t,(T_pA)^{\text{\'{e}t}} \cong \mathbb{Z}_p)$.

There is an adic \'{e}tale section of the $\mathbb{Z}_p^{\times}$-torsor $Y^{\mathrm{Ig}} \rightarrow Y^{\mathrm{ord}}$ defined by choosing an embedding $\mu_{p^{\infty}} \subset \mathcal{A}^{\mathrm{ord}}[p^{\infty}]$ (or equivalently $\hat{\mathbb{Z}}_{p,Y}(1) \subset T_p\mathcal{A}^{\mathrm{ord}}$), or in other words by choosing a trivialization
\begin{equation}\label{connectedtrivialization}\hat{\mathbb{Z}}_{p,Y^{\mathrm{ord}}} \xrightarrow{\sim} (T_p\mathcal{A}^{\mathrm{ord}})^{0}(-1)\subset T_p\mathcal{A}^{\mathrm{ord}}(-1)
\end{equation}
of the canonical line $(T_p\mathcal{A}^{\mathrm{ord}})^0 \subset T_p\mathcal{A}^{\mathrm{ord}}$; here $(T_p\mathcal{A}^{\mathrm{ord}})^0$ denotes the connected component of the universal ordinary Tate module. From (\ref{infinitepairing}), we get the isomorphism
\begin{equation}\label{ordinaryinfinitepairing}(T_p\mathcal{A}^{\mathrm{ord}})^{\text{\'{e}t}} \cong \Hom_{\hat{\mathbb{Z}}_{p,Y^{\mathrm{ord}}}}(\mathcal{A}^{\mathrm{ord}}[p^{\infty}]^0,\mu_{p^{\infty}}),
\end{equation}
where $(T_p\mathcal{A}^{\mathrm{ord}})^{\text{\'{e}t}}= T_p\hat{\mathcal{A}}_0^{\mathrm{ord}}(\overline{\mathbb{F}}_p)$ is the \'{e}tale quotient of $T_p\mathcal{A}^{\mathrm{ord}}$ and $\mathcal{A}^{\mathrm{ord}}[p^{\infty}]^0$ is the connected component of the $p$-divisible group of the universal ordinary elliptic curve. Hence to have a trivialization as in (\ref{connectedtrivialization}) is equivalent to having a trivialization
$$\hat{\mathbb{Z}}_{p,Y^{\mathrm{ord}}} \xrightarrow{\sim} (T_p\mathcal{A}^{\mathrm{ord}})^{\text{\'{e}t}}.$$

We wish to find a natural cover $\mathcal{Y}^{\mathrm{Ig}} \rightarrow Y^{\mathrm{Ig}}$ which is an affinoid subdomain of the $p$-adic universal cover $\mathcal{Y}$. Let $F = W(\overline{\mathbb{F}}_p)[1/p]$ so that $\mathcal{O}_F = W(\overline{\mathbb{F}}_p)$. Let $\mathcal{C}_{\mathrm{can}} \subset \mathcal{A}^{\mathrm{ord}}$ denote the canonical subgroup, defined over $\mathcal{O}_F$. Base-changing $\mathcal{A}^{\mathrm{ord}}$ to $\mathrm{Spa}(F,\mathcal{O}_F)$, by properties of the canonical subgroup, the isogeny
$$\pi : \mathcal{A}^{\mathrm{ord}} \rightarrow \mathcal{A}^{\mathrm{ord}}/\mathcal{C}_{\mathrm{can}}$$
over $\mathrm{Spa}(F,\mathcal{O}_F)$, is a lifting of the relative Frobenius morphism on the special fiber
$$\mathcal{A}_0^{\mathrm{ord}} \rightarrow \mathcal{A}_0^{\mathrm{ord},(p)}$$
where 
$$\mathcal{A}_0^{\mathrm{ord},(p)} = \mathcal{A}_0^{\mathrm{ord}} \times_{\Spec(\overline{\mathbb{F}}_p),\mathrm{Frob}} \Spec(\overline{\mathbb{F}}_p)$$ 
and $\mathrm{Frob} : \Spec(\overline{\mathbb{F}}_p) \rightarrow \Spec(\overline{\mathbb{F}}_p)$ is the absolute Frobenius (with $\mathrm{Frob}^*f = f^p$). Then let $\phi : Y^{\mathrm{ord}} \rightarrow Y^{\mathrm{ord}}$ be the classifying morphism such that $$\mathcal{A}^{\mathrm{ord}}/\mathcal{C}_{\mathrm{can}} = \mathcal{A}^{\mathrm{ord}} \times_{Y^{\mathrm{ord}}, \phi}Y^{\mathrm{ord}} =: \mathcal{A}^{\mathrm{ord},(\phi)}.$$
The map $\phi : Y^{\mathrm{ord}} \rightarrow Y^{\mathrm{ord}}$ is often known as the \emph{canonical lifting of Frobenius} or \emph{canonical Frobenius endomorphism}. Let $\phi^{-1} : \mathcal{A}^{\mathrm{ord},(\phi)} \rightarrow \mathcal{A}^{\mathrm{ord}}$ denote the natural projection. Recall that $T_p\mathcal{A}^{\mathrm{ord},0} \subset T_p\mathcal{A}^{\mathrm{ord}}$ is the so-called \emph{canonical line} is given by the line $T_p\mathcal{A}^{\mathrm{ord},0} = (T_p\mathcal{A}^{\mathrm{ord}})^{F(\phi) = p\phi}$ defined over $F$ on which $F(\phi) = \phi^{-1} \circ \pi$ acts as multiplication by $p$. The \emph{anti-canonical line} is given by the line $(T_p\mathcal{A}^{\mathrm{ord}})^{F(\phi) = \phi}$ defined over $F$ on which $F(\phi)$ acts as the identity. A choice of
$$T_p\mathcal{A}^{\mathrm{ord},\text{\'{e}t}} \cong (T_p\mathcal{A}^{\mathrm{ord}})^{F(\phi) = \phi}$$
induces an $F(\phi)$-equivariant splitting defined over $F$, often called a \emph{unit root splitting}, of the Hodge-Tate exact sequence
$$0 \rightarrow T_p\mathcal{A}^{\mathrm{ord},0} \rightarrow T_p\mathcal{A}^{\mathrm{ord}} \rightarrow T_p\mathcal{A}^{\mathrm{ord},\text{\'{e}t}} \rightarrow 0.$$

Now we can define the affinoid subdomain of the ordinary locus $\mathcal{Y}^{\mathrm{Ig}} \subset \mathcal{Y}^{\mathrm{ord}}$ defined over $\mathrm{Spa}(F,\mathcal{O}_F)$ as the sublocus classifying isomorphism classes of triples $(A,t,\alpha : \mathbb{Z}_p^{\oplus 2} \xrightarrow{\sim} T_pA)$ where $A$ is ordinary with the additional condition on the $p^{\infty}$-level structure $\alpha : \mathbb{Z}_p^{\oplus 2} \xrightarrow{\sim} T_pA$ that $\alpha$ is a unit root splitting, or explicitly:
\begin{equation}\label{levelstructure1}\alpha|_{\mathbb{Z}_p\oplus \{0\}} : \mathbb{Z}_p\xrightarrow{\sim} (T_pA)^{F(\phi) = p\phi} = (T_pA)^0 \subset T_pA
\end{equation}
and
\begin{equation}\label{levelstructure2}\alpha|_{\{0\}\oplus \mathbb{Z}_p} : \mathbb{Z}_p \xrightarrow{\sim} (T_pA)^{F(\phi) = \phi} \subset T_pA.
\end{equation}
This gives a $\mathbb{Z}_p^{\times}$-torsor of $Y^{\mathrm{Ig}}$. To see this, suppose that we are given a trivialization $T_pA^{\text{\'{e}t}} \cong \mathbb{Z}_p$. Then this induces 
\begin{equation}\label{given}\mathbb{Z}_p  \cong T_pA^{\text{\'{e}t}} \cong \Hom(A[p^{\infty}]^0,\mu_{p^{\infty}})
\end{equation}
via (\ref{infinitepairing}). Now for each choice of (\ref{levelstructure1}), there is a unique choice of (\ref{levelstructure2}) which induces via (\ref{infinitepairing}), the identification (\ref{given}). Since there are $\mathbb{Z}_p^{\times}$ ways to choose (\ref{levelstructure1}), this shows that $\mathcal{Y}^{\mathrm{Ig}} \rightarrow Y^{\mathrm{Ig}}$ is a $\mathbb{Z}_p^{\times}$-torsor, and $\mathcal{Y}^{\mathrm{Ig}} \rightarrow Y^{\mathrm{ord}}$ is a $\mathbb{Z}_p^{\times} \times \mathbb{Z}_p^{\times}$-torsor.

Again, \emph{choosing} a trivialization 
$$\alpha_{\infty} : \hat{\mathbb{Z}}_{p,Y^{\mathrm{ord}}}^{\oplus 2}\xrightarrow{\sim} T_p\mathcal{A}^{\mathrm{ord}}$$
with 
$$\alpha_{\infty}|_{\hat{\mathbb{Z}}_{p,Y^{\mathrm{ord}}}\oplus \{0\}} : \hat{\mathbb{Z}}_{p,Y^{\mathrm{ord}}} \xrightarrow{\sim}  (T_p\mathcal{A}^{\mathrm{ord}})^{F(\phi) = p\phi} = (T_p\mathcal{A}^{\mathrm{ord}})^0$$
and 
$$\alpha_{\infty}|_{\{0\}\oplus \hat{\mathbb{Z}}_{p,Y^{\mathrm{ord}}}} : \hat{\mathbb{Z}}_{p,Y^{\mathrm{ord}}} \xrightarrow{\sim} (T_p\mathcal{A}^{\mathrm{ord}})^{\text{\'{e}t}}$$
defines an adic \'{e}tale section 
$$Y^{\mathrm{ord}} \hookrightarrow \mathcal{Y}^{\mathrm{Ig}}$$
of the $\mathbb{Z}_p^{\times}\times\mathbb{Z}_p^{\times}$-torsor $\mathcal{Y}^{\mathrm{Ig}} \rightarrow Y^{\mathrm{ord}}$. 

\begin{proposition}\label{ordinaryperiod}On $\mathcal{Y}^{\mathrm{Ig}}$, we have
$$1/\mathbf{z} = 0.$$
As a consequence, on $\mathcal{Y}^{\mathrm{Ig}}$, we have
$$\mathbf{y}_{\mathrm{dR}}/t = 1.$$
\end{proposition}

\begin{proof}Since by definition $\alpha_{\infty}((1,0))$ trivializes the first piece of the Hodge-Tate filtration of $T_p\mathcal{A}^{\mathrm{ord}}$, then $\alpha_{\infty}((0,1))$ spans the first piece of the Hodge-Tate filtration. By definition of the Hodge-Tate period, the first piece of the Hodge-Tate filtration is also spanned by 
$$\alpha_{\infty,1} - 1/\mathbf{z}\cdot\alpha_{\infty,2} = \alpha_{\infty}((1,0)) - 1/\mathbf{z}\cdot\alpha_{\infty}((0,1))$$
on $\mathcal{Y}^{\mathrm{Ig}}$, then we have $1/\mathbf{z} = 0$ on $\mathcal{Y}^{\mathrm{Ig}}$. Hence by the $p$-adic Legendre relation (\ref{periodsrelation}), we have $\mathbf{y}_{\mathrm{dR}} = t$ on $\mathcal{Y}^{\mathrm{Ig}}$.
\end{proof}

We now recall Serre-Tate coordinates, which are defined naturally on $\mathcal{Y}^{\mathrm{Ig}}$, stated for the pro\'{e}tale site $Y_{\text{pro\'{e}t}}^{\mathrm{ord}}$. (One can check that all the isomorphisms defined below are functorial for the pro\'{e}tale objects, and so induce isomorphisms of sheaves on $Y_{\text{pro\'{e}t}}^{\mathrm{ord}}$.)

\begin{theorem}[Corollary 4.1.5., Theorem 4.3.1. quat \cite{Katz}]\label{SerreTatethm}We have the following natural isomorphisms of sheaves on $Y_{\text{pro\'{e}t}}^{\mathrm{ord}}$:
$$HT_{\mathcal{A}} : (T_p\mathcal{A}^{\mathrm{ord}})^{\text{\'{e}t}}\otimes_{\hat{\mathbb{Z}}_{p,Y^{\mathrm{ord}}}} \mathcal{O}_{Y^{\mathrm{ord}}} \xrightarrow{\sim} \omega_{\mathcal{A}}|_{Y^{\mathrm{ord}}}$$
\begin{align*}(HT_{\mathcal{A}})^{\vee,-1} : (T_p\mathcal{A}^{\mathrm{ord}})^0(-1) \otimes_{\hat{\mathbb{Z}}_{p,Y^{\mathrm{ord}}}} \mathcal{O}_{Y^{\mathrm{ord}}} &= \Hom_{\mathbb{Z}_p}((T_p\mathcal{A}^{\mathrm{ord}})^{\text{\'{e}t}},\hat{\mathbb{Z}}_{p,Y^{\mathrm{ord}}})\otimes_{\hat{\mathbb{Z}}_{p,Y^{\mathrm{ord}}}}\mathcal{O}_{Y^{\mathrm{ord}}} \\&\xrightarrow{\sim} \omega_{\mathcal{A}}^{-1}|_{Y^{\mathrm{ord}}}.
\end{align*}
Moreover, there is a natural isomorphism of formal schemes over $\mathcal{O}_F$ (viewed as functors of $\mathcal{O}_F$-algebras)
\begin{equation}\label{SerreTateintegral}\hat{Y}^{\mathrm{ord}}(\cdot) \cong \Hom_{\mathbb{Z}_p}(T_p\hat{\mathcal{A}}_0(\overline{\mathbb{F}}_p)\otimes_{\mathbb{Z}_p} T_p\hat{\mathcal{A}}_0(\overline{\mathbb{F}}_p), \hat{\mathbb{G}}_m(\cdot))
\end{equation}
which on the adic generic fibers induces a natural isomorphism of adic spaces over $\mathrm{Spa}(F,\mathcal{O}_F)$ (viewed as functors of $(F,\mathcal{O}_F)$-algebras)
\begin{equation}\label{SerreTate}Y^{\mathrm{ord}}(\cdot) \cong \Hom_{\mathbb{Z}_p}((T_p\mathcal{A}^{\mathrm{ord}})^{\text{\'{e}t}}\otimes_{\mathbb{Z}_p}(T_p\mathcal{A}^{\mathrm{ord}})^{\text{\'{e}t}},\hat{\mathbb{G}}_m(\cdot))\otimes_{\mathbb{Z}_p}\mathbb{Q}_p.
\end{equation}
This defines coordinates on the pro\'{e}tale $\mathbb{Z}_p^{\times}\times\mathbb{Z}_p^{\times}$-cover $\mathcal{Y}^{\mathrm{Ig}}$ in the following way: let $\alpha_{\infty,1}, \alpha_{\infty,2}$ be the basis of $(T_p\mathcal{A}^{\mathrm{ord}})|_{\mathcal{Y}^{\mathrm{Ig}}}$ defined in (\ref{levelstructure1}) and (\ref{levelstructure2}). Then the $\alpha_{\infty}|_{\{0\}\oplus\hat{\mathbb{Z}}_{p,\mathcal{Y}^{\mathrm{Ig}}}} : \hat{\mathbb{Z}}_{p,\mathcal{Y}^{\mathrm{Ig}}} \xrightarrow{\sim} (T_p\mathcal{A}^{\mathrm{ord}})^{\text{\'{e}t}}|_{\mathcal{Y}^{\mathrm{Ig}}}$ induces, via (\ref{SerreTate}), a canonical isomorphism of adic spaces
\begin{equation}\label{SerreTateiso}\begin{split}\mathcal{Y}^{\mathrm{Ig}}(\cdot)&\cong \Hom_{\mathbb{Z}_p}((T_p\mathcal{A}^{\mathrm{ord}})^{\text{\'{e}t}}|_{\mathcal{Y}^{\mathrm{Ig}}}\otimes_{\mathbb{Z}_p}(T_p\mathcal{A}^{\mathrm{ord}})^{\text{\'{e}t}}|_{\mathcal{Y}^{\mathrm{Ig}}},\hat{\mathbb{G}}_m(\cdot))\otimes_{\mathbb{Z}_p}\mathbb{Q}_p \\
&\xrightarrow{\alpha_{\infty}|_{\{0\}\oplus\hat{\mathbb{Z}}_{p,\mathcal{Y}^{\mathrm{Ig}}}}^{-1}} \Hom_{\mathbb{Z}_p}(\hat{\mathbb{Z}}_{p,\mathcal{Y}^{\mathrm{Ig}}},\hat{\mathbb{G}}_m(\cdot))\otimes_{\mathbb{Z}_p}\mathbb{Q}_p
\end{split}
\end{equation}
which is naturally a $\mathbb{Z}_p^{\times}\times\mathbb{Z}_p^{\times}$-pro\'{e}tale cover of $Y^{\mathrm{ord}}$. In fact, given a residue disc $D \subset Y^{\mathrm{ord}}$ and letting $\mathcal{D} = \mathcal{Y}^{\mathrm{ord}} \times_{Y^{\mathrm{ord}}} D$, since $(T_p\mathcal{A}^{\mathrm{ord}})^{\text{\'{e}t}}|_{\mathcal{D}}$ is constant on each geometric connected component of $\mathcal{D}$, then by (\ref{SerreTateiso}) $\mathcal{D}$ is a trivial $\mathbb{Z}_p^{\times}\times\mathbb{Z}_p^{\times}$-cover of $\hat{\mathbb{G}}_m(\cdot)\otimes_{\mathbb{Z}_p}{\mathbb{Q}_p}$, i.e. on each geometric component it is isomorphic to $\hat{\mathbb{G}}_m(\cdot)\otimes_{\mathbb{Z}_p}{\mathbb{Q}_p}$.
\end{theorem}

\begin{definition}\label{SerreTatecoordinate}The pullback of the natural coordinate $T$ on 
$$\mathbb{T} = \mathrm{Spa}(\mathbb{Q}_p\langle T^{\pm 1}\rangle,\mathbb{Z}_p\langle T^{\pm 1}\rangle)$$
is called the \emph{Serre-Tate coordinate} on $\mathcal{Y}^{\mathrm{Ig}}$. 
\end{definition}

\begin{definition}\label{SerreTateexpansion}Let
$$\hat{\mathcal{O}}_{F,\mathcal{Y}} = \hat{\mathbb{Z}}_{p,\mathcal{Y}}\otimes_{\mathbb{Z}_p} \mathcal{O}_F.$$
We call the map of $\mathcal{O}_F$-modules 
\begin{equation}T-\mathrm{exp} : \mathcal{O}_{\mathcal{Y}^{\mathrm{Ig}}} \xrightarrow{\sim} \hat{\mathcal{O}}_{F,\mathcal{Y}}\llbracket T - 1\rrbracket \otimes_{\mathbb{Z}_p}\mathbb{Q}_p,
\end{equation}
induced by (\ref{SerreTateiso}) the \emph{Serre-Tate expansion map}. Now let $y \in \mathcal{Y}^{\mathrm{Ig}}$ be a geometric point, and let $D(\lambda(y))$ denote the residue disc in $Y^{\mathrm{ord}}$ centered around $\lambda(y)$ (recalling that $\lambda : \mathcal{Y}^{\mathrm{Ig}} \rightarrow Y^{\mathrm{ord}}$ is the natural projection), and let $D(y) = D(\lambda(y)) \times_{Y^{\mathrm{ord}}} \mathcal{Y}^{\mathrm{Ig}}$. Using the fact that $dT$ is a generator of $\Omega_{\hat{\mathbb{G}}_m}^1$, one can show that 
\begin{equation}\label{STcomputation}T-\mathrm{exp}(f) = \sum_{n = 0}^{\infty} \frac{1}{n!}\left(\frac{d}{dT}\right)^nf|_{T = 1}(y)(T-1)^n
\end{equation}
where for a section $s$ of $\mathcal{O}_{\mathcal{Y}^{\mathrm{Ig}}}$, $s(y)$ denotes the image of $s$ in the residue field at $y$. Henceforth, we adopt the shorthand
$$f(T) := T-\mathrm{exp}(f).$$
\end{definition}

\begin{definition}\label{Katzcanonicaldifferential}We also have a natural section 
$$HT_{\mathcal{A}}(\alpha_{\infty,2}) = \omega_{\mathrm{can}}^{\mathrm{Katz}} = \frak{s}|_{\mathcal{Y}^{\mathrm{Ig}}} \in \omega_{\mathcal{A}}(\mathcal{Y}^{\mathrm{Ig}})$$
which is called \emph{Katz's canonical differential}. By \cite[Main Theorem 4.4.1]{Katz}, we have
$$\sigma(\omega_{\mathrm{can}}^{\mathrm{Katz},\otimes 2})= d\log T.$$
\end{definition}

\subsection{The Kodaira-Spencer isomorphism}Recall the Kodaira-Spencer map (known to be an \emph{isomorphism} of locally free $\mathcal{O}_Y$-modules in this setting) which is given by 
\begin{equation}\label{KS1}\sigma : \omega_{\mathcal{A}}^{\otimes 2} \xrightarrow{\sim} \Omega_Y^1, \hspace{2cm} \sigma(\omega_1\otimes \omega_2) = \langle \omega_1,\nabla(\omega_2)\rangle_{\mathrm{Poin}}
\end{equation}
where 
\begin{equation}\label{Poincare}\langle\cdot,\cdot\rangle_{\mathrm{Poin}} : \mathcal{H}_{\mathrm{dR}}^1(\mathcal{A}) \times \mathcal{H}_{\mathrm{dR}}^1(\mathcal{A}) \rightarrow \mathcal{O}_Y
\end{equation}
is the alternating, non-degenerate ``Poincar\'{e} pairing'', for which $\omega_{\mathcal{A}}$ is an isotropic subspace, and which encodes the (Serre) duality between $\omega_{\mathcal{A}} := R^0\pi_*\Omega_{\mathcal{A}/Y}^1$ and $\omega_{\mathcal{A}}^{-1} := R^1\pi_*\mathcal{O}_{\mathcal{A}}$. We extend (\ref{KS1}) $\mathcal{O}\mathbb{B}_{\mathrm{dR},Y}^{+}$-linearly to a map (hence an isomorphism)
\begin{equation}\label{KS2}\begin{split}&\sigma : \omega_{\mathcal{A}}^{\otimes 2}\otimes_{\mathcal{O}_Y}\mathcal{O}\mathbb{B}_{\mathrm{dR},Y}^{+} \xrightarrow{\sim} \Omega_Y^1\otimes_{\mathcal{O}_Y}\mathcal{O}\mathbb{B}_{\mathrm{dR},Y}^{+}\\
& \sigma(\omega_1\otimes \omega_2\otimes f) = \langle \omega_1,\nabla(\omega_2\otimes f)\rangle_{\mathrm{Poin}} = \langle \omega_1,\nabla(\omega_2)\otimes f\rangle_{\mathrm{Poin}}
\end{split}
\end{equation}
(where the last equality follows because $\omega_{\mathcal{A}}$ is an isotropic space under the Poincar\'{e} pairing). We note that we have the following commutative diagram 
\begin{equation}\label{KSdiagram}
\begin{tikzcd}
     & \omega_{\mathcal{A}}^{\otimes 2} \arrow{r}{\sim}[swap]{\sigma} \arrow[hookrightarrow]{d}{} & \Omega_{Y}^1 \arrow[hookrightarrow]{d}{} &\\
     & \omega_{\mathcal{A}}^{\otimes 2}\otimes_{\mathcal{O}_{Y}}\mathcal{O}\mathbb{B}_{\mathrm{dR},Y}^+ \arrow{r}{\sim}[swap]{\sigma} \arrow[rightarrow]{d}{\theta}& \Omega_{Y}^1\otimes_{\mathcal{O}_{Y}}\mathcal{O}\mathbb{B}_{\mathrm{dR},Y}^+ \arrow[rightarrow]{d}{\theta}&\\
     & \omega_{\mathcal{A}}^{\otimes 2}\otimes_{\mathcal{O}_Y} \hat{\mathcal{O}}_Y \arrow{r}{\sim}[swap]{\sigma}  &  \Omega_Y^1\otimes_{\mathcal{O}_Y}\hat{\mathcal{O}}_Y &\\
\end{tikzcd}.
\end{equation}

\begin{proposition}\label{Weilextension}We have the following commutative diagram
\begin{equation}
\begin{tikzcd}[column sep =small]\label{pairingdiagram}
   \mathcal{H}_{\mathrm{dR}}^1(\mathcal{A})\otimes_{\mathcal{O}_{Y}}\mathcal{O}\mathbb{B}_{\mathrm{dR},\mathcal{Y}}^+ \arrow[hookrightarrow]{d}{\iota_{\mathrm{dR}}} & \times &\mathcal{H}_{\mathrm{dR}}^1(\mathcal{A})\otimes_{\mathcal{O}_{Y}}\mathcal{O}\mathbb{B}_{\mathrm{dR},\mathcal{Y}}^+  \arrow[hookrightarrow]{d}{\iota_{\mathrm{dR}}}  \arrow{rr}{\langle \cdot, \cdot \rangle_{\mathrm{Poin}}}&  &\mathcal{O}\mathbb{B}_{\mathrm{dR},\mathcal{Y}}^+ \arrow[hookrightarrow]{d}{}&\\
  T_p\mathcal{A}\otimes_{\hat{\mathbb{Z}}_{p,Y}}\mathcal{O}\mathbb{B}_{\mathrm{dR},\mathcal{Y}}^+\cdot t^{-1} & \times & T_p\mathcal{A}\otimes_{\hat{\mathbb{Z}}_{p,Y}}\mathcal{O}\mathbb{B}_{\mathrm{dR},\mathcal{Y}}^+\cdot t^{-1} \arrow{rr}{\langle \cdot, \cdot \rangle\cdot t^{-1}} & &\mathcal{O}\mathbb{B}_{\mathrm{dR},\mathcal{Y}}^+\cdot t^{-2}.
\end{tikzcd}
\end{equation}
\end{proposition}

\begin{proof}

Since both $\langle\cdot,\cdot\rangle_{\mathrm{Poin}}$ and $\langle\cdot,\cdot\rangle$ are \emph{alternating} on spaces of rank 2, and $\omega_{\mathcal{A}}\otimes_{\mathcal{O}_Y}\mathcal{O}\mathbb{B}_{\mathrm{dR},\mathcal{Y}}^+$ is an isotropic subspace of $\mathcal{H}_{\mathrm{dR}}^1(\mathcal{A})\otimes_{\mathcal{O}_Y}\mathcal{O}\mathbb{B}_{\mathrm{dR},\mathcal{Y}}^+$, by Corollary \ref{quotientsection} it suffices to check that 
$$\langle \cdot,\cdot\rangle_{\mathrm{Poin}}' = \langle\iota_{\mathrm{dR}}(\cdot),(HT_{\mathcal{A}})^{\vee}(\cdot)\rangle\cdot t^{-1}$$
where here $\langle \cdot, \cdot \rangle_{\mathrm{Poin}}'$ denotes the pairing induced by the actual Poincar\'{e} pairing $\langle \cdot, \cdot \rangle_{\mathrm{Poin}}$ (via the isotropicity of $\omega_{\mathcal{A}}\otimes_{\mathcal{O}_Y}\mathcal{O}\mathbb{B}_{\mathrm{dR},\mathcal{Y}}^+$) on the quotient
\begin{align*}&\omega_{\mathcal{A}}\otimes_{\mathcal{O}_Y} \mathcal{O}\mathbb{B}_{\mathrm{dR},\mathcal{Y}}^+ \times \left((\mathcal{H}_{\mathrm{dR}}^1(\mathcal{A})\otimes_{\mathcal{O}_Y} \mathcal{O}\mathbb{B}_{\mathrm{dR},\mathcal{Y}}^+)/(\omega_{\mathcal{A}}\otimes_{\mathcal{O}_Y} \mathcal{O}\mathbb{B}_{\mathrm{dR},\mathcal{Y}}^+)\right)\\
& = (\omega_{\mathcal{A}}\otimes_{\mathcal{O}_Y}\mathcal{O}\mathbb{B}_{\mathrm{dR},\mathcal{Y}}^+)\times (\omega_{\mathcal{A}}^{-1}\otimes_{\mathcal{O}_Y} \mathcal{O}\mathbb{B}_{\mathrm{dR},\mathcal{Y}}^+).
\end{align*}
By Corollary \ref{WeilPoincarecomparison} we have that on the above subspace, 
$$\langle\cdot,(HT_{\mathcal{A}})^{\vee}(\cdot)\rangle_p = \langle\iota_{\mathrm{dR}}(\cdot),(HT_{\mathcal{A}})^{\vee}(\cdot)\rangle,$$ 
and so to finish the proof of the Proposition we are reduced to showing 
$$\langle\cdot,\cdot\rangle_{\mathrm{Poin}} = \langle\cdot,(HT_{\mathcal{A}})^{\vee}(\cdot)\rangle_p \cdot t^{-1}.$$

Given any section $w$ of $\omega_{\mathcal{A}}\otimes_{\mathcal{O}_Y}\mathcal{O}\mathbb{B}_{\mathrm{dR},\mathcal{Y}}^+$ and $\check{w}$ of $\omega_{\mathcal{A}}^{-1}\otimes_{\mathcal{O}_Y}\mathcal{O}\mathbb{B}_{\mathrm{dR},\mathcal{Y}}^+$, we have
\begin{equation}\begin{split}\langle w,(HT_{\mathcal{A}})^{\vee}(\check{w})\rangle_p &= \lim_{n \rightarrow \infty}p^n\int_{(HT_{\mathcal{A}})^{\vee}(\check{w})_n}w \overset{(1)}{=} (HT_{\mathcal{A}})^{\vee}(\check{w})(\iota_{\mathrm{dR}}(w)) \cdot t\\
&\overset{(2)}{=} \check{w}(HT_{\mathcal{A}}(\iota_{\mathrm{dR}}(w))) \cdot t \overset{(3)}{=} \check{w}(w) \cdot t= \langle w,\check{w}\rangle_{\mathrm{Poin}}' \cdot t.
\end{split}
\end{equation}
Here, the equality $(1)$ follows by the ``pullback formula'' given in the proof of Proposition \ref{compositionproposition}, $(2)$ follows since by definition $(HT_{\mathcal{A}})^{\vee}(\check{w}) = \check{w}\circ HT_{\mathcal{A}}$, and $(3)$ follows from Proposition \ref{compositionproposition}. 

\end{proof}

\begin{remark}Alternatively, one can use \cite[Proposition 4.11]{Scholze3} to prove Proposition \ref{Weilextension}.
\end{remark}

\begin{definition}\label{deRhamfundamentalperiod}We define the \emph{de Rham fundamental period} as the ratio
$${\mathbf{z}_{\mathrm{dR}}} := -\frac{\mathbf{x}_{\mathrm{dR}}}{\mathbf{y}_{\mathrm{dR}}} \in \mathcal{O}\mathbb{B}_{\mathrm{dR}}(\mathcal{Y}_x).$$
\end{definition}

The identities (\ref{x'transformationprop}) and (\ref{y'transformationprop}) imply that
\begin{proposition}
\begin{equation}\label{z'transformationprop}
\begin{split}\left(\begin{array}{ccc}a & b\\
c & d\\
\end{array}\right)^*{\mathbf{z}_{\mathrm{dR}}} = \frac{d{\mathbf{z}_{\mathrm{dR}}} + b}{c{\mathbf{z}_{\mathrm{dR}}} + a}
\end{split}
\end{equation}
for any $\left(\begin{array}{ccc}a & b\\
c & d\\
\end{array}\right) \in GL_2(\mathbb{Q}_p)$. 
\end{proposition}

\begin{remark}\label{CMremark}Note that ${\mathbf{z}_{\mathrm{dR}}} = -\mathbf{x}_{\mathrm{dR}}/\mathbf{y}_{\mathrm{dR}}$ specializes to values in $\overline{\mathbb{Q}} \subset \overline{\mathbb{Q}}_p$ at CM points, since the Hodge-de Rham filtration is canonically split over $\overline{\mathbb{Q}}$ by the CM action (and the fact that $\omega_{\mathcal{A}}$ and $\omega_{\mathcal{A}}^{-1}$ have complex conjugate complex structures). Similarly, $\mathbf{z}$ specializes to values in $\overline{\mathbb{Q}}_p$ by \cite{Howe}.
\end{remark}

\begin{definition}\label{denominator}Given a section $x \in \mathcal{O}\mathbb{B}_{\mathrm{dR}}(\mathcal{U}) = \mathcal{O}\mathbb{B}_{\mathrm{dR}}^+[1/t](\mathcal{U}) = \mathcal{O}\mathbb{B}_{\mathrm{dR}}^+(\mathcal{U})[1/t]$ where $\mathcal{U} \in \mathcal{Y}_{\text{pro\'{e}t}}$, we define the \emph{denominator} $i(x)$ to be the minimal $i \in \mathbb{Z}$ so that $t^ix \in \mathcal{O}\mathbb{B}_{\mathrm{dR}}^+(\mathcal{U})$. Note that $i(xy) = i(x)+i(y)$, and that if $i(x) > i(y)$, then $i(x+y) = i(x)$. Moreover, if $i(x) < 0$, then $x \in \ker\theta$. Hence, if $x \in \mathcal{O}(\mathcal{U}) \subset \mathcal{O}\mathbb{B}_{\mathrm{dR}}^+(\mathcal{U})\setminus \{0\}$, then $i(x) = 0$ since $\mathcal{O}\subset\mathcal{O}\mathbb{B}_{\mathrm{dR}}^+ \xrightarrow{\theta} \hat{\mathcal{O}}$ is the natural $p$-adic completion map.
\end{definition}

\begin{proposition}\label{invertible}Both $\mathbf{x}_{\mathrm{dR}}, \mathbf{y}_{\mathrm{dR}} \in t\cdot\mathcal{O}\mathbb{B}_{\mathrm{dR}}^+(\mathcal{Y}_x)$, and $\mathbf{x}_{\mathrm{dR}}, \mathbf{y}_{\mathrm{dR}} \not\in t^2\cdot\mathcal{O}\mathbb{B}_{\mathrm{dR}}^+(\mathcal{Y}_x)$. Hence 
$$\theta(\mathbf{x}_{\mathrm{dR}}/t), \theta(\mathbf{y}_{\mathrm{dR}}/t) \in \hat{\mathcal{O}}(\mathcal{Y}_x).$$
Defining the following affinoid subdomains
$$\mathcal{Y}_{x,\mathrm{dR}} := \{\mathbf{z} \neq 0, \theta(\mathbf{x}_{\mathrm{dR}}/t) \neq 0\}  \hspace{1cm} \mathcal{Y}_{y,\mathrm{dR}} := \{\mathbf{z} \neq 0, \theta(\mathbf{y}_{\mathrm{dR}}/t) \neq 0\}$$
of $\mathcal{Y}_x$, we in fact have
$$\mathcal{Y}_{y,\mathrm{dR}} = \mathcal{Y}_x$$
as well as
$$\mathcal{Y}_{x,\mathrm{dR}} = \mathcal{Y}_x \setminus \{\mathbf{y}_{\mathrm{dR}}/t = 1\}.$$ 
Hence,
$\mathbf{x}_{\mathrm{dR}} \in \mathcal{O}\mathbb{B}_{\mathrm{dR}}^+(\mathcal{Y}_{x,\mathrm{dR}})^{\times}$ and $\mathbf{y}_{\mathrm{dR}} \in \mathcal{O}\mathbb{B}_{\mathrm{dR}}^+(\mathcal{Y}_x)^{\times}$. 
In particular,
$$\mathbf{z}_{\mathrm{dR}} \in \mathcal{O}\mathbb{B}_{\mathrm{dR}}^+(\mathcal{Y}_x)$$
and
$$\mathbf{z}_{\mathrm{dR}} \in \mathcal{O}\mathbb{B}_{\mathrm{dR}}^+(\mathcal{Y}_{x,\mathrm{dR}})^\times.$$
\end{proposition}
\begin{proof}
Recall the Kodaira-Spencer isomorphism (\ref{KS1}) and its extension (\ref{KS2}) to $\mathcal{O}\mathbb{B}_{\mathrm{dR},\mathcal{Y}}^+$-coefficients. By the definition of $\sigma$ and Proposition \ref{Weilextension}, on $\mathcal{Y}_x$ we have 
\begin{equation}\begin{split}\label{KScalc}&\sigma(\frak{s}^{\otimes 2}) \\
&= \left\langle \mathbf{x}_{\mathrm{dR}}\alpha_{\infty,1}t^{-1} + \mathbf{y}_{\mathrm{dR}}\alpha_{\infty,2}t^{-1}, \frac{d}{d(1/\mathbf{z})}\mathbf{x}_{\mathrm{dR}}\alpha_{\infty,1}t^{-1} + \frac{d}{d(1/\mathbf{z})}\mathbf{y}_{\mathrm{dR}}\alpha_{\infty,2}t^{-1}\right\rangle_{\mathrm{Poin}}\otimes d(1/\mathbf{z})\\
&= \left\langle \mathbf{x}_{\mathrm{dR}}\alpha_{\infty,1}t^{-1} + \mathbf{y}_{\mathrm{dR}}\alpha_{\infty,2}t^{-1}, \frac{d}{d(1/\mathbf{z})}\mathbf{x}_{\mathrm{dR}}\alpha_{\infty,1}t^{-1} + \frac{d}{d(1/\mathbf{z})}\mathbf{y}_{\mathrm{dR}}\alpha_{\infty,2}t^{-1}\right\rangle \cdot t^{-1} \otimes d(1/\mathbf{z})\\
&= \left(\mathbf{x}_{\mathrm{dR}}\frac{d}{d(1/\mathbf{z})}\mathbf{y}_{\mathrm{dR}} - \mathbf{y}_{\mathrm{dR}}\frac{d}{d(1/\mathbf{z})}\mathbf{x}_{\mathrm{dR}}\right)\langle \alpha_{\infty,1}t^{-1},\alpha_{\infty,2}t^{-1}\rangle\cdot t^{-1}\otimes d(1/\mathbf{z}) \\
&= \left(\frac{\mathbf{x}_{\mathrm{dR}}}{t}\frac{d}{d(1/\mathbf{z})}\left(\frac{\mathbf{y}_{\mathrm{dR}}}{t}\right) - \frac{\mathbf{y}_{\mathrm{dR}}}{t}\frac{d}{d(1/\mathbf{z})}\left(\frac{\mathbf{x}_{\mathrm{dR}}}{t}\right)\right)\otimes d(1/\mathbf{z}) \\
&= -\left(\frac{\mathbf{y}_{\mathrm{dR}}}{t}\right)^2\frac{d}{d(1/\mathbf{z})}{\mathbf{z}_{\mathrm{dR}}}\otimes d(1/\mathbf{z}) \\
&= -\left(-\frac{1}{\mathbf{z}^2}\left(\frac{\mathbf{y}_{\mathrm{dR}}}{t}\right)^2 +\frac{1}{\mathbf{z}^2}\frac{\mathbf{y}_{\mathrm{dR}}}{t} + \mathbf{z}\frac{d}{d(1/\mathbf{z})}\left(\frac{\mathbf{y}_{\mathrm{dR}}}{t}\right)\right)\otimes d(1/\mathbf{z}).
\end{split}
\end{equation}
Here the third-to-last equality holds since $\langle \alpha_{\infty,2}t^{-1},\alpha_{\infty,1}^{-1}\rangle \cdot t^{-1} = t^{-2}$ by our definition of $t$ (\ref{generator}), and the last equality follows from the $p$-adic Legendre relation (\ref{periodsrelation}). Since $\frak{s}^{\otimes 2}$ is a generator of $\omega_{\mathcal{A}}^{\otimes 2}$ over $\mathcal{O}_{\mathcal{Y}_x}$, and $d(1/\mathbf{z})$ is a generator of $\Omega_{\mathcal{Y}_x}^1$ over $\mathcal{O}_{\mathcal{Y}_x}$, the commutativity of the diagram (\ref{KSdiagram}) implies that 
\begin{equation}\label{unit}-\left(\frac{\mathbf{y}_{\mathrm{dR}}}{t}\right)^2\frac{d}{d(1/\mathbf{z})}{\mathbf{z}_{\mathrm{dR}}} = -\frac{1}{\mathbf{z}^2}\left(\frac{\mathbf{y}_{\mathrm{dR}}}{t}\right)^2 +\frac{1}{\mathbf{z}^2}\frac{\mathbf{y}_{\mathrm{dR}}}{t} + \mathbf{z}\frac{d}{d(1/\mathbf{z})}\left(\frac{\mathbf{y}_{\mathrm{dR}}}{t}\right) \in \mathcal{O}(\mathcal{Y}_x)^{\times}.
\end{equation}

Note that since $t \in \mathbb{B}_{\mathrm{dR}}^+(\mathcal{Y})$ is a horizontal section, then $t$ commutes with $\frac{d}{d(1/\mathbf{z})}$, and so 
$$i\left(\frac{d}{d(1/\mathbf{z})}(\mathbf{y}_{\mathrm{dR}}/t)\right) \ge i(\mathbf{y}_{\mathrm{dR}}/t).$$
Hence on $\mathcal{Y}_x \cap \mathcal{Y}_y$ (where $\mathbf{z} \neq \infty$, and so $\mathbf{z} \in \mathcal{O}(\mathcal{Y}_x\cap\mathcal{Y}_y)$),
$$i\left(\frac{1}{\mathbf{z}^2}\mathbf{y}_{\mathrm{dR}}/t+\mathbf{z}\frac{d}{d(1/\mathbf{z})}(\mathbf{y}_{\mathrm{dR}}/t)\right) \ge i(\mathbf{y}_{\mathrm{dR}}/t).$$
If $i > 0$, then $i\left((\mathbf{y}_{\mathrm{dR}}/t)^2\right) = 2i(\mathbf{y}_{\mathrm{dR}}/t) > i(\mathbf{y}_{\mathrm{dR}}/t)$, and so
$$i\left(-\frac{1}{\mathbf{z}^2}(\mathbf{y}_{\mathrm{dR}}/t)^2 +\frac{1}{\mathbf{z}^2}(\mathbf{y}_{\mathrm{dR}}/t) + \mathbf{z}\frac{d}{d(1/\mathbf{z})}(\mathbf{y}_{\mathrm{dR}}/t)\right) = 2i\left(\frac{1}{\mathbf{z}^2}\mathbf{y}_{\mathrm{dR}}/t\right) > 0,$$
which contradicts (\ref{unit}) (by the remarks in Definition \ref{denominator}). So $i(\mathbf{y}_{\mathrm{dR}}/t) \le 0$, and now (\ref{periodsrelation}) implies that $i(\mathbf{x}_{\mathrm{dR}}/t) \le 0$ as well. If $\mathbf{z} = \infty$, i.e. $1/\mathbf{z} = 0$, then $\mathbf{y}_{\mathrm{dR}}/t = 1$ again by (\ref{periodsrelation}). Thus we have the first claim.

For the second claim, note that from (\ref{unit}) above and the fact that $\mathcal{O}\subset \mathcal{O}\mathbb{B}_{\mathrm{dR}}^+ \xrightarrow{\theta} \hat{\mathcal{O}}$ is the natural inclusion, we have 
$$-\theta\left(\frac{\mathbf{y}_{\mathrm{dR}}}{t}\right)^2\theta\left(\frac{d}{d(1/\mathbf{z})}{\mathbf{z}_{\mathrm{dR}}}\right) = -\left(\frac{\mathbf{y}_{\mathrm{dR}}}{t}\right)^2\frac{d}{d(1/\mathbf{z})}{\mathbf{z}_{\mathrm{dR}}} \in \mathcal{O}(\mathcal{Y}_x)^{\times}.$$
Hence $\theta\left(\frac{\mathbf{y}_{\mathrm{dR}}}{t}\right) \neq 0$ on $\mathcal{Y}_x$. 

For the third claim, by (\ref{periodsrelation}) on $\mathcal{Y}_x$ we have
$$\mathbf{z}_{\mathrm{dR}} = \mathbf{z}_{\mathrm{HT}}(1-t/\mathrm{y}_{\mathrm{dR}}).$$
Since $\mathbf{z}_{\mathrm{HT}} \neq 0$, $\mathbf{z}_{\mathrm{dR}} = 0 \iff \mathbf{y}_{\mathrm{dR}}/t = 1$.

For the final claim, we must show that $i(\mathbf{x}_{\mathrm{dR}}/t), i(\mathbf{y}_{\mathrm{dR}}/t) = 0$. By the previous paragraph, we know that $i(\mathbf{x}_{\mathrm{dR}}/t), i(\mathbf{y}_{\mathrm{dR}}/t) \ge 0$, and so $\mathbf{x}_{\mathrm{dR}}/t, \mathbf{y}_{\mathrm{dR}}/t \in \mathcal{O}\mathbb{B}_{\mathrm{dR}}^+(\mathcal{Y})$. By Proposition \ref{ordinaryperiod}, on $\mathcal{Y}^{\mathrm{Ig}} \subset \mathcal{Y}$ we have $\mathbf{y}_{\mathrm{dR}}/t = 1$. (We know that $\mathcal{Y}^{\mathrm{Ig}} \neq \emptyset$ because it surjectively covers $\emptyset \neq Y^{\mathrm{ord}} \subset Y$.) Hence $i(\mathbf{y}_{\mathrm{dR}}/t) = 0$ (on $\mathcal{Y}$). Now by the $p$-adic Legendre relation (\ref{periodsrelation}), we have
$$(\mathbf{x}_{\mathrm{dR}}/t)/\mathbf{z} + \mathbf{y}_{\mathrm{dR}}/t = 1$$
which, upon applying $\theta$, implies
$$\theta(\mathbf{x}_{\mathrm{dR}}/t)/\mathbf{z} + \theta(\mathbf{y}_{\mathrm{dR}}/t) = 1.$$
And so since $i(1/\mathbf{z}) = 0$, we have $i(\mathbf{x}_{\mathrm{dR}}/t) = 0$.  

For the final statement, one notes that $\mathbf{x}_{\mathrm{dR}}/t, \mathbf{y}_{\mathrm{dR}}/t$ are invertible in $\mathcal{O}\mathbb{B}_{\mathrm{dR}}^+$ if and only if $\theta(\mathbf{x}_{\mathrm{dR}}/t), \theta(\mathbf{y}_{\mathrm{dR}}/t)$ are invertible in $\hat{\mathcal{O}}$. 
\end{proof}

\begin{corollary}The affinoid subdomain 
$$\mathcal{V} = \mathcal{Y}^{\mathrm{ss}} \cap \{\theta({\mathbf{z}_{\mathrm{dR}}}) \not\in \mathbb{Q}_p\}$$
of $\mathcal{Y}^{\mathrm{ss}}$ is preserved by the action of $GL_2(\mathbb{Z}_p)$. That is, $\mathcal{V} \cdot GL_2(\mathbb{Z}_p) =  \mathcal{V}$. 
\end{corollary}
\begin{proof}From (\ref{y'transformationprop}), we have
$$\left(\begin{array}{ccc} a & b \\
c & d\\
\end{array}\right)^*(\mathbf{y}_{\mathrm{dR}}/t) = (bc-ad)^{-1}(c{\mathbf{z}_{\mathrm{dR}}} + a)(b\mathbf{z}^{-1} + d)\mathbf{y}_{\mathrm{dR}}/t$$
for any $\left(\begin{array}{ccc} a & b \\
c & d\\
\end{array}\right) \in GL_2(\mathbb{Z}_p)$. Applying $\theta : \mathcal{O}\mathbb{B}_{\mathrm{dR}}^+ \rightarrow \hat{\mathcal{O}}$, we get
$$\left(\begin{array}{ccc} a & b \\
c & d\\
\end{array}\right)^*(\theta(\mathbf{y}_{\mathrm{dR}}/t)) = (bc-ad)^{-1}(c\theta({\mathbf{z}_{\mathrm{dR}}}) + a)(b\mathbf{z}^{-1} + d)\theta(\mathbf{y}_{\mathrm{dR}}/t).$$
Suppose that the above quantity specializes to 0 at some point. Then since $\mathbf{z}$ cannot specialize to $-d/b$ (as we are on the supersingular locus, and $\mathbf{z}$ can only be $\mathbb{Q}_p$-valued at ordinary points), and $\theta(\mathbf{y}_{\mathrm{dR}}/t) \neq 0$ (since $\mathcal{Y}_{y,\mathrm{dR}} = \mathcal{Y}_x \subset \mathcal{Y}^{\mathrm{ss}}$ and we are on $\mathcal{Y}^{\mathrm{ss}}$ by assumption), then we must have $\theta({\mathbf{z}_{\mathrm{dR}}})$ specializes to the value $-a/c \in \mathbb{Q}_p$ at some point which contradicts our assumption.
\end{proof}

\subsection{The canonical differential}

\begin{definition}Henceforth, let 
$$\omega_{\mathrm{dR}} := \omega_{\mathcal{A}}\otimes_{\mathcal{O}_Y}\mathcal{O}\mathbb{B}_{\mathrm{dR},Y}^+, \hspace{.5cm} \omega_{\mathrm{dR}}^{-1} := \omega_{\mathcal{A}}\otimes_{\mathcal{O}_Y}\mathcal{O}\mathbb{B}_{\mathrm{dR},Y}^+.$$
Note that 
$$\omega_{\mathrm{dR}}|_{\mathcal{Y}}(n) = \omega_{\mathrm{dR}}|_{\mathcal{Y}}\cdot t^n, \hspace{.5cm} \omega_{\mathrm{dR}}^{-1}|_{\mathcal{Y}}(n) = \omega_{\mathrm{dR}}^{-1}|_{\mathcal{Y}}\cdot t^n.$$

\begin{definition}\label{canonicaldifferentialdef}We define the \emph{canonical differential} as
\begin{equation}\omega_{\mathrm{can}} := \frac{t}{\mathbf{y}_{\mathrm{dR}}}\cdot \frak{s} \in \omega_{\mathrm{dR}}(\mathcal{Y}_x).
\end{equation}
By (\ref{stransformationprop}) and (\ref{y'transformationprop}), we have
\begin{equation}\label{s'transformationprop}\left(\begin{array}{ccc} a & b \\
c & d\\
\end{array}\right)^*\omega_{\mathrm{can}} = (bc-ad)(c\mathbf{z}_{\mathrm{dR}}+a)^{-1}\cdot \omega_{\mathrm{can}}
\end{equation}
for any $\left(\begin{array}{ccc} a & b \\
c & d\\
\end{array}\right) \in GL_2(\mathbb{Z}_p)$.
\end{definition}

Note that
\begin{equation}\label{canonicaldifferentialcoordinates}\iota_{\mathrm{dR}}(\omega_{\mathrm{can}}) = -\mathbf{z}_{\mathrm{dR}}\cdot \alpha_{\infty,1} + \alpha_{\infty,2}
\end{equation}
and (using Corollary \ref{Weilextension})
\begin{equation}\label{canonicaldifferentialKScalculation}\begin{split}\sigma(\omega_{\mathrm{can}}^{\otimes 2}) &= \langle \omega_{\mathrm{can}},\nabla(\omega_{\mathrm{can}})\rangle_{\mathrm{Poin}} = \langle \iota_{\mathrm{dR}}(\omega_{\mathrm{can}}), \nabla(\iota_{\mathrm{dR}}(\omega_{\mathrm{can}}))\rangle\cdot t^{-1}\\
&=\langle -\mathbf{z}_{\mathrm{dR}}\cdot\alpha_{\infty,1} + \alpha_{\infty,2}, \nabla(-\mathbf{z}_{\mathrm{dR}}\cdot\alpha_{\infty,1} + \alpha_{\infty,2})\rangle\cdot t^{-1} = -d\mathbf{z}_{\mathrm{dR}}.
\end{split}
\end{equation}
\end{definition}
\begin{proposition}\label{generator2}In fact, $\omega_{\mathrm{can}}$ generates $\omega_{\mathrm{dR}}(\mathcal{Y}_x)$, and $d\mathbf{z}_{\mathrm{dR}}$ generates $\Omega_{\mathcal{Y}}^1\otimes_{\mathcal{O}_{\mathcal{Y}}} \mathcal{O}\mathbb{B}_{\mathrm{dR},\mathcal{Y}_x}^+$. 
\end{proposition}
\begin{proof}The fake Hasse invariant $\frak{s}$ is a generator of $\omega_{\mathcal{A}}(\mathcal{Y}_x)$. By Proposition \ref{invertible}, $\mathbf{y}_{\mathrm{dR}}/t$ is invertible in $\mathcal{O}\mathbb{B}_{\mathrm{dR}}^+(\mathcal{Y}_{x})$, so the first claim of the Proposition follows. Now the second claim follows from (\ref{generator2}), and the fact that the Kodaira-Spencer map (\ref{KSdiagram}) is an isomorphism.
\end{proof}

\begin{proposition}\label{canonicalextension}On $\mathcal{Y}^{\mathrm{Ig}}$ (see Section \ref{ordinarysection} for a definition), we have
$$\omega_{\mathrm{can}}^{\mathrm{Katz}} = \frak{s}|_{\mathcal{Y}^{\mathrm{Ig}}}.$$
In other words, $\omega_{\mathrm{can}} \in \omega_{\Delta,\mathcal{Y}}(\mathcal{Y}_x)$ extends Katz's canonical differential $\omega_{\mathrm{can}}^{\mathrm{Katz}}$ (see Definition \ref{Katzcanonicaldifferential}) from $\mathcal{Y}^{\mathrm{Ig}} \subset \mathcal{Y}_x$ to all of $\mathcal{Y}_x$. 
\end{proposition}
\begin{proof}
This follows from Proposition \ref{ordinaryperiod} and Definition \ref{SerreTatecoordinate}.
\end{proof}

\begin{proposition}\label{H10factors}The $p$-adic comparison injection $\iota_{\mathrm{dR}}$, restricted to the Hodge filtration $\omega_{\mathrm{dR},\mathcal{Y}_x} \subset \mathcal{H}_{\mathrm{dR}}^1(\mathcal{A})\otimes_{\mathcal{O}_Y} \mathcal{O}\mathbb{B}_{\mathrm{dR},\mathcal{Y}_x}^+$, i.e.
$$\omega_{\mathrm{dR},\mathcal{Y}_x}\overset{\iota_{\mathrm{dR}}}{\hookrightarrow} T_p\mathcal{A}\otimes_{\hat{\mathbb{Z}}_{p,Y}} \mathcal{O}\mathbb{B}_{\mathrm{dR},\mathcal{Y}_x}^+\cdot t^{-1},$$
factors through
$$\omega_{\mathrm{dR},\mathcal{Y}_x}\overset{\iota_{\mathrm{dR}}}{\hookrightarrow} T_p\mathcal{A}\otimes_{\hat{\mathbb{Z}}_{p,Y}} \mathcal{O}\mathbb{B}_{\mathrm{dR},\mathcal{Y}_x}^+,$$
which, reducing modulo $t$, induces 
$$\omega_{\Delta,\mathcal{Y}_x}\overset{\iota_{\mathrm{dR}}}{\hookrightarrow} T_p\mathcal{A}\otimes_{\hat{\mathbb{Z}}_{p,Y}} \mathcal{O}_{\Delta,\mathcal{Y}_x}.$$
Hence we have a natural map 
\begin{equation}\label{naturalmap}\omega_{\mathcal{A}}|_{\mathcal{Y}_x} \subset \omega_{\Delta,\mathcal{Y}_x}\overset{\iota_{\mathrm{dR}}}{\subset} T_p\mathcal{A}\otimes_{\hat{\mathbb{Z}}_{p,Y}} \mathcal{O}_{\Delta,\mathcal{Y}_x}\overset{\theta}{\twoheadrightarrow} T_p\mathcal{A}\otimes_{\hat{\mathbb{Z}}_{p,Y}} \hat{\mathcal{O}}_{\mathcal{Y}_x} \xrightarrow{HT_{\mathcal{A}}} \omega_{\mathcal{A}}\otimes_{\mathcal{O}_Y}\hat{\mathcal{O}}_{\mathcal{Y}_x},
\end{equation}
and in fact this map is the natural inclusion.
\end{proposition}

\begin{proof}This follows from Proposition (\ref{invertible}): since $\mathbf{x}_{\mathrm{dR}}/t, \mathbf{y}_{\mathrm{dR}}/t \in \mathcal{O}\mathbb{B}_{\mathrm{dR}}^+(\mathcal{Y}_x)$, then for the generator $\frak{s} \in \omega_{\Delta,\mathcal{Y}}(\mathcal{Y}_x)$ we have $\iota_{\mathrm{dR}}(\frak{s}) \in \mathcal{O}\mathbb{B}_{\mathrm{dR},\mathcal{Y}}(\mathcal{Y}_x)$, from which the first statement follows. The second statement follows immediately from the first. For the third statement, first note that
$$\theta \circ HT_{\mathcal{A}} = HT_{\mathcal{A}}\circ \theta$$
and so by Proposition \ref{compositionproposition}, (\ref{naturalmap}) is just the natural inclusion.  
\end{proof}

\begin{lemma}\label{nonvanishing}Consider the derivation
$$\frac{d}{d\mathbf{z}} : \mathcal{O}\mathbb{B}_{\mathrm{dR},\mathcal{Y}}^+ \rightarrow \mathcal{O}\mathbb{B}_{\mathrm{dR},\mathcal{Y}}^+.$$
We have
$$\frac{d}{d\mathbf{z}}{\mathbf{z}_{\mathrm{dR}}} \in \mathcal{O}\mathbb{B}_{\mathrm{dR},\mathcal{Y}}^+(\mathcal{Y}_x \cap \mathcal{Y}_y)^{\times}$$
is invertible. (Note that we need to restrict to $\mathcal{Y}_x \cap \mathcal{Y}_y \subset \mathcal{Y}_x$ in order to define $\frac{d}{d\mathbf{z}}$, since $d\mathbf{z}$ is a regular differential only on $\Omega_{\mathcal{Y}_x}^1|_{\mathcal{Y}_x\cap\mathcal{Y}_y}$.)
\end{lemma}

\begin{proof}
By Proposition \ref{generator2}, the line 
\begin{align*}\mathcal{O}\mathbb{B}_{\mathrm{dR}}^+(\mathcal{Y}_x) \cong \omega _{\mathrm{dR}}(\mathcal{Y}_x)& \overset{\iota_{\mathrm{dR}}}{\subset} (\mathcal{H}_{\text{\'{e}t}}^1(\mathcal{A})\otimes_{\hat{\mathbb{Z}}_{p,Y}}\mathcal{O}\mathbb{B}_{\mathrm{dR},\mathcal{Y}}^+)(\mathcal{Y}_x)\\
&\underset{\alpha_{\infty}^{-1}}{\xrightarrow{\sim}} (t^{-1}\cdot\mathcal{O}\mathbb{B}_{\mathrm{dR},\mathcal{Y}}^+)^{\oplus 2}(\mathcal{Y}_x)
\end{align*}
is generated by
\begin{equation}\omega_{\mathrm{can}} := \frac{t}{\mathbf{y}_{\mathrm{dR}}}\cdot \frak{s}.
\end{equation}
Then under the Kodaira-Spencer isomorphism \ref{KS1}, and using Theorem \ref{Weilextension} and the equation (\ref{canonicaldifferentialcoordinates}), we have that
\begin{equation}\label{basiselt}\begin{split}\sigma(\omega_{\mathrm{can}}\otimes_{\mathcal{O}\mathbb{B}_{\mathrm{dR},\mathcal{Y}}^+}\omega_{\mathrm{can}}) &= \langle \omega_{\mathrm{can}},\nabla(\omega_{\mathrm{can}})\rangle_{\mathrm{Poin}} = \langle \iota_{\mathrm{dR}}(\omega_{\mathrm{can}}), \nabla(\iota_{\omega_{\mathrm{can}}})\rangle\cdot t^{-1}\\
&=\left\langle -\mathbf{z}_{\mathrm{dR}}\alpha_{\infty,1} + \alpha_{\infty,2}, \nabla(-\mathbf{z}_{\mathrm{dR}}\alpha_{\infty,1} + \alpha_{\infty,2})\right\rangle \cdot t^{-1}\\
&= \left\langle -\mathbf{z}_{\mathrm{dR}}\alpha_{\infty,1} + \alpha_{\infty,2},-\frac{d}{d\mathbf{z}}{\mathbf{z}_{\mathrm{dR}}}\alpha_{\infty,1}\right\rangle\cdot t^{-1}\otimes d\mathbf{z} = -\frac{d}{d\mathbf{z}}{\mathbf{z}_{\mathrm{dR}}}\otimes d\mathbf{z}
\end{split}
\end{equation}
is a generator of $(\mathcal{O}\mathbb{B}_{\mathrm{dR},Y}^+\otimes_{\mathcal{O}_{Y}} \Omega_{Y}^1)(\mathcal{Y}_x \cap \mathcal{Y}_y)$ as an $\mathcal{O}\mathbb{B}_{\mathrm{dR},\mathcal{Y}}^+(\mathcal{Y}_x \cap \mathcal{Y}_y)$-module. Hence 
$$\frac{d}{d\mathbf{z}}{\mathbf{z}_{\mathrm{dR}}} \in \mathcal{O}\mathbb{B}_{\mathrm{dR}}^+(\mathcal{Y}_x \cap \mathcal{Y}_y)^{\times}$$
which is what we wanted to show.

\end{proof}

Note that by the $p$-adic Legendre relation (\ref{periodsrelation}) and Proposition \ref{invertible},
\begin{equation}\mathcal{Y}_x\cap \mathcal{Y}_y \subset \mathcal{Y}_x \cap \{\mathbf{y}_{\mathrm{dR}}/t = 1\} = \mathcal{Y}_{x,\mathrm{dR}}.
\end{equation}

\subsection{The ``horizontal'' lifting of the Hodge-Tate filtration}In this section, we construct a splitting 

Let $\frak{s}^{-1} \in \omega_{\mathcal{A}}^{-1}(\mathcal{Y}_x)$ denote the \emph{unique} section such that 
$$\langle \iota_{\mathrm{dR}}(\frak{s}),\frak{s}^{-1}\rangle = 1.$$
Then $\frak{s}^{-1}$ generates $\omega_{dR}^{-1}|_{\mathcal{Y}_x}$. Using the $p$-adic Legendre relation (\ref{periodsrelation}), we have the exact sequence
\begin{equation}\label{relativeHTsequenceOBdR} 0\rightarrow \omega_{\mathrm{dR}}^{-1}|_{\mathcal{Y}_x}\cdot t \rightarrow T_p\mathcal{A}\otimes_{\hat{\mathbb{Z}}_{p,Y}}\mathcal{O}\mathbb{B}_{\mathrm{dR},\mathcal{Y}_x}^+ \underset{\alpha_{\infty}^{-1}}{\xrightarrow{\sim}} (\mathcal{O}\mathbb{B}_{\mathrm{dR},\mathcal{Y}_x}^+)^{\oplus 2} \rightarrow \omega_{\mathrm{dR}}|_{\mathcal{Y}_x} \rightarrow 0
\end{equation}
where the second arrow is defined by sending $\frak{s}^{-1}t \mapsto \alpha_{\infty,1} - 1/\mathbf{z}\cdot\alpha_{\infty,2}$ and the penultimate arrow is defined by sending $\alpha_{\infty,2} \mapsto \frak{s}$ and $\alpha_{\infty,1} \mapsto 1/\mathbf{z}\cdot \frak{s}$.

By (\ref{periodsrelation}), we have the following relative $\mathcal{O}\mathbb{B}_{\mathrm{dR},\mathcal{Y}_x}^+$-Hodge-Tate decomposition on $\mathcal{Y}_x$:
\begin{equation}\label{HTdecomposition1}\begin{split}&\omega_{\mathrm{dR}}|_{\mathcal{Y}_x} \oplus \left(\omega_{\mathrm{dR}}^{-1}|_{\mathcal{Y}_x}\cdot t\right)= \frak{s}\cdot\mathcal{O}\mathbb{B}_{\mathrm{dR},\mathcal{Y}_x}^+\oplus \frak{s}^{-1}t\cdot\mathcal{O}\mathbb{B}_{\mathrm{dR},\mathcal{Y}_x}^+ \\
&\underset{\sim}{\xrightarrow{\iota_{\mathrm{dR}}}} \left((\mathbf{x}_{\mathrm{dR}}/t\cdot\alpha_{\infty,1} + \mathbf{y}_{\mathrm{dR}}/t\cdot \alpha_{\infty,2})\mathcal{O}\mathbb{B}_{\mathrm{dR},\mathcal{Y}_x}^+\right)\oplus \left((\alpha_{\infty,1} -1/\mathbf{z}\cdot\alpha_{\infty,2})\mathcal{O}\mathbb{B}_{\mathrm{dR},\mathcal{Y}_x}^+\right) \\
&\overset{(\ref{periodsrelation})}{=} \mathcal{O}\mathbb{B}_{\mathrm{dR},\mathcal{Y}_x}^{+,\oplus 2} \underset{\sim}{\xrightarrow{\alpha_{\infty}}}  T_p\mathcal{A}\otimes_{\hat{\mathbb{Z}}_{p,Y}}\mathcal{O}\mathbb{B}_{\mathrm{dR},\mathcal{Y}_x}^+
\end{split}
\end{equation}
\begin{definition}\label{periodsmodtdefinition}Henceforth, let
$$x_{\mathrm{dR}} := \mathbf{x}_{\mathrm{dR}}/t \pmod{t}, \hspace{.5cm} y_{\mathrm{dR}} := \mathbf{y}_{\mathrm{dR}}/t \pmod{t}, \hspace{.5cm} z_{\mathrm{dR}} := \mathbf{z}_{\mathrm{dR}} \pmod{t}.$$
\end{definition}

\begin{definition}\label{diagonalringdefinition}Henceforth, let 
$$\mathcal{O}_{\Delta,\mathcal{Y}} = \mathcal{O}\mathbb{B}_{\mathrm{dR},\mathcal{Y}}^+/(t)$$
and
$$\omega_{\Delta,\mathcal{Y}} := \omega_{\mathcal{A}}\otimes_{\mathcal{O}_Y}\mathcal{O}_{\Delta,\mathcal{Y}}.$$
We will also let $\omega_{\mathrm{can}}$ denote the image of $\omega_{\mathrm{can}}$ from Definition \ref{canonicaldifferentialdef} modulo $(t)$.
\end{definition}

Reducing (\ref{HTdecomposition1}) modulo $(t)$, we get a relative $\mathcal{O}_{\Delta,\mathcal{Y}}|_{\mathcal{Y}_x}$-Hodge-Tate decomposition on $\mathcal{Y}_x$:
\begin{equation}\label{HTdecomposition2}\begin{split}&\omega_{\Delta,\mathcal{Y}}|_{\mathcal{Y}_x} \oplus \omega_{\Delta,\mathcal{Y}}^{-1}\cdot t|_{\mathcal{Y}_x} = \frak{s}\cdot\mathcal{O}_{\Delta,\mathcal{Y}}|_{\mathcal{Y}_x} \oplus \frak{s}^{-1}t\cdot\mathcal{O}_{\Delta,\mathcal{Y}}|_{\mathcal{Y}_x} \\
&\underset{\sim}{\xrightarrow{\iota_{\mathrm{dR}}}} \left((x_{\mathrm{dR}}\cdot\alpha_{\infty,1} + y_{\mathrm{dR}}\cdot \alpha_{\infty,2})\mathcal{O}_{\Delta,\mathcal{Y}}|_{\mathcal{Y}_x} \right)\oplus \left((\alpha_{\infty,1} -1/\mathbf{z}\cdot\alpha_{\infty,2})\mathcal{O}_{\Delta,\mathcal{Y}}|_{\mathcal{Y}_x} \right) \\
&\overset{(\ref{diagonal})}{=} \mathcal{O}_{\Delta,\mathcal{Y}}|_{\mathcal{Y}_x} ^{\oplus 2} \underset{\sim}{\xrightarrow{\alpha_{\infty}}}  T_p\mathcal{A}\otimes_{\hat{\mathbb{Z}}_{p,Y}}\mathcal{O}_{\Delta,\mathcal{Y}}|_{\mathcal{Y}_x} 
\end{split}
\end{equation}

\begin{definition}Henceforth, denote the projection to the first factor above the above splitting by
\begin{equation}\label{splitdef2}\mathrm{split} :  T_p\mathcal{A}\otimes_{\hat{\mathbb{Z}}_{p,Y}}\mathcal{O}_{\Delta,\mathcal{Y}}|_{\mathcal{Y}_x} \twoheadrightarrow \omega_{\Delta,\mathcal{Y}}|_{\mathcal{Y}_x}.
\end{equation}
\end{definition}

One can use $\mathrm{split}$ to define a $p$-adic Maass-Shimura operator; however, for the purposes of defining a differential operator with satisfactory properties, such as coinciding in value at CM points with those of the complex Maass-Shimura operator, it will be necessary to define a ``horizontal'' lifting of the relative Hodge-Tate filtration, i.e. such that the lifting is generated by an element contained in the horizontal sections of the Gauss-Manin connection $\nabla$. To be more precise, this means that we will define a line $\mathcal{L} \subset T_p\mathcal{A} \otimes_{\hat{\mathbb{Z}}_{p,Y}} \mathcal{O}_{\Delta,\mathcal{Y}_x}$ such that $\theta(\mathcal{L})$ coincides with the usual Hodge filtration in (\ref{relativeHTfiltration2}), and that this line is horizontal means that $\nabla_w(\mathcal{L}) \subset \mathcal{L}$ for any non-vanishing section $w$ of $\Omega_{\mathcal{Y}}^1 \otimes_{\mathcal{O}_{\mathcal{Y}}} \mathcal{O}_{\Delta,\mathcal{Y}_x}$. That our $\mathcal{L}$ lifts the Hodge filtration will be immediate from its construction (see (\ref{modtheta})), and we prove the horizontalness of $\mathcal{L}$ in Proposition \ref{holomorphicallyhorizontal}; in fact, we show that $\mathcal{L}$ is the unique such line which has both these properties (see Proposition \ref{uniquehorizontal}).

Note that we have the natural map
\begin{equation}\label{horizontalsections}j : \mathcal{O}_{\mathcal{Y}} \rightarrow \hat{\mathcal{O}}_{\mathcal{Y}} \cong \mathbb{B}_{\mathrm{dR},\mathcal{Y}}^+/(t)  \subset \mathcal{O}\mathbb{B}_{\mathrm{dR},\mathcal{Y}}^+/(t)
\end{equation}
which is a section of $\theta : \mathcal{O}\mathbb{B}_{\mathrm{dR},\mathcal{Y}}^+ \twoheadrightarrow \hat{\mathcal{O}}_{\mathcal{Y}}$, i.e.
\begin{equation}\label{jsection}\theta \circ j = \mathrm{id},
\end{equation}
since $\mathcal{O}_{\mathcal{Y}} \subset \mathcal{O}\mathbb{B}_{\mathrm{dR},\mathcal{Y}}^+ \overset{\theta}{\twoheadrightarrow} \hat{\mathcal{O}}_{\mathcal{Y}}$ is the natural inclusion. This embeds $\hat{\mathcal{O}}_{\mathcal{Y}}$ into the horizontal sections of the $\hat{\mathcal{O}}_{\mathcal{Y}}$-linear connection
$$\nabla : \mathcal{O}\mathbb{B}_{\mathrm{dR},\mathcal{Y}}^+/(t) \rightarrow \mathcal{O}\mathbb{B}_{\mathrm{dR},\mathcal{Y}}^+/(t) \otimes \Omega_{\mathcal{Y}}^1$$
which is induced by the $\mathbb{B}_{\mathrm{dR},\mathcal{Y}}^+$-linear connection 
$$\nabla : \mathcal{O}\mathbb{B}_{\mathrm{dR},\mathcal{Y}}^+ \rightarrow \mathcal{O}\mathbb{B}_{\mathrm{dR},\mathcal{Y}}^+ \otimes \Omega_{\mathcal{Y}}^1$$
since $t \in \mathbb{B}_{\mathrm{dR}}^+(\mathcal{Y})$ is a horizontal section. 

\begin{definition}\label{horizontaldefinition}
Let 
$$\bar{z} = j(\mathbf{z})$$
and note that
\begin{equation}\label{modtheta}\theta(\bar{z}) = \hat{\mathbf{z}} = \theta(\mathbf{z})
\end{equation}
and in particular
$$\mathbf{z}-\bar{z} \in \ker\theta.$$
\end{definition}

We may, at times, conflate the notations $\mathbf{z}$ and $\hat{\mathbf{z}}$, as they denote the same section under the natural map $\mathcal{O} \rightarrow \hat{\mathcal{O}}$. 

\subsection{The ``horizontal'' relative Hodge-Tate decomposition}

\begin{proposition}We have
\begin{equation}\label{diagonal}x_{\mathrm{dR}}/\bar{z} + y_{\mathrm{dR}} \in \mathcal{O}_{\Delta,\mathcal{Y}}(\mathcal{Y}_x)^{\times}.
\end{equation}
\end{proposition}

\begin{proof}By \ref{periodsrelation}, we have
\begin{align*}&\frac{x_{\mathrm{dR}}}{\bar{z}}\left(-\left(\frac{\mathbf{z}}{\bar{z}}-1\right) + \left(\frac{\mathbf{z}}{\bar{z}}-1\right)^2 - \cdots\right) + \frac{x_{\mathrm{dR}}}{\bar{z}} + y_{\mathrm{dR}} \\
&= x_{\mathrm{dR}}/(\bar{z} + (\mathbf{z}-\bar{z})) + y_{\mathrm{dR}} = x_{\mathrm{dR}}/\mathbf{z} + y_{\mathrm{dR}} = 1
\end{align*}
and so
$$\frac{x_{\mathrm{dR}}}{\bar{z}} + y_{\mathrm{dR}} = 1 + \frac{x_{\mathrm{dR}}}{\bar{z}}\left(\left(\frac{\mathbf{z}}{\bar{z}}-1\right) - \left(\frac{\mathbf{z}}{\bar{z}}-1\right)^2 + \cdots \right) \in \mathcal{O}_{\Delta,\mathcal{Y}}(\mathcal{Y}_x)^{\times}$$
since $\left(\frac{\mathbf{z}}{\bar{z}}-1\right) \in \ker\theta$.
\end{proof}

Henceforth, let
\begin{equation}\label{horizontalkernel}\mathcal{L} := \left((\alpha_{\infty,1} -1/\bar{z}\cdot\alpha_{\infty,2})\mathcal{O}_{\Delta,\mathcal{Y}}|_{\mathcal{Y}_x} \right).
\end{equation}
Now we have the following $\mathcal{O}_{\Delta,\mathcal{Y}}|_{\mathcal{Y}_x}$-Hodge-Tate decomposition
\begin{equation}\label{HTdecomposition3}\begin{split}&\omega_{\Delta,\mathcal{Y}}|_{\mathcal{Y}_x}\oplus \mathcal{L}\\
&\underset{\sim}{\xrightarrow{\iota_{\mathrm{dR}}}}\left((x_{\mathrm{dR}}\cdot\alpha_{\infty,1} + y_{\mathrm{dR}}\cdot \alpha_{\infty,2})\mathcal{O}_{\Delta,\mathcal{Y}}|_{\mathcal{Y}_x} \right)\oplus \left((\alpha_{\infty,1} -1/\bar{z}\cdot\alpha_{\infty,2})\mathcal{O}_{\Delta,\mathcal{Y}}|_{\mathcal{Y}_x} \right) \\
&\overset{(\ref{diagonal})}{=} \mathcal{O}_{\Delta,\mathcal{Y}}|_{\mathcal{Y}_x} ^{\oplus 2} \underset{\sim}{\xrightarrow{\alpha_{\infty}}}  T_p\mathcal{A}\otimes_{\hat{\mathbb{Z}}_{p,Y}}\mathcal{O}_{\Delta,\mathcal{Y}}|_{\mathcal{Y}_x} 
\end{split}
\end{equation}
where the first factor in the above decomposition is given by 
$$\iota_{\mathrm{dR}} : \omega_{\Delta,\mathcal{Y}} \hookrightarrow T_p\mathcal{A} \otimes_{\hat{\mathbb{Z}}_{p,Y}}\mathcal{O}_{\Delta,\mathcal{Y}_x}.$$

\begin{definition}Henceforth, denote the projection to the first factor above the above splitting by
\begin{equation}\label{splitdef}\overline{\mathrm{split}} :  T_p\mathcal{A}\otimes_{\hat{\mathbb{Z}}_{p,Y}}\mathcal{O}_{\Delta,\mathcal{Y}}|_{\mathcal{Y}_x} \twoheadrightarrow \omega_{\Delta,\mathcal{Y}}|_{\mathcal{Y}_x}.
\end{equation}
\end{definition}

\begin{proposition}\label{holomorphicallyhorizontal}Let $U \rightarrow \mathcal{Y}_x$ be any pro\'{e}tale cover. We have that $\mathcal{L}|_U$ is horizontal for $\nabla$. That is, for any nonvanishing $w \in \Omega_{\mathcal{Y}}^1\otimes_{\mathcal{O}_{\mathcal{Y}}} \mathcal{O}_{\Delta,\mathcal{Y}_x}(U)$, we have
$$\nabla_w(\mathcal{L}|_U) \subset \mathcal{L}|_U.$$
\end{proposition}

\begin{proof}
Note that the kernel $\mathcal{L}$ of $\overline{\mathrm{split}}$ is generated by the section $\alpha_{\infty,1} - 1/\bar{z}\cdot \alpha_{\infty,2}$ by definition, and so we can write any section of $\mathcal{L}|_U$ as $f \cdot (\alpha_{\infty,1} - 1/\bar{z}\cdot \alpha_{\infty,2})$. Then
\begin{align*}\nabla_w(f\cdot (\alpha_{\infty,1} - 1/\bar{z}\cdot \alpha_{\infty,2})) &= \nabla_w(f)\cdot (\alpha_{\infty,1} - 1/\bar{z}\cdot \alpha_{\infty,2}) + f\cdot \nabla(\alpha_{\infty,1} - 1/\bar{z}\cdot \alpha_{\infty,2}) \\
&= \nabla_w(f)\cdot (\alpha_{\infty,1} - 1/\bar{z}\cdot \alpha_{\infty,2})
\end{align*}
which is a section of $\mathcal{L}|_U$. 
\end{proof}

\begin{proposition}\label{uniquehorizontal}$\mathcal{L}$ is the unique line in $T_p\mathcal{A}\otimes_{\hat{\mathbb{Z}}_{p,Y}} \mathcal{O}_{\Delta,\mathcal{Y}_x}$ with $\theta(\mathcal{L}) = \omega_{\mathcal{A}}^{-1} \otimes_{\mathcal{O}_Y} \hat{\mathcal{O}}_{\mathcal{Y}_x}$ (i.e. $\mathcal{L}$ lifts the Hodge-Tate filtration as in (\ref{relativeHTfiltration2})), and which is horizontal in the sense of Proposition \ref{holomorphicallyhorizontal}.
\end{proposition}

\begin{proof}For any line $\mathcal{L}' \subset T_p\mathcal{A}\otimes_{\hat{\mathbb{Z}}_{p,Y}} \mathcal{O}_{\Delta,\mathcal{Y}_x}$ with $\theta(\mathcal{L}') = \omega_{\mathcal{A}}^{-1} \otimes_{\mathcal{O}_Y} \hat{\mathcal{O}}_{\mathcal{Y}_x} = \langle \alpha_{\infty,1} - 1/\hat{\mathbf{z}}\alpha_{\infty,2}\rangle \cdot \hat{\mathcal{O}}_{\mathcal{Y}_x}$, we must have $\mathcal{L}' = \langle\alpha_{\infty,1} - (1/\bar{z} + X) \alpha_{\infty,2}\rangle \mathcal{O}_{\Delta,\mathcal{Y}_x}$ where 
\begin{equation}\label{Xinkernel}X \in \ker(\theta : \mathcal{O}_{\Delta,\mathcal{Y}_x} \twoheadrightarrow \hat{\mathcal{O}}_{\mathcal{Y}_x}).
\end{equation}
From the calculation in the proof of Proposition (\ref{holomorphicallyhorizontal}), we see that $\mathcal{L}$ is horizontal if and only if $X$ is a horizontal section of $\nabla : \mathcal{O}_{\Delta,\mathcal{Y}_x} \rightarrow \mathcal{O}_{\Delta,\mathcal{Y}_x} \otimes_{\mathcal{O}_{\mathcal{Y}_x}} \Omega_{\mathcal{Y}_x}^1$ (recall that $X$ being a horizontal section means $\nabla(X) = 0$). However, the sheaf of horizontal sections of $\nabla$ (by its construction) is simply the subsheaf $j : \hat{\mathcal{O}}_{\mathcal{Y}_x} \hookrightarrow \mathcal{O}_{\Delta,\mathcal{Y}_x}$ as defined in (\ref{horizontalsections}), so if $X$ is horizontal we have $X \in j(\hat{\mathcal{O}}_{\mathcal{Y}_x})$, say $X = j(x)$. But from (\ref{jsection}), we see that $\theta(X) = \theta(j(x)) = x$, and so from (\ref{Xinkernel}) we have $0 = \theta(X) = x$, which implies $X = j(x) = 0$. 
\end{proof}

By direct calculation, on $\mathcal{Y}_x$ we have
\begin{equation}\label{splitcalculation}\overline{\mathrm{split}}(\alpha_{\infty,1}) = \frac{1}{\bar{z}-{z_{\mathrm{dR}}}}\cdot \omega_{\mathrm{can}}
\end{equation}
and
\begin{equation}\label{canonicaldifferentialcalculation}\nabla(\omega_{\mathrm{can}}) = \nabla(\iota_{\mathrm{dR}}(\omega_{\mathrm{can}})) = \nabla(-z_{\mathrm{dR}}\cdot\alpha_{\infty,1} + \alpha_{\infty,2}) = -\alpha_{\infty,1}\cdot dz_{\mathrm{dR}} \xmapsto{\overline{\mathrm{split}}} \frac{1}{{z_{\mathrm{dR}}}-\bar{z}}\cdot \omega_{\mathrm{can}}
\end{equation}
where
$$dz_{\mathrm{dR}} := \nabla(z_{\mathrm{dR}}).$$

\begin{proposition}\label{STcoincide}On the locus $\mathcal{Y}^{\mathrm{Ig}} \subset \mathcal{Y}_x$, we have that 
$$d\mathbf{z}_{\mathrm{dR}} = d\log T \in \Omega^1(\mathcal{Y}^{\mathrm{Ig}})$$ 
where $T$ is the Serre-Tate coordinate defined in the statement of Theorem \ref{SerreTatethm}. 
\end{proposition}

\begin{proof}By (\ref{KScalc}) and Proposition \ref{ordinaryperiod}, on $\mathcal{Y}^{\mathrm{Ig}}$ we have that 
\begin{equation}\label{KScalc2}\sigma(\frak{s}^{\otimes 2}) = \left(\frac{\mathbf{y}_{\mathrm{dR}}}{t}\right)^2d\mathbf{z}_{\mathrm{dR}} = d\mathbf{z}_{\mathrm{dR}}
\end{equation}
and hence by the commutativity of the diagram (\ref{KSdiagram})
$$d\mathbf{z}_{\mathrm{dR}}\in \Omega^1(\mathcal{Y}_x).$$
By the theorem cited in Definition (\ref{Katzcanonicaldifferential}) and (\ref{KScalc2}), we must have $d\mathbf{z}_{\mathrm{dR}} = d\log T$.

\end{proof}

From the relative $p$-adic de Rham comparison theorem (\ref{comparison}), by taking $k^{\mathrm{th}}$ symmetric powers of each side viewed as $\mathcal{O}\mathbb{B}_{\mathrm{dR},Y}^+$-modules, one has the comparison
\begin{equation}\label{SymmetricComparison}\begin{split}\Sym_{\mathcal{O}_{Y}}^k\mathcal{H}_{\mathrm{dR}}^1(\mathcal{A})\otimes_{\mathcal{O}_Y}\mathcal{O}\mathbb{B}_{\mathrm{dR},Y}^+ &\cong\Sym_{\mathcal{O}\mathbb{B}_{\mathrm{dR},Y}^+}^k(\mathcal{H}_{\mathrm{dR}}^1(\mathcal{A})\otimes_{\mathcal{O}_Y}\mathcal{O}\mathbb{B}_{\mathrm{dR},Y}^+) \\&\overset{\iota_{\mathrm{dR}}}{\subset}\Sym_{\mathcal{O}\mathbb{B}_{\mathrm{dR},Y}^+}^k(T_p\mathcal{A}\otimes_{\hat{\mathbb{Z}}_{p,Y}}\mathcal{O}\mathbb{B}_{\mathrm{dR},Y}^+) 
\end{split}
\end{equation}
compatible with the inherited connections and filtrations. Here, the filtration on the left hand side is given by the convolution of the filtration on $\Sym_{\mathcal{O}_{Y}}^k\mathcal{H}_{\mathrm{dR}}^1(\mathcal{A})$ inherited from the Hodge-de Rham filtration
\begin{equation}\label{filtration}\omega_{\mathcal{A}}^{\otimes k} \subset \omega_{\mathcal{A}}^{\otimes k-1}\otimes_{\mathcal{O}_Y}\mathcal{H}_{\mathrm{dR}}^1(\mathcal{A}) \subset \omega_{\mathcal{A}}^{\otimes k-2}\otimes_{\mathcal{O}_Y}\Sym_{\mathcal{O}_Y}^2\mathcal{H}_{\mathrm{dR}}^1(\mathcal{A}) \subset \ldots \subset \Sym_{\mathcal{O}_Y}^k\mathcal{H}_{\mathrm{dR}}^1(\mathcal{A})
\end{equation}
with the filtration on $\mathcal{O}\mathbb{B}_{\mathrm{dR},Y}^+$. From (\ref{SymmetricComparison}), we have
\begin{equation}\begin{split}\label{SymmetricComparison2}\Sym_{\mathcal{O}_Y}^k\mathcal{H}_{\mathrm{dR}}^1(\mathcal{A})\otimes_{\mathcal{O}_Y}\mathcal{O}\mathbb{B}_{\mathrm{dR},Y}^+ &\cong\Sym_{\mathcal{O}\mathbb{B}_{\mathrm{dR},Y}^+}^k\left(\mathcal{H}_{\mathrm{dR}}^1(\mathcal{A})\otimes_{\mathcal{O}_Y}\mathcal{O}\mathbb{B}_{\mathrm{dR},Y}^+\right) \\
&\overset{\iota_{\mathrm{dR}}}{\subset} \Sym_{\mathcal{O}\mathbb{B}_{\mathrm{dR},Y}^+}^k\left(T_p\mathcal{A}\otimes_{\hat{\mathbb{Z}}_{p,Y}}\mathcal{O}\mathbb{B}_{\mathrm{dR},Y}^+\right)
\end{split}
\end{equation}
compatible with connections and filtrations. Reducing modulo $(t)$, the splitting (\ref{splitdef}) induces a splitting on symmetric powers:
\begin{definition}Henceforth, denote the map on symmetric powers induced by (\ref{splitdef}) also by
\begin{equation}\label{splitdef2}\overline{\mathrm{split}} : \Sym_{\mathcal{O}_{\Delta,\mathcal{Y}}}^k\left(T_p\mathcal{A}\otimes_{\hat{\mathbb{Z}}_{p,Y}}\mathcal{O}_{\Delta,\mathcal{Y}}\right)|_{\mathcal{Y}_x} \twoheadrightarrow \left(\omega_{\Delta,\mathcal{Y}}\right)|_{\mathcal{Y}_x}^{\otimes_{\mathcal{O}_{\Delta,\mathcal{Y}}} k}.
\end{equation}
\end{definition}

\subsection{The $p$-adic Maass-Shimura operator}
In this section, we define our $p$-adic Maass-Shimura operator using the horizontal splitting (\ref{splitdef2}). 

\begin{definition}We define a map 
$$\partial_k : \left(\omega_{\Delta,\mathcal{Y}}\right)|_{\mathcal{Y}_x}^{\otimes_{\mathcal{O}_{\Delta,\mathcal{Y}}} k} \rightarrow \left(\omega_{\Delta,\mathcal{Y}}\right)|_{\mathcal{Y}_x}^{\otimes_{\mathcal{O}_{\Delta,\mathcal{Y}}} k+2}$$
as the following composition:
\begin{equation}\label{padicMaassShimura}\begin{split}\left(\omega_{\Delta,\mathcal{Y}}\right)^{\otimes_{\mathcal{O}_{\Delta,\mathcal{Y}}} k} &\overset{\iota_{\mathrm{dR}}}{\subset} \Sym_{\mathcal{O}_{\Delta,\mathcal{Y}}}^k\left(T_p\mathcal{A}\otimes_{\hat{\mathbb{Z}}_{p,Y}}\mathcal{O}_{\Delta,\mathcal{Y}}\right)|_{\mathcal{Y}_x}\\
&\xrightarrow{\nabla} \left(\Sym_{\mathcal{O}_{\Delta,\mathcal{Y}}}^k\left(T_p\mathcal{A}\otimes_{\hat{\mathbb{Z}}_{p,Y}}\mathcal{O}_{\Delta,\mathcal{Y}}\right)\otimes_{\mathcal{O}_{\mathcal{Y}}}\Omega_{\mathcal{Y}}^1\right)|_{\mathcal{Y}_x}\\
&\xrightarrow{\sigma^{-1}} \left(\Sym_{\mathcal{O}_{\Delta,\mathcal{Y}}}^k\left(T_p\mathcal{A}\otimes_{\hat{\mathbb{Z}}_{p,Y}}\mathcal{O}_{\Delta,\mathcal{Y}}\right)\otimes_{\mathcal{O}_{\mathcal{Y}}}\omega_{\mathcal{A}}|_{\mathcal{Y}}^{\otimes 2}\right)|_{\mathcal{Y}_x}\\
&\xrightarrow{\overline{\mathrm{split}}} \left((\omega_{\Delta,\mathcal{Y}})^{\otimes_{\mathcal{O}_{\Delta,\mathcal{Y}}} k}\otimes_{\mathcal{O}_{\mathcal{Y}}}\omega_{\mathcal{A}}|_{\mathcal{Y}}^{\otimes 2}\right)|_{\mathcal{Y}_x} \cong (\omega_{\Delta,\mathcal{Y}})|_{\mathcal{Y}_x}^{\otimes_{\mathcal{O}_{\Delta,\mathcal{Y}}} k+2}.
\end{split}
\end{equation}
\end{definition}

\subsection{The $p$-adic Maass-Shimura operator in coordinates and generalized $p$-adic modular forms}
Given any pro\'{e}tale $\mathcal{U} \rightarrow \mathcal{Y}_x$ and 
$$\omega \in (\omega_{\Delta,\mathcal{Y}})^{\otimes_{\mathcal{O}_{\Delta,\mathcal{Y}}} k}(\mathcal{U}),$$
we write 
$$\omega = F\cdot (\omega_{\mathrm{can}})^{\otimes_{\mathcal{O}_{\Delta,\mathcal{Y}}} k} \xmapsto{\iota_{\mathrm{dR}}} F\cdot (-z_{\mathrm{dR}}\alpha_{\infty,1} +\alpha_{\infty,2})^{\otimes_{\mathcal{O}\mathbb{B}_{\mathrm{dR},\mathcal{Y}}} k}$$
where $\omega_{\mathrm{can}}$ is defined in Definition \ref{canonicaldifferentialdef}. For brevity, write $\omega_{\mathrm{can}}^{\otimes k} = \omega_{\mathrm{can}}^{\otimes_{\mathcal{O}_{\Delta,\mathcal{Y}}} k}$ and $\iota_{\mathrm{dR}}(\omega_{\mathrm{can}})^{\otimes k} = \iota_{\mathrm{dR}}(\omega_{\mathrm{can}})^{\otimes_{\mathcal{O}_{\Delta,\mathcal{Y}}} k}$. Then $\partial_k(\omega)$ is computed as the composition
\begin{align*}\omega& \mapsto\nabla(\omega) = \nabla(\iota_{\mathrm{dR}}(\omega)) = \nabla\left(\iota_{\mathrm{dR}}(F\cdot\omega_{\mathrm{can}}^{\otimes k})\right) \\
&= \frac{d}{d{z_{\mathrm{dR}}}}F\cdot \iota_{\mathrm{dR}}(\omega_{\mathrm{can}})^{\otimes k}\otimes dz_{\mathrm{dR}} + \sum_{i = 1}^k\iota_{\mathrm{dR}}(\omega_{\mathrm{can}})^{\otimes i-1}\otimes\nabla\left(\iota_{\mathrm{dR}}(\omega_{\mathrm{can}})\right)\otimes \iota_{\mathrm{dR}}(\omega_{\mathrm{can}})^{\otimes k-i+1} \\
&= \frac{d}{dz_{\mathrm{dR}}}F\cdot \iota_{\mathrm{dR}}(\omega_{\mathrm{can}})^{\otimes k}\cdot dz_{\mathrm{dR}} + \sum_{i = 1}^k\iota_{\mathrm{dR}}(\omega_{\mathrm{can}})^{\otimes i-1}\cdot (-\alpha_{\infty,1} \cdot dz_{\mathrm{dR}})\otimes \iota_{\mathrm{dR}}(\omega_{\mathrm{can}})^{\otimes k-i+1} \\
&= \frac{d}{d{z_{\mathrm{dR}}}}F\cdot \iota_{\mathrm{dR}}(\omega_{\mathrm{can}})^{\otimes k}\cdot d{z_{\mathrm{dR}}} - k\cdot\iota_{\mathrm{dR}}(\omega_{\mathrm{can}})^{\otimes k-1}\cdot \alpha_{\infty,1} \cdot d{z_{\mathrm{dR}}}\\
&\xmapsto{\sigma^{-1},(\ref{canonicaldifferentialKScalculation}) \pmod{t}}\frac{d}{d{z_{\mathrm{dR}}}}F\cdot \iota_{\mathrm{dR}}(\omega_{\mathrm{can}})^{\otimes k+2} - k\cdot\iota_{\mathrm{dR}}(\omega_{\mathrm{can}})^{\otimes k+1}\cdot \alpha_{\infty,1} \\
&\xmapsto{\overline{\mathrm{split}}}\left(\frac{d}{d{z_{\mathrm{dR}}}} + \frac{k}{{z_{\mathrm{dR}}}-\bar{z}}\right)F\cdot \omega_{\mathrm{can}}^{\otimes k+2}
\end{align*}
where the third line uses the calculation
$$\nabla(\iota_{\mathrm{dR}}(\omega_{\mathrm{can}})) = \nabla(-{z_{\mathrm{dR}}}\alpha_{\infty,1} + \alpha_{\infty,2}) = -\alpha_{\infty,1}\cdot d{z_{\mathrm{dR}}}.$$
and the last arrow uses (\ref{splitcalculation}).

Now given $F \in \mathcal{O}_{\Delta,\mathcal{Y}}(\mathcal{U})$, we define
$$\delta_k : \mathcal{O}_{\Delta,\mathcal{Y}}(\mathcal{U}) \rightarrow \mathcal{O}_{\Delta,\mathcal{Y}}(\mathcal{U})$$
by
$$\partial_k\left(F\cdot \omega_{\mathrm{can}}^{\otimes k}\right) = \left(\delta_kF\right)\cdot\omega_{\mathrm{can}}^{\otimes k+2},$$
the above calculation (\ref{MSformula}) shows that
\begin{equation}\label{MSformula}\delta_k = \frac{d}{d{z_{\mathrm{dR}}}} + \frac{k}{{z_{\mathrm{dR}}}-\bar{z}}.
\end{equation}

\begin{definition}\label{weightdefinition}Given an open subset $\mathcal{U} \subset \mathcal{Y}$, $k \in \mathbb{Z}$, and a subgroup $\Gamma \subset GL_2(\mathbb{Z}_p)$ with $\mathcal{U}\cdot\Gamma = \mathcal{U}$, we define a \emph{$p$-adic modular form $F$ on $\mathcal{U}$ for $\Gamma$ of weight $k$} to be a $F\in \mathcal{O}_{\Delta,\mathcal{Y}}(\mathcal{U})$ such that
\begin{equation}\label{Ftransformationprop}
\begin{split}
\left(\begin{array}{ccc}a & b\\
c & d\\
\end{array}\right)^*F = (bc-ad)^{-k}(c{z_{\mathrm{dR}}}+a)^kF
\end{split}
\end{equation}
for any $\left(\begin{array}{ccc}a & b\\
c & d\\
\end{array}\right) \in \Gamma$. 
We also make an analogous definition for $k \in \mathbb{Z}/(p-1) \times \mathbb{Z}_p$, provided that $(bc-ad)^{-k}(c{z_{\mathrm{dR}}}+a)^k$ makes sense on $\mathcal{U}$. (Note that we embed $\mathbb{Z} \subset \mathbb{Z}/(p-1)\times \mathbb{Z}_p$ diagonally to make all notions of weight compatible.) Let $M_{k,\Delta}(\Gamma)(\mathcal{U})$ denote the space of $p$-adic modular forms on $\mathcal{U}$ for $\Gamma$ of weight $k$.
\end{definition}

Recall the natural projection $\lambda : \mathcal{Y} \rightarrow Y$.

\begin{proposition}Let $U \subset \lambda(\mathcal{Y}_x)$ be an open subset, let $\mathcal{U} = \lambda^{-1}(U)$, and let $k \in \mathbb{Z}_{\ge 0}$. Then we have
$$\omega_{\mathcal{A}}^{\otimes k}(U) \xrightarrow{\lambda^*} \omega_{\mathcal{A}}^{\otimes k}(\mathcal{U}) \rightarrow \omega_{\Delta,\mathcal{Y}}^{\otimes k}(\mathcal{U}) \xrightarrow{\sim} M_{k,\Delta}(GL_2(\mathbb{Z}_p))(\mathcal{U})\cdot \omega_{\mathrm{can}}^{\otimes k} \underset{\sim}{\xrightarrow{\cdot (\omega_{\mathrm{can}}^{\otimes k})^{-1}}}M_{k,\Delta}(GL_2(\mathbb{Z}_p))(\mathcal{U}).$$
\end{proposition}

\begin{proof}
The first arrow is obvious. Since $\frak{s}$ trivializes $\omega_{\mathcal{A}}(\mathcal{U})$, we can write any $\omega \in \omega_{\mathcal{A}}^{\otimes k}(\mathcal{U})$ as 
$$\omega = f\cdot \frak{s}^{\otimes k} = f\left(y_{\mathrm{dR}}/t\right)^k\cdot \omega_{\mathrm{can}}^{\otimes k}.$$ Then 
\begin{equation}
\begin{split}
\left(\begin{array}{ccc}a & b\\
c & d\\
\end{array}\right)^*f = (b/\mathbf{z}+d)^kF
\end{split}
\end{equation}
and combining this with (\ref{y'transformationprop}), we see that $f\left(y_{\mathrm{dR}}/t\right)^k$ is a $p$-adic modular form on $\mathcal{U}$ for $GL_2(\mathbb{Z}_p)$ of weight $k$. The rest of the arrows are obvious after invoking (\ref{s'transformationprop}) and using the fact that $\omega_{\mathrm{can}}$ generates $\omega_{\Delta,\mathcal{Y}}(\mathcal{U})$ (Proposition \ref{generator2}).

\end{proof}

\begin{definition}We let 
$$\partial_k^j = \partial_{k+2j-2}\circ \partial_{k+2j-4}\circ\cdots\circ \partial_{k+2}\circ\partial_k$$
and
$$\delta_k^j = \delta_{k+2j-2}\circ \delta_{k+2j-4}\circ\cdots\circ \delta_{k+2}\circ \delta_k.$$
It is clear by its definition that $\partial_k^j$ and $\delta_k^j$ are $\hat{\mathcal{O}}_{\mathcal{Y}} = \mathbb{B}_{\mathrm{dR},\mathcal{Y}}^+/(t)$-linear. 
\end{definition}

\begin{definition}\label{weightdefinition}Given an open subset $\mathcal{U} \subset \mathcal{Y}$, $k \in \mathbb{Z}$, and a subgroup $\Gamma \subset GL_2(\mathbb{Z}_p)$ with $\mathcal{U}\cdot\Gamma = \mathcal{U}$, we define a \emph{$p$-adic modular form $F$ on $\mathcal{U}$ for $\Gamma$ of weight $k$} to be a $F\in \hat{\mathcal{O}}_{\mathcal{Y}}(\mathcal{U})$ such that
\begin{equation}\label{Ftransformationprop}
\begin{split}
\left(\begin{array}{ccc}a & b\\
c & d\\
\end{array}\right)^*F = (bc-ad)^{-k}(c{z_{\mathrm{dR}}}+a)^kF
\end{split}
\end{equation}
for any $\left(\begin{array}{ccc}a & b\\
c & d\\
\end{array}\right) \in \Gamma$. 
Let $M_{k,\Delta}(\Gamma)(\mathcal{U})$ denote the space of $p$-adic modular forms on $\mathcal{U}$ for $\Gamma$ of weight $k$.
\end{definition}

\begin{proposition}\label{weightproposition}We have
$$\delta_k^j : M_{k,\Delta}(\Gamma)(\mathcal{U}) \rightarrow M_{k+2j,\Delta}(\Gamma)(\mathcal{U})$$
and 
\begin{equation}\label{MSformula}\delta_k^j = \sum_{i = 0}^j\binom{k-1+j}{i}\binom{j}{i}\frac{i!}{({z_{\mathrm{dR}}}-\bar{z})^i}\left(\frac{d}{d{z_{\mathrm{dR}}}}\right)^{j-i}.
\end{equation}
\end{proposition}
\begin{proof}This is a direct calculation, using (\ref{ztransformationprop}), (\ref{x'transformationprop}) and (\ref{z'transformationprop}). More precisely, one uses induction: if $F \in \mathcal{O}_{\Delta,\mathcal{Y}}(\mathcal{U})$ has weight $k'$, then by direct computation one verifies that
$$\left(\frac{d}{d{z_{\mathrm{dR}}}} + \frac{k'}{{z_{\mathrm{dR}}} - \bar{z}}\right)F$$
has weight $k' + 2$ in the sense of (\ref{Ftransformationprop}). Another ``coordinate free" proof is to observe that each step in the construction of the assignment
$$f \mapsto \partial_{k'}f$$
is $\Gamma$-equivariant (and $GL_2(\mathbb{Z}_p)$-equivariant if $\mathcal{U} = \mathcal{Y}_x$).
\end{proof}

\begin{proposition}\label{horizontalcommute}We have
$$\partial_k^j = \overline{\mathrm{split}}\circ \nabla^j \circ \iota_{\mathrm{dR}}$$
\end{proposition}
\begin{proof}By definition, we have
$$\partial_k^j = (\overline{\mathrm{split}} \circ \nabla \circ\iota_{\mathrm{dR}})^j.$$
Since the $p$-adic de Rham comparison $\iota_{\mathrm{dR}}$ (see (\ref{OBdRlocalsystem2})) is compatible with connections, we have
\begin{equation}\label{comm1}\iota_{\mathrm{dR}}\circ\nabla = \nabla\circ \iota_{\mathrm{dR}}.
\end{equation}
By Proposition \ref{holomorphicallyhorizontal}, we have
\begin{equation}\label{comm2}\nabla \circ \overline{\mathrm{split}} = \overline{\mathrm{split}}\circ \nabla.
\end{equation}
Hence, we have
$$\partial_k^j = (\overline{\mathrm{split}} \circ \nabla \circ\iota_{\mathrm{dR}})^j \overset{(\ref{comm1})}{=} (\overline{\mathrm{split}}\circ\nabla)^j\circ \iota_{\mathrm{dR}} \overset{(\ref{comm2})}{=}  \overline{\mathrm{split}}\circ \nabla^j\circ\iota_{\mathrm{dR}}
$$
which is what we wanted to show.
\end{proof}

\subsection{Comparison between the complex and $p$-adic Maass-Shimura operators at CM points}Let $K/\mathbb{Q}$ be an imaginary quadratic field, with ring of integers $\mathcal{O}_K$.

Let $i_p : \overline{\mathbb{Q}} \hookrightarrow \mathbb{C}_p$ denote our previously fixed embedding (\ref{fixedembedding}), and now fix an embedding $i_{\infty} : \overline{\mathbb{Q}} \hookrightarrow \mathbb{C}$. Let $A/\overline{\mathbb{Q}}$ be an elliptic curve with CM by an order $\mathcal{O}\subset\mathcal{O}_K$. In the complex counterpart of the Hodge-de Rham filtration, we have 
$$H^{1,0}(A/\mathbb{C}) = \Omega_{A/\mathbb{C}}^1, \hspace{1cm} H^{0,1}(A/\mathbb{C}) = \mathrm{Lie}(A/\mathbb{C}),$$ 
and in fact these respectively correspond to the $\gamma$ and $\overline{\gamma}$-eigenspaces for any $\gamma \in \mathcal{O} = \End(A/\overline{\mathbb{Q}})$, acting as scalars on these eigenspaces under the above embedding $i_{\infty}$. 

Let $\tau$ denote the standard coordinate on the complex upper half-plane $\mathcal{H}^+$. The real-analytic Hodge splitting of the complex Hodge-de Rham filtration gives rise to the complex Maass-Shimura operator
$$\frak{d}_k = \frac{1}{2\pi i}\left(\frac{d}{d\tau} + \frac{k}{\tau-\bar{\tau}}\right).$$
The real-analytic Hodge splitting is induced by complex conjugation acting on the complex structure of de Rham cohomology, in particular interchanging the Hodge pieces (i.e. $\overline{H^{p,q}(A/\mathbb{C})} = H^{q,p}(A/\mathbb{C})$) and giving rise to ``opposing filtrations'' in Deligne's sense. Thus, as for a CM curve $A/\overline{\mathbb{Q}}$ we have
$$H^{1,0}(A/\mathbb{C}) := \Omega_{A/\mathbb{C}}^1 = \Omega_{A/\overline{\mathbb{Q}}}^1\otimes_{\overline{\mathbb{Q}}}\mathbb{C} = H^{1,0}(A/\overline{\mathbb{Q}})\otimes_{\overline{\mathbb{Q}}}\mathbb{C},$$
the splitting of the Hodge-de Rham filtration induced by the eigendecomposition under the $\mathcal{O}_K$-action coincides with the Hodge splitting. In fact, we have the following coincidence of the values of our $p$-adic and complex Maass-Shimura operators at CM points (Theorem \ref{MScomparison}). Let $$\frak{d}_k^j = \frak{d}_{k+2j-2} \circ \frak{d}_{k+2j-4} \circ \cdots \circ \frak{d}_{k+2}\circ \frak{d}_k.$$

\begin{definition}\label{perioddefinition}Let $A/\overline{\mathbb{Q}}$ be an elliptic curve with complex multiplication by an order $\mathcal{O}\subset\mathcal{O}_K$, and let $t$ be any $\Gamma_1(N)$-level structure, $\alpha : \mathbb{Z}_p^{\oplus 2} \xrightarrow{\sim} T_pA$ be any $\Gamma(p^{\infty})$-level structure. Fix any $\omega_0 \in \Omega_{A/\overline{\mathbb{Z}}}^1$, and fix an isomorphism $\mathbb{C}/\mathcal{O} \cong A$ which defines a differential $2\pi i dz \in \Omega_{A/\overline{\mathbb{Q}}}^1$, where $z$ is the standard coordinate on $\mathbb{C}$. Define $\Omega_{\infty}(A,t) \in \mathbb{C}^{\times}$ and $\Omega_p(A,t,\alpha) \in (B_{\mathrm{dR}}^{+})^{\times}$ by
\begin{equation}\label{CMperiodequation1}2 \pi i dz(A,t) = \Omega_{\infty}(A,t)\cdot\omega_0,\hspace{1cm} \omega_{\mathrm{can}}(A,t,\alpha) = \Omega_p(A,t,\alpha)\cdot \omega_0.
\end{equation}
\end{definition}

\begin{definition}\label{Shimuraactiondefinition}We recall the Shimura reciprocity law on $\mathcal{Y}$. Fix an order $\mathcal{O} \subset \mathcal{O}_K$. Let $\mathbb{I}^{\frak{f}}$ be the semigroup of ideals of $\mathcal{O}_K$ prime to a fixed ideal $\frak{f} \subset \mathcal{O}_K$. Let $\frak{a} \in \mathbb{I}^{(pN)}$ with $\frak{a} \subset \mathcal{O}$. With $A$ as in Definition \ref{perioddefinition}, define 
$$\frak{a} \star A := A/A[\frak{a}]$$
where $A[\frak{a}]$ denotes the $\frak{a}$-torsion subgroup of $A$ using the identification $\End_{\overline{\mathbb{Q}}}(A) = \mathcal{O}$. Then the natural projection $\pi_{\frak{a}} : A \twoheadrightarrow \frak{a}\star A$ takes $\Gamma_1(N)$-level structures to $\Gamma_1(N)$-level structure, and $\Gamma(p^{\infty})$-level structures to $\Gamma(p^{\infty})$-level structures, so that we have a natural action
$$\frak{a}\star(A,t) = (\frak{a}\star A,\pi_{\frak{a}}(t)), \hspace{1cm} \frak{a}\star(A,t,\alpha) = (A/A[\frak{a}],\pi_{\frak{a}}(t),\pi_{\frak{a}}(\alpha)).$$
\end{definition}

\begin{proposition}\label{sameperiod}
Now let $\omega_{0,\frak{a}} \in \Omega_{\frak{a}\star A/K_p}^1$ be the unique differential such that $\pi_{\frak{a}}^*\omega_{0,\frak{a}} = \omega_0$ (viewing $\Omega_{A/\overline{\mathbb{Z}}}^1 \subset \Omega_{A/\mathbb{C}_p}^1$ using $i_p$), and define $\Omega_{\infty}(\frak{a}\star(A,t)) \in \mathbb{C}^{\times}$ and $\Omega_p(\frak{a}\star(A,t,\alpha)) \in (B_{\mathrm{dR}}^+)^{\times}$ by 
\begin{equation}\label{CMperiodequation2} 2 \pi i dz(\frak{a}\star(A,t)) = \Omega_{\infty}(\frak{a}\star(A,t))\cdot\omega_{0,\frak{a}},\hspace{1cm} \omega_{\mathrm{can}}(\frak{a}\star(A,t,\alpha)) = \Omega_p(\frak{a}\star(A,t,\alpha))\cdot \omega_{0,\frak{a}}.
\end{equation}
Then
$$\Omega_{\infty}(A,t) = \Omega_{\infty}(\frak{a}\star(A,t)), \hspace{1cm} \Omega_p(A,t,\alpha) = \Omega_p(\frak{a}\star(A,t,\alpha)).$$
\end{proposition}

\begin{proof}This follows immediately from (\ref{CMperiodequation1}) and (\ref{CMperiodequation2}), after noting that by definitions (in particular, the functoriality of the complex uniformization $\mathbb{C}/(\mathbb{Z} + \mathbb{Z}\tau) \cong \mathcal{A}$ which defines $2\pi i dz$ and the functoriality of the construction of $\omega_{\mathrm{can}}$ as defined in Definition \ref{canonicaldifferentialdef}), we have
$$\pi_{\frak{a}}^* 2\pi i dz(\frak{a}\star(A,t)) = 2\pi i dz(A,t), \hspace{1cm} \pi_{\frak{a}}^*\omega_{\mathrm{can}}(\frak{a}\star(A,t,\alpha)) = \omega_{\mathrm{can}}(A,t,\alpha).$$
\end{proof}

\begin{definition}\label{ddefinition}Henceforth, let
$$d_k^j = \theta \circ \partial_k^j$$
and
\begin{equation}\label{MSformulatheta}\begin{split}\theta_k^j &= \theta\circ \delta_k^j = \theta\left(\sum_{i = 0}^j\binom{k-1+j}{i}\binom{j}{i}\frac{i!}{(z_{\mathrm{dR}}-\bar{z})^i}\left(\frac{d}{dz_{\mathrm{dR}}}\right)^{j-i}\right) \\
&= \sum_{i = 0}^j\binom{k-1+j}{i}\binom{j}{i}\frac{i!}{\theta(z_{\mathrm{dR}}-\bar{z})^i}\theta\circ\left(\frac{d}{dz_{\mathrm{dR}}}\right)^{j-i}\\
&= \sum_{i = 0}^j\binom{k-1+j}{i}\binom{j}{i}i!\left(-\frac{\theta(y_{\mathrm{dR}})}{\mathbf{z}}\right)^i\theta\circ\left(\frac{d}{dz_{\mathrm{dR}}}\right)^{j-i}\
\end{split}
\end{equation}
where the first line follows from (\ref{MSformula}) and the last equality follows from (\ref{periodsrelation}). 
\end{definition}

\begin{definition}\label{weightdefinitiontheta}Given an open subset $\mathcal{U} \subset \mathcal{Y}$, $k \in \mathbb{Z}$, and a subgroup $\Gamma \subset GL_2(\mathbb{Z}_p)$ with $\mathcal{U}\cdot\Gamma = \mathcal{U}$, we define a \emph{$p$-adic modular form $F$ on $\mathcal{U}$ for $\Gamma$ of weight $k$} to be a $F\in \hat{\mathcal{O}}_{\mathcal{Y}}(\mathcal{U})$ such that
\begin{equation}\label{Ftransformationprop}
\begin{split}
\left(\begin{array}{ccc}a & b\\
c & d\\
\end{array}\right)^*F = (bc-ad)^{-k}(c{z_{\mathrm{dR}}}+a)^kF
\end{split}
\end{equation}
for any $\left(\begin{array}{ccc}a & b\\
c & d\\
\end{array}\right) \in \Gamma$. 
We also make an analogous definition for $k \in \mathbb{Z}/(p-1) \times \mathbb{Z}_p$, provided that $(bc-ad)^{-k}(c{z_{\mathrm{dR}}}+a)^k$ makes sense on $\mathcal{U}$. (Note that we embed $\mathbb{Z} \subset \mathbb{Z}/(p-1)\times \mathbb{Z}_p$ diagonally to make all notions of weight compatible.) Let $M_{k,\Delta}(\Gamma)(\mathcal{U})$ denote the space of $p$-adic modular forms on $\mathcal{U}$ for $\Gamma$ of weight $k$.
\end{definition}

\begin{proposition}\label{weightpropositiontheta}We have
$$\theta_k^j : \hat{M}_k(\Gamma)(\mathcal{U}) \rightarrow \hat{M}_{k+2j}(\Gamma)(\mathcal{U}).$$
\end{proposition}

\begin{proof}This follows from applying $\theta$ to Proposition \ref{weightproposition}.
\end{proof}

\begin{theorem}\label{MScomparison}Let $A/\overline{\mathbb{Q}}$ be an elliptic curve with CM by an order $\mathcal{O}\subset \mathcal{O}_K$, and let $\alpha : \mathbb{Z}_p^{\oplus 2} \xrightarrow{\sim} T_pA$ be any $p^{\infty}$-level structure, and $t \in A[N]$ any $\Gamma_1(N)$-level structure, and fix $\omega_0 \in \Omega_{A/\overline{\mathbb{Z}}}^1$ as in Definition \ref{perioddefinition}. Let $\omega \in \omega_{\mathcal{A}}(U)$ for any $U \rightarrow Y$, and write 
$$\omega = F\cdot (2\pi i dz)^{\otimes k}$$ 
and 
$$\omega|_{\lambda^{-1}(U) \cap \mathcal{Y}_x} = f\cdot \omega_{\mathrm{can}}^{\otimes k}$$
on $\lambda^{-1}(U) \cap \mathcal{Y}_x $. 
Suppose $(A,t) \in U$ and $(A,t,\alpha) \in \lambda^{-1} \cap \mathcal{Y}_x $. 
Then for any $j \in \mathbb{Z}$, we have the following coincidence of values (in $\overline{\mathbb{Q}}$): 
\begin{align*}i_p^{-1}(\theta(\Omega_p(A,t,\alpha))^{-(k+2j)} \cdot \theta_k^jf(A,t,\alpha)) 
&= i_{\infty}^{-1}(\Omega_{\infty}(A,t)^{-(k+2j)}\cdot\frak{d}_k^jF(A,t,\omega_0)).
\end{align*}
\end{theorem}
\begin{proof}Let $\overline{\mathrm{split}}_{\mathbb{C}} : \mathcal{H}_{\mathrm{dR}}^1(\mathcal{A}) \otimes_{\mathcal{O}^{\mathrm{hol}}}\mathcal{O}^{\mathrm{r.\hspace{.01cm}an.}}$ denote the antiholomorphic splitting of complex relative de Rham cohomology. Using $2\pi i dz$, view $\frak{d}_k^j$ as an operator $\omega_{\mathcal{A}}^{\otimes k} \rightarrow \omega_{\mathcal{A}}^{\otimes k+2j}$, and using $\omega_{\mathrm{can}}$, view $\delta_k^j$ as an opeartor $\omega_{\mathcal{A}}^{\otimes k} \rightarrow \omega_{\mathcal{A}}^{\otimes k+2j}$. Since the Gauss-Manin connection is holomorphically horizontal, i.e. for any holomorphic differential $\omega$, $\nabla_{\omega}(\omega') = 0$ for any antiholomorphic $\omega'$, then we have 
$$\frak{d}_k^j = \overline{\mathrm{split}}_{\mathbb{C}} \circ \nabla^j.$$
This is analogous with the horizontalness of our $p$-adic Maass-Shimura operator
$$\partial_k^j = \overline{\mathrm{split}}\circ \nabla^j \circ \iota_{\mathrm{dR}}$$
as proven in Proposition \ref{horizontalcommute}.

The statement of the Theorem can hence be reformulated as
\begin{equation}\label{statement1}i_p^{-1}\left(\theta\left(\left(\overline{\mathrm{split}}\circ \nabla^jf \right)(A,t,\alpha)\right)\right)  =  i_{\infty}^{-1}\left(\left(\overline{\mathrm{split}}_{\mathbb{C}}\circ \nabla^jf\right)(A,t)\right)
\end{equation}
for any section $f \in \Sym^k\mathcal{H}_{\mathrm{dR}}^1(\mathcal{A})(U)$ where $U$ is as in the statement of the Theorem, where $(A,t,\alpha)$ and $(A,t)$ denote specialization to those points. First note that by functoriality of $\theta$, we have
$$\theta\left(\left(\overline{\mathrm{split}}\circ \nabla^jf \right)(A,t,\alpha)\right) = \left(\theta\circ \overline{\mathrm{split}}\circ \nabla^jf \right)(A,t,\alpha)$$
and so (\ref{statement1}) is equivalent to
\begin{equation}\label{statement2}i_p^{-1}\left(\left(\theta\circ \overline{\mathrm{split}}\circ \nabla^jf \right)(A,t,\alpha)\right) =  i_{\infty}^{-1}\left(\left(\overline{\mathrm{split}}_{\mathbb{C}}\circ \nabla^jf\right)(A,t)\right)
\end{equation}

By the functoriality of $\theta\circ \overline{\mathrm{split}}$ and $\overline{\mathrm{split}}_{\mathbb{C}}$, for any section $f'$ in $$\Sym_{\mathcal{O}_{\Delta,\mathcal{Y}}}^{\otimes k+2j} \left(\mathcal{H}_{\mathrm{dR}}^1(\mathcal{A})\otimes_{\mathcal{O}_Y}\mathcal{O}_{\Delta,\mathcal{Y}}\right)(\lambda^{-1}(U)\cap\mathcal{Y}_x) \cap\Sym_{\mathcal{O}_{\Delta,\mathcal{Y}}}^{\otimes k+2j}\left(T_p\mathcal{A}\otimes_{\hat{\mathbb{Z}}_{p,Y}}\mathcal{O}_{\Delta,\mathcal{Y}}\right)(\lambda^{-1}(U)\cap\mathcal{Y}_x),$$
where the intersection is taken in 
$$\Sym_{\mathcal{O}_{\Delta,\mathcal{Y}}}^{\otimes k+2j}\left(T_p\mathcal{A}\otimes_{\hat{\mathbb{Z}}_{p,Y}}\mathcal{O}_{\Delta,\mathcal{Y}}\cdot t^{-1}\right)(\lambda^{-1}(U)\cap\mathcal{Y}_x)$$
using $\iota_{\mathrm{dR}}$, letting $f'(A,t,\alpha)$ denote its specialization to $(A,t,\alpha)$, that
$$\left(\left(\theta\circ\overline{\mathrm{split}}\right)(f')\right)(A,t,\alpha) = \left(\theta\circ\overline{\mathrm{split}}\right)(f'(A,t,\alpha))$$
and similarly for any section $f$ of $\left(\mathcal{H}_{\mathrm{dR}}^1(\mathcal{A}) \otimes_{\mathcal{O}^{\mathrm{hol}}}\mathcal{O}^{\mathrm{r.\hspace{.01cm}an.}}\right)(U)$, letting $f(A,t)$ denote its specialization to $(A,t)$, that
$$\left(\overline{\mathrm{split}}_{\mathbb{C}}(f')\right)(A,t) = \overline{\mathrm{split}}(f'(A,t)).$$ 
Thus taking $f' = \nabla^jf$, can rewrite (\ref{statement2}) as
\begin{equation}\label{statement3}i_p^{-1}\left(\left(\theta\circ \overline{\mathrm{split}}\right)_{(A,t,\alpha)}\left(\left(\nabla^jf \right)(A,t,\alpha)\right)\right) =  i_{\infty}^{-1}\left(\overline{\mathrm{split}}_{\mathbb{C},(A,t)}\left(\left(\nabla^jf\right)(A,t)\right)\right)
\end{equation}
where $\left(\theta\circ\overline{\mathrm{split}}\right)_{(A,t,\alpha)}$ and $\overline{\mathrm{split}}_{\mathbb{C},(A,t)}$ denote the specializations of the splittings $\overline{\mathrm{split}}$ and $\overline{\mathrm{split}}_{\mathbb{C}}$ to the points $(A,t,\alpha)$ and $(A,t)$, respectively. Now by (\ref{modtheta}), we have
$$\theta\circ \overline{\mathrm{split}} = \theta\circ \mathrm{split}$$
where $\mathrm{split}$ is defined as in Definition \ref{splitdef2}. Hence (\ref{statement3}) is in turn equivalent to
\begin{equation}\label{statement3}i_p^{-1}\left(\left(\theta\circ \mathrm{split}\right)_{(A,t,\alpha)}\left(\left(\nabla^jf \right)(A,t,\alpha)\right)\right) =  i_{\infty}^{-1}\left(\overline{\mathrm{split}}_{\mathbb{C},(A,t)}\left(\left(\nabla^jf\right)(A,t)\right)\right)
\end{equation}
where again where $\left(\theta\circ\mathrm{split}\right)_{(A,t,\alpha)}$ and $\overline{\mathrm{split}}_{\mathbb{C},(A,t)}$ denote the specializations of the splittings $\overline{\mathrm{split}}$ and $\overline{\mathrm{split}}_{\mathbb{C}}$ to the points $(A,t,\alpha)$ and $(A,t)$, respectively.

Since $\nabla^j f$ is a section defined over $U \rightarrow Y$, then $\left(\nabla^jf\right)(A,t) = \left(\nabla^jf\right)(A,t,\alpha)$; moreover, we have $\nabla^jf(A,t) \in H_{\mathrm{dR}}^1(A/\overline{\mathbb{Q}})$ since $(A,t) \in U(\overline{\mathbb{Q}})$ by the theory of complex multiplication. So now it suffices to show that
\begin{equation}\label{statement4}i_p^{-1}\left(\left(\theta\circ \mathrm{split}\right)_{(A,t,\alpha)}\right) =  i_{\infty}^{-1}\left(\overline{\mathrm{split}}_{\mathbb{C},(A,t)}\right).
\end{equation}

For this, we first note that the $\mathcal{O} = \End(A/\overline{\mathbb{Q}})$-action induces a splitting
\begin{equation}\label{complexdecomposition}H_{\mathrm{dR}}^1(A/\overline{\mathbb{Q}}) \cong \Omega_{A/\overline{\mathbb{Q}}}^1 \oplus H^{0,1}(A/\overline{\mathbb{Q}})
\end{equation}
functorial in $\overline{\mathbb{Q}}$-algebras, where the first factor is the subspace on which $\mathcal{O}$ acts through multiplication, and the second factor is the subspace on which $\mathcal{O}$ acts through the complex conjugation of multiplication; by the compatibility of complex structures with the CM action, (\ref{complexdecomposition}) induces, upon tensoring with $\otimes_{\overline{\mathbb{Q}}}\mathbb{C}$, the complex Hodge decomposition
\begin{equation}\mathcal{H}_{\mathrm{dR}}^1(\mathcal{A})(A,t) = H_{\mathrm{dR}}^1(A/\mathbb{C}) \cong (H^{1,0}(A)\otimes_{\mathbb{Z}}\mathbb{C})\oplus \left(H^{0,1}(A)\otimes_{\mathbb{Z}}\mathbb{C}\right).
\end{equation}
Hence the projection onto the first factor in (\ref{complexdecomposition}) is just $\overline{\mathrm{split}}_{\mathbb{C},(A,t)}$, and projection onto the first factor in (\ref{complexdecomposition}) is just $\overline{\mathrm{split}}_{\mathbb{C},(A,t)}$ restricted to $H_{\mathrm{dR}}^1(A/\overline{\mathbb{Q}}) \subset \mathcal{H}_{\mathrm{dR}}^1(\mathcal{A})(A,t)$.

The CM action also induces a splitting
\begin{equation}\label{padicdecomposition1}\left(T_p\mathcal{A}\otimes_{\hat{\mathbb{Z}}_{p,Y}}\mathcal{O}_{\Delta,\mathcal{Y}}\right)(A,t,\alpha) \cong \omega_{\Delta,\mathcal{Y}}(A,t,\alpha) \oplus \omega_{\Delta,\mathcal{Y}}^{-1}(1)(A,t,\alpha)
\end{equation}
where the first factor is the subspace on which $\mathcal{O}$ acts through multiplication, and the second factor is the subspace on which $\mathcal{O}$ acts through the complex conjugation of multiplication. 

Suppose that $p$ is inert or ramified in $K$, so that complex conjugation acts on $\mathcal{O} \subset \mathcal{O}_{K_p}$. Then looking at the action of $\mathcal{O}$ on the tangent space of $A$, the fact that the Hodge-Tate filtration
$$\omega_{\Delta,\mathcal{Y}}^{-1}(1)(A,t,\alpha) \subset \left(T_p\mathcal{A}\otimes_{\hat{\mathbb{Z}}_{p,Y}}\mathcal{O}_{\Delta,\mathcal{Y}}\right)(A,t,\alpha)$$
is the subspace on which $\End(A/\overline{\mathbb{Q}})$ acts through the complex conjugation of multiplication implies that the projection onto the first factor in (\ref{padicdecomposition1}) is given by $\mathrm{split}_{(A,t,\alpha)}$, and the projection onto the first factor in (\ref{complexdecomposition}) is just $\mathrm{split}_{(A,t,\alpha)}$ restricted to $H_{\mathrm{dR}}^1(A/\overline{\mathbb{Q}}) \subset \mathcal{H}_{\mathrm{dR}}^1(\mathcal{A})(A,t)$. When $p$ is split in $K$, the latter statement follows from the well-known fact that the CM splitting induces the unit root splitting upon base-change to $\mathbb{C}_p$, see \cite[Lemma 5.1.27]{Katz2}.

Hence in all, we have
$$i_p^{-1}\left(\mathrm{split}_{(A,t,\alpha)}\right) = i_{\infty}^{-1}\left(\overline{\mathrm{split}}_{\mathbb{C},(A,t)}\right).$$

Now we claim that there is a unique embedding $\overline{\mathbb{Q}}_p \subset \mathcal{O}_{\Delta,\mathcal{Y}}(A,t,\alpha)$ such that its composition with $\theta : \mathcal{O}_{\Delta,\mathcal{Y}}(A,t,\alpha) \twoheadrightarrow \hat{\mathcal{O}}_{\mathcal{Y}}(A,t,\alpha) = \mathbb{C}_p$ is the natural inclusion. For this, note that the pro\'{e}tale stalk $\mathcal{O}_{\Delta,\mathcal{Y},y}$ by Proposition \ref{localdescription} (reduced modulo $t$) is a power series ring over the complete local ring $\mathcal{O}_{\mathcal{Y},y}$, and so is Henselian. Let $\frak{m} \subset \mathcal{O}_{\mathcal{Y},y}$ be the maximal ideal of the local ring $\mathcal{O}_{\mathcal{Y},y}$. Then $\mathcal{O}_{\Delta,\mathcal{Y}}(A,t,\alpha) = \mathcal{O}_{\Delta,\mathcal{Y},y} \otimes_{\mathcal{O}_{\mathcal{Y},y}} \mathcal{O}_{\mathcal{Y},y}/\frak{m}$ is also a Henselian local ring, with maximal ideal given by $\ker\theta$ (generated by an indeterminate in the above power series description of $\mathcal{O}_{\Delta,\mathcal{Y},y}$) and residue field $\mathbb{C}_p$. (Recall that the tensor $\otimes_{\mathcal{O}_{\mathcal{Y},y}}$ is given by the natural embedding $\mathcal{O}_{\mathcal{Y},y} \subset \mathcal{O}\mathbb{B}_{\mathrm{dR},\mathcal{Y},y}^+ \overset{\mod t}{\twoheadrightarrow} \mathcal{O}_{\Delta,\mathcal{Y},y}$. We can thus describe the maximal ideal of $\mathcal{O}_{\Delta,\mathcal{Y},y}$ as the ideal generated by the image of the maximal ideal of $\mathcal{O}_{\mathcal{Y},y}$ under this embedding, and $\ker\theta$.) Now since $\overline{\mathbb{Q}}_p \subset \mathbb{C}_p$, by Hensel's lemma the inclusion $\mathbb{Q}_p \subset \mathcal{O}_{\Delta,\mathcal{Y}}(A,t,\alpha)$ lifts uniquely to an inclusion $\overline{\mathbb{Q}}_p \subset \mathcal{O}_{\Delta,\mathcal{Y}}(A,t,\alpha)$ whose composition with $\theta$ (the reduction map to the residue field) is the natural inclusion $\overline{\mathbb{Q}}_p \subset \mathbb{C}_p$. 

Since $\mathrm{split}_{(A,t,\alpha)}$ is defined over $\overline{\mathbb{Q}}_p$ and the the map
$$\overline{\mathbb{Q}}_p \subset \mathcal{O}_{\Delta,\mathcal{Y}}(A,t,\alpha) \xrightarrow{\theta} \mathbb{C}_p$$
is just the natural inclusion, we have
$$\theta\left(i_p^{-1}\left(\mathrm{split}_{(A,t,\alpha)}\right)\right) = i_p^{-1}\left(\mathrm{split}_{(A,t,\alpha)}\right)$$
and so
$$i_p^{-1}\left(\left(\theta\circ\mathrm{split}\right)_{(A,t,\alpha)}\right) = \theta\left(i_p^{-1}\left(\mathrm{split}_{(A,t,\alpha)}\right)\right) = i_p^{-1}\left(\mathrm{split}_{(A,t,\alpha)}\right) =  i_{\infty}^{-1}\left(\overline{\mathrm{split}}_{\mathbb{C},(A,t)}\right).$$

Now the rest of the Theorem follows from the definition of $\Omega_p$ and $\Omega_{\infty}$ (taken with respect to the fixed $\omega_0 \in \Omega_{A/\overline{\mathbb{Q}}}^1$).

\end{proof}

\subsection{Relation of $d_k^j$ to the ordinary Atkin-Serre operator}

Recall the Atkin-Serre operator
$$d_{\mathrm{AS},k} : \omega_{\mathcal{A}}|_{\mathcal{Y}^{\mathrm{Ig}}}^{\otimes k} \rightarrow \omega_{\mathcal{A}}|_{\mathcal{Y}^{\mathrm{Ig}}}^{\otimes k+2}$$
which acts, with respect to the Serre-Tate coordinate $T$, as
$$d_{\mathrm{AS},k} (F\cdot \omega_{\mathrm{can}}^{\mathrm{Katz},\otimes k}) = \left(\theta_{\mathrm{AS}}F\right) \cdot\omega_{\mathrm{can}}^{\mathrm{Katz},\otimes k+2}$$
where 
$$\theta_{\mathrm{AS}} = \frac{Td}{dT}.$$
Note that the operator $\theta_{\mathrm{AS}}^j$, defined in terms of the Serre-Tate coordinate $T$, does not depend on the weight $k$, unlike $\theta_k^j$. Let $d_{\mathrm{AS},k}^j = d_{\mathrm{AS},k+2j-2}\circ \cdots d_k$ denote the $j$-fold composition, and similarly with $\theta_{\mathrm{AS}}^j$. 

\begin{proposition}\label{recoverAtkinSerre}On $\mathcal{Y}^{\mathrm{Ig}}$, we have
$$d_k^j = d_{\mathrm{AS},k}^j$$
and
$$\theta_k^j = \theta_{\mathrm{AS}}^j.$$
\end{proposition}

\begin{proof}From the $p$-adic Legendre relation (\ref{periodsrelation}), we have on $\mathcal{Y}^{\mathrm{Ig}}$
$$\theta(1/(z_{\mathrm{dR}} - \bar{z})) = -\theta(y_{\mathrm{dR}})/\theta(\mathbf{z}) = 0.$$

Hence from (\ref{MSformula}), we have 
\begin{equation}\label{ASequation1}\theta_k^j = \theta\circ \delta_k^j = \theta\left(\frac{d}{dz_{\mathrm{dR}}}\right)^j.
\end{equation}
By Proposition \ref{STcoincide}, we have
\begin{equation}\label{ASequation2}\frac{d}{d z_{\mathrm{dR}}} = \frac{d}{d\log T} = \frac{Td}{dT}
\end{equation}
and so combining (\ref{ASequation1}) and (\ref{ASequation2}), we have
$$d_k^j = \theta(d_{\mathrm{AS},k}^j) = d_{\mathrm{AS},k}^j$$
where the last equality follows from the fact that $\mathcal{O} \subset \mathcal{O}\mathbb{B}_{\mathrm{dR}}^+ \xrightarrow{\theta} \hat{\mathcal{O}}$ is the natural inclusion. 

\end{proof}

Hence from now on, we can regard our theory of $p$-adic analysis on $\mathcal{Y}_x$ using $d_k^j$ as an extension of Katz's theory of $p$-adic analysis on $\mathcal{Y}^{\mathrm{Ig}}$ using $d_{\mathrm{AS},k}^j$. 

\section{$p$-adic Analysis of $d_k^j$}\label{overconvergencesection}

In this section, we analyze $p$-adic analytic properties of $d_k^j$ and $\theta_k^j$, in particular how $\theta_k^jF$ behaves as a function of $j$ for elements $F \in \mathcal{O}_{\Delta,\mathcal{Y},y}$ in the stalk at a geometric point $y \in \mathcal{Y}_x(\mathbb{C}_p,\mathcal{O}_{\mathbb{C}_p})$ of the period ring $\mathcal{O}_{\Delta,\mathcal{Y}}$, which we view as the space of ``germs at $y$ of nearly holomorphic functions". Crucial to this study will be the ``$q_{\mathrm{dR}}$-expansion map" 
$$\mathcal{O}_{\mathcal{Y}} \overset{q-\mathrm{exp}}{\hookrightarrow} \hat{\mathcal{O}}_{\mathcal{Y}}\llbracket q_{\mathrm{dR}}-1\rrbracket \subset \mathcal{O}_{\Delta,\mathcal{Y}},$$
which on the supersingular locus factors the natural inclusion $\mathcal{O}_{\mathcal{Y}^{\mathrm{ss}}} \subset \mathcal{O}_{\Delta,\mathcal{Y}^{\mathrm{ss}}}$ (Proposition \ref{factorproposition}). The $q_{\mathrm{dR}}$-expansion map can be viewed naturally as analogue (or an extension) of the Serre-Tate $T$-expansion map, and similarly is injective and so satisfies a ``$q_{\mathrm{dR}}$-expansion principle". Indeed, on the natural cover $\mathcal{Y}^{\mathrm{Ig}}$ of $Y^{\mathrm{ord}}$, the $q_{\mathrm{dR}}$-expansion map recovers the Serre-Tate $T$-expansion (Theorem \ref{STexpcoincide}). The coordinate descriptions (\ref{MSformula}) and (\ref{MSformulatheta}) allow us to compute the action of the $p$-adic Maass-Shimura operators $d_k^j$ and $\theta_k^j$ on $q_{\mathrm{dR}}$-expansions, and hence study their $p$-adic analytic properties.

Understanding the analytic properties of $\theta_k^jF$ will be important in Chapter \ref{padicLfunctionsection} for constructing our $p$-adic $L$-function and establishing our $p$-adic Waldspurger formula.

\subsection{$q_{\mathrm{dR}}$-expansions}\label{q'coordinates}
We retain the notation of the previous sections and let $\mathcal{U}$ denote an affinoid subdomain of $\mathcal{Y}$.  Let $F$ be a complete nonarchimedean field with ring of integers $\mathcal{O}_F$. Recall that the unit polydisc over $\mathrm{Spa}(F,\mathcal{O}_F)$ given by
$$\mathrm{Spa}(F\langle z\rangle, \mathcal{O}_F \langle z \rangle)$$
and the rational subdomain the adic \emph{open} punctured polydisc of radius $r$ over $\mathrm{Spa}(F,\mathcal{O}_F)$ given by 
$$D_r^0 := \bigcup_{\pi \in \mathbb{C}_p, |\pi| < r}\bigcup_{\pi'\in \mathbb{C}_p, |\pi'| < |\pi|} \mathrm{Spa}\left(F\langle \pi^{-1}z, \pi' z^{-1} \rangle, \mathcal{O}_F\langle \pi^{-1}z, \pi' z^{-1}\rangle\right).$$
Also recall the adic torus
$$\mathbb{T} = \mathrm{Spa}\left(F\langle T^{\pm 1}\rangle, \mathcal{O}_F\langle T^{\pm 1}\rangle\right)$$
and the rational subdomain (in fact an open polydisc) consisting of ``units of distance less than $r$ away from 1''
\begin{align*}&\mathbb{T}_r^0 = \bigcup_{\pi \in \mathbb{C}_p, |\pi| < r}\bigcup_{\pi'\in \mathbb{C}_p, |\pi'| < |\pi|} \\
&\mathrm{Spa}\left(F\langle T^{\pm 1}, \pi^{-1}(T-1),\pi'(T-1)^{-1}\rangle, \mathcal{O}_F\langle T^{\pm 1}, \pi^{-1}(T-1),\pi'(T-1)^{-1}\rangle\right).
\end{align*}
Hence by our above convention, we identify
$$D_{\infty}^0 = \mathbb{A}^1 = \mathbb{P}_x^1 \subset \mathbb{P}^1$$
where $z = -x/y$ in terms of the canonical homogeneous coordinates on $\mathbb{P}^1 = \mathbb{P}(F^{\oplus 2})$. 

The $p$-adic exponential gives an isomorphism of adic spaces over $\mathrm{Spa}(F,\mathcal{O}_F)$
\begin{equation}\label{exponentialisomorphism}
\exp : D_{p^{-1/(p-1)}}^0 \xrightarrow{\sim} \mathbb{T}_{p^{-1/(p-1)}}^0
\end{equation}
which sends 
$$D_{p^{-1/(p-1)}}^0 \ni z \mapsto \exp z = 1 + z + \frac{z^2}{2!} + \ldots \in \mathbb{T}_{p^{-1/(p-1)}}^0,$$
with inverse given by the $p$-adic logarithm 
\begin{equation}\label{logarithmicisomorphism} \log : \mathbb{T}_{p^{-1/(p-1)}}^0 \xrightarrow{\sim} D_{p^{-1/(p-1)}}^0
\end{equation}
which sends 
$$\mathbb{T}_{p^{-1/(p-1)}} \ni T \mapsto \log T = (T-1) - \frac{(T-1)^2}{2} + \frac{(T-1)^3}{3} -  \ldots \in D_{p^{-1/(p-1)}}.$$

Recall that for $(F,\mathcal{O}_F) = (\mathbb{Q}_p,\mathbb{Z}_p)$, the Hodge-Tate period map is a smooth map of adic spaces (over $\mathrm{Spa}(\mathbb{Q}_p,\mathbb{Z}_p)$)
$$\pi_{\mathrm{HT}} : \mathcal{Y}_x \rightarrow D_{\infty}^0$$
which restricts to a \emph{pro\'{e}tale} morphism on the open affinoid subdomain $\mathcal{Y}_x^{\mathrm{ss}} := \mathcal{Y}_x \cap \mathcal{Y}^{\mathrm{ss}}$
$$\pi_{\mathrm{HT}} : \mathcal{Y}_x^{\mathrm{ss}} \rightarrow D_{\infty}^0 \cap \Omega^2$$
where we recall $\Omega^2 \subset \mathbb{P}^1$ is Drinfeld's upper half-plane (an open rational subdomain of $\mathbb{P}^1$). 

Now define
$$\mathcal{U}_r = \pi_{\mathrm{HT}}^{-1}(D_r^0), \hspace{1cm} \mathcal{U}_r^{\mathrm{ss}} = \mathcal{U}_r \cap \mathcal{Y}^{\mathrm{ss}}  = \pi_{\mathrm{HT}}^{-1}(D_r^0 \cap \Omega^2).$$ Note that the $GL_2(\mathbb{Q}_p)$ action induces \emph{isomorphisms}
$$[p^k] : \mathcal{U}_r \xrightarrow{\sim} \mathcal{U}_{rp^{-k}},\hspace{1cm} (A,\alpha) \mapsto (A,\alpha) \cdot \left(\begin{array}{ccc} 1 & 0 \\
0 & p^k\\
\end{array}\right)$$
of adic spaces over $\mathrm{Spa}(\mathbb{Q}_p,\mathbb{Z}_p)$. (This is an isomorphism since the matrix is invertible.) We interpret $[p^k]$ as ``contracting toward the origin by a factor of $p^k$". Now we are supplied with a family of morphisms
$$\rho_k := \exp \circ \pi_{\mathrm{HT}} \circ [p^k] : \mathcal{U}_{p^{k-1/(p-1)}} \rightarrow \mathbb{T}_{p^{-1/(p-1)}}^0 \subset \mathbb{T}$$
which restrict to pro\'{e}tale morphisms
\begin{equation}\label{contractionetale}\rho_k : \mathcal{U}_{p^{k-1/(p-1)}}^{\mathrm{ss}} \rightarrow \mathbb{T}_{p^{-1/(p-1)}} \cap \Omega^2 \subset \mathbb{T}.
\end{equation}

Let $L$ be any algebraically closed perfectoid field containing $\mathbb{Q}_p$ (for example, we could take $L = \mathbb{C}_p$). Base-change everything to $\mathrm{Spa}(L,\mathcal{O}_L)$ and, for the rest of this section suppress the subscript ``$L$" for brevity. 

Now let $y \in \mathcal{U}_{p^{k-1/(p-1)}}^{\mathrm{ss}}$ be any geometric point, which under the natural pro\'{e}tale maps $\lambda : \mathcal{U}_{p^{k-1/(p-1)}}^{\mathrm{ss}} \rightarrow Y$ and $\rho_k : \mathcal{U}_{p^{k-1/(p-1)}}^{\mathrm{ss}} \rightarrow \mathbb{T}$, induces geometric points $\lambda(y) \in Y$ and $\rho_k(y) \in \mathbb{T}$. Since $\rho_k$ is pro\'{e}tale, given any sheaf $\mathcal{F}$ on $Y_{\text{pro\'{e}t}}$ and $\mathcal{G}$ on $\mathbb{T}_{\text{pro\'{e}t}}$, we have an identification of pro\'{e}tale stalks
$$\mathcal{F}_y = \mathcal{F}_{\lambda(y)}, \hspace{1cm} \mathcal{G}_y = \mathcal{G}_{\rho_k(y)}.$$
If further $\mathcal{F}|_{\rho_k^{-1}(\mathcal{V})} = \rho_k^*\mathcal{G}|_{\mathcal{V}}$ for any pro\'{e}tale neighborhood $\mathcal{V} \rightarrow \mathbb{T}$ of $\rho_k(y)$, then 
\begin{equation}\label{pullbackstalk}\mathcal{F}_y = \mathcal{F}_{\lambda(y)} = \mathcal{G}_{\rho_k(y)}.
\end{equation}

Since $\rho_k$ is pro\'{e}tale (in fact profinite-\'{e}tale), by the first fundamental exact sequence for $\Omega^1$ (\cite[Proposition 1.6.3, 1.6.8]{Huber}), we have
\begin{equation}\label{differentialpullback}\Omega_{\mathcal{U}_{p^{k-1/(p-1)}}^{\mathrm{ss}}}^1 = \rho_k^*\Omega_{\mathbb{T},\rho_k(y)}^1
\end{equation}
and so by (\ref{pullbackstalk}), we have
\begin{equation}\label{differentialstalk}\Omega_{Y,\lambda(y)}^1 = \Omega_{\mathcal{Y},y}^1 = \Omega_{\mathbb{T},\rho_k(y)}^1.
\end{equation}

Let $Y^+/\mathbb{Z}_p$ denote the adicifation of the (localization at $p\nmid N$ of the) Katz-Mazur integral model $Y^+/\mathbb{Z}[1/N]$ for $Y$, so that $\mathcal{O}_Y^+ = \mathcal{O}_{Y^+}$, and the exterior differential maps
\begin{equation}\label{integraldifferential}
d : \mathcal{O}_{Y^+} \rightarrow \Omega_{Y^+}^1.
\end{equation}
which induces a map on any stalk
\begin{equation}\label{integraldifferentialstalk}
d : \mathcal{O}_{Y^+,x} \rightarrow \Omega_{Y^+,x}^1.
\end{equation}
Hence we have
$$\mathcal{O}_{Y,\lambda(y)} = \mathcal{O}_{\mathcal{Y},y} = \mathcal{O}_{\mathbb{T},\rho_k(y)}$$
and
\begin{equation}\label{integralstalk}\mathcal{O}_{Y^+,\lambda(y)}  = \mathcal{O}_{Y,\lambda(y)}^+ = \mathcal{O}_{\mathcal{Y},y}^+ = \mathcal{O}_{\mathbb{T},\rho_k(y)}^+
\end{equation}
and by \cite[Lemma 4.2 (ii)]{Scholze} and \cite[Proposition 14.3.1]{Conrad3}, we have
\begin{equation}\label{integralstalkcriterion}\mathcal{O}_{Y^+,\lambda(y)} = \mathcal{O}_{Y,\lambda(y)}^+ = \mathcal{O}_{\mathcal{Y},y}^+ = \mathcal{O}_{\mathbb{T},\rho_k(y)}^+ = \{f \in \mathcal{O}_{Y,\lambda(y)} = \mathcal{O}_{\mathcal{Y},y} = \mathcal{O}_{\mathbb{T},\rho_k(y)} : |f(y)| \le 1\}
\end{equation}
where $f(y) := f \pmod{\frak{p}_y}$, where $\frak{p}_y$ is the prime ideal corresponding to the equivalence class of valuations associated with $y$. Finally, we have by (\ref{differentialstalk})
\begin{equation}\label{integraldifferentialstalk}\Omega_{Y^+,\lambda(y)}^1 \subset \Omega_{Y,
\lambda(y)}^1 = \Omega_{\mathcal{Y},y}^1 = \Omega_{\mathbb{T},\rho_k(y)}^1.
\end{equation}
Again, we have the description
\begin{equation}\label{integraldifferentialdescription}\begin{split}&\Omega_{Y^+,\lambda(y)}^1 \\
&= \{w \in \Omega_{Y,
\lambda(y)}^1 = \Omega_{\mathcal{Y},y}^1 = \Omega_{\mathbb{T},\rho_k(y)}^1 : |w/w_0(y)| \le 1 \hspace{.2cm} \text{for any generator} \hspace{.2cm} w_0 \in \Omega_{Y^+,\lambda(y)}^1\}
\end{split}
\end{equation}
where the generator $w_0$ exists above by the moduli interpretation of the Katz-Mazur integral model $Y^+$ of $Y$.

Note that for any $s\in \mathbb{Q}$, we have an isomorphism of adic spaces over $\mathrm{Spa}(L,\mathcal{O}_L)$
$$p^s : D_{p^r}^0 \rightarrow D_{p^{r-s}}^0$$
given by $x \mapsto p^s x$. Choose any $b \in \mathbb{Z}$ such that $s+b \ge 0$. Now we define an \'{e}tale map
$$\rho_k^s : \mathcal{U}_{p^{k-1/(p-1)}}^{\mathrm{ss}} \rightarrow \mathbb{T}$$
given by the composition 
$$\rho_k^s = \exp \circ p^{s+b} \circ \pi_{\mathrm{HT}} \circ [p^k]$$
(the $b$ is necessary in the above definition so that $p^{a+b} \circ \pi_{\mathrm{HT}} \circ [p^k](\mathcal{U}_{p^{k-1/(p-1)}}^{\mathrm{ss}})$ is in the radius of convergence for $\exp$). We denote
\begin{equation}\label{qpower}\frak{q}^{p^s} := (\rho_k^s)^*T^{1/p^b} \in \mathcal{O}_{\mathcal{Y}}(\mathcal{U}_{p^{k-1/(p-1)}}\times_{\mathbb{T}} \mathbb{T}^{1/p^b})
\end{equation}
where $q^{1/p^b} \in \mathcal{O}_{\mathbb{T}}(\mathbb{T}^{1/p^b})$ and 
$$\mathbb{T}^{1/p^b} := \mathrm{Spa}(L \langle T^{\pm 1/p^b}\rangle, \mathcal{O}_L\langle T^{\pm 1/p^b}\rangle) \xrightarrow{x \mapsto x^{p^b}} \mathbb{T}$$
is the standard finite \'{e}tale cover. 

Recall that we have the canonical differential $\frac{dq}{q} \in \Omega_{\mathbb{T},\rho_k(y)}^1$. Now suppose that $y\in \tilde{\mathcal{U}}_{p^{k-1/(p-1)}}(L,\mathcal{O}_L)$ is a geometric point, so that the valuation group of the equivalence class of valuation associated wth $y$ is $\mathbb{Q}$, and let $a \in \mathbb{Q}$ be the smallest rational number such that under the identification (\ref{integraldifferentialstalk}), 
\begin{equation}\label{integralcorrection}p^a\frac{dq}{q} \in \Omega_{Y^+,\lambda(y)}^1 \subset \Omega_{Y,\lambda(y)}^1 = \Omega_{\mathbb{T},\rho_k^a(y)}^1,
\end{equation}
or in other words, such that $p^a\frac{dq}{q}$ is a generator of $\Omega_{Y^+,\lambda(y)}^1$; such $a$ exists by (\ref{integraldifferentialdescription}). 
Now 
$$p^a \frac{dq}{q} = \frac{dq^{p^a}}{q^{p^a}}$$
and so since $q^{p^a} = \exp(p^a\log q) = \exp(p^{a+n}\log q^{1/p^n})$ (the last series is a convergent power series in $q^{\mathrm{1/p^n}}-1$ for all large $n \in \mathbb{Z}_{\ge 0}$, and so belongs to the pro\'{e}tale stalk since  $q^{1/p^n} \in \mathcal{O}_{Y^+,\lambda(y)}$) is a unit in $\mathcal{O}_{Y^+,\lambda(y)} = \mathcal{O}_{Y,\lambda(y)}^+ = \mathcal{O}_{\mathcal{Y},y}^+ = \mathcal{O}_{\mathbb{T},\rho_k(y)}^+$ under the identification (\ref{integralstalk}), we have
\begin{equation}\label{integraldifferentialgenerator}dq^{p^a} \in \Omega_{Y^+,\lambda(y)}^1 \subset \Omega_{Y,\lambda(y)}^1 = \Omega_{\mathbb{T},\rho_k^a(y)}^1
\end{equation}
is a generator. Hence by (\ref{integraldifferential}) and (\ref{integraldifferentialgenerator}), we have
\begin{equation}\label{integralderivative} \frac{d}{dq^{p^a}} : \mathcal{O}_{\mathcal{Y},y}^+ = \mathcal{O}_{Y^+,\lambda(y)} \rightarrow \mathcal{O}_{Y^+,\lambda(y)} = \mathcal{O}_{Y^+,y}.
\end{equation}
By standard theory, the completion (by maximal prime ideal corresponding to $y$) of the \'{e}tale stalk of $\mathcal{O}_{Y^+}$ at $y$ is isomorphic to the completion of the analytic stalk of $\mathcal{O}_{Y^+}$ at $y$, which is isomorphic to $\mathcal{O}_{\mathbb{C}_p}\llbracket q^{p^a} - q^{p^a}(y) \rrbracket$. We can write the pro\'{e}tale stalk $\mathcal{O}_{Y^+,y}$ as a direct limit of \'{e}tale stalks (which are henselizations of noetherian local rings, in particular subrings of the completions), and so the pro\'{e}tale stalk maps to the direct limit of the completions of \'{e}tale stalks, each of which is isomorphic to the power series ring $\mathcal{O}_{\mathbb{C}_p}\llbracket q^{p^a} - q^{p^a}(y) \rrbracket$. Hence there is a natural inclusion of the pro\'{e}tale stalk $\mathcal{O}_{Y^+,y}$ into $\mathcal{O}_{\mathbb{C}_p}\llbracket q^{p^a} - q^{p^a}(y) \rrbracket$. Writing any element $f \in \mathcal{O}_{Y^+,\lambda(y)}$ as
$$f = \sum_{n=0}^{\infty} a_n(q^{p^a}-q^{p^a}(y))^n, \hspace{1cm} a_n \in \mathcal{O}_{\mathbb{C}_p},$$
we see by the usual Taylor coefficient formula that
$$a_n = \frac{1}{n!}\left(\frac{d}{dq^{p^a}}\right)^n(f)(y) \in \mathcal{O}_{\mathbb{C}_p}$$
and so by (\ref{integralstalkcriterion}), we have
\begin{equation}\label{integralderivative2} \frac{1}{n!}\left(\frac{d}{dq^{p^a}}\right)^n : \mathcal{O}_{\mathcal{Y},y}^+ = \mathcal{O}_{Y^+,\lambda(y)} \rightarrow \mathcal{O}_{Y^+,\lambda(y)} = \mathcal{O}_{\mathcal{Y},y}^+.
\end{equation}

Alternatively, to prove (\ref{integralderivative2}), we could have invoked \cite[Chapter 2 Proposition 2.6, see also Remark 2.7]{BerthelotOgus}.

Applying Proposition \ref{localdescription} with respect to $\rho_k^a : \mathcal{U}_{p^{k-1/(p-1)}}^{\mathrm{ss}} \rightarrow \mathbb{T}$, we get 
$$\tilde{\mathcal{U}}_{p^{k-1/(p-1)}}^{\mathrm{ss}} := \mathcal{U}_{p^{k-1/(p-1)}}^{\mathrm{ss}}\times_{\rho_k^a,\mathbb{T}} \tilde{\mathbb{T}} \in Y_{\text{pro\'{e}t}}$$
where $\tilde{\mathbb{T}}$ is defined in (\ref{infinitytorus}), on which we have an isomorphism of sheaves on the localized site $Y_{\text{pro'{e}t}}/\tilde{\mathcal{U}}_{p^{k-1/(p-1)}}^{\mathrm{ss}}$
$$\mathcal{O}\mathbb{B}_{\mathrm{dR},\tilde{\mathcal{U}}_{p^{k-1/(p-1)}}^{\mathrm{ss}}}^+ = \mathbb{B}_{\mathrm{dR},\tilde{\mathcal{U}}_{p^{k-1/(p-1)}}^{\mathrm{ss}}}^+\llbracket \frak{q}^{p^a} - [\frak{q}^{p^a,\flat}]\rrbracket$$
where 
$$\frak{q}^{p^a} = (\rho_k^a)^*T$$
is as in (\ref{qpower}), and
$$\frak{q}^{p^a,\flat} = (\frak{q}^{p^a},(\frak{q}^{p^a})^{1/p},(\frak{q}^{p^a})^{1/p^2},\ldots) \in \mathcal{O}_Y^{+,\flat}(\tilde{\mathcal{U}}_{p^{k-1/(p-1)}}^{\mathrm{ss}}),$$
so the Teichm\"{u}ller lift 
$$[\frak{q}^{p^a,\flat}] \in \mathbb{B}_{\mathrm{dR},Y}^+(\tilde{\mathcal{U}}_{p^{k-1/(p-1)}}^{\mathrm{ss}}) \subset \mathcal{O}\mathbb{B}_{\mathrm{dR},Y}^+(\tilde{\mathcal{U}}_{p^{k-1/(p-1)}}^{\mathrm{ss}}).$$
Reducing modulo $t$, we have a natural local description of $\mathcal{O}_{\Delta,\mathcal{Y}}^+$
$$\mathcal{O}_{\Delta,\tilde{\mathcal{U}}_{p^{k-1/(p-1)}}^{\mathrm{ss}}}^+ = \hat{\mathcal{O}}_{\tilde{\mathcal{U}}_{p^{k-1/(p-1)}}^{\mathrm{ss}}}^+\llbracket q^{p^a} - \theta(q)^{p^a}\rrbracket$$
where
$$q^{p^a} := \frak{q}^{p^a} \pmod{t},$$
and we note that
$$\theta(q)^{p^a} = [\frak{q}^{p^a,\flat}] \pmod{t}.$$

\begin{proposition}With the same notation as in the above discussion, let $\tilde{y}$ denote any point in $\tilde{\mathcal{U}}_{p^{k-/1(p-1)}}^{\mathrm{ss}}(L,\mathcal{O}_L)$ lying above $y \in \mathcal{U}_{p^{k-1/(p-1)}}^{\mathrm{ss}}(L,\mathcal{O}_L)$. Then we have that the natural inclusion 
\begin{equation}\label{qexp}\mathcal{O}_{\mathcal{Y},y}^+ \subset \mathcal{O}_{\Delta,\tilde{\mathcal{U}}_{p^{k-1/(p-1)}}^{\mathrm{ss}},\tilde{y}}^+ = \hat{\mathcal{O}}_{\tilde{\mathcal{U}}_{p^{k-1/(p-1)}}^{\mathrm{ss}},\tilde{y}}^+\llbracket q^{p^a} - \theta(q)^{p^a}\rrbracket
\end{equation}
factors through
\begin{equation}\label{integralstalkqexp}\mathcal{O}_{\tilde{\mathcal{U}}_{p^{k-1/(p-1)}}^{\mathrm{ss}},\tilde{y}}^+ \subset \hat{\mathcal{O}}_{\tilde{\mathcal{U}}_{p^{k-1/(p-1)}}^{\mathrm{ss}},\tilde{y}}^+\llbracket q^{p^a} - \theta(q)^{p^a}\rrbracket.
\end{equation}
Moreover, (\ref{integralstalkqexp}) extends to the $p$-adic completion
\begin{equation}\label{integralstalkqexpcomplete}\hat{\mathcal{O}}_{\tilde{\mathcal{U}}_{p^{k-1/(p-1)}}^{\mathrm{ss}},\tilde{y}}^+ \subset \hat{\mathcal{O}}_{\tilde{\mathcal{U}}_{p^{k-1/(p-1)}}^{\mathrm{ss}},\tilde{y}}^+\llbracket q^{p^a} - \theta(q)^{p^a}\rrbracket.
\end{equation}
\end{proposition}

\begin{proof}Given $f \in \mathcal{O}_{\mathcal{Y},y}^+$, let 
$$\sum_{n = 0}^{\infty} a_n(f)(q^{p^a}-\theta(q^{p^a}))^n, \hspace{1cm} a_n(f) \in \hat{\mathcal{O}}_{\tilde{\mathcal{U}}_{p^{k-1/(p-1)}}^{\mathrm{ss}},\tilde{y}}$$
denote its image under the map (\ref{qexp}). We have a commutative diagram
\begin{equation}
\begin{tikzcd}
     & \mathcal{O}_{\tilde{\mathcal{U}}_{p^{k-1/(p-1)}}^{\mathrm{ss}}} \arrow{r}{\nabla} \arrow[hookrightarrow]{d}{} & \Omega_{\tilde{\mathcal{U}}_{p^{k-1/(p-1)}}^{\mathrm{ss}}}^1 \arrow[hookrightarrow]{d}{}&\\
     & \mathcal{O}\mathbb{B}_{\mathrm{dR},\tilde{\mathcal{U}}_{p^{k-1/(p-1)}}^{\mathrm{ss}}}^+ \arrow{r}{\nabla} & \mathcal{O}\mathbb{B}_{\mathrm{dR},\tilde{\mathcal{U}}_{p^{k-1/(p-1)}}^{\mathrm{ss}}}^+\otimes_{\mathcal{O}_{\tilde{\mathcal{U}}_{p^{k-1/(p-1)}}^{\mathrm{ss}}}}\Omega_{\tilde{\mathcal{U}}_{p^{k-1/(p-1)}}^{\mathrm{ss}}}^1&\\
\end{tikzcd}
\end{equation}
which, reducing mod $t$, trivializing $\Omega_{\tilde{\mathcal{U}}_{p^{k-1/(p-1)}}^{\mathrm{ss}},\tilde{y}}^1 \cong \mathcal{O}_{\tilde{\mathcal{U}}_{p^{k-1/(p-1)}}^{\mathrm{ss}},\tilde{y}}$ using the section $dq^{p^a}$ and extending to the $p$-adic completion, implies that
\begin{equation}
\begin{tikzcd}\label{coefficientdiagram}
     & \hat{\mathcal{O}}_{\tilde{\mathcal{U}}_{p^{k-1/(p-1)}}^{\mathrm{ss}},\tilde{y}} \arrow{r}{\frac{d}{dq^{p^a}}} \arrow[hookrightarrow]{d}{} \hspace{1cm}& \hspace{1cm}\hat{\mathcal{O}}_{\tilde{\mathcal{U}}_{p^{k-1/(p-1)}}^{\mathrm{ss}},\tilde{y}} \arrow[hookrightarrow]{d}{}&\\
     & \hat{\mathcal{O}}_{\tilde{\mathcal{U}}_{p^{k-1/(p-1)}}^{\mathrm{ss}},\tilde{y}}\llbracket q^{p^a} - \theta(q^{p^a})\rrbracket \arrow{r}{\frac{d}{dq^{p^a}} = \frac{d}{d(q^{p^a}-\theta(q^{p^a}))}} \hspace{1cm}&\hspace{1cm} \hat{\mathcal{O}}_{\tilde{\mathcal{U}}_{p^{k-1/(p-1)},\tilde{y}}^{\mathrm{ss}}}\llbracket q^{p^a} - \theta(q^{p^a})\rrbracket&\\
\end{tikzcd}
\end{equation}
is commutative. Note that since $\mathcal{O}\subset \mathcal{O}\mathbb{B}_{\mathrm{dR}}^+ \xrightarrow{\theta} \hat{\mathcal{O}}$ is the natural inclusion, then (\ref{coefficientdiagram}) implies that 
$$a_n(f) = \theta\left(\frac{1}{n!}\left(\frac{d}{dq^{p^a}}\right)^n\right).$$
Now by (\ref{integralderivative2}), we in fact we have that
$$a_n(f) = \theta\left(\frac{1}{n!}\left(\frac{d}{dq^{p^a}}\right)^n\right) \in \mathcal{O}_{\mathcal{Y},y}^+ = \mathcal{O}_{\tilde{\mathcal{U}}_{p^{k-1/(p-1)}}^{\mathrm{ss}},\tilde{y}}^+$$
by identifying pro\'{e}tale stalks.

Now (\ref{integralstalkqexpcomplete}) follows because 
$$\hat{\mathcal{O}}_{\tilde{\mathcal{U}}_{p^{k-1/(p-1)}}^{\mathrm{ss}},\tilde{y}}^+  = \varprojlim_n\mathcal{O}_{\tilde{\mathcal{U}}_{p^{k-1/(p-1)}}^{\mathrm{ss}},\tilde{y}}^+/p^n.$$
\end{proof}

\begin{definition}\label{qHTdefinition}Let $y$ be as above. Henceforth, let
$$q_{\mathrm{HT}}^{p^a} := q^{p^a}/\theta(q^{p^a}) \in \mathcal{O}_{\Delta,\tilde{\mathcal{U}}_{p^{k-1/(p-1)}}^{\mathrm{ss}},\tilde{y}}$$
so that $\theta(q_{\mathrm{HT}}^{p^a}) = 1$. Also let
$$q_{\mathrm{dR}} := \exp(z_{\mathrm{dR}} - \theta(z_{\mathrm{dR}})) \in \mathcal{O}_{\Delta,\mathcal{Y}}(\mathcal{Y}_x)$$
where the fact that $q_{\mathrm{dR}}$ defines a section on all of $\mathcal{Y}_x$ follows from Proposition (\ref{invertible}). We also denote
$$dq_{\mathrm{dR}} := \nabla(q_{\mathrm{dR}}).$$
Note that
$$dq_{\mathrm{dR}} = q_{\mathrm{dR}}dz_{\mathrm{dR}}.$$

Furthermore, for any $b \in \mathbb{Q}$ (and a choice of an element $p^b \in \mathbb{C}_p$ with $p$-adic absolute value $|p|^b = (1/p)^b$), we can define 
$$q_{\mathrm{dR}}^{1/p^b} := \exp\left(\frac{z_{\mathrm{dR}}-\theta(z_{\mathrm{dR}})}{p^b}\right).$$
\end{definition}

\begin{lemma}\label{y'lemma}Let $y$ be as above. We have
\begin{equation}\label{calculation}y_{\mathrm{dR}}^2 \cdot \frac{dz_{\mathrm{dR}}}{dq^{p^a}} \in \mathcal{O}_{\mathcal{Y},y}^{\times}.
\end{equation}
As a consequence, we have
\begin{equation}\label{calculation2}\theta(y_{\mathrm{dR}})^2\cdot\theta\left(\frac{dq_{\mathrm{dR}}}{dq^{p^a}}\right) \in \mathcal{O}_{\mathcal{Y},y}^{\times}.
\end{equation}
\end{lemma}

\begin{proof}By the calculation (\ref{basiselt}) and the diagram (\ref{KSdiagram}), we have
\begin{equation}\label{calculation2}\sigma(\frak{s}^{\otimes 2}) = y_{\mathrm{dR}}^2\cdot\frac{d{z_{\mathrm{dR}}}}{dq^{p^q}}\otimes dq^{p^q}.
\end{equation}
Then since $\frak{s}^{\otimes 2}$ trivializes $\omega_{\mathcal{A}}|_{\mathcal{Y},y} \cong \mathcal{O}_{\mathcal{Y},y}$ and $dq^{p^a}$ trivializes $\Omega_{\mathcal{Y},y}^1 \cong \mathcal{O}_{\mathcal{Y},y}$, so we must have that (\ref{calculation}) holds. For the second statement, we have that $\mathcal{O}_{\mathcal{Y},y} \subset \mathcal{O}_{\Delta,\mathcal{Y},y} \overset{\theta}{\twoheadrightarrow} \hat{\mathcal{O}}_{\mathcal{Y},y}$ is the natural inclusion, and so 
$$\theta(y_{\mathrm{dR}})^2\cdot\theta\left(\frac{dz_{\mathrm{dR}}}{dq^{p^a}}\right) \in \mathcal{O}_{\mathcal{Y},y}^{\times}.$$
Now the statement follows by observing
$$\theta\left(\frac{dq_{\mathrm{dR}}}{dq^{p^a}}\right) = \theta\left(q_{\mathrm{dR}}\frac{dz_{\mathrm{dR}}}{dq^{p^a}}\right) = \theta\left(\frac{dz_{\mathrm{dR}}}{dq^{p^a}}\right)$$
since $\theta(q_{\mathrm{dR}}) = \theta(1 + (\mathbf{z}_{\mathrm{dR}} - \theta(\mathbf{z}_{\mathrm{dR}})) + (\mathbf{z}_{\mathrm{dR}} - \theta(\mathbf{z}_{\mathrm{dR}}))^2/2! + \ldots) = 1$. 
\end{proof}

\begin{definition}\label{qdRexp}Let $b \in \mathbb{Q}$. We define the ``$q_{\mathrm{dR}}^{1/p^b}$-expansion map" as
\begin{equation}\label{qdrexp}q_{\mathrm{dR}}^{1/p^b}-\mathrm{exp} : \mathcal{O}_{\mathcal{Y}_x} \hookrightarrow \hat{\mathcal{O}}_{\mathcal{Y}_x} \llbracket q_{\mathrm{dR}}^{1/p^b}-1\rrbracket
\end{equation}
given by 
$$f \mapsto f(q_{\mathrm{dR}}^{1/p^b}) := \sum_{n = 0}^{\infty}\theta\left(\frac{1}{n!}\left(\frac{d}{dq_{\mathrm{dR}}^{1/p^b}}\right)^nf\right)(q_{\mathrm{dR}}^{1/p^b}-1)^n.$$
Note that $q_{\mathrm{dR}}^{1/p^b}-\mathrm{exp}$ is an injection, since $\theta \circ q_{\mathrm{dR}}^{1/p^b}-\mathrm{exp}$ is the natural inclusion $\mathcal{O}_{\mathcal{Y}^{\mathrm{ss}}} \subset \hat{\mathcal{O}}_{\mathcal{Y}^{\mathrm{ss}}}$, and in particular injective.

Note that we have a natural inclusion
\begin{equation}\label{naturalinclusion}\hat{\mathcal{O}}_{\mathcal{Y}_x}\llbracket q_{\mathrm{dR}}^{1/p^b} - 1\rrbracket \subset \mathcal{O}_{\Delta,\mathcal{Y}_x}
\end{equation}
where $\hat{\mathcal{O}}_{\mathcal{Y}_x} \subset \mathcal{O}_{\Delta,\mathcal{Y}_x}$ via (\ref{horizontaldefinition}), and $q_{\mathrm{dR}}^{1/p^b} -1$ is naturally an element of $\mathcal{O}_{\Delta,\mathcal{Y}_x}$ (and in fact is in $\ker\theta$) by (\ref{qHTdefinition}). Since $dz_{\mathrm{dR}}$ generates $\Omega_{\Delta,\mathcal{Y}_x}^1$ by Proposition \ref{generator2}. 

One checks that by construction, the map (\ref{qdrexp}) commutes with connections, where the connection on $\mathcal{O}_{\mathcal{Y}_x}$ is the natural one defined earlier, and the connection on $\hat{\mathcal{O}}_{\mathcal{Y}_x}\llbracket q_{\mathrm{dR}}^{1/p^b} - 1\rrbracket$ is induced by the connection on $\mathcal{O}_{\Delta,\mathcal{Y}_x}$ via the inclusion (\ref{naturalinclusion}).

When $b = 0$, so that $q_{\mathrm{dR}}^{1/p^b} = q_{\mathrm{dR}}$, we abbreviate
$$q_{\mathrm{dR}}-\mathrm{exp} = q_{\mathrm{dR}}^{1/p^b}-\mathrm{exp}, \hspace{1cm} f(q_{\mathrm{dR}}) = f(q_{\mathrm{dR}}^{1/p^b}).$$

\end{definition}

\begin{proposition}\label{factorproposition}On $\mathcal{Y}^{\mathrm{ss}}$, we have that the natural inclusion (\ref{naturalinclusion}) is in fact an equality
\begin{equation}\label{equality}\mathcal{O}_{\Delta,\mathcal{Y}^{\mathrm{ss}}}\llbracket q_{\mathrm{dR}}^{1/p^b}-1 \rrbracket =\mathcal{O}_{\Delta,\mathcal{Y}^{\mathrm{ss}}}.
\end{equation}
Then natural inclusion
\begin{equation}\label{naturalinclusion2}\mathcal{O}_{\mathcal{Y}^{\mathrm{ss}}} \subset \mathcal{O}_{\Delta,\mathcal{Y}^{\mathrm{ss}}}
\end{equation}
factors as
$$\mathcal{O}_{\mathcal{Y}^{\mathrm{ss}}} \overset{q_{\mathrm{dR}}^{1/p^b}-\mathrm{exp}}{\hookrightarrow} \hat{\mathcal{O}}_{\mathcal{Y}^{\mathrm{ss}}}\llbracket q_{\mathrm{dR}}^{1/p^b} - 1\rrbracket \overset{(\ref{equality})}{=} \mathcal{O}_{\Delta,\mathcal{Y}^{\mathrm{ss}}}.$$
\end{proposition}

\begin{proof}We want to show (\ref{equality}). First, note that as subsheafs of $\mathcal{O}_{\Delta,\mathcal{Y}^{\mathrm{ss}}}$ (using (\ref{horizontaldefinition})), we have
\begin{equation}\label{periodequality}\hat{\mathcal{O}}_{\mathcal{Y}^{\mathrm{ss}}}\llbracket z - \bar{z}\rrbracket = \hat{\mathcal{O}}_{\mathcal{Y}^{\mathrm{ss}}}\llbracket z_{\mathrm{dR}} - \theta(z_{\mathrm{dR}}) \rrbracket = \hat{\mathcal{O}}_{\mathcal{Y}^{\mathrm{ss}}}\llbracket q_{\mathrm{dR}}^{1/p^b}-1\rrbracket.
\end{equation}
The first equality follows from Proposition \ref{nonvanishing}, since $\mathcal{Y}^{\mathrm{ss}} \subset \mathcal{Y}_x \cap \mathcal{Y}_y$. The second equality is a completely formal, given by a ``change of variables" 
$$z_{\mathrm{dR}}-\theta(z_{\mathrm{dR}}) \rightarrow \exp \left(\frac{z_{\mathrm{dR}}-\theta(z_{\mathrm{dR}})}{p^b}\right).$$ Under (\ref{periodequality}), (\ref{naturalinclusion}) is given by 
\begin{equation}\label{naturalinclusion'}\hat{\mathcal{O}}_{\mathcal{Y}^{\mathrm{ss}}}\llbracket z - \bar{z}\rrbracket \subset \mathcal{O}_{\Delta,\mathcal{Y}^{\mathrm{ss}}}
\end{equation}
where $\hat{\mathcal{O}}_{\mathcal{Y}^{\mathrm{ss}}} \subset \mathcal{O}_{\Delta,\mathcal{Y}^{\mathrm{ss}}}$ through (\ref{horizontaldefinition}), and $z-\bar{z} \mapsto z-\bar{z}$. 

To construct the inverse of (\ref{naturalinclusion'}), we again follow the argument of (\cite[1.6.10]{Scholze}). We claim it suffices to show that there is a unique map 
\begin{equation}\label{zexp}\mathcal{O}_{\mathcal{Y}^{\mathrm{ss}}} \hookrightarrow \hat{\mathcal{O}}_{\mathcal{Y}^{\mathrm{ss}}}\llbracket z - \bar{z}\rrbracket
\end{equation}
sending $z \mapsto \bar{z} + (z-\bar{z})$ which is compatible with connections, and such that the composition with the map $\theta : \hat{\mathcal{O}}_{\mathcal{Y}^{\mathrm{ss}}} \llbracket z - \bar{z} \rrbracket \rightarrow \hat{\mathcal{O}}_{\mathcal{Y}^{\mathrm{ss}}}$ is the natural inclusion. Given (\ref{zexp}), we have a natural map
$$\mathcal{O}_{\mathcal{Y}^{\mathrm{ss}}} \otimes_{W(\kappa)} \hat{\mathcal{O}}_{\mathcal{Y}^{\mathrm{ss}}} \rightarrow \hat{\mathcal{O}}_{\mathcal{Y}^{\mathrm{ss}}}\llbracket z - \bar{z} \rrbracket$$
by extending $\otimes_{W(\kappa)} \hat{\mathcal{O}}_{\mathcal{Y}^{\mathrm{ss}}}$-linearly, which is easily checked to be the inverse of (\ref{naturalinclusion'}). 

To prove the existence of (\ref{zexp}), note that the Hodge-Tate period map gives a pro\'{e}tale map $\mathcal{Y}^{\mathrm{ss}} \rightarrow \Omega \subset \mathbb{P}_x^1 = \mathbb{A}^1 = \bigcup_{r>0} \mathrm{Spa}(L\langle p^{-r}z\rangle, \mathcal{O}_L\langle p^{-r}z\rangle)$. Now note that we have a map 
$$\bigcap_{r > 0}W(\kappa)[1/p]\langle p^{-r}z\rangle \rightarrow \hat{\mathcal{O}}_{\mathcal{Y}_x^{\mathrm{ss}}}\llbracket z - \bar{z}\rrbracket$$
given by $z \mapsto \bar{z} + (z-\bar{z})$. Now by the same argument as in Claim \ref{claim} using Hensel's lemma (\emph{mutatis mutandis}), this extends uniquely to a homomorphism
$$\mathcal{O}_{\mathcal{Y}^{\mathrm{ss}}} \rightarrow \hat{\mathcal{O}}_{\mathcal{Y}^{\mathrm{ss}}}\llbracket z - \bar{z}\rrbracket$$
whose composition with $\theta : \hat{\mathcal{O}}_{\mathcal{Y}^{\mathrm{ss}}} \llbracket z - \bar{z} \rrbracket \rightarrow \hat{\mathcal{O}}_{\mathcal{Y}^{\mathrm{ss}}}$ is the natural inclusion (which also shows that it is an injection). So we have shown (\ref{zexp}), and so we are done.

Now to prove the factoring statement, we recall that 
\begin{equation}\label{powerseries}\mathcal{O}_{\mathcal{Y}^{\mathrm{ss}}} \subset \mathcal{O}_{\Delta,\mathcal{Y}^{\mathrm{ss}}} = \hat{\mathcal{O}}_{\mathcal{Y}^{\mathrm{ss}}}\llbracket q_{\mathrm{dR}}^{1/p^b}-1\rrbracket
\end{equation}
(\ref{naturalinclusion2}) commutes with connections, and that it composed with $\theta : \mathcal{O}_{\Delta,\mathcal{Y}^{\mathrm{ss}}} \twoheadrightarrow \hat{\mathcal{O}}_{\mathcal{Y}^{\mathrm{ss}}}$ is the natural inclusion, hence given a section $f$ of $\mathcal{O}_{\mathcal{Y}^{\mathrm{ss}}}$, each coefficient of its image under (\ref{powerseries}) is given by Taylor's formula, and so its image is just $f(q_{\mathrm{dR}}^{1/p^b})$. 
\end{proof}

We have the following proposition relating $q_{\mathrm{dR}}-\mathrm{exp}$ to the Serre-Tate expansion. Given a $q_{\mathrm{dR}}$-expansion $f(q_{\mathrm{dR}}) = (q_{\mathrm{dR}}-\mathrm{exp})(f)$ or $T$-expansion $f(T) = (T-\mathrm{exp})(f)$, let $f(q_{\mathrm{dR}})(y) = (q_{\mathrm{dR}}-\mathrm{exp})(f)(y)$ and $(T) = (T-\mathrm{exp})(f)$ denote
the corresponding power series obtained by evaluating each coefficient in the fiber at $y$ (i.e. reducing modulo $\frak{p}_y$, the prime ideal corresponding to $y$). 
\begin{theorem}\label{STexpcoincide}For a point $y = (A,\alpha) \in \mathcal{Y}^{\mathrm{Ig}}$ where $A$ is an ordinary elliptic curve with complex multiplication, we have 
$$q_{\mathrm{dR}}-\mathrm{exp}(\cdot)(y) = T-\mathrm{exp}|_{D(y)}(\cdot)(y)$$
where $T-\mathrm{exp}$ is the Serre-Tate expansion map (see Definition \ref{SerreTateexpansion}), and $D(y)$ is the preimage in $\mathcal{Y}^{\mathrm{Ig}}$ of the ordinary residue disc in $Y^{\mathrm{ord}}$ centered around the geometric point of $Y$ corresponding to $\lambda(y)$. (Recall that $\lambda : \mathcal{Y} \rightarrow Y$ is the natural projection.)
\end{theorem}

\begin{proof}First, note that by a strictly formal computation, for a formal variable $Y$, we have that
$$\theta\left(\frac{d}{d\exp(Y)}\right)^n|_{\exp Y = 1} = P_n\left(\left(\frac{d}{dY}\right)^n|_{Y = 0}\right)$$
for some easily computable polynomial $P_n(X) \in \mathbb{Q}[X]$. 
On $\mathcal{Y}^{\mathrm{Ig}}$, by Theorem \ref{STcoincide}, we then have
\begin{align*}\theta\left(\frac{1}{n!}\left(\frac{d}{dq_{\mathrm{dR}}}\right)^nf\right) = \frac{1}{n!}P_n\left(\theta\left(\left(\frac{d}{dz_{\mathrm{dR}}}\right)^nf\right)\right) &= \frac{1}{n!}P_n\left(\frac{d}{d\log T}|_{\log T = 0}f(y)\right) \\
&= \frac{1}{n!}\left(\frac{d}{dT}\right)^n|_{T = 1}f(y),
\end{align*}
where for a section $s$ of $\mathcal{O}|_{\mathcal{Y}^{\mathrm{Ig}}}$, $s(y)$ denotes evaluation at $y$ (i.e. the image of $s$ in the residue field $y$). Now we are done by (\ref{STcomputation}).
\end{proof}

Let 
$$Y_{0r} = Y(\Gamma_1(N)\cap \Gamma_0(p^r))$$
and recall that $Y = Y_1(N) = Y_{00}$. Recall the $U_p$ operator which can be viewed as a correspondence $Y_{0r} \leftarrow Y_{0r} \rightarrow Y$ and and $V_p$ viewed as a correspondence $Y_{0r} \leftarrow Y_{0r} \rightarrow Y$. Hence they can be viewed as Hecke operators which act on modular forms $w \in \omega_{\mathcal{A}}^{\otimes k}(Y_{0r})$ by considering the pullback $w|_{Y_{0r+1}}$, so that $U_p^*(w|_{Y_{0r+1}}), V_p^*(w|_{Y_{0r+1}}) \in \omega_{\mathcal{A}}^{\otimes k}(Y_{0r+1})$. Given a modular form $w \in \omega_{\mathcal{A}}^{\otimes k}(Y)$, we can define the \emph{$p$-stabilization} 
\begin{equation}\label{pstabilization}((U_pV_p)^*-(V_pU_p)^*)w \in \omega_{\mathcal{A}}^{\otimes k}(Y_{02}).
\end{equation}
We can find lifts of them in terms of the $GL_2(\mathbb{Q}_p)$ acting on the infinite level modular curve $\mathcal{Y}$; in other words, we can find correspondences $\tilde{U}_p, \tilde{V}_p$ on $\mathcal{Y}$ defined in terms of the $GL_2(\mathbb{Q}_p)$-action such that
$$\tilde{U}_p^*w|_{\mathcal{Y}_x} = (U_p^*w)|_{\mathcal{Y}_x}, \hspace{1cm} \tilde{V}_p^*w|_{\mathcal{Y}_x} = (V_p^*w)|_{\mathcal{Y}_x}.$$
In fact, one such choice of $\tilde{U}_p, \tilde{V}_p$ is given by 
$$\tilde{V}_p = \left(\begin{array}{ccc} 1 & 0 \\
0 & p\\
\end{array}\right), \hspace{1cm} \tilde{U}_p = \frac{1}{p}\sum_{j = 0}^{p-1}\left(\begin{array}{ccc} p & j \\
0 & 1\\
\end{array}\right).
$$
We recall here that $GL_2(\mathbb{Q}_p)$ acts on functions $f$ by $\gamma^*f(x) = f(x\gamma)$.  

Again, given a $T$-expansion $f(T)$ and a geometric point $y \in \mathcal{Y}(\mathbb{C}_p,\mathcal{O}_{\mathbb{C}_p})$, let $f(T)(y)$ denote the specialization of $f(T)$ at (i.e. image in the fiber at) $y$. 

\begin{theorem}[\cite{Brakocevic} Lemma 8.2, \cite{LiuZhangZhang} Proposition A.0.1] \label{stabilizationtheorem}Suppose that $y \in \mathcal{Y}^{\mathrm{Ig}}$ is an ordinary CM point. Let $w \in \omega_{\mathcal{A}}^{\otimes k}(Y)$, so that $((U_pV_p)^* - (V_pU_p)^*)w \in \omega_{\mathcal{A}}^{\otimes k}(Y_{02})$. Write $w|_{\mathcal{Y}^{\mathrm{Ig}}} = f\cdot 
\left(\omega_{\mathrm{can}}^{\mathrm{Katz}}\right)^{\otimes k}$ in $\omega_{\mathcal{A}}^{\otimes k}(\mathcal{Y}^{\mathrm{Ig}})$. Then 
\begin{align*}(((U_pV_p)^* - (V_pU_p)^*)f)|_{\mathcal{Y}^{\mathrm{Ig}}}(T)(y) &= (((\tilde{U}_p\tilde{V}_p)^* - (\tilde{V}_p\tilde{U}_p)^*)f|_{\mathcal{Y}^{\mathrm{Ig}}})(T)(y) \\
&= f|_{\mathcal{Y}^{\mathrm{Ig}}}(T)(y) - \frac{1}{p}\sum_{j = 0}^{p-1}f|_{\mathcal{Y}^{\mathrm{Ig}}}(\zeta_p^j T)(y).
\end{align*}
Furthermore, when $k = 2$, letting $\log_w$ denote the $p$-adic logarithm $Y \rightarrow \mathbb{C}_p$ associated with $w$, and writing 
$$F|_{\mathcal{Y}} = \log_w|_{\mathcal{Y}},$$
we have
\begin{equation}\label{BrookslogformulaSTexpansion}\begin{split}
&\lim_{m \rightarrow \infty}\left(\frac{Td}{dT}\right)^{-1+p^m(p-1)}\left((((U_pV_p)^*- (V_pU_p)^*)f)|_{\mathcal{Y}^{\mathrm{Ig}}}(T)\right) \\
&=\left(\lim_{m \rightarrow \infty} \theta_{\mathrm{AS}}^{-1 + p^m(p-1)}((U_pV_p)^* - (V_pU_p)^*)f\right)|_{\mathcal{Y}^{\mathrm{Ig}}}(T) \\
&= (((U_pV_p)^* - (V_pU_p)^*)F)|_{\mathcal{Y}^{\mathrm{Ig}}}(T) = (((\tilde{U}_p\tilde{V}_p)^* - (\tilde{V}_p\tilde{U}_p)^*)F|_{\mathcal{Y}^{\mathrm{Ig}}})(T) \\
&= F|_{\mathcal{Y}^{\mathrm{Ig}}}(T)  - \frac{1}{p}\sum_{j = 0}^{p-1}F|_{\mathcal{Y}^{\mathrm{Ig}}}(\zeta_p^j T).
\end{split}
\end{equation}
\end{theorem}

The last goal of this section is to prove our ``Key Lemma'', which will be essential for our $p$-adic analytic computations, in particular for showing the convergence of the sequence of $p$-adic Maass-Shimura derivatives of a rigid function in stalks at supersingular CM points. First, let $b \in \mathbb{Q}$, and choose an element $p^b \in \mathbb{C}_p$ (of absolute value $|p|^b = (1/p)^b$) such that
\begin{equation}\label{integralcorrection2}
\left|\theta\left(\frac{dq_{\mathrm{dR}}}{dq_{\mathrm{HT}}^{p^a}}\right)(y)\right| = |p^b|.
\end{equation}
Then recalling
$$q_{\mathrm{dR}}^{1/p^b} := \exp\left(\frac{z_{\mathrm{dR}}-\theta(z_{\mathrm{dR}})}{p^b}\right),$$
we have by (\ref{integralcorrection2})
\begin{equation}\label{integralcorrection3}
\left|\theta\left(\frac{dq_{\mathrm{dR}}^{1/p^b}}{dq_{\mathrm{HT}}^{p^a}}\right)(y)\right| = 1.
\end{equation}

First, let 
$$\mathcal{U}^a := \{|\mathbf{z}| > a\}.$$

\begin{lemma}[Key Lemma]\label{keylemma}Suppose that $y \in \mathcal{Y}(L,\mathcal{O}_L) \cap \mathcal{U}^{p^{p/(p^2-1)}}$ 
where we recall that $L$ was the algebraically closed perfectoid field containing $\mathbb{Q}_p$ fixed at the beginning of this section.
We have that 
\begin{equation}\label{keylemma1}q^{p^a} \in \hat{\mathcal{O}}_{\mathcal{Y},y}^+\llbracket q_{\mathrm{dR}}^{1/p^b} -1\rrbracket.
\end{equation}
As a consequence, we have
\begin{equation}\label{keylemma2}\mathcal{O}_{\mathcal{Y},y}^+ \subset \hat{\mathcal{O}}_{\mathcal{Y},y}^+\llbracket q_{\mathrm{dR}}^{1/p^b} - 1\rrbracket
\end{equation}
and this extends to the $p$-adic completion
\begin{equation}\label{keylemma3}\hat{\mathcal{O}}_{\mathcal{Y},y}^+ \subset \hat{\mathcal{O}}_{\mathcal{Y},y}^+\llbracket q_{\mathrm{dR}}^{1/p^b} - 1\rrbracket.
\end{equation}
\end{lemma}

The argument essentially comes down to an application of the Dieudonn\'{e}-Dwork lemma. We recall a result of Chojecki-Hansen-Johansson that in particular implies that the subgroup 
$$\langle \alpha_1 \pmod{p}\rangle$$
called the \emph{pseudo-canonical subgroup} in by Chojecki-Hansen-Johansson, coincides with the canonical subgroup on $\mathcal{U}^{p^{p/(p^2-1)}}$. 

\begin{lemma}[\cite{ChojeckiHansenJohansson}, Lemma 2.14]\label{canonicalsubgrouplemma} Assume that $|\mathbf{z}(A,\alpha)| > p^{p/(p^2-1)}$. Then the pseudocanonical subgroup $\langle \alpha_1 \pmod p\rangle \subset A[p]$ is in fact the canonical subgroup of $A$.
\end{lemma}

View $\mathcal{Y}$ as an adic space over $\mathrm{Spa}(W(\overline{\mathbb{F}}_p)[1/p],W(\overline{\mathbb{F}}_p))$. 
By Lemma \ref{canonicalsubgrouplemma}, the isogeny 
$$\mathcal{A}_{\infty} \rightarrow \mathcal{A}_{\infty}\cdot \left(\begin{array}{ccc}1 & 0\\
0 & p\\
\end{array}\right) = \left(\mathcal{A}/\langle \alpha_{\infty,1} \pmod p \rangle, \alpha_{\infty}\cdot \left(\begin{array}{ccc}1 & 0\\
0 & p\\
\end{array}\right)\right),$$
defined over $\mathrm{Spa}(F,\mathcal{O}_F)$, gives rise via the universal property of $(\mathcal{A},\alpha_{\infty})$ to a classifying map $\phi : \mathcal{Y}_x = \mathcal{U}^0 \rightarrow \mathcal{U}^0 = \mathcal{Y}_x$ such that
$$\mathcal{A}_{\infty}\cdot \left(\begin{array}{ccc}1 & 0\\
0 & p\\
\end{array}\right) = \mathcal{A}_{\infty}\times_{\mathcal{U}^0,\phi}\mathcal{U}^0.$$
The restriction
$$\phi : \mathcal{U}^{p^{p/(p^2-1)}} \rightarrow \mathcal{U}^{p^{p/(p^2-1)-1}}$$ 
gives a lifting of the $p$-power Frobenius morphism on the special fiber (this being Frobenius-linear). Hence, by Lemma \ref{canonicalsubgrouplemma}, the induced map
\begin{equation}\left(\begin{array}{ccc}1 & 0\\
0 & p\\
\end{array}\right)^* : \mathcal{O}(\mathcal{U}^{p^{p/(p^2-1)-1}}) \rightarrow \mathcal{O}(\mathcal{U}^{p^{p/(p^2-1)}})
\end{equation}
is a lifting of the $p$-power Frobenius morphism.

\begin{proof}[Proof of Lemma \ref{keylemma}]

Note that by (\ref{z'transformationprop}) and (\ref{ztransformationprop}), and that by definition $z_{\mathrm{dR}} = \mathbf{z}_{\mathrm{dR}} \pmod{t}$, then for any integer $r \ge 0$ we have
\begin{equation}\label{z'contractionprop}\left(\begin{array}{ccc}1 & 0\\
0 & p^r\\
\end{array}\right)^*z_{\mathrm{dR}} = p^rz_{\mathrm{dR}}, \hspace{2cm} \left(\begin{array}{ccc}1 & 0\\
0 & p^r\\
\end{array}\right)^*(z - \bar{z}) = p^r(z - \bar{z})
\end{equation}
and from Definition \ref{qHTdefinition} we see that
\begin{equation}\label{q'contractionprop}\left(\begin{array}{ccc}1 & 0\\
0 & p^r\\
\end{array}\right)^*q_{\mathrm{dR}}= q_{\mathrm{dR}}^{p^r}, \hspace{2cm} \left(\begin{array}{ccc}1 & 0\\
0 & p^r\\
\end{array}\right)^*q_{\mathrm{HT}}  = q_{\mathrm{HT}}^{p^r}.
\end{equation}

We now need the following modified version of the Dieudonn\'{e}-Dwork lemma, which addresses certain non-standard liftings of Frobenius (specifically, liftings of Frobenius to power series rings in one variable $X$ which do not send $X \mapsto X^p$).

\begin{lemma}[Dieudonn\'{e}-Dwork]\label{Dworklemma} 

Let $F$ be a complete nonarchimedean field, and let $\frak{p}$ denote the maximal ideal of $\mathcal{O}_F$. Suppose we have a lifting 
$$\phi : \mathcal{O}_F\llbracket X \rrbracket \rightarrow \mathcal{O}_F\llbracket X \rrbracket$$
of the Frobenius morphism, in the sense that $\phi \pmod{\frak{p}}$ is equal to the composition
$$(\mathcal{O}_F/\frak{p}\mathcal{O}_F)\llbracket X\rrbracket = \overline{\mathbb{F}}_p\llbracket X\rrbracket \xrightarrow{x\mapsto x^p} \overline{\mathbb{F}}_p\llbracket X\rrbracket = (\mathcal{O}_F/\frak{p}\mathcal{O}_F)\llbracket X\rrbracket,$$
and also suppose that $\phi(X) \in X\mathcal{O}_F\llbracket X\rrbracket$. Note that this extends to a morphism
$$\phi : F\llbracket X \rrbracket \rightarrow F\llbracket X\rrbracket$$ 
on generic fibers.

Let $Q(X) = \sum_{n = 0}^{\infty}a_nX^n \in 1 + XF\llbracket X\rrbracket$, and assume that 
\begin{equation}\label{keyassumption}a_1 \in \mathcal{O}_F.
\end{equation}
Then $Q(X) \in 1+ XF\llbracket X \rrbracket$ is in $1 + X\mathcal{O}_F\llbracket X\rrbracket$ with constant term in $1 + \frak{p}\mathcal{O}_F$ if and only if we have
\begin{equation}\label{Dworkcriterion}\frac{\phi(Q(X))}{Q(X)^p} \in 1 + \frak{p}X\mathcal{O}_F\llbracket X\rrbracket.
\end{equation}
\end{lemma}

\begin{remark}We note that there is no analogue of the assumption (\ref{keyassumption}) on $a_1$ made in the original Dieudonn\'{e}-Dwork lemma \cite[Lemma 1]{Dwork}. It is made here because we address a slightly modified version of the Frobenius lifting considered in loc. cit. Namely, instead of specifiying that $\phi(X) = X^p$, we consider a more general $\phi$ where $\phi(X) \in X\mathcal{O}_F\llbracket X \rrbracket$. 
\end{remark}

\begin{proof}[Proof of Lemma \ref{Dworklemma}]We adapt the proof of \cite[Lemma 1]{Dwork}. Let $Q(X) = \sum_{n = 0}^{\infty}a_nX^n$. Suppose that $Q(X) \in 1 + X\mathcal{O}_F\llbracket X\rrbracket$. Then since $\phi$ is a lifting of Frobenius, we have 
$$Q(X)^p \equiv \sum_{n=0}^{\infty}a_n^pX^{pn} \equiv \phi\left(\sum_{n = 0}^{\infty}a_nX^n\right) = \phi(Q(X)) \pmod{\frak{p}\mathcal{O}_F\llbracket X\rrbracket}$$
which gives one direction of the statement.

Now suppose that (\ref{Dworkcriterion}) is satisfied. Then we have
\begin{equation}\label{Dworkrelation}\left(\sum_{n = 0}^{\infty}a_nX^n\right)^p = \left(\sum_{n = 0}^{\infty}\phi(a_n)\phi(X)^n\right)\left(1+\sum_{n = 1}^{\infty}b_nX^n\right)
\end{equation}
where $b_n \in \frak{p}\mathcal{O}_F$ for all $n \ge 0$. We show that $a_n \in \mathcal{O}_F$ for all $n \ge 0$, proceeding by induction on $n$. By assumption, $a_0 = 1$ and $a_1 \in \mathcal{O}_F$. Now assume that $a_0, a_1, \ldots, a_{n -1} \in \mathcal{O}_F$ and $n > 1$. Note that since $\phi(X) \in X\mathcal{O}_F\llbracket X\rrbracket$, then we can write $\phi(X) = X(X^{p-1} + g(X))$ where $g(X) \in \frak{p}\mathcal{O}_F\llbracket X\rrbracket$. Now comparing terms of degree at most $n$ on both sides of (\ref{Dworkrelation}), we have 
\begin{equation}\label{Dworkcongruence}pa_nX^n + \sum_{j = 0}^{n-1}a_j^pX^{jp} \equiv \sum_{j = 0}^n\phi(a_j)X^j(X^{p-1}+g(X))^j \pmod{(\frak{p},X^{n+1})\mathcal{O}_F\llbracket X\rrbracket}.
\end{equation}
Since $a_j \in \hat{R}$ for $0\le j \le n-1$, we have 
$$\phi(a_j)X^j(X^{p-1}+g(X))^j \equiv a_j^pX^{jp} \pmod{\frak{p}\mathcal{O}_F\llbracket X\rrbracket}$$
for $0 \le j \le n-1$. Hence we can clear all terms indexed by $0 \le j \le n-1$ from both sides of (\ref{Dworkcongruence}), and we get
\begin{equation}\label{Dworkcongruence}pa_nX^n \equiv \phi(a_n)X^n(X^{p-1}+g(X))^n \equiv \phi(a_n)g(0)^nX^n \pmod{(\frak{p},X^{n+1})\mathcal{O}_F\llbracket X\rrbracket}.
\end{equation}
Since $g(0)^n \in p^n\mathcal{O}_F$, say $g(0) = p^nu$, (\ref{Dworkcongruence}) implies
\begin{equation}\label{Dworkcongruence2}a_n-\phi(a_n)p^{n-1}u \in \mathcal{O}_F.
\end{equation}

Note that since $\phi : \mathcal{O}_F \rightarrow \mathcal{O}_F$, then $|\phi(a_n)| > |a_n|$. Now since $n > 1$, we have $|\phi(a_n)p^{n-1}u| < |a_n|$, and so (\ref{Dworkcongruence2}) implies that $a_n \in \mathcal{O}_F$. This completes the induction.

\end{proof}

Now we return to our proof of Lemma (\ref{keylemma}). Consider
$$q_{\mathrm{HT}}^{p^a} =  \sum_{n = 0}^{\infty}a_n(q_{\mathrm{dR}}^{1/p^b}-1)^n\in \hat{\mathcal{O}}_{\mathcal{Y},y}\llbracket q_{\mathrm{dR}}^{1/p^b}-1\rrbracket.$$

By (\ref{integralstalkcriterion}), we have
\begin{equation}\label{citecriterion}\hat{\mathcal{O}}_{\mathcal{Y},y}^+ = \{f \in \hat{\mathcal{O}}_{\mathcal{Y},y} : |f(y)| \le 1\}.
\end{equation}
So now to check that 
\begin{equation}\label{qHTinqdR} q_{\mathrm{HT}}^{p^a} = \sum_{n = 0}^{\infty}a_n(q_{\mathrm{dR}}^{1/p^b}-1)^n \in \hat{\mathcal{O}}_{\mathcal{Y},y}^+\llbracket q_{\mathrm{dR}}^{1/p^b}-1\rrbracket,
\end{equation}
it suffices to show that the element
$$q_{\mathrm{HT}}^{p^a} \pmod{\frak{p}_y} = \sum_{n = 0}^{\infty}a_n(y)(q_{\mathrm{dR}}^{1/p^b}-1)^n \in (\hat{\mathcal{O}}_{\mathcal{Y},y}/\frak{p}_y)\llbracket q_{\mathrm{dR}}^{1/p^b} - 1\rrbracket = L\llbracket q_{\mathrm{dR}}^{1/p^b}-1\rrbracket$$
is actually in the subring of integral power series, i.e.
$$q_{\mathrm{HT}}^{p^a} \pmod{\frak{p}_y} = \sum_{n = 0}^{\infty}a_n(y)(q_{\mathrm{dR}}^{1/p^b}-1)^n \in \mathcal{O}_L\llbracket q_{\mathrm{dR}}^{1/p^b}-1\rrbracket,$$
where $\frak{p}_y$ is the prime ideal associated with (the equivalence class of valuations corresponding to) $y$, since then we would have $|a_n(y)| \le 1$ for all $n \ge 0$ and so by (\ref{citecriterion}) we have $a_n \in \hat{\mathcal{O}}_{\mathcal{Y},y}^+$ for all $n \ge 0$.

Now we apply Lemma \ref{Dworklemma} to 
$$\phi = \left(\begin{array}{ccc}1 & 0\\
0 & p\\
\end{array}\right)^*, \hspace{1cm} X = q_{\mathrm{dR}}^{1/p^b}-1,\hspace{1cm} Q(X) =  q_{\mathrm{HT}}^{p^a} \pmod{\frak{p}_y}$$
so that letting $y' = y\cdot \left(\begin{array}{ccc}1 & 0\\
0 & p\\
\end{array}\right)^{-1}$, $\phi : \hat{\mathcal{O}}_{\mathcal{Y},y} \rightarrow \hat{\mathcal{O}}_{\mathcal{Y},y'}$ induces a lifing of Frobenius on residue fields $\phi : L \rightarrow L$, and since $\phi^*(q_{\mathrm{HT}}^{p^a}-1) = (q_{\mathrm{HT}}^{p^a})^p - 1$ we have
$$\phi(X) = X((X+1)^{p-1} + (X+1)^{p-2} + \ldots + 1)$$
and by Definition (\ref{integralcorrection2}) and (\ref{integralstalkcriterion}), the assumption (\ref{keyassumption}) on 
$$a_1 = \theta\left(\frac{dq_{\mathrm{HT}}}{dq_{\mathrm{dR}}^{1/p^b}}\right)_y = \theta\left(\frac{dz_{\mathrm{HT}}}{dz_{\mathrm{dR}}}p^b\right)_y$$
is satisfied. (Here, the subscript $y$ denotes the image in the pro\'{e}tale stalk at $y$.) Hence, we have (\ref{qHTinqdR}). Now (\ref{keylemma2}) follows from (\ref{integralstalkqexpcomplete}) and (\ref{qHTinqdR}), after making the identification
$$\hat{\mathcal{O}}_{\tilde{\mathcal{U}}_{p^{k-1/(p-1)}}^{\mathrm{ss}},\tilde{y}}^+  = \hat{\mathcal{O}}_{\mathcal{Y},y}^+.$$

Finally, (\ref{keylemma3}) follows because $\hat{\mathcal{O}}_{\mathcal{Y},y}^+ = \varprojlim_n \mathcal{O}_{\mathcal{Y},y}^+/p^n$, so reducing (\ref{keylemma2}) modulo $p^n$ for every $n \in \mathbb{Z}_{\ge 0}$, one sees that (\ref{keylemma2}) extends to the $p$-adic completion.
\end{proof}

\begin{corollary}\label{y'corollary}In the situation of Lemma \ref{keylemma}, we have
$$y_{\mathrm{dR}}^2 \in \hat{\mathcal{O}}_{\mathcal{Y},y}^+\llbracket q_{\mathrm{dR}}^{1/p^b}-1\rrbracket [1/p] = \hat{\mathcal{O}}_{\mathcal{Y},y}^+\llbracket q_{\mathrm{dR}}^{1/p^b} - 1 \rrbracket [1/p].$$
\end{corollary}

\begin{proof}By (\ref{calculation}), we have
\begin{align*}y_{\mathrm{dR}}^2 \in \frac{dq^{p^a}}{dz_{\mathrm{dR}}}p^b\cdot \mathcal{O}_{\mathcal{Y},y} = q_{\mathrm{dR}}^{1/p^b}\frac{dq^{p^a}}{dq_{\mathrm{dR}}^{1/p^b}}\cdot \mathcal{O}_{\mathcal{Y},y}^+[1/p] &\overset{(\ref{keylemma2})}{\subset} q_{\mathrm{dR}}^{1/p^b}\frac{dq^{p^a}}{dq_{\mathrm{dR}}^{1/p^b}}\cdot \hat{\mathcal{O}}_{\mathcal{Y},y}\llbracket q_{\mathrm{dR}}^{1/p^b}-1\rrbracket [1/p] \\
&= \hat{\mathcal{O}}_{\mathcal{Y},y}\llbracket q_{\mathrm{dR}}^{1/p^b}-1\rrbracket [1/p].
\end{align*}
Here the last equality follows because clearly 
$$q_{\mathrm{dR}}^{1/p^b} = 1 + (q_{\mathrm{dR}}^{1/p^b}-1) \in \hat{\mathcal{O}}_{\mathcal{Y},y}\llbracket q_{\mathrm{dR}}^{1/p^b}-1\rrbracket$$
and 
$$\frac{dq^{p^a}}{dq_{\mathrm{dR}}^{1/p^b}}\in \hat{\mathcal{O}}_{\mathcal{Y},y}\llbracket q_{\mathrm{dR}}^{1/p^b}-1\rrbracket$$
by (\ref{qHTinqdR}). 
\end{proof}

\subsection{The $p$-adic Maass-Shimura operator $\theta_k^j$ in $q_{\mathrm{dR}}$-coordinates}

Now we can rewrite (\ref{MSformula}) as
\begin{equation}\label{MSformulaqdelta}\begin{split}\delta_k^j &= \sum_{i = 0}^j\binom{j+k-1}{i}\binom{j}{i}i!\frac{1}{(z_{\mathrm{dR}}-\bar{z})^i}\left(\frac{1}{p^b}\frac{q_{\mathrm{dR}}^{1/p^b}d}{dq_{\mathrm{dR}}^{1/p^b}}\right)^{j-i},
\end{split}
\end{equation}
and (\ref{MSformulatheta}) as 
\begin{equation}\label{MSformulaq}\begin{split}\theta_k^j &= \sum_{i = 0}^j\binom{j+k-1}{i}\binom{j}{i}i!\frac{1}{(\theta(z_{\mathrm{dR}})-\theta(\bar{z}))^i}\theta\circ \left(\frac{1}{p^b}\frac{q_{\mathrm{dR}}^{1/p^b}d}{dq_{\mathrm{dR}}^{1/p^b}}\right)^{j-i}\\
&= \sum_{i = 0}^j\binom{j+k-1}{i}\binom{j}{i}i!\left(-\frac{\theta(y_{\mathrm{dR}})}{\mathbf{z}}\right)^i\left(\frac{1}{p^b}\right)^{j-i}\theta\circ \left(\frac{q_{\mathrm{dR}}^{1/p^b}d}{dq_{\mathrm{dR}}^{1/p^b}}\right)^{j-i}\\
&= \sum_{i = 0}^{\infty}c_i(j)\left(-\frac{\theta(y_{\mathrm{dR}})}{\mathbf{z}}\right)^i\left(\frac{1}{p^b}\right)^{j-i}\theta\circ \left(\frac{q_{\mathrm{dR}}^{1/p^b}d}{dq_{\mathrm{dR}}^{1/p^b}}\right)^{j-i}
\end{split}
\end{equation}
where the second equality follows from the $p$-adic Legendre relation (\ref{periodsrelation}), and 
\begin{equation}\label{cijdefinition}
c_i(j) := \begin{cases} \binom{j+k-1}{i}\binom{j}{i}i! & i \le j\\
0 & i > j\\
\end{cases}.
\end{equation}
Note that $c_i(j)$ extends to a $p$-adic continuous function in $j \in \mathbb{Z}/(p-1) \times \mathbb{Z}_p \rightarrow \mathbb{Z}_p$. Here
$$\mathbb{Z} \subset \mathbb{Z}/(p-1) \times \mathbb{Z}_p$$
is embedded diagonally and is dense with respect to the $p$-adic topology (by the Chinese remainder theorem). 

Note that $\frac{q_{\mathrm{dR}}^{1/p^b}d}{dq_{\mathrm{dR}}^{1/p^b}}$ has the following simple action on power series
$$\frac{q_{\mathrm{dR}}^{1/p^b}d}{dq_{\mathrm{dR}}^{1/p^b}}\left(\sum_{n = 0}^{\infty} a_n(q_{\mathrm{dR}}^{1/p^b}-1)^n\right) = \sum_{n = 1}^{\infty} na_nq_{\mathrm{dR}}^{1/p^b}(q_{\mathrm{dR}}^{1/p^b}-1)^{n-1}$$
where $a_n \in \hat{\mathcal{O}}_{\mathcal{Y},y}^+$. In particular, on polynomials, we can write
$$\sum_{n = 0}^da_n(q_{\mathrm{dR}}^{1/p^b}-1)^n = \sum_{n = 0}^d b_n\left(q_{\mathrm{dR}}^{1/p^b}\right)^n$$
for appropriate $b_n$, and hence
$$\left(\frac{q_{\mathrm{dR}}^{1/p^b}d}{dq_{\mathrm{dR}}^{1/p^b}}\right)^j\left(\sum_{n = 0}^{d}a_n(q_{\mathrm{dR}}^{1/p^b}-1)^n\right) = \left(\frac{q_{\mathrm{dR}}^{1/p^b}d}{dq_{\mathrm{dR}}^{1/p^b}}\right)^j\left(\sum_{n = 0}^{d}b_n\left(q_{\mathrm{dR}}^{1/p^b}\right)^n\right) = \sum_{n = 0}^{d}n^j\left(b_nq_{\mathrm{dR}}^{1/p^b}\right)^n.$$
Now endow $\hat{\mathcal{O}}_{\mathcal{Y},y}\llbracket q_{\mathrm{dR}}^{1/p^b}-1\rrbracket$ with the $p$-adic uniform convergence topology, i.e. the metric topology induced by the ``$p$-adic Gauss norm''
$$\left|\sum_{n = 0}^{\infty}a_n(q_{\mathrm{dR}}^{1/p^b}-1)^n\right| := \sup_n|a_n|.$$
Note that this induces a stronger $p$-adic topology on the image of $\hat{\mathcal{O}}_{\mathcal{Y},y} \overset{(\ref{keylemma3})}{\subset} \hat{\mathcal{O}}_{\mathcal{Y},y}^+\llbracket q_{\mathrm{dR}}^{1/p^b}-1\rrbracket [1/p]$, since the weakest topology on $\hat{\mathcal{O}}_{\mathcal{Y},y}^+\llbracket q_{\mathrm{dR}}^{1/p^b}-1\rrbracket [1/p]$ into which the $p$-adic topology on $\hat{\mathcal{O}}_{\mathcal{Y},y}$ embeds is just the weakest topology which makes $\theta : \hat{\mathcal{O}}_{\mathcal{Y},y}^+\llbracket q_{\mathrm{dR}}^{1/p^b}-1\rrbracket [1/p] \twoheadrightarrow \hat{\mathcal{O}}_{\mathcal{Y},y}$ continuous, and open sets in this latter topology consists of sets which induce open sets on the constant term (recall $\theta$ takes a power series $F(q_{\mathrm{dR}}^{1/p^b})$ to its constant term). Hence in order to show $p$-adic analytic properties of an element of $F \in \mathcal{O}_{\Delta,\mathcal{Y}},y$ where $y \in \mathcal{Y}^{\mathrm{ss}}$, it suffices to show such properties for $F(q_{\mathrm{dR}}^{1/p^b})$.

\begin{lemma}\label{limitlemma}Suppose $f \in \hat{\mathcal{O}}_{\mathcal{Y},y} = \hat{\mathcal{O}}_{\mathcal{Y},y}^+[1/p]$. Then under the inclusion (\ref{keylemma3}), we have
\begin{equation}\label{integrality}y_{\mathrm{dR}}^k f \in \hat{\mathcal{O}}_{\mathcal{Y},y}^+\llbracket q_{\mathrm{dR}}^{1/p^b} -1\rrbracket [1/p]
\end{equation}
for any even integer $k \ge 0$. The same conclusion also holds if $p > 2$ for any $k \in \mathbb{Z}_{\ge 0}$.
\end{lemma}
\begin{proof}By (\ref{keylemma3}), we have 
$$f \in \hat{\mathcal{O}}_{\mathcal{Y},y}^+[1/p] \subset \hat{\mathcal{O}}_{\mathcal{Y},y}^+\llbracket q_{\mathrm{dR}}^{1/p^b} -1\rrbracket [1/p]$$
and by Corollary \ref{y'corollary}, we have 
$$y_{\mathrm{dR}}^2 \in \hat{\mathcal{O}}_{\mathcal{Y},y}^+\llbracket q_{\mathrm{dR}}^{1/p^b} -1\rrbracket [1/p].$$
If $p > 2$, then we can use the Taylor series expansion of $y_{\mathrm{dR}}^2$ to see that 
$$y_{\mathrm{dR}} \in \hat{\mathcal{O}}_{\mathcal{Y},y}^+\llbracket q_{\mathrm{dR}}^{1/p^b} -1\rrbracket [1/p].$$
The conclusion now immediately follows.
\end{proof}

From this lemma, we see that integral properties of modular forms from finite level are somewhat preserved upon passing to finite level.  Suppose now that $w \in \omega_{\mathcal{A}}^{\otimes k}(Y^+)$ (where we recall $Y^+/\mathbb{Z}[1/n]$ is the Katz-Mazur model of $Y/\mathbb{Q}$) is a normalized eigenform. Then since $\mathcal{Y}_x \twoheadrightarrow Y$ and $\omega_{\mathrm{can}} = \frak{s}/y_{\mathrm{dR}}$ and $\frak{s}$ trivializes $\omega_{\Delta,\mathcal{Y}}(\mathcal{Y}_x)$, we can write
$$w|_{\mathcal{Y}_x} = y_{\mathrm{dR}}^kf\cdot \omega_{\mathrm{can}}^{\otimes k}$$
where $f \in \mathcal{O}_{\mathcal{Y}}(\mathcal{Y}_x)$. Recall $b \in \mathbb{Q}$ as in (\ref{integralcorrection2}). Let 
$$(y_{\mathrm{dR}}f)^{\flat}(q_{\mathrm{dR}}^{1/p^b}) \in \hat{\mathcal{O}}_{\mathcal{Y}_x}\llbracket q_{\mathrm{dR}}^{1/p^b} - 1\rrbracket$$
denote the \emph{$p$-stabilization of $y_{\mathrm{dR}}^kf$}, which is 
\begin{equation}\begin{split}\label{stabilizedqdRexp}p^r(y_{\mathrm{dR}}^kf)^{\flat}(q_{\mathrm{dR}}^{1/p^b}) &:= p^r(y_{\mathrm{dR}}^kf)(q_{\mathrm{dR}}^{1/p^b}) - \frac{1}{p}\sum_{j = 0}^{p-1}p^r(y_{\mathrm{dR}}^kf)(\zeta_p^jq_{\mathrm{dR}}^{1/p^b}) \\
&= \sum_{n = 0}^{\infty} a_n(q_{\mathrm{dR}}^{1/p^b}-1)^n - \frac{1}{p}\sum_{j = 0}^{p-1}\sum_{n = 0}^{\infty}a_n(\zeta_p^jq_{\mathrm{dR}}^{1/p^b}-1)^n.
\end{split}
\end{equation}

Let 
$r \in \mathbb{Q}$ be minimal such that 
\begin{equation}\label{integralcorrection3}p^r(y_{\mathrm{dR}}^kf)_y^{\flat}(q_{\mathrm{dR}}^{1/p^b}) \in \hat{\mathcal{O}}_{\mathcal{Y},y}^+\llbracket q_{\mathrm{dR}}^{1/p^b}-1\rrbracket,
\end{equation}
where the subscript $y$, as usual, denotes the image in the stalk at $y$. The fact that such $r \in \mathbb{Q}$ exists follows from Lemma \ref{limitlemma} and the definition of the stabilization (\ref{stabilizedqdRexp}). 
(One can in fact use Theorem \ref{integraltheorem} in order to get upper and lower bounds on $r$ in terms of $f$ and $k$.)
Now let
$$p^r(y_{\mathrm{dR}}^kf)_y(q_{\mathrm{dR}}^{1/p^b}) = \sum_{n = 0}^{\infty}a_n(q_{\mathrm{dR}}^{1/p^b} - 1)^n$$
where $a_n \in \hat{\mathcal{O}}_{\mathcal{Y},y}^+$. 

\begin{theorem}\label{keycorollary}Let $f, y, y_{\mathrm{dR}}$ be as above satisfying the assumptions of Lemma \ref{limitlemma}, and 
\begin{equation}\label{integralfiberassumption}|\theta(y_{\mathrm{dR}})(y)p^b/\mathbf{z}(y)| < p^{1/p-1}
\end{equation}
for $y$. 

Then the function $\mathbb{Z} \rightarrow \hat{\mathcal{O}}_{\mathcal{Y},y}^+$ defined by 
$$j \mapsto (p^b\theta_k)^j((y_{\mathrm{dR}}^kf)^{\flat}_y(q_{\mathrm{dR}}^{1/p^b})),$$
where the subscript $y$ denotes the image in the stalk at $y$, extends to a continuous function in $j \in \mathbb{Z}/(p-1) \times \mathbb{Z}_p \rightarrow \hat{\mathcal{O}}_{\mathcal{Y},y}$ by defining
\begin{equation}\label{limitdefinition}(p^b\theta_k)^j((y_{\mathrm{dR}}^kf)^{\flat}_y(q_{\mathrm{dR}}^{1/p^b})) := \lim_{m \rightarrow \infty} (p^b\theta_k)^{j_m}((y_{\mathrm{dR}}^kf)^{\flat}_y(q_{\mathrm{dR}}))
\end{equation}
where writing $j = \sum_{n = 0}^{\infty} \alpha_np^n$ uniquely with $0 \le \alpha_n \le p-1$, then $j_m = \sum_{n = 0}^m\alpha_np^n$. (In particular, this limit exists.)

Moreover, for any $j_0 \in \mathbb{Z}_{\ge 0}$, we have
\begin{equation}\label{limitisstabilization}\lim_{m\rightarrow \infty}(p^b\theta_k)^{j_0+(p-1)p^m}((y_{\mathrm{dR}}^kf)_y(q_{\mathrm{dR}}^{1/p^b})) = (p^b\theta_k)^{j_0}((y_{\mathrm{dR}}^kf)^{\flat}(q_{\mathrm{dR}}^{1/p^b})).\end{equation}

\end{theorem}

\begin{proof}By (\ref{MSformulaqdelta}), we have
\begin{equation}\label{analyticexpressiondelta}(p^b\delta_k)^j((y_{\mathrm{dR}}^kf)^{\flat}_y(q_{\mathrm{dR}}^{1/p^b})) = \sum_{i=0}^{\infty}c_i(j)\frac{(p^b)^i}{(z_{\mathrm{dR}}-\bar{z})_y^i}\left(\frac{q_{\mathrm{dR}}^{1/p^b}d}{dq_{\mathrm{dR}}^{1/p^b}}\right)^{j-i}((y_{\mathrm{dR}}^kf)^{\flat}_y(q_{\mathrm{dR}}^{1/p^b}))
\end{equation}
for $j \in \mathbb{Z}_{\ge 0}$, where again the subscript $y$ denotes the image in the stalk at $y$.

By (\ref{MSformulaq}), we have
\begin{equation}\label{analyticexpression}(p^b\theta_k)^j((y_{\mathrm{dR}}^kf)^{\flat}_y(q_{\mathrm{dR}}^{1/p^b})) = \sum_{i = 0}^{\infty}c_i(j)\left(-\frac{\theta(y_{\mathrm{dR}})p^b}{\mathbf{z}}\right)_y^i\theta\circ \left(\frac{q_{\mathrm{dR}}^{1/p^b}d}{dq_{\mathrm{dR}}^{1/p^b}}\right)^{j-i}((y_{\mathrm{dR}}^kf)^{\flat}_y(q_{\mathrm{dR}}^{1/p^b}))
\end{equation}
for $j \in \mathbb{Z}_{\ge 0}$.
Note that since the $q_{\mathrm{dR}}^{1/p^b}$-expansion map (\ref{qdrexp}) is injective and commutes with derivations, it suffices to prove all statements on $q_{\mathrm{dR}}^{1/p^b}$-expansions.

First, note that for any polynomial 
$$g(q_{\mathrm{dR}}^{1/p^b}) = \sum_{n = 0, p\nmid n}^d b_n\left(q_{\mathrm{dR}}^{1/p^b}\right)^n \in \hat{\mathcal{O}}_{\mathcal{Y},y}^+[q_{\mathrm{dR}}^{1/p^b}] = \hat{\mathcal{O}}_{\mathcal{Y},y}^+[q_{\mathrm{dR}}^{1/p^b}-1],$$

we have
$$\left(\frac{q_{\mathrm{dR}}^{1/p^b}d}{dq_{\mathrm{dR}}^{1/p^b}}\right)^jg(q_{\mathrm{dR}}^{1/p^b}) = \sum_{n = 0,p\nmid n}^dn^jb_n\left(q_{\mathrm{dR}}^{1/p^b}\right)^n
$$
which, by Fermat's little theorem, is a continuous function in $j \in \mathbb{Z}/(p-1) \times \mathbb{Z}_p \rightarrow \hat{\mathcal{O}}_{\mathcal{Y},y}^+[q_{\mathrm{dR}}^{1/p^b}] = \hat{\mathcal{O}}_{\mathcal{Y},y}^+[q_{\mathrm{dR}}^{1/p^b}-1]$,
where $\hat{\mathcal{O}}_{\mathcal{Y},y}^+[q_{\mathrm{dR}}^{1/p^b}-1]$ has the $p$-adic uniform convergence topology. (Note that it is important that $b_n \in \hat{\mathcal{O}}_{\mathcal{Y},y}^+$ for all $0 \le n \le d$ for this last assertion.) 

In particular, writing 
$$\left(\frac{q_{\mathrm{dR}}^{1/p^b}d}{dq_{\mathrm{dR}}^{1/p^b}}\right)^jg(q_{\mathrm{dR}}^{1/p^b}) = \sum_{n = 0,p\nmid n}^d c_{n,j}(q_{\mathrm{dR}}^{1/p^b}-1)^n,$$
by the above discussion we see that each coefficient $c_{n,j}$ is a continuous function in $j \in \mathbb{Z}/(p-1) \times \mathbb{Z}_p \rightarrow \hat{\mathcal{O}}_{\mathcal{Y},y}^+$.

Moreover, note that for a more general polynomial 
$$h(q_{\mathrm{dR}}^{1/p^b}) = \sum_{n = 0}^d b_n\left(q_{\mathrm{dR}}^{1/p^b}\right)^n \in \hat{\mathcal{O}}_{\mathcal{Y},y}^+[q_{\mathrm{dR}}^{1/p^b}] = \hat{\mathcal{O}}_{\mathcal{Y},y}^+[q_{\mathrm{dR}}^{1/p^b}-1],$$
we have
$$\left(\frac{q_{\mathrm{dR}}^{1/p^b}d}{dq_{\mathrm{dR}}^{1/p^b}}\right)^jh(q_{\mathrm{dR}}^{1/p^b}) = \sum_{n = 0}^dn^jb_n\left(q_{\mathrm{dR}}^{1/p^b}\right)^n.
$$
Letting $j = j_0 + (p-1)p^m$ for any $j_0 \in \mathbb{Z}$, we thus clearly see that 
\begin{equation}\label{polynomiallimit}\lim_{m \rightarrow \infty}\left(\frac{q_{\mathrm{dR}}^{1/p^b}d}{dq_{\mathrm{dR}}^{1/p^b}}\right)^{j_0 + p^m(p-1)} = \sum_{n = 0,p\nmid n}^dn^{j_0}b_n\left(q_{\mathrm{dR}}^{1/p^b}\right)^n.
\end{equation}

We now need the following lemma.

\begin{lemma}\label{holomorphicpart}Let $p^r$ be as in (\ref{integralcorrection3}) for $(y_{\mathrm{dR}}^kf)_y^{\flat}(q_{\mathrm{dR}}^{1/p^b})$. Each coefficient of 
\begin{equation}\left(\frac{q_{\mathrm{dR}}^{1/p^b}d}{dq_{\mathrm{dR}}^{1/p^b}}\right)^j(p^r(y_{\mathrm{dR}}^kf)^{\flat}_y(q_{\mathrm{dR}}^{1/p^b}))(q_{\mathrm{dR}}^{1/p^b}) =: \sum_{n = 0}^{\infty} a_{0,n,j}(q_{\mathrm{dR}}^{1/p^b}-1)^n,
\end{equation}
where $a_{0,n,j} \in \hat{\mathcal{O}}_{\mathcal{Y},y}^+$, is a continuous function in $j \in \mathbb{Z}/(p-1) \times \mathbb{Z}_p \rightarrow \hat{\mathcal{O}}_{\mathcal{Y},y}^+$. 

Moreover, for any $j_0 \in \mathbb{Z}$, we have
\begin{equation}\label{powerserieslimit}\lim_{m \rightarrow \infty}\left(\frac{q_{\mathrm{dR}}^{1/p^b}d}{dq_{\mathrm{dR}}^{1/p^b}}\right)^{j_0 + p^m(p-1)}((y_{\mathrm{dR}}^kf)_y(q_{\mathrm{dR}}^{1/p^b})) = \left(\frac{q_{\mathrm{dR}}^{1/p^b}d}{dq_{\mathrm{dR}}^{1/p^b}}\right)^{j_0}((y_{\mathrm{dR}}^kf)_y^{\flat}(q_{\mathrm{dR}}^{1/p^b})).
\end{equation}
\end{lemma}
\begin{proof}[Proof of Lemma \ref{holomorphicpart}] Note that since $a_{0,n,j}$ can be expressed recursively in terms of $a_{0,n,j-1}$ and $a_{0,n+1,j-1}$, then $a_{0,n,j}$ can be expressed entirely in terms of the coefficients $a_{0,n,0}, a_{0,n+1,0}, \ldots a_{0,n+j,0}$. Hence for any $N \ge 0$, there exists a polynomial truncation 
$$(p^r(y_{\mathrm{dR}}^kf)^{\flat})_{y}^{(N)}(q_{\mathrm{dR}}^{1/p^b}) = \sum_{n = 0}^N a_n(q_{\mathrm{dR}}^{1/p^b}-1)^n \in \hat{\mathcal{O}}_{\mathcal{Y},y}^+[q_{\mathrm{dR}}^{1/p^b}-1]$$
of $(p^r(y_{\mathrm{dR}}^kf)^{\flat}_y)(q_{\mathrm{dR}}^{1/p^b})$ such that letting 
\begin{equation}\label{derivativesanalytic}\left(\frac{q_{\mathrm{dR}}^{1/p^b}d}{dq_{\mathrm{dR}}^{1/p^b}}\right)^j (p^r(y_{\mathrm{dR}}^kf)^{\flat})_{y}^{(N)}(q_{\mathrm{dR}}^{1/p^b}) = \sum_{n = 0}^Na_{0,n,j}^{(N)}(q_{\mathrm{dR}}^{1/p^b}-1)^n,
\end{equation}
we have
$$a_{0,n,j} = a_{0,n,j}^{(N)}$$
for all $n + j \le N$. By the previous paragraph, we know that each $a_{0,n,j}^{(N)}$ is a continuous function $j \in \mathbb{Z}/(p-1)\times \mathbb{Z}_p \rightarrow \hat{\mathcal{O}}_{\mathcal{Y},y}^+$.
In particular, all the $a_{0,n,j}^{(N)}$'s patch together and show that $a_{0,n,j}$ is continuous function in $j \in \mathbb{Z}/(p-1)\times\mathbb{Z}_p \rightarrow \hat{\mathcal{O}}_{\mathcal{Y},y}^+$.
Moreover, by the uniform convergence of the coefficients from the previous paragraph, we have that (\ref{derivativesanalytic}) is a continuous function in $j$, giving $\hat{\mathcal{O}}_{\mathcal{Y},y}\llbracket q_{\mathrm{dR}}^{1/p^b}-1\rrbracket$ the uniform convergence topology.

Now write for any $j \in \mathbb{Z}_{\ge 0}$
$$\left(\frac{q_{\mathrm{dR}}^{1/p^b}d}{dq_{\mathrm{dR}}^{1/p^b}}\right)^j(y_{\mathrm{dR}}^kf)_y(q_{\mathrm{dR}}) = \sum_{n = 0}^{\infty}b_{0,n,j}(q_{\mathrm{dR}}^{1/p^b}-1)^n.$$
Then (\ref{powerserieslimit}) follows by the same argument as above, expressing $b_{0,n,j}$ in terms of the coefficients $b_{0,n,0}, b_{0,n+1,0}, \ldots b_{0,n+j,0}$, then considering truncations $b_{0,n,j}^{(N)}$ and using (\ref{polynomiallimit}). 

\end{proof}

Fix any $j_0 \in \mathbb{Z}_{\ge 0}$. By (\ref{analyticexpressiondelta}), we have 
\begin{equation}\label{deltaanalytic}\begin{split}&\left(\frac{q_{\mathrm{dR}}^{1/p^b}d}{dq_{\mathrm{dR}}^{1/p^b}}\right)^j((p^b\delta_k)^{j_0}(p^r(y_{\mathrm{dR}}^kf)^{\flat}_y)(q_{\mathrm{dR}}^{1/p^b})) \\
&= \left(\frac{q_{\mathrm{dR}}^{1/p^b}d}{dq_{\mathrm{dR}}^{1/p^b}}\right)^j\left(\sum_{i=0}^{\infty}c_i(j_0)\frac{(p^b)^i}{(z_{\mathrm{dR}}-\bar{z})_y^i}\left(\frac{q_{\mathrm{dR}}^{1/p^b}d}{dq_{\mathrm{dR}}^{1/p^b}}\right)^{j_0-i}(p^r(y_{\mathrm{dR}}^kf)_y^{\flat}(q_{\mathrm{dR}}^{1/p^b}))\right).
\end{split}
\end{equation}
Note that for all large $i \gg 0$, we have
\begin{equation}\label{nearlyholomorphicterm}\begin{split}c_i(j)\frac{q_{\mathrm{dR}}^{1/p^b}d}{dq_{\mathrm{dR}}^{1/p^b}}\frac{(p^b)^i}{(z_{\mathrm{dR}}-\bar{z})_y^i} = c_i(j)\frac{d}{dz_{\mathrm{dR}}}\frac{(p^b)^{i+1}}{(z_{\mathrm{dR}}-\bar{z})_y^i} &= -c_i(j)\frac{i(p^b)^{i+1}}{(z_{\mathrm{dR}}-\bar{z})_y^{i+1}} \\
&\overset{\theta}{\mapsto} -i\cdot \left(\frac{\theta(y_{\mathrm{dR}})p^b}{\mathbf{z}}\right)_y^{i+1} \in \hat{\mathcal{O}}_{\mathcal{Y},y}^+
\end{split}
\end{equation}
where the last inclusion follows since
\begin{equation}\label{nearlyholomorphictermintegral}c_i(j)\left(-\frac{\theta(y_{\mathrm{dR}})p^b}{\mathbf{z}}\right)_y\in \hat{\mathcal{O}}_{\mathcal{Y},y}^+
\end{equation}
for all large $i \gg 0$; this follows from our assumption (\ref{integralfiberassumption}), the fact that $i!|c_i(j)$ for all $j \in \mathbb{Z}_{\ge 0}$, $|i!|$ becomes arbitrarily close to $p^{1/p-1}$ as $i \rightarrow \infty$, and (\ref{integralstalkcriterion}). In particular, one can see from Lemma \ref{holomorphicpart}, (\ref{deltaanalytic}) and (\ref{nearlyholomorphicterm}) that
\begin{equation}\label{alsoanalytic}\theta \circ \left(\frac{q_{\mathrm{dR}}^{1/p^b}d}{dq_{\mathrm{dR}}^{1/p^b}}\right)^j(((p^b\delta_k)^{j_0}p^r(y_{\mathrm{dR}}^kf)^{\flat}_y)(q_{\mathrm{dR}}^{1/p^b})) = a_{j_0,0,j}
\end{equation}
is a continuous function in $j \in \mathbb{Z}/(p-1) \times \mathbb{Z}_p \rightarrow \hat{\mathcal{O}}_{\mathcal{Y},y}^+$.
Now we can rewrite (\ref{analyticexpression}) as
\begin{equation}\label{analyticexpression2}\theta_{k+2j_0}^j((p^b\delta_k)^{j_0}(p^r(y_{\mathrm{dR}}^kf)^{\flat}_y)(q_{\mathrm{dR}}^{1/p^b})) = \sum_{i = 0}^{\infty}c_i(j)\left(-\frac{\theta(y_{\mathrm{dR}})p^b}{\mathbf{z}}\right)_y^ia_{j_0,0,j-i}\end{equation}
where we see by the previous paragraph that each $a_{j_0,0,j-i}$ is a continuous function in $j \in \mathbb{Z}/(p-1) \times \mathbb{Z}_p \rightarrow \hat{\mathcal{O}}_{\mathcal{Y},y}$.

Now we show that the extension of $j \mapsto (p^b\theta_k)^j(p^r(y_{\mathrm{dR}}^kf)^{\flat}_y(q_{\mathrm{dR}}^{1/p^b}))$ to $j \in \mathbb{Z}/(p-1) \times \mathbb{Z}_p$ is well-defined (i.e. the limit in (\ref{limitdefinition}) exists). 

For this, given $j = \sum_{n = 0}^{\infty} \alpha_np^n \in \mathbb{Z}/(p-1) \times \mathbb{Z}_p$ where $0 \le \alpha_n \le p-1$, let $j_m = \sum_{n = 0}^m \alpha_np^n$. Given any $\epsilon > 0$, choose $N > 0$ such that $|p^N| < \epsilon$ and such that 
\begin{equation}\label{ainequality}|a_{j_N,n,x} - a_{j_N,n,y}| < \epsilon
\end{equation}
for any $x, y \in \mathbb{Z}/(p-1) \times \mathbb{Z}_p$ with $|x-y| < p^N$ and for all $n \in \mathbb{Z}_{\ge 0}$. For each individual $n$, we can do this since $\mathbb{Z}/(p-1) \times \mathbb{Z}_p$ is compact and $\mathbb{Z}/(p-1)\times \mathbb{Z}_p \ni x \mapsto a_{j_N,0,x} \in \hat{\mathcal{O}}_{\mathcal{Y},y}^+$ as in (\ref{alsoanalytic}) are analytic and hence continuous, and hence uniformly continuous. The fact that such $N$ exists for all $n \in \mathbb{Z}_{\ge 0}$ simultaneously follows from the statement that (\ref{derivativesanalytic}) is a continuous function in $j$ giving $\hat{\mathcal{O}}_{\mathcal{Y},y}^+\llbracket q_{\mathrm{dR}}^{1/p^b}-1\rrbracket$ the uniform convergence topology, which we proved above. 

Note that for all $m > N$, we have for any $|x| \le |p^N|$,
\begin{equation}\label{cijinequality}|c_i(x)| \le |p^N|, \hspace{1cm} |c_i(1-k+x)| \le |p^N|
\end{equation}
if $i > 0$, by Definition (\ref{cijdefinition}). Then we have for all $m, m' > N$, since 
\begin{equation}\label{Cauchyargument}\begin{split}&|(p^b\theta_k)^{j_m}(p^r(y_{\mathrm{dR}}^kf)^{\flat}_y(q_{\mathrm{dR}}^{1/p^b})) - (p^b\theta_k)^{j_{m'}}(p^r(y_{\mathrm{dR}}^kf)^{\flat}_y(q_{\mathrm{dR}}^{1/p^b}))| \\
&= |\theta_{k+2j_N}^{j_m - j_N}\circ (p^b\delta_k)^{j_N}(p^r(y_{\mathrm{dR}}^kf)^{\flat}_y(q_{\mathrm{dR}}^{1/p^b})) - \theta_{k+2j_N}^{j_{m'} - j_N}\circ (p^b\delta_k)^{j_N}(p^r(y_{\mathrm{dR}}^kf)^{\flat}_y(q_{\mathrm{dR}}^{1/p^b}))|\\
&= \left|\sum_{i = 0}^{\infty}c_i(j_{m'}-j_N)\left(-\frac{\theta(y_{\mathrm{dR}})}{\mathbf{z}}\right)_y^i\theta\circ \left(\frac{q_{\mathrm{dR}}^{1/p^b}d}{dq_{\mathrm{dR}}^{1/p^b}}\right)^{j_{m'}-j_N-i}((p^b\delta_k)^{j_N}(p^r(y_{\mathrm{dR}}^kf)^{\flat})(q_{\mathrm{dR}}^{1/p^b})) \right.\\
&\left.- \sum_{i = 0}^{\infty}c_i(j_m-j_N)\left(-\frac{\theta(y_{\mathrm{dR}})}{\mathbf{z}}\right)_y^i\theta\circ \left(\frac{q_{\mathrm{dR}}^{1/p^b}d}{dq_{\mathrm{dR}}^{1/p^b}}\right)^{j_m-j_N-i}((p^b\delta_k)^{j_N}(p^r(y_{\mathrm{dR}}^kf)^{\flat}_y)(q_{\mathrm{dR}}^{1/p^b}))\right|\\
&\overset{|j_m - j_N| \le |p|^N, (\ref{cijinequality})}{\le} \max\left(|p^N|,\left|\theta\circ \left(\frac{q_{\mathrm{dR}}^{1/p^b}d}{dq_{\mathrm{dR}}^{1/p^b}}\right)^{j_{m'}-j_N}((p^b\delta_k)^{j_N}(p^r(y_{\mathrm{dR}}^kf)^{\flat}_y)(q_{\mathrm{dR}}^{1/p^b})) \right.\right.\\
&\hspace{4cm}\left.\left.- \theta\circ \left(\frac{q_{\mathrm{dR}}^{1/p^b}d}{dq_{\mathrm{dR}}^{1/p^b}}\right)^{j_m-j_N}((p^b\delta_k)^{j_N}(p^r(y_{\mathrm{dR}}^kf)^{\flat}_y)(q_{\mathrm{dR}}^{1/p^b}))\right|\right)\\
&\overset{(\ref{ainequality})}{=} \max\left(|p^N|,\left|a_{j_N,0,j_{m'}}- a_{j_N,0,j_m}\right|\right) < \epsilon.
\end{split}
\end{equation}
So we have shown that 
$$\left\{(p^b\theta_k)^{j_m}(p^r(y_{\mathrm{dR}}^kf)^{\flat}_y(q_{\mathrm{dR}}^{1/p^b}))\right\}_{m \in \mathbb{Z}_{\ge 0}}$$
is a Cauchy sequence in the $p$-adic topology. Now the statement that (\ref{limitdefinition}) exists follows from the fact that $\hat{\mathcal{O}}_{\mathcal{Y},y}$ is $p$-adically complete. 

We now show that the function defined by (\ref{limitdefinition}) is continuous in $j \in \mathbb{Z}/(p-1) \times \mathbb{Z}_p \rightarrow \hat{\mathcal{O}}_{\mathcal{Y},y}$.

Consider the sequence of truncations
\begin{equation}\label{analyticexpression2}(p^b\theta_k)^j(p^r(y_{\mathrm{dR}}^kf)^{\flat}_y(q_{\mathrm{dR}}^{1/p^b}))^{(N)} = \sum_{i = 0}^{N}c_i(j)\left(-\frac{\theta(y_{\mathrm{dR}})p^b}{\mathbf{z}}\right)_y^ia_{0,0,j-i}\end{equation}
each of which is a continuous function in $j \in \mathbb{Z}/(p-1)\times \mathbb{Z}_p \rightarrow \hat{\mathcal{O}}_{\mathcal{Y},y}^+$, since the $a_{0,0,j-i}$ as in (\ref{alsoanalytic}) are continuous.
 
We claim that the sequence $(p^b\theta_k)^j(p^r(y_{\mathrm{dR}}^kf)^{\flat}_y(q_{\mathrm{dR}}^{1/p^b}))^{(N)}$ for $N \in \mathbb{Z}_{\ge 0}$ converges uniformly to $(p^b\delta_k)^j(p^r(y_{\mathrm{dR}}^kf)^{\flat}_y(q_{\mathrm{dR}}^{1/p^b}))$ as a function in $j \in \mathbb{Z}/(p-1)\times \mathbb{Z}_p \rightarrow \hat{\mathcal{O}}_{\mathcal{Y},y}$, and hence $(p^b\delta_k)^j(p^r(y_{\mathrm{dR}}^kf)^{\flat}_y(q_{\mathrm{dR}}^{1/p^b}))$ is a continuous function in $j \in \mathbb{Z}/(p-1)\times \mathbb{Z}_p \rightarrow \hat{\mathcal{O}}_{\mathcal{Y},y}$ by the standard argument from analysis.

Let us now show the uniform convergence. Given $j \in \mathbb{Z}/(p-1) \times \mathbb{Z}_p$, suppose we are given any $\epsilon > 0$. Choose $t \in \mathbb{Z}_{\ge 0}$ so that $|p^t| < \epsilon$. Note that for all $i \ge p^t$
\begin{equation}\label{cijdivisibility}c_i(j) \in p^t\mathbb{Z}_p
\end{equation}
for all $j \in \mathbb{Z}/(p-1)\times \mathbb{Z}_p$ (since one sees from the definition of $c_i(j)$ that for every residue class $\mathbb{Z}_p/(p^t)$, $c_i(j)$ is divisible by some representative of that residue class).
Then for all $N \ge p^t$, we have
\begin{align*}&|(p^b\theta_k)^j(p^r(y_{\mathrm{dR}}^kf)^{\flat}_y(q_{\mathrm{dR}}^{1/p^b})) - (p^b\theta_k)^j(p^r(y_{\mathrm{dR}}^kf)^{\flat}_y(q_{\mathrm{dR}}^{1/p^b}))^{(N)}| \\
&= \left|\sum_{n = N}^{\infty}c_i(j)\left(-\frac{\theta(y_{\mathrm{dR}})p^b}{\mathbf{z}}\right)_y^ia_{0,0,j-i}\right|\\
&\le \max_{N \le i < \infty}\left(c_i(j)\left(-\frac{\theta(y_{\mathrm{dR}})p^b}{\mathbf{z}}\right)_y^ia_{0,0,j-i}\right)\overset{(\ref{nearlyholomorphictermintegral}), (\ref{alsoanalytic}), (\ref{cijdivisibility}), }{\le} |p^t| < \epsilon.
\end{align*}
This gives the desired uniform convergence.

Now (\ref{limitisstabilization}) follows from (\ref{powerserieslimit}) in Lemma \ref{holomorphicpart} and the calculation (\ref{nearlyholomorphicterm}). 

\end{proof}

\section{The $p$-adic $L$-function and $p$-adic Waldspurger formula}\label{padicLfunctionsection}
In this section, let $K$ be an imaginary quadratic field. We construct a one-variable $p$-adic $L$-function for Rankin-Selberg families $(w,\chi^{-1})$, where $f$ is an eigenform of weight $k$ for $\Gamma_1(N)$ and $\chi$ runs through central anticyclotomic characters of $K$. This will be done by applying Theorem (\ref{keycorollary}) in the case where $y$ is a supersingular CM point. We hence construct a $p$-adic $L$-function, which is a continuous function on a space of $p$-adic Hecke characters (the closure of $p$-adic central critical characters for $f$ inside the space of functions $\mathbb{A}_K^{\times,(p\infty)} \rightarrow \mathcal{O}_{\mathbb{C}_p}$, where $\mathbb{A}_K^{\times,(p\infty)}$ is the group of id\`{e}les prime to $p\infty$, with the uniform convergence topology). We use Theorem \ref{MScomparison} to show that our $p$-adic $L$-function satisfies an ``approximate interpolation property", namely that special values of our $p$-adic $L$-function are limits of algebraic normalizations of central $L$-values $L(w,\chi^{-1},0) := L((\pi_w)_K \times \chi^{-1},1/2)$ as $\chi$ varies through a sequence of central critical characters in the interpolation range. 

We conclude by showing that in the case $k = 2$, a special value of our $p$-adic $L$-function is equal to the evaluation at a certain Heegner point of the formal logarithm of a certain differential in the same Hecke isotypic component away from $p$ as that of $f$. We show that if this special value is nonzero, then the Heegner point descends to a point of infinite order in $A_f(K)$ for a $GL_2$-type abelian variety $A_f$ attached to $f$. In particular, combined with the approximate interpolation property described in the previous paragraph, this gives a new criterion to produce a point of infinite order in $A_f(K)$, namely showing that a sequence of central critical $L$-values has a non-zero limit. We make a few remarks on how this criterion can be combined with a method (and some results) of Rubin in the case when $f$ is congruent modulo $p$ with an Eisenstein series in order give new congruence criteria for producing points of infinite order in $A_f(K)$.

\subsection{Periods of supersingular CM points}\label{CMsatisfyassumptions}Now fix an imaginary quadratic field $K = \mathbb{Q}(\sqrt{-d})$, where $d$ is squarefree, and with fundamental discriminant $d_K$. Let $\mathcal{O} \subset \mathcal{O}_K$ be a fixed suborder. Let $H$ denote the ring class field over $K$ associated with $\mathcal{O}$.

Suppose that $A/H$ is an elliptic curve with CM by $\mathcal{O}$. For any field extension $F/H$, there is a canonical algebraic splitting 
\begin{equation}\label{algsplitting}H_{\mathrm{dR}}^1(A/F) \cong H^{1,0}(A/F) \oplus H^{0,1}(A/F) = \Omega_{A/F}^1 \oplus (\Omega_{A/F}^1)^*
\end{equation}
of the Hodge-de Rham sequence
$$0 \rightarrow H^{1,0}(A/F) = \Omega_{A/F}^1 \rightarrow H_{\mathrm{dR}}^1(A/F) \rightarrow H^{0,1}(A/F) = (\Omega_{A/F}^1)^* \rightarrow 0$$
given by the action of $\mathcal{O}$: for $\gamma \in \mathcal{O} = \End(A/F)$, $\gamma^*$ acts on $H_{\mathrm{dR}}^1(A/F)$ as multiplication by $\gamma$ on $H^{1,0}(A/F)$ and as multiplication by $\overline{\gamma}$ on $H^{0,1}(A/F)$. In fact,  for any integral model $A/\mathcal{O}_F$ the $\mathcal{O}$-action also induces an algebraic splitting 
\begin{equation}\label{integralalgsplitting}H_{\mathrm{dR}}^1(A/F) \cong H^{1,0}(A/\mathcal{O}_F) \oplus H^{0,1}(A/\mathcal{O}_F) = \Omega_{A/\mathcal{O}_F}^1 \oplus (\Omega_{A/\mathcal{O}_F}^1)^*
\end{equation}
of the Hodge filtration:
\begin{equation}\label{integralCMsplitting}0 \rightarrow H^{1,0}(A/\mathcal{O}_F) = \Omega_{A/\mathcal{O}_F}^1 \rightarrow H_{\mathrm{dR}}^1(A/\mathcal{O}_F) \rightarrow H^{0,1}(A/\mathcal{O}_F) = (\Omega_{A/\mathcal{O}_F}^1)^* \rightarrow 0.
\end{equation}
By virtue of it splitting the Hodge filtration, this splitting is also compatible with Poincar\'{e} duality, that is, under the specialization of the Poincar\'{e} pairing (i.e. the de Rham cup product)
$$H_{\mathrm{dR}}^1(A/\mathcal{O}_F) \times H_{\mathrm{dR}}^1(A/\mathcal{O}_F) \rightarrow \mathcal{O}_F$$
we have that $H^{1,0}(A/\mathcal{O}_F) = \Omega_{A/\mathcal{O}_F}^1$ and $H^{0,1}(A/\mathcal{O}_F)$ are dual.

Moreover, recall the integral Hodge-Tate complex as recalled in Theorem \ref{integraltheorem}
\begin{equation}\label{definedoverK}\begin{split}0 \rightarrow \mathrm{Lie}(A/\mathcal{O}_{\mathbb{C}_p})(1) \overset{(HT_A)^{\vee}(1)}{\rightarrow} T_pA\otimes_{\mathbb{Z}_p} \mathcal{O}_{\mathbb{C}_p} \xrightarrow{HT_A} &H^{1,0}(A/\mathcal{O}_F)\otimes_{\mathcal{O}_F}\mathcal{O}_{\mathbb{C}_p} \\
&= \Omega_{A/\mathcal{O}_{\mathbb{C}_p}}^1 \rightarrow 0.
\end{split}
\end{equation}

Consider $(A,\alpha) \in \mathcal{Y}(\mathbb{C}_p,\mathcal{O}_{\mathbb{C}_p})$ for any $p^{\infty}$-level structure $\alpha : \mathbb{Z}_p^{\oplus 2} \xrightarrow{\sim} T_pA$. Specializing Proposition \ref{H10factors} to to $(A,\alpha)$, we have an injective map
\begin{equation}\label{inclusionrestrict}\Omega_{A/F}^1 \subset \Omega_{A/F}^1 \otimes_{F} \mathcal{O}_{\Delta,\mathcal{Y}}(A,\alpha)  \overset{\iota_{\mathrm{dR}}(A,\alpha)}{\subset} T_pA\otimes_{\mathbb{Z}_p}\mathcal{O}_{\Delta,\mathcal{Y}}(A,\alpha) \overset{\theta}{\twoheadrightarrow} T_pA\otimes_{\mathbb{Z}_p} \mathbb{C}_p 
\end{equation}
whose composition with the natural map $T_pA \otimes_{\mathbb{Z}_p}\mathbb{C}_p \xrightarrow{HT_A} \Omega_{A/\mathbb{C}_p}^1$ is the natural inclusion $\Omega_{A/F}^1 \subset \Omega_{A/\mathbb{C}_p}^1$. The next proposition gives more information on (\ref{inclusionrestrict}) in the case where $A$ has complex multiplication.

\begin{proposition}\label{integralsurjective}For an elliptic curve $A$ with CM by some $\mathcal{O} \subset \mathcal{O}_K$, the map (\ref{inclusionrestrict}) factors through an injective map
$$\Omega_{A/\mathcal{O}_{F}}^1 \hookrightarrow T_pA\otimes_{\mathbb{Z}_p}\mathcal{O}_{\mathbb{C}_p}.$$
As a result, the map 
$$HT_A : T_pA\otimes_{\mathbb{Z}_p} \mathcal{O}_{\mathbb{C}_p} \rightarrow \Omega_{A/\mathcal{O}_{\mathbb{C}_p}}^1$$
from (\ref{definedoverK}) is surjective. Furthermore, if $|\mathbf{z}(A,\alpha)| \ge 1$, then 
$$\frak{s}(A,\alpha) \in \Omega_{A/\mathcal{O}_{\mathbb{C}_p}}^1$$
is a generator.
\end{proposition}

\begin{proof}Recall our map 
$$\theta : \mathcal{O}_{\Delta,\mathcal{Y}}(A,\alpha) \overset{\theta}{\twoheadrightarrow} \hat{\mathcal{O}}_{\mathcal{Y}}(A,\alpha) = \mathbb{C}_p.$$
Let 
$$\mathcal{O}_{\Delta,\mathcal{Y}}(A,\alpha)^+ := \theta^{-1}(\mathcal{O}_{\mathbb{C}_p}).$$
Then the first part of the statement follows if we can show that 
$$\iota_{\mathrm{dR}}(A,\alpha)(\Omega_{A/\mathcal{O}_F}^1) \subset T_pA\otimes_{\mathbb{Z}_p}\mathcal{O}_{\Delta,\mathcal{Y}}(A,\alpha)^+.$$

From (\ref{integralalgsplitting}), we have a \emph{subspace} $H^{0,1}(A/\mathcal{O}_F) \subset H_{\mathrm{dR}}^1(A/\mathcal{O}_F)$ with orthogonal complement $\Omega_{A/\mathcal{O}_F}^1$ under the de Rham cup product on $H_{\mathrm{dR}}^1(A/\mathcal{O}_F)$. 
For brevity, for the rest of the proof of the Proposition, denote the specialization of $t(A,\alpha)$ of $t \in \mathbb{B}_{\mathrm{dR}}^+(\mathcal{Y})$ simply by $t = t(A,\alpha)$. Recall that by Proposition \ref{Weilextension}, we have that the de Rham cup product composed with $\iota_{\mathrm{dR}}(A,\alpha)$ coincides with $t^{-1}$ times the Weil pairing $\langle \cdot, \cdot \rangle \cdot t^{-1}$. Hence 
$$H^{0,1}(A/\mathcal{O}_F)(1) \otimes_{\mathcal{O}_F} \mathcal{O}_{\Delta,\mathcal{Y}}(A,\alpha),$$ being the $t$ times the Poincar\'{e} dual of $\Omega_{A/\mathcal{O}_F}^1 \otimes_{\mathcal{O}_F} \mathcal{O}_{\Delta,\mathcal{Y}}(A,\alpha)$, viewed naturally via $\iota_{\mathrm{dR}}(A,\alpha)$ as a \emph{subspace} of $T_pA\otimes_{\mathbb{Z}_p}\mathcal{O}_{\Delta,\mathcal{Y}}(A,\alpha)$, is identified with the dual of 
$$\iota_{\mathrm{dR}}(A,\alpha)(\Omega_{A/\mathcal{O}_F}^1 \otimes_{\mathcal{O}_F} \mathcal{O}_{\Delta,\mathcal{Y}}(A,\alpha))$$
under the twisted Weil pairing $\langle \cdot, \cdot \rangle \cdot t^{-1}$. By the $p$-adic Legendre relation (\ref{periodsrelation}) and Proposition \ref{HTarrows}, this dual subspace is the Hodge-Tate filtration 
$$\mathcal{H}^{0,1}(\mathcal{A}) \otimes_{\mathcal{O}_Y} \mathcal{O}_{\Delta,\mathcal{Y}}(A,\alpha) \overset{(HT_A)^{\vee}(1)}{\hookrightarrow} T_pA \otimes_{\mathbb{Z}_p} \mathcal{O}_{\Delta,\mathcal{Y}}(A,\alpha).$$

In summary, we have a natural identification
\begin{equation}\label{naturalidentification}H^{0,1}(A/\mathcal{O}_F)(1)\otimes_{\mathcal{O}_F} \mathcal{O}_{\Delta,\mathcal{Y}}(A,\alpha)  = \mathcal{H}^{0,1}(\mathcal{A}) \otimes_{\mathcal{O}_Y} \mathcal{O}_{\Delta,\mathcal{Y}}(A,\alpha) \overset{(HT_A)^{\vee}(1)}{\hookrightarrow} T_pA \otimes_{\mathbb{Z}_p} \mathcal{O}_{\Delta,\mathcal{Y}}(A,\alpha)
\end{equation}
which identifies this subspace with the dual of $\iota_{\mathrm{dR}}(A,\alpha)(\Omega_{A/\mathcal{O}_F}^1 \otimes_{\mathcal{O}_F} \mathcal{O}_{\Delta,\mathcal{Y}}(A,\alpha))$ under the twisted Weil pairing $\langle \cdot, \cdot \rangle\cdot t^{-1}$. Tensoring with $\otimes_{\mathcal{O}_{\Delta,\mathcal{Y}}(A,\alpha)}\left(\mathcal{O}_{\Delta,\mathcal{Y}}(A,\alpha)/(\ker\theta)\right) = \otimes_{\mathcal{O}_{\Delta,\mathcal{Y}}(A,\alpha)}\mathbb{C}_p$, we get the natural inclusion
\begin{equation}\label{H01inclusion}H^{0,1}(A/\mathcal{O}_F)(1) \hookrightarrow H^{0,1}(A/\mathcal{O}_F)(1) \otimes_{\mathcal{O}_F} \mathbb{C}_p = \mathrm{Lie}(A/\mathbb{C}_p)(1) \overset{(HT_A)^{\vee}(1)}{\hookrightarrow} T_pA \otimes_{\mathbb{Z}_p} \mathbb{C}_p.
\end{equation}

On the other hand, we have from the second arrow in (\ref{definedoverK}) that
$$\mathrm{Lie}(A/\mathcal{O}_{\mathbb{C}_p})(1) \overset{(HT_A)^{\vee}(1)}{\hookrightarrow} T_pA\otimes_{\mathbb{Z}_p} \mathcal{O}_{\mathbb{C}_p}$$
and so (\ref{H01inclusion}) factors as
$$H^{0,1}(A/\mathcal{O}_F)(1) \subset \mathrm{Lie}(A/\mathcal{O}_{\mathbb{C}_p})(1) \overset{(HT_A)^{\vee}(1)}{\hookrightarrow} T_pA\otimes_{\mathbb{Z}_p} \mathcal{O}_{\mathbb{C}_p}.$$
Hence, in turn we see that (\ref{naturalidentification}) factors through
\begin{equation}\label{naturalidentification2}H^{0,1}(A/\mathcal{O}_F)(1) \overset{(HT_A)^{\vee}(1)}{\hookrightarrow} T_pA \otimes_{\mathbb{Z}_p} \mathcal{O}_{\Delta,\mathcal{Y}}(A,\alpha)^+.
\end{equation}

By the above discussion, the Weil pairing gives a duality
$$\langle \cdot, \cdot \rangle \cdot t^{-1} : (T_pA \otimes_{\mathbb{Z}_p} \mathcal{O}_{\Delta,\mathcal{Y}}(A,\alpha)^+)(-1) \times (T_pA \otimes_{\mathbb{Z}_p} \mathcal{O}_{\Delta,\mathcal{Y}}(A,\alpha)^+) \rightarrow \mathcal{O}_{\Delta,\mathcal{Y}}(A,\alpha)^+$$
compatible under $\iota_{\mathrm{dR}}(A,\alpha)$ with Poincar\'{e} duality. Hence we see from (\ref{naturalidentification2}) that the Poincar\'{e} dual $\Omega_{A/\mathcal{O}_F}^1$ of the subspace
$$H^{0,1}(A/\mathcal{O}_F) \overset{(HT_A)^{\vee}(1)}{\hookrightarrow} T_pA \otimes_{\mathbb{Z}_p} \mathcal{O}_{\Delta,\mathcal{Y}}(A,\alpha)^+$$
is mapped into 
$$T_pA\otimes_{\mathbb{Z}_p} \mathcal{O}_{\Delta,\mathcal{Y}}(A,\alpha)^+$$
by $\iota_{\mathrm{dR}}(A,\alpha);$
in other words,
$$\iota_{\mathrm{dR}}(A,\alpha)(\Omega_{A/\mathcal{O}_F}^1) \subset T_pA\otimes_{\mathbb{Z}_p}\mathcal{O}_{\Delta,\mathcal{Y}}(A,\alpha)^+$$
which is what we wanted to show.

For the second part of the Proposition, recall that 
$$HT_A \circ (\ref{inclusionrestrict}) = \mathrm{id},$$
and so from the first part of the Proposition, we have that
$$\Omega_{A/\mathcal{O}_F}^1 \hookrightarrow T_pA\otimes_{\mathbb{Z}_p}\mathcal{O}_{\mathbb{C}_p} \xrightarrow{HT_A} \Omega_{A/\mathcal{O}_{\mathbb{C}_p}}^1$$
is the natural inclusion. Extending by $\otimes_{\mathcal{O}_F}\mathcal{O}_{\mathbb{C}_p}$-linearity, we see that 
$$\Omega_{A/\mathcal{O}_{\mathbb{C}_p}}^1 \hookrightarrow T_pA\otimes_{\mathbb{Z}_p}\mathcal{O}_{\mathbb{C}_p} \xrightarrow{HT_A} \Omega_{A/\mathcal{O}_{\mathbb{C}_p}}^1$$
is the identity. In particular, the second arrow is surjective. 

For the final claim, suppose that $|\mathbf{z}(A,\alpha)| \ge 1$. Recall that $\frak{s}(A,\alpha) = HT_A(\alpha_2) \in \Omega_{A/\mathcal{O}_{\mathbb{C}_p}}^1$, and that 
$$HT_A(\alpha_1) = HT_A(\alpha_2)/\mathbf{z}(A,\alpha) = 1/\mathbf{z}(A,\alpha) \cdot \frak{s}(A,\alpha) \in \Omega_{A/\mathcal{O}_{\mathbb{C}_p}}^1.$$
Since $1/\mathbf{z}(A,\alpha) \in \mathcal{O}_{\mathbb{C}_p}$, we see that $\frak{s}(A,\alpha)$ generates the image of 
$$HT_A : T_pA\otimes_{\mathbb{Z}_p}\mathcal{O}_{\mathbb{C}_p} \rightarrow \Omega_{A/\mathcal{O}_{\mathbb{C}_p}}^1.$$
But by the previous part of the Proposition, $HT_A$ is surjective, and so $\frak{s}(A,\alpha) \in \Omega_{A/\mathcal{O}_{\mathbb{C}_p}}^1$ is a generator.

\end{proof}

\begin{corollary}\label{periodunit}Let $y = (A,\alpha) \in \mathcal{Y}(\mathbb{C}_p,\mathcal{O}_{\mathbb{C}_p})$ be a CM point (i.e. $A$ is an elliptic curve with CM by some order $\mathcal{O} \subset \mathcal{O}_K$). In the notation of Definition \ref{perioddefinition} and (\ref{integralcorrection2}) for $y$, we have
$$|p^b\theta(\Omega_p(A,\alpha))^2| = 1.$$
\end{corollary}

\begin{proof}Let $\omega_0 \in \Omega_{A/\mathcal{O}_{\mathbb{Q}}}^1$ be an integral generator as in Definition \ref{perioddefinition}. From (\ref{calculation}), we have
\begin{equation}\label{integraldifferentialequation}\frac{\omega_0^{\otimes 2}}{\theta(\Omega_p(A,\alpha))^2} = \omega_{\mathrm{can}}(A,\alpha)^{\otimes 2} = \frac{\frak{s}(A,\alpha)}{y_{\mathrm{dR}}(A,\alpha)^2} \overset{\sigma}{\mapsto} p^b\cdot \frac{1}{p^b}\theta\left(\frac{dq_{\mathrm{dR}}}{dq^{p^a}}\right)(A,\alpha)\otimes dq^{p^a}(A,\alpha).
\end{equation}
By (\ref{integralcorrection}), we have that $dq^{p^q}(A,\alpha) \in \Omega_{A/\mathcal{O}_{\mathbb{C}_p}}^1$ is an integral generator, and by (\ref{integralcorrection2}), we have that 
$$\left|\frac{1}{p^b}\theta\left(\frac{dq_{\mathrm{dR}}}{dq^{p^a}}\right)(A,\alpha)\right| = 1,$$
and so
$$\omega_0' := \frac{1}{p^b}\theta\left(\frac{dq_{\mathrm{dR}}}{dq^{p^a}}\right)(A,\alpha)\cdot dq^{p^a}(A,\alpha) \in \Omega_{A/\mathcal{O}_{\mathbb{C}_p}}^1$$
is also an integral generator. Hence from (\ref{integraldifferentialequation}), we have
$$\omega_0 = p^b\theta(\Omega_p(A,\alpha))^2 \cdot \omega_0'.$$
Since $\omega_0$ and $\omega_0'$ are integral generators, we have that
$$|p^b\theta(\Omega_p(A,\alpha))^2| = 1$$
as desired.
\end{proof}

Using Proposition \ref{integralsurjective}, we now analyze the Hodge-Tate, de Rham and $\Omega_p$ periods of points in $\mathcal{C}(\mathcal{O})$. We first deal with the case where $p$ is ramified in $K$.

\begin{proposition}\label{CMproposition1}Suppose $p$ is ramified in $K$. Then suppose $A$ is an elliptic curve with CM by $\mathcal{O} \subset \mathcal{O}_K$, say $\mathcal{O} = \mathbb{Z} + c\mathcal{O}_K$ with $c \in \mathbb{Z}_{> 0}$. Then there is a $p^{\infty}$-level structure $\alpha : \mathbb{Z}_p^{\oplus 2} \xrightarrow{\sim} T_pA$ such that $\mathbf{z}(A,\alpha), \mathbf{z}_{\mathrm{dR}}(A,\alpha) \in K_p$, and
$$\mathbf{z}(A,\alpha) = -1/(c\sqrt{-d}), \mathbf{z}_{\mathrm{dR}}(A,\alpha) = 1/(c\sqrt{-d}), \theta(y_{\mathrm{dR}})(A,\alpha) = -1/2.$$

Furthermore, if $p > 2$, we have that $b = 0$ as defined in (\ref{integralcorrection2}) for $y = (A,\alpha) \in \mathcal{Y}(\mathbb{C}_p,\mathcal{O}_{\mathbb{C}_p})$, and 
$$|\theta(\Omega_p(A,\alpha))| = 1$$
for $\Omega_p(A,\alpha)$ as defined in Definition \ref{perioddefinition}.
\end{proposition}

\begin{proof}Consider the element $c\sqrt{-d} \in \mathcal{O}_K$, and let $[c\sqrt{-d}] \in \End(A/H)$ denote the corresponding endomorphism. Note that the composition 
$$A \xrightarrow{[c\sqrt{-d}]} A\xrightarrow{[c\sqrt{-d}]} A$$
is simply multiplication by $-c^2d$. Hence the composition of the induced maps on Tate modules
$$T_pA \xrightarrow{[c\sqrt{-d}]} T_pA \xrightarrow{[c\sqrt{-d}]} T_pA$$
is also multiplication by $-c^2d$. We can thus choose a level structure $\alpha : \mathbb{Z}_p^{\oplus 2} \xrightarrow{\sim} T_pA$ such that $(e_1,e_2)$ is a rational canonical basis for $[c\sqrt{-d}]$ and the induced linear map
$$\mathbb{Z}_p^{\oplus 2} \underset{\sim}{\xrightarrow{\alpha}} T_pA \xrightarrow{[c\sqrt{-d}]} T_pA \underset{\sim}{\xrightarrow{\alpha^{-1}}}\mathbb{Z}_p^{\oplus 2}$$
is the matrix
$$[c\sqrt{-d}] = \left(\begin{array}{cc} 0 & 1 \\ -c^2d & 0\\ \end{array}\right) \in GL_2(\mathbb{Q}_p).$$
For this choice of $\alpha$, we have by (\ref{ztransformationprop}) and (\ref{z'transformationprop})
\begin{equation}\label{CMids1}\begin{split}&[c\sqrt{-d}]^*\mathbf{z}(A,\alpha) = -1/(c^2d\cdot \mathbf{z}(A,\alpha)),\hspace{1cm} [c\sqrt{-d}]^*\mathbf{z}_{\mathrm{dR}}(A,\alpha) = -1/(c^2d\cdot \mathbf{z}_{\mathrm{dR}}(A,\alpha)).
\end{split}
\end{equation}
On the other hand, by the discussion at the beginning of the section, we have
\begin{equation}\label{CMids2}\begin{split}&[c\sqrt{-d}]^*\mathbf{z}(A,\alpha) = \mathbf{z}(A,\alpha), \hspace{1cm} [c\sqrt{-d}]^*\mathbf{z}_{\mathrm{dR}}(A,\alpha) = \mathbf{z}_{\mathrm{dR}}(A,\alpha).
\end{split}
\end{equation}
Hence from (\ref{CMids1}) and (\ref{CMids2}), we have
\begin{equation}\label{CMids3}\mathbf{z}(A,\alpha) = \pm 1/(c\sqrt{-d}), \hspace{1cm}  -\mathbf{x}_{\mathrm{dR}}(A,\alpha)/\mathbf{y}_{\mathrm{dR}}(A,\alpha) = \mathbf{z}_{\mathrm{dR}}(A,\alpha) = \pm 1/(c\sqrt{-d}).
\end{equation}
Now from the $p$-adic Legendre relation (\ref{periodsrelation}) and (\ref{x'y'definition}), we have
\begin{equation}\label{recallLegendre}\begin{split}\pm \mathbf{y}_{\mathrm{dR}}(A,\alpha) + \mathbf{y}_{\mathrm{dR}}(A,\alpha) &= \pm \mathbf{x}_{\mathrm{dR}}(A,\alpha)\cdot c\sqrt{-d} + \mathbf{y}_{\mathrm{dR}}(A,\alpha) \\
&= \mathbf{x}_{\mathrm{dR}}(A,\alpha)/\mathbf{z}(A,\alpha) + \mathbf{y}_{\mathrm{dR}}(A,\alpha) = t(A,\alpha)
\end{split}
\end{equation}
and where the precise sign in $\pm$ in (\ref{recallLegendre}) is the same sign for which
$$\mathbf{z}(A,\alpha) = \pm \mathbf{z}_{\mathrm{dR}}(A,\alpha)$$
holds. Since $t(A,\alpha) \neq 0$ (since $t(A,\alpha)$ generates $\Fil^1 B_{\mathrm{dR}}^+$), we thus have
$$\mathbf{z}(A,\alpha) = -\mathbf{z}_{\mathrm{dR}}(A,\alpha) = \pm 1/(c\sqrt{-d})$$
and 
$$\mathbf{x}_{\mathrm{dR}}(A,\alpha)/t(A,\alpha) = \pm 1/(2c\sqrt{-d}), \hspace{1cm} \mathbf{y}_{\mathrm{dR}}(A,\alpha)/t(A,\alpha) = \pm 1/2.$$
Now the first part of the statement follows upon noting that $\theta(\mathbf{y}_{\mathrm{dR}}/t) = \theta(y_{\mathrm{dR}})$. 

Now suppose that $p > 2$. Since by the first part of the Proposition we have $|\mathbf{z}(A,\alpha)| = |1/(c\sqrt{-d})| \ge 1$, then by Propositon (\ref{integralsurjective}) we have that $\frak{s}(A,\alpha) \in \Omega_{A/\mathcal{O}_{\mathbb{C}_p}}$ is a generator. Hence by (\ref{calculation2}), we have
$$|\theta(y_{\mathrm{dR}}(A,\alpha))^2|\cdot \left|\theta\left(\frac{dz_{\mathrm{dR}}}{dz}\right)(A,\alpha)\right| = 1.$$
Since we have shown $|\theta(y_{\mathrm{dR}}(A,\alpha))| = |1/2| = 1$, then this implies that
$$\left|\theta\left(\frac{dz_{\mathrm{dR}}}{dz}\right)(A,\alpha)\right| = 1$$
and so $b = 0$ as in Definition (\ref{integralcorrection2}). Now that $|\theta(\Omega_p(A,\alpha))| = 1$ follows from Definition \ref{perioddefinition}, since $\omega_{\mathrm{can}}(A,\alpha) = \frak{s}(A,\alpha)/y_{\mathrm{dR}}(A,\alpha)^2$, $|y_{\mathrm{dR}}(A,\alpha)| = |1/2| = 1$. 

\end{proof}

To deal with the $p$ inert in $K$ case, we need to recall another result on canonical subgroups. For an abelian variety $A$ over $\mathcal{O}_{\mathbb{C}_p}$, let $\mathrm{Ha}(A)$ denote the Hasse invariant, truncated so that $|\mathrm{Ha}(A)| \in [1/p,1]$.
\begin{proposition}[Theorem 4.2.5 of \cite{Conrad2}, see also p. 2 of \cite{Conrad}]\label{Conradbound} Suppose $A$ is an elliptic curve. Then $A$ admits a canonical subgroup (of order $p^n$) if and only if $|\mathrm{Ha}(A)| > p^{-1/p^{n-2}(p+1)}$. 
\end{proposition}

\begin{proposition}\label{CMproposition2}Suppose $p$ is inert in $K$. Then suppose $A$ is an elliptic curve with CM by $\mathcal{O}_K$. Then for any $p^{\infty}$-level structure $\alpha : \mathbb{Z}_p^{\oplus 2} \xrightarrow{\sim} T_pA$, we have $\mathbf{z}(A,\alpha), \mathbf{z}_{\mathrm{dR}}(A,\alpha) \in K_p$, and
$$|\mathbf{z}(A,\alpha)| = 1, \hspace{1cm} |\mathbf{z}_{\mathrm{dR}}(A,\alpha)| = 1, \hspace{1cm} |\theta(\mathbf{y}_{\mathrm{dR}}(A,\alpha))| = |\theta(y_{\mathrm{dR}}(A,\alpha))| = 1.$$
In fact $\mathbf{z}(A,\alpha), \mathbf{z}_{\mathrm{dR}}(A,\alpha) \in \mathcal{O}_{K_p}^{\times}$. Furthermore we have $b = 0$ as defined in (\ref{integralcorrection2}) for $y = (A,\alpha) \in \mathcal{Y}(\mathbb{C}_p,\mathcal{O}_{\mathbb{C}_p})$, and 
$$|\theta(\Omega_p(A,\alpha))| = 1$$ 
for $\Omega_p(A,\alpha)$ as defined in Definition \ref{perioddefinition}. Moreover, for any two choices $(A,\alpha), (A,\alpha')$ of points of $\mathcal{Y}(\mathbb{C}_p,\mathcal{O}_{\mathbb{C}_p})$ above $(A) \in Y(F,\mathcal{O}_F)$, we have that 
$$\theta(\Omega_p)(A,\alpha)/\theta(\Omega_p)(A,\alpha') \in \mathcal{O}_{K_p}^{\times}.$$
\end{proposition}
\begin{proof}Since $p$ is inert in $K$, by Deuring's theorem $A$ is supersingular and so $|\mathrm{Ha}(A)| < 1$. Since $\mathrm{Ha}(A) \in K_p$, its $p$-adic valuation is an integral power of $p$, and so $|\mathrm{Ha}(A)| = 1/p$. Now by Proposition \ref{Conradbound}, we have that $A$ does \emph{not} have a canonical subgroup. Hence, by Lemma \ref{canonicalsubgrouplemma}, $|z(A,\alpha)| \le p^{p/(p^2-1)}$ for any $p^{\infty}$-level structure $\alpha$ on $A$. 

By the discussion at the beginning of this section, the Hodge-Tate sequence splits over $K_p$. Hence, for any $\alpha$, we have 
$$\alpha_1 + 1/\mathbf{z}(A,\alpha)\alpha_2 = \frak{s}^{-1}(A,\alpha) \in \Omega_{A/\mathcal{O}_{K_p}}^1 \subset T_pA\otimes_{\mathbb{Z}_p}\mathcal{O}_{K_p} \xrightarrow{\alpha^{-1}} \mathcal{O}_{K_p}^{\oplus 2}.$$
Hence $\mathbf{z}(A,\alpha) \in K_p$. In fact, by \cite[Theorem 1.2 (2)]{Howe}, we have
\begin{equation}\label{HTperiodgenerates}K_p = \mathbb{Q}_p(\mathbf{z}(A,\alpha)).
\end{equation}
In particular $|\mathbf{z}(A,\alpha)|$ is an integral power of $p$, and so the previous paragraph implies that $|\mathbf{z}(A,\alpha)| \le 1$. Since this holds for all $(A,\alpha)$, repeating the argument with $(A,\alpha)\cdot \left(\begin{array}{cc} 0 & 1 \\ 1 & 0\\ \end{array}\right)$ and using the fact that
$$\mathbf{z}\left((A,\alpha)\cdot \left(\begin{array}{cc} 0 & 1 \\ 1 & 0\\ \end{array}\right)\right) = 1/\mathbf{z}(A,\alpha),$$
we also obtain $|\mathbf{z}(A,\alpha)| \le 1$, so in all we have $|\mathbf{z}(A,\alpha)| = 1$.

Recall that for any field extension $F/H$, complex conjugation switches the $H^{1,0}(A/F)$ and $H^{0,1}(A/F)$ pieces of $H_{\mathrm{dR}}^1(A/F)$. Let $\alpha^c : \mathbb{Z}_p^{\oplus 2} \xrightarrow{\sim} T_pA$ denote the $p^{\infty}$-level structure obtained from applying $c$ to $\alpha$ (note that $\alpha^c$ is well-defined since all the $p$-power torsion points of $A/F$ are defined over $\overline{\mathbb{Q}}$), and let
$$\gamma = \left(\begin{array}{ccc} a & b\\
c & d\\
\end{array}\right) \in GL_2(\mathbb{Z}_p)$$
such that $\alpha \cdot \gamma = \alpha^c$. By the above discussion $c$ switches the lines given by the Hodge and Hodge-Tate filtrations, and hence we have
\begin{equation}\label{conjugationswitch}\mathbf{z}_{\mathrm{dR}}(A,\alpha) = \mathbf{z}(A,\alpha^c)^c = \frac{d\mathbf{z}(A,\alpha)^c + b}{c\mathbf{z}(A,\alpha)^c + a}
\end{equation}
which implies
\begin{equation}\label{z'rational}\mathbf{z}_{\mathrm{dR}}(A,\alpha) \in K_p.
\end{equation}
From (\ref{conjugationswitch}), we have
\begin{equation}\label{moduluseq}|\mathbf{z}_{\mathrm{dR}}(A,\alpha)| = \frac{|d\mathbf{z}(A,\alpha)^c + b|}{|c\mathbf{z}(A,\alpha)^c + a|}.
\end{equation}
From (\ref{HTperiodgenerates}), since $|\mathbf{z}(A,\alpha)| = 1$ we see that $\mathbf{z}(A,\alpha) \pmod{p}$ generates the residue field $\mathbb{F}_{p^2}$ of $\mathcal{O}_{K_p}$ over $\mathbb{F}_p$, and hence the same is true for $\mathbf{z}(A,\alpha)^c$. Hence, 
$$d\mathbf{z}(A,\alpha)^c + b, c\mathbf{z}(A,\alpha)^c + a \not\equiv 0 \pmod{p}$$
and thus (\ref{moduluseq}) implies that $|\mathbf{z}_{\mathrm{dR}}(A,\alpha)| = 1$. Hence in all, $\mathbf{z}(A,\alpha), \mathbf{z}_{\mathrm{dR}}(A,\alpha) \in \mathcal{O}_{K_p}^{\times}.$ 

Note that since $|\mathbf{z}(A,\alpha)| = 1$, then $(A,\alpha) \in \mathcal{Y}_x$. By the $p$-adic Legendre relation (\ref{periodsrelation}), we have
\begin{equation}\label{ydRequation}\mathbf{y}_{\mathrm{dR}}(A,\alpha) = \frac{\mathbf{z}(A,\alpha)}{\mathbf{z}(A,\alpha) - \mathbf{z}_{\mathrm{dR}}(A,\alpha)}.
\end{equation}
Since on the special fiber $A \otimes_{\mathcal{O}_{K_p}} \mathcal{O}_{K_p}/p\mathcal{O}_{K_p}$, the $\pm 1$ eigenspaces of the CM action $\mathcal{O} = \End(A/H)$ are linearly independent (i.e. the lines on which $\mathcal{O}$ acts through multiplication and the complex conjugation of multiplication, respectively, are linearly independent), then we must have
$$\mathbf{z}(A,\alpha) \not\equiv \mathbf{z}_{\mathrm{dR}}(A,\alpha) \pmod{p\mathcal{O}_{K_p}}$$
and so
$$\mathbf{z}(A,\alpha) - \mathbf{z}_{\mathrm{dR}}(A,\alpha) \in \mathcal{O}_{K_p}^{\times}.$$ 
Hence, by (\ref{ydRequation}), we have
\begin{equation}\label{ydRunit}\mathbf{y}_{\mathrm{dR}}(A,\alpha) \in \mathcal{O}_{K_p}^{\times}.
\end{equation}
Now since $\mathbf{z}_{\mathrm{dR}}(A,\alpha) \in \mathcal{O}_{K_p}^{\times}$, then 
$$-\mathbf{x}_{\mathrm{dR}}(A,\alpha)/\mathbf{y}_{\mathrm{dR}}(A,\alpha) = \mathbf{z}_{\mathrm{dR}}(A,\alpha) \in \mathcal{O}_{K_p}^{\times}.$$
Hence, by (\ref{ydRunit}), we have;
\begin{equation}\label{xdRunit}\mathbf{x}_{\mathrm{dR}}(A,\alpha) \in \mathcal{O}_{K_p}^{\times}.
\end{equation}

For the second part of the Proposition, note that since $|\mathbf{z}(A,\alpha)| = 1$ from the first part of the Proposition, then by Proposition \ref{integralsurjective}, $\frak{s}(A,\alpha) \in \Omega_{A/\mathcal{O}_{\mathbb{C}_p}}$ is a generator. Hence by (\ref{calculation2}), we have that
$$|\theta(y_{\mathrm{dR}}(A,\alpha))^2|\cdot \left|\theta\left(\frac{dz_{\mathrm{dR}}}{dz}\right)(A,\alpha)\right| = 1.$$
Since we have shown $|\theta(y_{\mathrm{dR}}(A,\alpha))| = 1$, then this implies
$$\left|\theta\left(\frac{dz_{\mathrm{dR}}}{dz}\right)(A,\alpha)\right| = 1$$
and so $b = 0$ as in Definition (\ref{integralcorrection2}). Now that $|\theta(\Omega_p(A,\alpha))| = 1$ follows from Definition \ref{perioddefinition} since $\omega_{\mathrm{can}}(A,\alpha) = \frak{s}(A,\alpha)/y_{\mathrm{dR}}(A,\alpha)^2$, $|y_{\mathrm{dR}}(A,\alpha)| = 1$ and $\frak{s}(A,\alpha) \in \Omega_{A/\mathcal{O}_{K_p}}^1$ is a generator.

Finally, we show that given any two choices $\alpha, \alpha'$ of $p^{\infty}$-level structure, we have
\begin{equation}\label{omegapsamemodulus}\theta(\Omega_p)(A,\alpha)/\theta(\Omega_p)(A,\alpha') \in \mathcal{O}_{K_p}^{\times}.
\end{equation}
We can write
$$(A,\alpha') = \left(\begin{array}{ccc} a & b\\
c & d\\
\end{array}\right)$$
for some
$$\left(\begin{array}{ccc} a & b\\
c & d\\
\end{array}\right) \in GL_2(\mathbb{Z}_p).$$
Recall from \ref{canonicaldifferentialdef} that
$$\omega_{\mathrm{can}}((A,\alpha')) = (bc-ad)(c\mathbf{z}_{\mathrm{dR}}((A,\alpha))+a)^{-1}\cdot\omega_{\mathrm{can}}((A,\alpha)),$$
whereas for $\omega_0 \in \Omega_{A/F}^1$ as in Definition \ref{perioddefinition},
$$\omega_0 = \omega_0((A,\alpha)) = \omega_0((A,\alpha')).$$
So now from Definition \ref{perioddefinition}, we see that
$$\Omega_p((A,\alpha')) = (bc-ad)^{-1}(c\mathbf{z}_{\mathrm{dR}}((A,\alpha))+a)\cdot \Omega_p((A,\alpha))$$
which implies
$$\theta_p(\Omega_p)((A,\alpha')) = (bc-ad)^{-1}(c\theta(\mathbf{z}_{\mathrm{dR}})((A,\alpha))+a)\cdot \theta(\Omega_p)((A,\alpha)).$$
From (\ref{z'rational}) we see that $\mathbf{z}_{\mathrm{dR}}((A,\alpha)) \in K_p$ and so $\theta(\mathbf{z}_{\mathrm{dR}})((A,\alpha)) \in K_p$ (since $K_p \subset B_{\mathrm{dR}}^+ \overset{\theta}{\twoheadrightarrow} \mathbb{C}_p$ is the natural inclusion). Hence $\theta(\Omega_p)((A,\alpha))/\theta(\Omega_p)((A,\alpha')) \in K_p$. Now from Corollary \ref{periodunit} and (\ref{omegapsamemodulus}) we have
$$|\theta(\Omega_p)((A,\alpha))| = |\theta(\Omega_p)((A,\alpha'))|.$$
This, combined with $\theta(\Omega_p)(y)/\theta(\Omega_p)(y') \in K_p$ implies $\theta(\Omega_p)(y)/\theta(\Omega_p)(y') \in \mathcal{O}_{K_p}^{\times}$. Now we are done.

\end{proof}

\subsection{The supersingular CM torus}\label{CMtorus}In this section, we study the geometry of the supersingular CM torus at infinite level, as well as properties of its periods in order to show that Corollary (\ref{keycorollary}) applies to $y$ in this locus.

Consider orders $\mathcal{O}_f = \mathbb{Z} + f\mathcal{O}_K \subset \mathcal{O}_K \subset K$, where $K = \mathbb{Q}(\sqrt{-d})$ with $d>0$ squarefree is an imaginary quadratic field in which $p$ is inert or ramified, let $d_K$ denote the fundamental discriminant of $K/\mathbb{Q}$, and let $\frak{p} \subset \mathcal{O}_K$ denote the prime ideal above $p$. Suppose that conductor $c$ of $\mathcal{O}_c$ is prime to $p$. Let $K_p$ denote the $p$-adic completion of $K$ (under the embedding $i_p : \overline{\mathbb{Q}} \hookrightarrow \overline{\mathbb{Q}}_p$ from (\ref{fixedembedding})). Let $\mathcal{C}\ell(\mathcal{O}_c)$ denote the class group of $K$.

For the rest of this section, assume the \emph{Heegner hypothesis} for $N$, namely
$$\text{for each prime $\ell|N$, we have $\ell$ is split or ramified in $K$, and if $\ell^2|N$ then $\ell$ is split in $K$.}$$
Let
\begin{equation}\label{epsilondefinition}\varepsilon = \begin{cases} 0 & \text{$p$ inert in $K$}\\
1 & \text{$p$ ramified in $K$}.
\end{cases}
\end{equation}
Then $Np^{\varepsilon}$ also satisfies the Heegner hypothesis. (Recall that we assumed $p\nmid N$.)

Let $\hat{\mathbb{Z}} := \prod_{\ell < \infty} \mathbb{Z}_{\ell}$. Let $K_n(Np^{\varepsilon}) \subset M_2(\hat{\mathbb{Z}})$ be the order
\begin{align*}K_n(Np^{\varepsilon}) &= \left\{\left(\begin{array}{cc} a & b \\ c & d\\ \end{array}\right) \in M_2(\hat{\mathbb{Z}}) : b \equiv 0 \pmod{p^n\hat{\mathbb{Z}}},\right. \\
&\left.c \equiv 0 \pmod{p^{n+\varepsilon}N\hat{\mathbb{Z}}}, d \equiv 1 \pmod{p^n\hat{\mathbb{Z}}}\right\}.
\end{align*}
Let $\alpha, \beta \in \mathbb{Z}$ such that $\alpha^2 - 4Np^{\varepsilon}\beta = c^2d_K$. Fix an embedding 
\begin{equation}\label{fixedtorusembedding}i : \mathbb{A}_K \hookrightarrow M_2(\mathbb{A}_{\mathbb{Q}})
\end{equation}
given at a place $v$ of $K$ for $x_v + y_v\frac{c^2d_K+c\sqrt{d_K}}{2} \in K_v$, where $x_v, y_v \in \mathbb{Q}_{\ell}$ and $\ell$ is the prime below $v$, by 
$$x_v \mapsto \left(\begin{array}{cc} x_v & 0 \\ 0 & x_v\\ \end{array}\right), \hspace{1cm} y_v \frac{c^2d_K+c\sqrt{d_K}}{2} \mapsto y_v \cdot \left(\begin{array}{cc} \frac{c^2d_K + \alpha}{2} & -1 \\ Np^{\varepsilon}\beta & \frac{c^2d_K-\alpha}{2} \\ \end{array}\right)$$
so that $K \cap K_0(Np^{\varepsilon}) = \hat{\mathcal{O}}_{c} = \mathcal{O}_{c} \otimes_{\mathbb{Z}} \hat{\mathbb{Z}}$, and more generally 
$$K \cap K_n(Np^{\varepsilon}) = \hat{\mathcal{O}}_c \times (1+p^n\frak{p}^{\varepsilon}\mathcal{O}_{K_p}).$$
Note that (\ref{fixedtorusembedding}) also gives rise to an embedding of algebraic groups over $\mathbb{Q}$, 
$$\mathrm{Res}_{K/\mathbb{Q}}\mathbb{G}_m \hookrightarrow GL_2/\mathbb{Q}.$$

Let $Y_n := Y(\Gamma_1(N)\cap \Gamma_0(p^{\varepsilon}) \cap \Gamma(p^n))$, and let us briefly once more consider $Y_n$ as an algebraic scheme (and not its associated adic space). In particular, $Y_0 = Y(\Gamma_0(p^{\varepsilon}))$. Let 
$$K_n(Np^{\varepsilon})^{\times} := (K_n(Np^{\varepsilon}) \cap GL_2(\mathbb{A}_{\mathbb{Q}})) \subset GL_2(\mathbb{A}_{\mathbb{Q}})$$
denote the open compact subgroup corresponding to $\Gamma_1(N)\cap \Gamma(p^n)$. Recall the uniformization
$$Y_n(\mathbb{C}) = GL_2(\mathbb{Q})\backslash \mathcal{H}^{\pm} \times GL_2(\mathbb{A}_{\mathbb{Q},f})/K_n(Np^{\varepsilon})^{\times}$$
with Hecke action $GL_2(\mathbb{A}_{\mathbb{Q},f})$ given by right multiplication. Let $C_n(\mathcal{O}_c)$ denote the scheme consisting of points with CM by an order $\mathcal{O}_c \subset \mathcal{O}_K$ and $p^n\frak{p}^{\varepsilon}$-level structure, so that as a complex space we have
$$C_n(\mathcal{O}_{c})(\mathbb{C}) :=  K^{\times}\backslash \mathbb{A}_{K,f}^{\times}/\left(\hat{\mathcal{O}}_c^{\times} \times (1+p^n\frak{p}^{\varepsilon}\mathcal{O}_{K_p})\right).$$
Then $C_n(\mathcal{O}_{c})$ is a scheme defined over $\mathbb{Q}$, where the right action $\mathbb{A}_{K,f}^{\times} \subset GL_2(\mathbb{A}_{\mathbb{Q},f})$ is given by the Hecke action. For simplicity, denote $C_0(\mathcal{O}_{c}) = C(\mathcal{O}_{c})$. 

Then the embedding (\ref{fixedtorusembedding}) induces an embedding
$$\mathcal{O}_{K_p} \hookrightarrow M_2(\mathbb{Z}_p)$$
and hence
$$\mathcal{O}_{K_p}^{\times} \hookrightarrow GL_2(\mathbb{Z}_p)$$
so that we may view $C_n(\mathcal{O}_{c}) \subset Y_n$ as a subscheme. Recall that the main theorem of complex multiplication says that every point of $C_n(\mathcal{O}_{c})(\overline{K})$ is defined over the maximal abelian extension $K^{\mathrm{ab}}$ of $K$, and that the Galois action $\Gal(K^{\mathrm{ab}}/K)$ is given by the action of $\mathbb{A}_K^{\times} \hookrightarrow GL_2(\mathbb{A}_{\mathbb{Q}})$ under the reciprocity law. Now viewing $C_n(\mathcal{O}_{c})$ and $Y_n$ once more as adic spaces over $\mathrm{Spa}(\mathbb{Q}_p,\mathbb{Z}_p)$, define an object $\mathcal{C}(\mathcal{O}_c) \in C(\mathcal{O}_{c})_{\text{pro\'{e}t}}$ by
$$\mathcal{C}(\mathcal{O}_c) = \varprojlim_n C_n(\mathcal{O}_{c}) \subset \varprojlim_n Y_n = \mathcal{Y},$$
so that the Galois group of any geometric fiber of $\mathcal{C}(\mathcal{O}_c) \rightarrow C(\mathcal{O}_{c})$ is $\mathcal{O}_{K_p}^{\times}$ if $p$ is inert in $K$, and is $1+\frak{p}\mathcal{O}_{K_p}$ if $p$ is ramified in $K$, viewed naturally inside the Galois group $\Gal(\mathcal{Y}/Y) = GL_2(\mathbb{Z}_p)$ via the embedding (\ref{fixedtorusembedding}) above. By construction, we see that $\mathcal{C}(\mathcal{O}_c) \subset \mathcal{Y}$ is a closed adic subspace. 

\begin{proposition}\label{CMtorussameperiods}The Hodge-Tate, Hodge-de Rham, and $\theta(y_{\mathrm{dR}})$ periods of all points $y \in \mathcal{C}(\mathcal{O}_c)(\overline{K}_p,\mathcal{O}_{\overline{K}_p})$ all have the same $p$-adic valuation.

In other words, given any $y, y' \in \mathcal{C}(\mathcal{O}_c)(\overline{K}_p,\mathcal{O}_{\overline{K}_p})$ as above, we have
$$\mathbf{z}(y)/\mathbf{z}(y'), \mathbf{z}_{\mathrm{dR}}(y)/\mathbf{z}_{\mathrm{dR}}(y'), \theta(y_{\mathrm{dR}})(y)/\theta(y_{\mathrm{dR}})(y') \in \mathcal{O}_{K_p}^{\times}$$
and so in particular
$$|\mathbf{z}(y)| = |\mathbf{z}(y')|, |\mathbf{z}_{\mathrm{dR}}(y)| = |\mathbf{z}_{\mathrm{dR}}(y')|, |\theta(y_{\mathrm{dR}})(y)| = |\theta(y_{\mathrm{dR}})(y')|.$$
In fact, we have $\mathbf{z}(y), \mathbf{z}_{\mathrm{dR}}(y) \in K_p$ for all $y \in \mathcal{C}(\mathcal{O}_c)(\overline{K}_p,\mathcal{O}_{\overline{K}_p})$, as well as
\begin{equation}\begin{split}&|\mathbf{z}(y)| = \begin{cases} 1 & \text{$p$ inert in $K$} \\
 |1/(c\sqrt{-d})| & \text{$p$ ramified in $K$}\\ \end{cases}, |\mathbf{z}_{\mathrm{dR}}(y)| = \begin{cases} 1 & \text{$p$ inert in $K$} \\
 |1/(c\sqrt{-d})| & \text{$p$ ramified in $K$} \end{cases}, \\
 &|\theta(y_{\mathrm{dR}})(y)| = \begin{cases} 1 & \text{$p$ inert in $K$} \\
 |1/2| & \text{$p$ ramified in $K$} \\ \end{cases}.
 \end{split}
\end{equation}

Moreover, we have that the numbers $a, b\in \mathbb{Q}$ defined as in (\ref{integralcorrection}) and (\ref{integralcorrection2}), as well as the $p$-adic absolute value $|\theta(\Omega_p)(y)|$, are the same for any $y \in \mathcal{C}(\mathcal{O}_c)(\overline{K}_p,\mathcal{O}_{\overline{K}_p})$. In fact, for any $y, y' \in \mathcal{C}(\mathcal{O}_c)(\overline{K}_p,\mathcal{O}_{\overline{K}_p})$ we have
$$\theta(\Omega_p)(y)/\theta(\Omega_p)(y') \in \mathcal{O}_{K_p}^{\times}.$$
If $p > 2$ or $p$ is inert, we further have $b = 0$ and $|\theta(\Omega_p)(y)| = 1$ for any $y \in \mathcal{C}(\mathcal{O}_c)(\overline{K}_p,\mathcal{O}_{\overline{K}_p})$.

\end{proposition}

\begin{proof}
Note that the $GL_2(\mathbb{Q}_p)$-action on $\mathcal{Y}$ than the action of the local Hecke algebra $GL_2(\mathbb{Q}_p) \subset GL_2(\mathbb{A}_{\mathbb{Q}})$. As discussed in the beginning of this section, using the embedding $i : \mathbb{A}_K^{\times} \hookrightarrow GL_2(\mathbb{A}_{\mathbb{Q}})$ from (\ref{fixedtorusembedding}), the $\mathbb{A}_K^{\times}$-action on $\mathcal{C}(\mathcal{O}_c)$ acts through the Hecke action, and by the main theorem of complex multiplication this coincides with the Artin reciprocity map. This action of $\mathbb{A}_K^{\times}$ hence factors through the ray class group (over $K$ of conductor $p^{\infty}$)
\begin{equation}\label{rayclassgroup}K^{\times}\backslash \mathbb{A}_K^{\times}/(\prod_{v<\infty, v\nmid p}\mathcal{O}_{K_v}^{\times}\cdot \mathbb{C}^{\times}),
\end{equation}
which acts simply transitively on $\mathcal{C}(\mathcal{O}_c)$, and $\mathcal{O}_{K_p}^{\times}$ acts simply transitively on any geometric fiber of $\mathcal{C}(\mathcal{O}_c) \rightarrow C(\mathcal{O}_c)$ (and can be identified as the Galois group of any geometric fiber in $\mathcal{C}(\mathcal{O}_c) \rightarrow C(\mathcal{O}_c)$). 

Now we can view (\ref{rayclassgroup}) as 
$$\mathcal{C}\ell(\mathcal{O}_c)\mathcal{O}_{K_p}^{\times}$$
and through the Hecke (i.e. $GL_2(\mathbb{Q}_p)$)-action this acts simply transitively on the sublocus $\mathcal{C}(\mathcal{O}_c)(\overline{K}_p,\mathcal{O}_{\overline{K}_p})$. Now the representatives of $\mathcal{C}\ell(\mathcal{O}_c)$ can be chosen to be units at $p$, and so these representatives act trivially through the $GL_2(\mathbb{Q}_p)$-action and hence do not affect Hodge-Tate periods. So now to prove the Lemma, it suffices to show that for any single $y \in \mathcal{C}(\mathcal{O}_c)(\overline{K}_p,\mathcal{O}_{\overline{K}_p})$ and any $g \in \mathcal{O}_{K_p}^{\times} \overset{i}{\subset} GL_2(\mathbb{Z}_p)$, that 
$$|\mathbf{z}(y\cdot \gamma)| = |\mathbf{z}(y)|, |\mathbf{z}_{\mathrm{dR}}(y\cdot \gamma)| = |\mathbf{z}_{\mathrm{dR}}(y)|, |\theta(y_{\mathrm{dR}}(y\cdot \gamma)| = |\theta(y_{\mathrm{dR}})(y)|.$$

For this, when $p$ is inert in $K$ fix any point $y = (A,\alpha) \in \mathcal{C}(\overline{K}_p,\mathcal{O}_{\overline{K}_p})$. When $p$ is ramified, let $y = (A,\alpha)$ be as in Proposition \ref{CMproposition1}; the projection of $y$ along $\mathcal{Y} \rightarrow Y(\Gamma_1(N) \cap \Gamma_0(N))$ gives rise to a point $(A,t,A[\frak{p}])$ (i.e. a point in $C(\mathcal{O}_c)(\overline{K}_p,\mathcal{O}_{\overline{K}_p})$) because $\alpha$ was chosen to be a rational canonical basis for the endomorphism $c\sqrt{-d} \in \mathcal{O}_c = \End(A/F) = \End(\hat{A}/F)$, the kernel of which is $A[\frak{p}]$. Then by Propositons \ref{CMproposition1} and \ref{CMproposition2}, we have
\begin{equation}\label{modulusformulas}\begin{split}&|\mathbf{z}(y)| = \begin{cases} 1 & \text{$p$ inert in $K$} \\
 |1/(c\sqrt{-d})| & \text{$p$ ramified in $K$}\\ \end{cases}, \mathbf{z}_{\mathrm{dR}}(y)| = \begin{cases} 1 & \text{$p$ inert in $K$} \\
 |1/(c\sqrt{-d})| & \text{$p$ ramified in $K$} \end{cases}, \\
 &|\theta(y_{\mathrm{dR}}(y))| = \begin{cases} 1 & \text{$p$ inert in $K$} \\
 |1/2| & \text{$p$ ramified in $K$} \\ \end{cases}.
 \end{split}
\end{equation}

If $p$ is inert in $K$, then $y \cdot \gamma$ for any $\gamma \in \mathcal{O}_{K_p}^{\times}$ also falls under the purview of Proposition \ref{CMproposition2}, and so we have that $y \cdot \gamma$ satisfies the same equations (\ref{modulusformulas}) as $y$ and the ratios of all the relevant periods associated with $y$ and $y\cdot \gamma$ are in $\mathcal{O}_{K_p}^{\times}$ and so we are done with the first part of the Proposition in this case.

If $p$ is ramified in $K$, then since $\gamma \in (\mathcal{O}_{K_p}^{\times} \cap \Gamma_0(p)) = 1+ \frak{p}\mathcal{O}_{K_p} \subset GL_2(\mathbb{Z}_p)$ is induced by (\ref{fixedtorusembedding}), we have 
$$\gamma = \left(\begin{array}{cc} a & b \\ c & d\\ \end{array}\right)$$
for some $a, b, c, d \in \mathbb{Z}_p$ necessarily with $a \in \mathbb{Z}_p^{\times}$ (in order for $\gamma$ to have determinant in $\mathbb{Z}_p^{\times}$) and $c \in p\mathbb{Z}_p$. Hence by (\ref{ztransformationprop}), (\ref{z'transformationprop}) and (\ref{y'transformationprop}) we have
\begin{equation}\label{transformationequations}\begin{split}&\mathbf{z}(y\cdot \gamma) = \frac{d\mathbf{z}(y) + b}{c\mathbf{z}(y)+a}, \hspace{1cm}\mathbf{z}_{\mathrm{dR}}(y\cdot \gamma) = \frac{d\mathbf{z}_{\mathrm{dR}}(y) + b}{c\mathbf{z}_{\mathrm{dR}}(y)+a}, \\
&\theta(y_{\mathrm{dR}})(y\cdot \gamma) = (c\theta(\mathbf{z}_{\mathrm{dR}})(y)+a)(b\theta(\mathbf{z})^{-1}(y)+d)(bc-ad)^{-1}\theta(y_{\mathrm{dR}})(y).
\end{split}
\end{equation}
Moreover, since $a \in \mathbb{Z}_p^{\times}$ and $c \in p\mathbb{Z}_p$ we have
\begin{align*}&|d\mathbf{z}(y)+b| = |\mathbf{z}(y)|, \hspace{.3cm} |c\mathbf{z}(y)+a| = 1, \hspace{.3cm} |d\mathbf{z}_{\mathrm{dR}}(y)+b| = |\mathbf{z}_{\mathrm{dR}}(y)|, \hspace{.3cm} |c\mathbf{z}_{\mathrm{dR}}(y)+a| = 1\\
& |c\theta(\mathbf{z}_{\mathrm{dR}}(y)+a| = 1, \hspace{.3cm} |b\theta(\mathbf{z})^{-1}(y)+d| = 1, \hspace{.3cm} |(bc-ad)| = 1.
\end{align*}
and so 
\begin{equation}\label{periodformulas}\begin{split}&|\mathbf{z}(y\cdot \gamma)| = \left|\frac{d\mathbf{z}(y) + b}{c\mathbf{z}(y)+a}\right| = |\mathbf{z}(y)|, \hspace{1cm}|\mathbf{z}_{\mathrm{dR}}(y\cdot \gamma)| = \left|\frac{d\mathbf{z}_{\mathrm{dR}}(y) + b}{c\mathbf{z}_{\mathrm{dR}}(y)+a}\right| = |\mathbf{z}_{\mathrm{dR}}(y)|, \\
&|\theta(y_{\mathrm{dR}})(y\cdot \gamma)| = |(c\theta(\mathbf{z}_{\mathrm{dR}})(y) + a)(b\theta(\mathbf{z})^{-1}(y) + d)(bc-ad)^{-1}\theta(y_{\mathrm{dR}})(y)| = |\theta(y_{\mathrm{dR}}(y))|.
\end{split}
\end{equation}
Recall from Proposition \ref{CMproposition1} that we have $\mathbf{z}(y), \mathbf{z}_{\mathrm{dR}}(y), \theta(y_{\mathrm{dR}})(y) \in K_p^{\times}$. From (\ref{transformationequations}) we see that $\mathbf{z}(y\cdot \gamma), \mathbf{z}_{\mathrm{dR}}(y\cdot \gamma), \theta(y_{\mathrm{dR}})(y\cdot \gamma) \in K_p$ and furthermore from (\ref{periodformulas}) we see that 
$$\mathbf{z}(y)/\mathbf{z}(y\cdot \gamma), \mathbf{z}_{\mathrm{dR}}(y)/\mathbf{z}_{\mathrm{dR}}(y\cdot \gamma), \theta(y_{\mathrm{dR}})(y)/\theta(y_{\mathrm{dR}})(y\cdot \gamma) \in \mathcal{O}_{K_p}^{\times}$$
and so we have proven the first part of the Proposition in the ramified case.

Now we prove the second part of the Proposition. Letting $y' = (A,\alpha)$ as before, we have by the above that $|\mathbf{z}_{\mathrm{HT}}(y')| = |\mathbf{z}(y)|, |\mathbf{z}_{\mathrm{dR}}(y')| = |\mathbf{z}_{\mathrm{dR}}(y)|, |\theta(y_{\mathrm{dR}})(y')| = |\theta(y_{\mathrm{dR}})(y)|$ for all $y' \in \mathcal{C}(\mathcal{O}_c)(K_p^{\mathrm{ab}},\mathcal{O}_{K_p}^{\mathrm{ab}})$. Now applying the same arguments as in the proofs of the last parts of Propositions \ref{CMproposition1} and \ref{CMproposition2}, we have proven the statement on $a,b$.

Now we address the periods $\theta(\Omega_p)(y)$. Letting $y' = y \cdot \gamma$ as above, from Definition \ref{canonicaldifferentialdef}, we see that
$$\omega_{\mathrm{can}}(y') = (bc-ad)(c\mathbf{z}_{\mathrm{dR}}(y)+a)^{-1}\cdot\omega_{\mathrm{can}}(y),$$
whereas for $\omega_0 \in \Omega_{A/F}^1$ as in Definition \ref{perioddefinition},
$$\omega_0 = \omega_0(y) = \omega_0(y').$$
So now from Definition \ref{perioddefinition}, we see that
$$\Omega_p(y') = (bc-ad)^{-1}(c\mathbf{z}_{\mathrm{dR}}(y)+a)\cdot \Omega_p(y)$$
which implies
$$\theta_p(\Omega_p)(y') = (bc-ad)^{-1}(c\theta(\mathbf{z}_{\mathrm{dR}})(y)+a)\cdot \theta(\Omega_p)(y).$$
Again, from (the proofs of) Propositions \ref{CMproposition1} and \ref{CMproposition2} we see that $\mathbf{z}_{\mathrm{dR}}(y) \in K_p$ and so $\theta(\mathbf{z}_{\mathrm{dR}})(y) \in K_p$ (since $K_p \subset B_{\mathrm{dR}}^+ \overset{\theta}{\twoheadrightarrow} \mathbb{C}_p$ is the natural inclusion). Hence $\theta(\Omega_p)(y)/\theta(\Omega_p)(y') \in K_p$. Now from Corollary \ref{periodunit} and the previous paragraph we have
$$|\theta(\Omega_p)(y)| = |\theta(\Omega_p)(y')|.$$
This, combined with $\theta(\Omega_p)(y)/\theta(\Omega_p)(y') \in K_p$ implies $\theta(\Omega_p)(y)/\theta(\Omega_p)(y') \in \mathcal{O}_{K_p}^{\times}$. Now we are done.
\end{proof}

\subsection{Weights}\label{closedweights}In this section, we slightly refine our study of the notion of weights for the operator $(p^b\theta_k)^jf$ (for $b$ as in (\ref{integralcorrection2})), and apply it restrictions of sections to the closed adic subspaces $\mathcal{V} \subset \mathcal{Y}_x$ of the open affinoid subdomain $\mathcal{Y}_x$ of $\mathcal{Y}$. Recall our fixed embedding $K \hookrightarrow M_2(\mathbb{Q})$ from (\ref{fixedtorusembedding}), which induces an embedding $\mathcal{O}_{K_p} \hookrightarrow M_2(\mathbb{Z}_p)$ and $\mathcal{O}_{K_p}^{\times} \hookrightarrow GL_2(\mathbb{Z}_p)$ in the case when $p > 2$ or $p$ is not inert in $K$. In particular, one sees by the proof of Proposition \ref{CMtorussameperiods} that for such embeddings, we have
\begin{equation}\label{modulus1}|cz_{\mathrm{dR}}(y)+a| = 1
\end{equation}
for any
$$\left(\begin{array}{ccc}a & b\\
c & d\\
\end{array}\right) \in \Gamma \subset GL_2(\mathbb{Z}_p)$$
with $\Gamma = 1+p\mathcal{O}_{K_p}$ when $p$ is ramified in $K$, and $\Gamma  = \mathcal{O}_{K_p}^{\times}$ when $p$ is inert in $K$. 

Given a closed immersion of adic spaces
$$i : \mathcal{V} \hookrightarrow \mathcal{Y}_x,$$ 
let $\mathcal{O}_{\mathcal{V}}$ denote the pro\'{e}tale structure sheaf on the adic space $\mathcal{V}$, with $p$-adic completion $\hat{\mathcal{O}}_{\mathcal{V}}$. Henceforth, denote
$$\hat{\mathcal{O}}_{\mathcal{Y}_x}|_{\mathcal{V}} := i^{-1}\hat{\mathcal{O}}_{\mathcal{V}}$$
the restriction of $\hat{\mathcal{O}}_{\mathcal{Y}_x}$ to $\mathcal{V}$. 

Note that since $\mathcal{Y}^{\mathrm{ss}} \subset \mathcal{Y}$ is an open affinoid subdomain, then $\mathcal{Y}^{\mathrm{ord}} \subset \mathcal{Y}$ is a closed adic subspace, and so $\mathcal{Y}_x^{\mathrm{ord}} := \mathcal{Y}_x \cap \mathcal{Y}^{\mathrm{ord}} \subset \mathcal{Y}_x$ is a closed adic subspace of $\mathcal{Y}_x$. 

\begin{definition}\label{CMweightdefinition}Let $\Gamma \subset GL_1(\mathcal{O}_{K_p}) = \mathcal{O}_{K_p}^{\times} \subset GL_2(\mathbb{Z}_p)$, where $\mathcal{V} \subset \mathcal{Y}_x$ is an affinoid open or closed adic subspace, be a subgroup such that $\mathcal{V} \cdot \Gamma = \Gamma$. A section $F \in \hat{\mathcal{O}}_{\mathcal{Y}_x}|_{\mathcal{V}}$ if $\mathcal{V} \subset \mathcal{Y}_x$ is closed and $\mathcal{\mathcal{O}}_{\mathcal{Y}_x}(\mathcal{V})$ if $\mathcal{V}$ is open is said to have \emph{weight $k$ for $\Gamma$ on $\mathcal{V}$} if 
$$\left(\begin{array}{ccc}a & b\\
c & d\\
\end{array}\right)^*F = (ad-bc)^{-k}(c\theta(z_{\mathrm{dR}})|_{\mathcal{V}}+a)^kF$$
for all $\left(\begin{array}{ccc}a & b\\
c & d\\
\end{array}\right) \in \Gamma$. 
\end{definition}

\begin{remark}\label{weightremark}Note that if $F \in \mathcal{O}_{\Delta,\mathcal{Y}}(\mathcal{Y}_x)$ is of weight $k$ in the sense of Definition \ref{weightdefinition} for any subgroup $\Gamma \subset GL_2(\mathbb{Z}_p)$, so that $\mathcal{Y}^{\mathrm{ss}} \cdot \Gamma = \mathcal{Y}^{\mathrm{ss}}$, then $\theta(F)  \in \hat{\mathcal{O}}(\mathcal{Y}^{\mathrm{ss}})$ is of weight $k$ on $\mathcal{Y}^{\mathrm{ss}}$ for $\Gamma$ in the sense of Definition \ref{CMweightdefinition}. Moreover, if $\mathcal{V} \cdot \Gamma = \mathcal{V}$ then $\theta(F)|_{\mathcal{V}}$ has weight $k$ for $\Gamma$ on $\mathcal{V}$.
\end{remark}

\begin{proposition}\label{weightchangeproposition}Suppose $F \in \mathcal{O}_{\Delta,\mathcal{Y}}(\mathcal{Y}_x)$ has weight $k$ for $\Gamma \subset GL_2(\mathbb{Z}_p)$ in the sense of Definition \ref{weightdefinition}, and suppose $\mathcal{V} \cdot \Gamma = \mathcal{V}$. Then for any $j \in \mathbb{Z}_{\ge 0}$, $\theta_k^jF|_{\mathcal{V}}$ has weight $k + 2j$ for $\Gamma$ on $\mathcal{V}$. 
\end{proposition}

\begin{proof}This follows immediately from Remark \ref{weightremark} and by Proposition \ref{weightpropositiontheta}. 

\end{proof}

\begin{proposition}\label{weightchangelimitproposition}Suppose that $F$ has weight $k$ on $\mathcal{Y}_x$ for a subgroup $\Gamma \subset GL_2(\mathbb{Z}_p)$ as defined in Definition \ref{weightdefinition}, and let $\mathcal{V} \subset \mathcal{Y}$ be an open affinoid subdomain or closed adic subspace $\mathcal{V} \cdot \Gamma = \mathcal{V}$. Then the limit
$$\lim_{m \rightarrow \infty}\theta_k^{j_0 + (p-1)p^m}F|_{\mathcal{V}}$$
exists for $j_0 \in \mathbb{Z}$, then it has weight $k+2j_0$ for $\Gamma$ on $\mathcal{V}$.
\end{proposition}

\begin{proof}Note first that $|c\theta(z_{\mathrm{dR}}) + a| = 1$ by (\ref{modulus1}), and so 
$$\lim_{m \rightarrow \infty}(c\theta(z_{\mathrm{dR}})+a)^{j_0+2(j_0+(p-1)p^m)} = (c\theta(z_{\mathrm{dR}})+a)^{k+2j_0}.$$
For any $\left(\begin{array}{ccc}a & b\\
c & d\\
\end{array}\right) \in \Gamma$, we compute, using Proposition \ref{weightchangeproposition}, that
\begin{align*}&\left(\begin{array}{ccc}a & b\\
c & d\\
\end{array}\right)^*\lim_{m \rightarrow \infty}\theta_k^{j_0 + (p-1)p^m}F|_{\mathcal{V}} \\
&= \lim_{m \rightarrow \infty}\left(\begin{array}{ccc}a & b\\
c & d\\
\end{array}\right)^*\theta_k^{j_0 + (p-1)p^m}F|_{\mathcal{V}}\\
&= \lim_{m \rightarrow \infty}(ad-bc)^{k+2(j_0+(p-1)p^m)}(c\theta(z_{\mathrm{dR}})+a)^{j_0+2(j_0+(p-1)p^m)}\theta_k^{j_0+ (p-1)p^m}F|_{\mathcal{V}}\\
&= \lim_{m \rightarrow \infty}(ad-bc)^{k+2(j_0+(p-1)p^m)}\lim_{m \rightarrow \infty}(c\theta(z_{\mathrm{dR}})+a)^{j_0+2(j_0+(p-1)p^m)}\lim_{m \rightarrow \infty}\theta_k^{j_0 + (p-1)p^m}F|_{\mathcal{V}}\\
&= (ad-bc)^{k+2a}(c\theta(z_{\mathrm{dR}})+a)^{k+2j_0}\lim_{m \rightarrow \infty}\theta_k^{j_0 + (p-1)p^m}F|_{\mathcal{V}}.\\
\end{align*}
\end{proof}

\subsection{Construction of the $p$-adic $L$-function}\label{padicLfunctionsection} In this section, we construct our 1-variable anticyclotomic $p$-adic $L$-function attached to a given normalized eigenform (i.e. newform or normalized Eisenstein series). Again let $w \in \omega_{\mathcal{A}}^{\otimes k}(Y)$ correspond to a normalized new eigenform of level $\Gamma_1(N)$ and weight $k$ (where $k \ge 0$ is an integer) and of nebentype $\epsilon$, as in the statement of Theorem \ref{MScomparison}. Note then that letting $N_{\epsilon}$ denote the conductor of $\epsilon$, we have $N_{\epsilon}|N$. 

\subsection{Preliminaries for the construction}

Retaining the notation of the previous section, let $K/\mathbb{Q}$ denote an imaginary quadratic field with $K = \mathbb{Q}(\sqrt{-d})$ and $d > 0$ squarefree. Henceforth let $d_K$ denote the fundamental discriminant of $K$. We will make the following \emph{Heegner hypothesis} for $N$, namely
$$\text{for each prime $\ell|N$, we have $\ell$ is split or ramified in $K$, and if $\ell^2|N$ then $\ell$ is split in $K$.}$$
We make the Heegner hypothesis primarily for convenience in this document, and plan to remove or weaken this assumption in forthcoming work. The Heegner hypothesis guarantees the existence of an integral ideal $\frak{N} \subset \mathcal{O}_K$ such that
$$\mathcal{O}_K/\frak{N} = \mathbb{Z}/N.$$
We henceforth fix a choice of such an $\frak{N}$. Since $N_{\epsilon}|N$, this determines an ideal $\frak{N}_{\epsilon}|\frak{N}$ such that
$$\mathcal{O}_K/\frak{N}_{\epsilon} = \mathbb{Z}/N_{\epsilon}.$$

Recall the fixed embedding $i_p : \overline{\mathbb{Q}} \hookrightarrow \overline{\mathbb{Q}}_p$ (from (\ref{fixedembedding}), which determines a $p$-adic completion $K_p \hookrightarrow \overline{\mathbb{Q}}_p$ of $K$. Let $q$ denote the order of the residue field of $\mathcal{O}_{K_p}$. 

\begin{definition}Henceforth, given a complex-valued Hecke character $\chi : K^{\times} \backslash \mathbb{A}_K^{\times} \rightarrow \mathbb{C}^{\times}$ (which we henceforth assume arises from an algebraic Hecke character of ``type $A_0$'' in Weil's sense), we let $\check{\chi}$ denote its $p$-adic avatar. We let $\frak{f}(\chi)$ denote the conductor of $\chi$.

We call a Hecke character $\chi$ over $K$ \emph{central critical with respect to $w$} if it is of infinity type $(k+j,-j)$ for $j \in \mathbb{Z}$ and if it satisfies the following compatibility of central characters:
$$\chi|_{\mathbb{A}_{\mathbb{Q}}^{\times}} = \epsilon^{-1}.$$

Let $\mathbb{N}_K : \mathbb{A}_K^{\times} \rightarrow \mathbb{C}^{\times}$ denote the norm character normalized to have infinity type $(1,1)$. 
\end{definition}

Let $\pi_w$ denote the unitary automorphic representation of $GL_2(\mathbb{A}_{\mathbb{Q}})$ generated by $w$, and $\pi_{\chi^{-1}}$ the unitary automorphic representation of $GL_2(\mathbb{A}_{\mathbb{Q}})$ generated by $\theta_{\chi^{-1}}$. Central criticality forces the automorphic representation $\pi_w \times \pi_{\chi^{-1}}$ to be self-dual, and the center of $L(\pi_w \times \pi_{\chi^{-1}},s)$ the functional equation to be $s = 1/2$. The Heegner hypothesis implies that the local root number 
$$\epsilon_{\ell}(\pi_w \times \pi_{\chi^{-1}},1/2) = \epsilon_{\ell}((\pi_w)_K \otimes \chi^{-1},1/2) = \prod_{v|\ell}\epsilon_v((\pi_w)_K \otimes \chi^{-1},1/2) =  +1$$
for all rational primes $\ell$ and primes $v$ of $K$, where here $(\pi_w)_K$ denotes the base-change of $\pi_w$ to $K$. 
Hence we satisfy the Saito-Tunnell local conditions so that the period integrals, whose squares we eventually relate to $L$-values via a calculation of Waldspurger, which we interpolate are not a priori 0. More precisely, we can view these period integrals as functionals on the (unitary) automorphic representation $(\pi_w)_K \times \chi^{-1}$ over $\mathbb{A}_K$ attached to the pair $(w,\chi^{-1})$, or in other words as elements of the restricted tensor product 
$$\bigotimes_v'\Hom_{K_v^{\times}}((\pi_w)_K \otimes \chi,\mathbb{C}).$$
By the well-known theorem of Saito and Tunnell (see, for example, \cite[Chapter 1, Theorem 1.3]{YuanZhangZhang}), this space is of dimension less than or equal to 1, and is of dimension 1 if and only if the local root numbers $\epsilon_v(w,\chi^{-1},1/2) = +1$ for all finite places $v$ of $K$, and $\epsilon_{\infty}(w,\chi^{-1},1/2) = -1$. In particular, the global root number satisfies
$$\epsilon(w,\chi,1/2) = -1.$$

\begin{example}\label{examplefamily}As an example a family of central critical characters $\{\chi_j\}$ with $\epsilon = 1$, we can take the character defined on ideals by
$$\chi_{-k/2 + h_Kwt}(\frak{a}) = \mathbb{N}_K(a)^{k/2}\left(\frac{a}{\bar{a}}\right)^{wt}$$
where $a \in K^{\times}$ is a generator of $\frak{a}^{h_K}$, where $h_K = \#\mathcal{C}\ell(\mathcal{O}_K)$ and $w =\#\mathcal{O}_K^{\times}/2$. (One can check that the value $\chi_j(\frak{a})$ does not depend on the choice of $a$.) Then letting $j = -k/2 + h_Kwt$, $\chi_j$ gives rise to a Hecke character which is unramified at all finite places, and which has infinity type $(k+j,-j)$. 
\end{example}

Let
\begin{align*}\mathcal{O}_c := \mathbb{Z} + c\mathcal{O}_K \subset \mathcal{O}_K,& \hspace{1cm} \hat{\mathcal{O}}_c^{\times} = \prod_{v\nmid \infty} (\mathcal{O}_c\otimes_{\mathcal{O}_K} \mathcal{O}_{K_v})^{\times}, \hspace{1cm} \hat{\mathcal{O}}_K^{\times} = \prod_{v \nmid \infty} \mathcal{O}_{K_v}^{\times},\\
&\frak{N}_c = \frak{N} \cap \mathcal{O}_c, \hspace{1cm}\frak{N}_{\epsilon,c} = \frak{N}_{\epsilon} \cap \mathcal{O}_c.
\end{align*}

Henceforth, as in \cite[Section 4]{BertoliniDarmonPrasanna} and \cite[Section 8.2]{Brakocevic}, we say a Hecke character $\phi : K^{\times} \backslash \mathbb{A}_K^{\times} \rightarrow \mathbb{C}^{\times}$ has \emph{finite type $(c,\frak{N},\epsilon)$} if
we have $c| \frak{f}(\chi)|c\frak{N}_{\epsilon}$ and 
$$\chi|_{\hat{\mathcal{O}_c}^{\times}} = \psi_{\epsilon}$$
where here, $\psi_{\epsilon}$ is the composition
$$\hat{\mathcal{O}}_c^{\times} \twoheadrightarrow (\hat{\mathcal{O}}_c/(\frak{N}_{\epsilon,c}\hat{\mathcal{O}}_{c}))^{\times} = (\mathcal{O}_K/\frak{N}_{\epsilon})^{\times} = (\mathbb{Z}/N_{\epsilon}\mathbb{Z})^{\times} \xrightarrow{\epsilon^{-1}} \mathbb{C}^{\times}.$$

One can check that each member of the family of examples in Example \ref{examplefamily} has finite type $(c,\frak{N},1)$ (in fact each member has finite type $(1,1,1)$).

\section{Construction of the continuous $p$-adic $L$-function}\label{constructionsection}

Let $\Sigma$ denote the set of central critical characters $\chi$ of infinity type $(k+j,-j)$, where $j \in \mathbb{Z}$ and $j \ge 0$ or $-1 \ge j \ge 1-k$, and of finite type $(c,\frak{N},\epsilon)$, and which satisfy the following conditions on the local signs of the functional equation:
$$\epsilon_{\ell}(w,\chi^{-1}) = +1 \hspace{1cm} \text{for all finite primes $\ell$}.$$
Under our assumptions, this holds automatically except for $\ell$ in the set
$$S_w := \{\ell : \ell|(N,d_K), \ell\nmid N_{\epsilon}\}.$$

Let $\Sigma_+ \subset \Sigma$ denote the subset of such $\chi$ with $j \ge 0$ and $\Sigma_- = \Sigma \setminus \Sigma_+$. Then given $\chi \in \Sigma_+$, $\chi$ is central critical with respect to $w$. Using Waldspurger's calculation involving the Rankin-Selberg formula, we will interpolate (square roots of) central values $L(\pi_w \times \pi_{(\chi\phi)^{-1}},1/2)$ when $\chi \in \Sigma_+$. Let $\check{\Sigma} \subset \Hom_{\mathrm{cts}}(\Gamma^-,\mathcal{O}_{\mathbb{C}_p})$ denote the set of $p$-adic avatars of elements of $\Sigma$, and similarly with $\check{\Sigma}_+$ and $\check{\Sigma}_-$. We consider $\check{\Sigma}$ as a subspace of the space $\mathrm{Fun}(\mathbb{A}_K^{\times,(p\infty)},\mathcal{O}_{\mathbb{C}_p})$ of functions $\mathbb{A}_K^{\times,(p\infty)} \rightarrow \mathcal{O}_{\mathbb{C}_p}$ equipped with the uniform convergence topology. We let $\overline{\check{\Sigma}}_+$ denote the closure of $\check{\Sigma}_+$ in $\mathrm{Fun}(\mathbb{A}_K^{\times,(p\infty)},\mathcal{O}_{\mathbb{C}_p})$. One can in fact show, using Example \ref{examplefamily}, that $\check{\Sigma}_- \subset \overline{\check{\Sigma}}_+$. 

Henceforth, let $h_c = \#\mathcal{C}\ell(\mathcal{O}_c)$ and let $\{\frak{a}\}$ be a fixed full set of representatives of $\mathcal{C}\ell(\mathcal{O}_c)$ which are prime to $\frak{N}cp$. Recall the Shimura action of the ideals $\mathbb{I}^{p\frak{N}}$ of $\mathcal{O}_c$ which are prime to $p\frak{N}$ on the set of elliptic curves $A$ with CM by $\mathcal{O}_c$:
$$\frak{a}\star A = A/A[\frak{a}], \hspace{1cm} \pi_{\frak{a}} : A \twoheadrightarrow \frak{a} \star A$$
as in Definition \ref{Shimuraactiondefinition}. Then $\pi_{\frak{a}}$ sends $\Gamma_1(N)$-level structures to $\Gamma_1(N)$-level structures and $\Gamma(p^{\infty})$-level structures to $\Gamma(p^{\infty})$-level structures, and so gives an action of $\mathbb{I}^{p\frak{N}}$ on pairs $(A,t)$ and triples $(A,t,\alpha)$
$$\frak{a}\star(A,t) = (\frak{a}\star A,\pi_{\frak{a}}(t)),\hspace{1cm} \frak{a}\star(A,t,\alpha) = (\frak{a}\star A,\pi_{\frak{a}}(t),\pi_{\frak{a}}(\alpha)).$$
Given a differential $\omega \in \Omega_{A/\mathbb{C}_p}^1$, let $\omega_{\frak{a}} \in \Omega_{(\frak{a}\star A)/\mathbb{C}_p}^1$ be the unique differential such that $\pi_{\frak{a}}^*\omega_{\frak{a}} = \omega$. 

Now fix a CM elliptic curve $A/H$, a $\Gamma_1(N)$-level structure $t : \mathbb{Z}/N \xrightarrow{\sim} A[\frak{N}] \subset A[N]$, and a $\Gamma(p^{\infty})$-level structure $\alpha : \mathbb{Z}_p^{\oplus 2} \xrightarrow{\sim} T_pA$ for any $(A,\alpha) = (A,t,\alpha) \in \mathcal{Y}(\mathbb{C}_p,\mathcal{O}_{\mathbb{C}_p})$ above the point $(A,t) \in Y(\mathbb{C}_p,\mathcal{O}_{\mathbb{C}_p})$.
Now write the pullback of the normalized eigenform $w$ to $\mathcal{Y}_x$ as 
$$w|_{\mathcal{Y}_x} = y_{\mathrm{dR}}^kf\cdot \omega_{\mathrm{can}}^{\otimes k},$$
where $f \in \mathcal{O}(\mathcal{Y}_x)$. Let $(y_{\mathrm{dR}}^kf)^{\flat}(q_{\mathrm{dR}})$ be as in (\ref{stabilizedqdRexp}). 

Recall the complex-analytic universal cover $\mathcal{H}^+ \rightarrow Y$, where $\mathcal{H}^+$ denotes the complex upper half-plane, and write
$$w|_{\mathcal{H}^+} = F \cdot (2\pi idz)^{\otimes k}$$
where $dz$ is induced by the standard holomorphic differential on $\mathbb{C}$ via the uniformization of the universal elliptic curve $\mathbb{C}/(\mathbb{Z}+\mathbb{Z}\tau)$. 

Let $\chi \in \Sigma_+$ be a central critical character for $w$ of infinity type $(k+j,-j)$ where $j \ge 0$ and of finite type $(c,\frak{N},\epsilon)$. We recall the classical result of Waldspurger relating the algebraic part of the central critical value 
$$L(F,\chi^{-1},0) := L((\pi_w)_K \times\chi^{-1},1/2)$$
to the toric period, as explicitly calculated in our situation by Bertolini-Darmon-Prasanna.

\begin{theorem}[\cite{BertoliniDarmonPrasanna} Theorem 5.4]\label{BDPalgebraicity}
Let $\chi \in \Sigma_+$ be as above, and suppose that $c$ and $d_K$ are odd, and let $w_K = \#\mathcal{O}_K^{\times}$. 
Then we have 
\begin{equation}\label{complexLvalue}C(w,\chi,c)\cdot L(F,\chi^{-1},0) =  \sigma(w,\chi)\cdot \left(\sum_{[\mathfrak{a}] \in \mathcal{C}\ell(\mathcal{O}_c)}(w,\mathbb{N}_K^j\chi)^{-1}(\frak{a})\frak{d}_k^jF(\frak{a}\star (A,t))\right)^2,
\end{equation}
where the representatives $\frak{a}$ of classes of $\mathcal{C}\ell(\mathcal{O}_c)$ are chosen to be prime to $\frak{N}_c$, and \begin{align*}&C(w,\chi,c) \\
&= \frac{1}{4}\pi^{k+2j-1}\Gamma(j+1)\Gamma(k+j)w_k|d_K|^{1/2}\cdot c\cdot \mathrm{vol}(\mathcal{O}_c)^{-(k+2j)}\cdot 2^{\# S_w}\cdot \prod_{\ell|c}\frac{\ell - \epsilon_K(\ell)}{\ell-1} \in \mathbb{C}^{\times}
\end{align*}
and
$$\sigma(w,\chi) \in \mathbb{C}^{\times}, \hspace{1cm} |\sigma(w,\chi)| = 1,$$
are constants depending on $w, \chi$ and $c$. (For a precise definition of $\sigma(w,\chi)$, see (5.1.11) of loc. cit. and the discussion preceding it.) Here, we note that the sum (\ref{complexLvalue}) is well-defined (i.e., independent of the choice of representatives $\frak{a}$ of elements of $\mathcal{C}\ell(\mathcal{O}_c)$) by the central-criticality and $(c,\frak{N},\epsilon)$-finite typeness of $\chi$. 
\end{theorem}

\begin{assumption}\label{goodchoice}Now let $y = (A,\alpha) = (A,t,\alpha) \in \mathcal{C}(\mathcal{O})(\overline{K}_p,\mathcal{O}_{\overline{K}_p})$ for $\mathcal{C}(\mathcal{O})$ as defined in Section \ref{CMtorus}. Hence $\langle \alpha_1 \pmod{p} \rangle = A[\frak{p}]$ if $p$ is ramified in $K$. (Recall that $\frak{p} \subset \mathcal{O}_K$ is the prime above $p$.) 
\end{assumption}

Then Theorem \ref{MScomparison} along with Proposition \ref{sameperiod}, Proposition \ref{CMtorussameperiods} and Propositon \ref{CMproposition2} (in the case when $p = 2$ is inert) establishes the equality of algebraic numbers
\begin{equation}\label{Lvaluecomparison}\begin{split}&i_{\infty}^{-1}\left(\frac{C(w,\chi,c)\sigma(w,\chi)^{-1}}{\Omega_{\infty}^{2(k+2j)}}\cdot L(F,\chi^{-1},0)\right) \\
&i_{\infty}^{-1}\left(\frac{1}{\Omega_{\infty}^{2(k+2j)}(A,t)}\cdot \left(\sum_{[\mathfrak{a}] \in \mathcal{C}\ell(\mathcal{O}_c)}(w,\mathbb{N}_K^j\chi)^{-1}(\frak{a})\frak{d}_k^jF(\frak{a}\star (A,t))\right)^2\right) \\
&= i_p^{-1}\left(\left(\frac{1}{p^b\theta\Omega_p(A,t,\alpha)^2}\right)^{k+2j}\cdot\left(\sum_{[\mathfrak{a}] \in \mathcal{C}\ell(\mathcal{O}_K)}(\check{\mathbb{N}}_K^j\check{\chi})^{-1}(\frak{a})(p^b\theta_k^j)(y_{\mathrm{dR}}^kf)(\frak{a}\star (A,t,\alpha))\right)^2\right).
\end{split}
\end{equation}
We note that
$$|p^b\Omega_p(A,t,\alpha)^2| = 1$$
by Corollary \ref{periodunit}.

Again the sum in the last line of (\ref{Lvaluecomparison}) is well-defined, which is seen as follows. For any $a \in K^{\times}$ which is prime to $pc\frak{N}$, if we change all representatives $\frak{a}$ to $\frak{a}a$, then using our previously fixed embeddings $K^{\times} \hookrightarrow GL_2(\mathbb{Q}) \overset{i_p}{\hookrightarrow} GL_2(\mathbb{Q}_p)$ and invoking Proposition \ref{weightproposition}, we have
\begin{equation}\label{invariance1}\begin{split}(p^b\theta_k)^j(y_{\mathrm{dR}}^kf)(\frak{a}a\star (A,t,\alpha)) &= (p^b\theta_k)^j(y_{\mathrm{dR}}^kf)(\frak{a}\star (A,t,\alpha\cdot a)) \\
&= a^{k+2j}\epsilon(a)(p^b\theta_k)^j(y_{\mathrm{dR}}^kf)(\frak{a}\star (A,t,\alpha)).
\end{split}
\end{equation}
We also have
$$(\mathbb{N}_K^j\chi)^{-1}(a\frak{a}) = a^{-(k+2j)}\epsilon^{-1}(a)(\mathbb{N}_K^j\chi)^{-1}(\frak{a}),$$
since $\chi$ has infinity type $(k+j,-j)$ and finite type $(c,\frak{N},\epsilon)$. This with (\ref{invariance1}) gives the desired invariance. 

\begin{definition}\label{padicLfunctiondefinition}Recall that $A/F$ is a fixed elliptic curve with CM by $\mathcal{O}_c \subset \mathcal{O}_K$, and $(A,\alpha) = (A,t,\alpha)$ is chosen to be as in Assumptions (\ref{goodchoice}). For this choice, we define the \emph{anticyclotomic $p$-adic $L$-function attached to $w$} as a function 
$$\mathcal{L}_{p,\alpha}(w,\cdot) : \overline{\check{\Sigma}}_+ \rightarrow \mathbb{C}_p$$
given by 
$$\mathcal{L}_{p,\alpha}(w,\check{\chi}) := \sum_{[\mathfrak{a}] \in \mathcal{C}\ell(\mathcal{O}_K)}(\check{\mathbb{N}}_K^j\check{\chi})^{-1}(\frak{a})(p^b\theta_k)^j((y_{\mathrm{dR}}^kf)^{\flat}(q_{\mathrm{dR}}))(\frak{a}\star (A,t,\alpha))$$
if $\check{\chi}|_{\mathcal{O}_{K_p}^{\times}}(x_p) = x_p^{k+j}\overline{x_p}^{-j}$.
\end{definition}

\begin{proposition}\label{choiceofalpha}The $p$-adic $L$-function $\mathcal{L}_{p,\alpha}(w,\cdot)$ depends on the choice of $\alpha$ in the fixed point $(A,t,\alpha) \in \mathcal{Y}(\mathbb{C}_p,\mathcal{O}_{\mathbb{C}_p})$ as in Assumption \ref{goodchoice} up to a continuous character $\overline{\check{\Sigma}}_+  \rightarrow \mathcal{O}_{K_p}^{\times}$. 
\end{proposition}

\begin{proof}Let $(A,t,\alpha), (A,t,\alpha')$ be as in Assumption \ref{goodchoice}. Then we can write
$$(A,t,\alpha') = (A,t,\alpha)\cdot \left(\begin{array}{ccc}a & b\\
c & d\\
\end{array}\right)$$
for some $\gamma = \left(\begin{array}{ccc}a & b\\
c & d\\
\end{array}\right) \in GL_2(\mathbb{Z}_p)$. 
By Theorem \ref{keycorollary} and Proposition \ref{weightchangelimitproposition}, we see that 
\begin{align*}&((y_{\mathrm{dR}}^kf)^{\flat}(q_{\mathrm{dR}}))(\frak{a}\star (A,t,\alpha)\cdot \gamma) \\
&= (bc-ad)^{-(k+2j)}(cz_{\mathrm{dR}}(A,y,\alpha)+a)^{k+2j}((y_{\mathrm{dR}}^kf)^{\flat}(q_{\mathrm{dR}}))(\frak{a}\star (A,t,\alpha)).
\end{align*}
Now by Propositions \ref{CMproposition2} and \ref{CMtorussameperiods}, we see that 
$$(bc-ad)^{-(k+2j)}(cz_{\mathrm{dR}}(A,y,\alpha)+a)^{k+2j} \in \mathcal{O}_{K_p}^{\times}$$
and hence we have that 
$$\frac{\mathcal{L}_{p,\alpha'}(w,\cdot)}{\mathcal{L}_{p,\alpha}(w,\cdot)} : \overline{\check{\Sigma}}_+ \rightarrow \mathcal{O}_{K_p}^{\times}$$
is a character given by 
$$\check{\chi} \mapsto (bc-ad)^{-(k+2j)}(cz_{\mathrm{dR}}(A,y,\alpha)+a)^{k+2j}$$
if $\check{\chi}|_{\mathcal{O}_{K_p}^{\times}}(x_p) = x_p^{k+j}\overline{x_p}^{-j}$.
\end{proof}

\begin{theorem}\label{thm:maintheorem}
$$\mathcal{L}_{p,\alpha}(w,\cdot) : \overline{\check{\Sigma}}_+ \rightarrow \mathbb{C}_p$$
is a continuous function.  
\end{theorem}

\begin{proof}Let $\check{\chi}_1, \check{\chi}_2 \in \overline{\check{\Sigma}}_+$ with $\check{\chi}_i|_{\mathcal{O}_{K_p}^{\times}}(x_p) = x_p^{k+j_i}\overline{x_p}^{-j_i}$ for $i = 1,2$ and which satisfy
$$\check{\chi}_1(\frak{a}) \equiv \check{\chi}_2(\frak{a}) \pmod{p^m\mathcal{O}_{\mathbb{C}_p}}$$
for $\frak{a} \in \mathbb{A}_{K}^{\times,(p\infty)}$. Then evaluating on id\`{e}les congruent with $1 \pmod{\frak{N}}$,, we see that 
$$j_1 \equiv j_2 \pmod{(p-1)p^{m-1}}.$$
Now the continuity of $\mathcal{L}_{p,\alpha}(w,\cdot)$ follows from the continuity statements in Theorem \ref{keycorollary}. 
\end{proof}

\subsection{Interpolation}

We have the following ``approximate interpolation property'' for $\mathcal{L}_{p,\alpha}(w,\cdot)$. Here we make the assumption that $d_K$ is \emph{odd} in order to use Theorem \ref{BDPalgebraicity}, though we expect to remove this assumption in the future.
\begin{theorem}\label{approximateinterpolationproperty}Suppose that the fundamental discriminant $d_K$ is odd, and suppose $\{\chi_j\} \subset \Sigma_+$ is a sequence of algebraic Hecke characters where $\chi_j$ has infinity type $(k+j,-j)$, such that $\check{\chi}_j \rightarrow \check{\chi}$ in $\overline{\check{\Sigma}}_+ \setminus \check{\Sigma}_+$. Then we have 
\begin{equation}\label{approximateinterpolationformula}\begin{split}&\mathcal{L}_{p,\alpha}(w,\check{\chi})^2 \\
&= \lim_{\check{\chi}_j \rightarrow \check{\chi}}\left(\frac{1}{p^b\theta(\Omega_p)(A,t,\alpha)^2}\right)^{k+2j} i_{\infty}^{-1}\left(\frac{C(w,\chi_j,c)\sigma(w,\chi_j)^{-1}}{\Omega_{\infty}(A,t)^{2(k+2j)}}\cdot L(F,\chi_j^{-1},0)\right). 
\end{split}
\end{equation}
\end{theorem}
\begin{proof}This follows immediately from (\ref{Lvaluecomparison}) and Theorem \ref{keycorollary} (in particular, from (\ref{limitisstabilization})).
\end{proof}

\begin{remark}\label{interpolationremark}One can show that the Euler factor at $p$ which one usually removes from the $L$-value in a $p$-adic interpolation formula is \emph{constant} for the $L$-values $L(F,\chi_j^{-1},0)$. Hence it is possible that the terms in the right-hand side of the limit in (\ref{approximateinterpolationformula}) form a $p$-adically continuous function. One strategy to address this, which is part of forthcoming work, is to base-change to a degree 4 CM extension where the unique prime above $p$ in the maximal totally real field (a real quadratic field) splits in the CM extension. In this case one can construct an analytic anticyclotomic $p$-adic $L$-function with the usual $p$-adic interpolation property, using Serre-Tate coordinates. In particular, it interpolates $L$-values and their quadratic twists whose limits give rise, as in (\ref{approximateinterpolationformula}), to a product of our anticyclotomic $p$-adic $L$-function times some constant Euler factors. Hence values of the product of our $p$-adic $L$-function and a quadratic twist are specializations of an analytic function, from which one may draw conclusions regarding the continuity questions raised above.
\end{remark}

\subsection{The $p$-adic Waldspurger formula}\label{Waldspurgersection}

Retain the notation and setting of the earlier part of this chapter. In particular $p \nmid Nc$, $(c,N) = 1$, $N$ satisfies the Heegner hypothesis with respect to $K$ which determined an ideal $\frak{N}|N$, and $A$ was a fixed elliptic curve with CM by $\mathcal{O}_c$ and $\Gamma_1(N)$-level structure $t$ determined by the cyclic ideal $\frak{N}$. We have the following special value formula (Theorem \ref{padicWaldspurgerformula}) in the case when $k = 2$. We view this as analogous to the $p$-adic Waldspurger formulas of \cite{BertoliniDarmonPrasanna} and \cite{LiuZhangZhang} in the case where $p$ is inert or ramified in $K$. For the Eisenstein case, we view this formula as an analogue of the ``$p$-adic Kronecker limit formula" in \cite[Chapter X]{Katz3}. In fact our approach can be used to recover these special value formulas, using the fact that the $q_{\mathrm{dR}}$-expansion recovers the Serre-Tate expansion (Theorem \ref{STcoincide}) on the cover $\mathcal{Y}^{\mathrm{Ig}} \rightarrow Y^{\mathrm{ord}}$ and Theorem \ref{recoverAtkinSerre}. 

As before, let $w$ be a weight $k = 2$ normalized $\Gamma_1(N)$-new eigenform with nebentype $\epsilon$. In this section, we will briefly consider compactified modular curves $X(\Gamma) = \overline{Y(\Gamma)}$, where we recall that $Y(\Gamma)$ is the modular curve associated with a finite-index subgroup $\Gamma \subset SL_2(\mathbb{Z})$. Recall $\varepsilon$ as defined in (\ref{epsilondefinition}), and let
$$X(N,p^n,\epsilon) = X(\Gamma_1(N) \cap \Gamma(p^n) \cap \Gamma_1(p^{n+\varepsilon})), \hspace{1cm} J(N,p^n,\varepsilon) := \mathrm{Jac}(X(N,p^n,\epsilon)).$$
Then let
$$h_n : X(N,p^n,\varepsilon) \rightarrow J(N,p^n,\varepsilon)$$ 
denote the (analytification of the usual algebraic) Abel-Jacobi map over $\mathbb{Q}$, sending the cusp at infinity $(\infty) \mapsto 0$. Thus $w = h_0^*\omega$ where $\omega = g(q)dq/q \in \Omega_{J(N,p^n,\varepsilon))/\mathbb{Q}}^1$ and $g(q) = \sum_{n = 0}^{\infty}a_nq^n$ is the $q$-expansion of $w$. Then the $p$-stabilization $w^{\flat}$ as in (\ref{pstabilization}) satisfies $h_2^*\omega^{\flat} = w^{\flat}$ where $\omega^{\flat} = g^{\flat}(q)dq/q \in \Omega_{J(N,p^n,\varepsilon)/\mathbb{Q}}^1$ and $g^{\flat}(q) = \sum_{n = 0,p\nmid n}^{\infty}a_nq^n$ is the $q$-expansion of $w^{\flat}$, viewed as a form of level $\Gamma_1(N)\cap \Gamma(p^2)$; it is an old eigenform. Let
$$\log_{w^{\flat}} : J(N,p^2,\varepsilon)(\mathbb{C}_p,\mathcal{O}_{\mathbb{C}_p}) \rightarrow \mathbb{C}_p$$
be the formal logarithm attached to $\omega^{\flat}$, obtained by integrating the $q$-expansion $g^{\flat}(q)dq/q$ in the residue disc containing the origin, specifying that it takes the value 0 at the origin, and then extending linearly. From its $q$-expansion, we see that $\log_{w^{\flat}}$ is rigid on the residue disc containing the origin and hence extends uniquely via linearity to a locally analytic function on all of $J(N,p^2,\varepsilon)$. Pulling back by $h_2^*$, we have a locally analytic function
$$\log_{w^{\flat}} := h_2^*\log_{w^{\flat}} : X(N,p^2,\varepsilon)(\mathbb{C}_p,\mathcal{O}_{\mathbb{C}_p}) \rightarrow \mathbb{C}_p.$$
We know from \cite[Proposition A.0.1]{LiuZhangZhang} that $\log_{w^{\flat}}$ is a Coleman primitive of $w^{\flat}$. 

Let $H_c$ denote the ring class field attached to $\mathcal{O}_c$, and for $1 \le r \le \infty$ let $H_c(p^r)$ denote the compositum of $H_c$ and the ray class field $K(p^r)$ of conductor $p^r$, viewed as an extension of $H_c$ (recall $p \nmid c$). Then 
$$\Gal(H_c(p^r)/K) \cong  \mathcal{C}\ell(\mathcal{O}_c)\times \left(\mathcal{O}_{K_p}^{\times}/(1+p^r\mathcal{O}_{K_p})\right). 
$$ 
Then let $K_{p^r}/K$ denote the abelian extension corresponding to the quotient of Galois groups 
$$\Gal(H_c(p^r)/K)/\Gal(H_c/K) \cong \mathcal{O}_{K_p}^{\times}/(1+p^r\mathcal{O}_{K_p}).
$$

Now fix $(A,t,\alpha) \in \mathcal{C}(\overline{K}_p,\mathcal{O}_{\overline{K}_p})$ as in Assumption \ref{goodchoice} (note that by Proposition \ref{CMproposition2} and Theorem \ref{keycorollary} that $b = 0$ for all $(A,t,\alpha) \in \mathcal{C}(\overline{K}_p,\mathcal{O}_{\overline{K}_p})$). Given a point $P \in X(N,p^n)(\mathbb{C}_p,\mathcal{O}_{\mathbb{C}_p})$, let $[P]$ denote its associated divisor. Let $\chi : \mathcal{C}\ell(\mathcal{O}_c) \rightarrow \overline{\mathbb{Q}}^{\times}$ be a character. In the notation of Definition \ref{Shimuraactiondefinition}, define a Heegner point
\begin{align*}&P_{K,\alpha}(\chi) \\
&= \sum_{[\frak{a}] \in \mathcal{C}\ell(\mathcal{O}_c)} \chi^{-1}([\frak{a}])([\frak{a}\star(A,t,(\alpha_1 \hspace{-.4cm}\pmod{p^{2+\varepsilon}},\alpha_2 \hspace{-.4cm}\pmod{p^2})] - [(\infty)]) \\
&\hspace{10cm}\in J(N,p^2,\varepsilon)(H_c[p^2])^{\chi}
\end{align*}
where $J(N,p^2,\varepsilon)(H_c[p^2])^{\chi}$ denotes the $\chi$-isotypic component of the action of the group $\Gal(H_c[p^2]/K[p^2]) \cong \Gal(H_c/K)$. Note that the trace of the above Heegner point
$$P_K(\chi) := \sum_{\sigma \in \Gal(K_{p^2}/K)}P_{K,\alpha}^{\sigma} \in J(N,p^2,\varepsilon)(H_c)^{\chi}$$
where $J(N,p^2,\varepsilon)(H_c)^{\chi}$ denotes the $\chi$-isotypic component of the action $\Gal(H_c/K)$. We note that when $\chi = 1$, we have $P_{K,\alpha}(1) \in J(N,p^2,\varepsilon)(K_{p^2})$ and also $P_K(1) \in J(N,p^2,\varepsilon)(K)$. 

\begin{theorem}[``$p$-adic Waldspurger formula'']\label{padicWaldspurgerformula}Assume that $p > 2$ or $p$ is inert in $K$. Let $(A,t,\alpha) \in \mathcal{C}(\overline{K}_p,\mathcal{O}_{\overline{K}_p}) = \mathcal{C}(K_p^{\mathrm{ab}},\mathcal{O}_{K_p^{\mathrm{ab}}})$ (we recall the last equality follows from the theory of CM). 
We have, for any ring class character $\chi : \mathcal{C}\ell(\mathcal{O}_c) \rightarrow \overline{\mathbb{Q}}^{\times}$
$$\mathcal{L}_{p,\alpha}(w,\check{\mathbb{N}}_K\chi) = \begin{cases} \frac{1}{p^2(p^2-1)}\log_{w^{\flat}}P_K(\chi) & \text{$p$ is inert in $K$}\\
\frac{1}{p^3(p-1)} \log_{w^{\flat}}P_K(\chi) & \text{$p$ is ramified in $K$}
\end{cases}.$$

In particular, if $\mathcal{L}_{p,\alpha}(w,\check{\mathbb{N}}_K\chi)\neq 0$, then $P_K(\chi)$ projects via a modular parametrization
to a non-torsion point in $A_w(H_c)^{\chi}$, where $A_w/\mathbb{Q}$ is the $GL_2$-type abelian variety associated uniquely up to isogeny with $w$.
\end{theorem}

\begin{proof}
We need to slightly refine the statement of Theorem \ref{keycorollary}. Recall the $q_{\mathrm{dR}}$-expansion map from Definition \ref{qdRexp}.
\begin{equation}\label{qdR}\mathcal{O}_{\mathcal{Y}_x} \hookrightarrow \hat{\mathcal{O}}_{\mathcal{Y}_x}\llbracket q_{\mathrm{dR}}-1\rrbracket
\end{equation}
which is an injective map of pro\'{e}tale sheaves. Letting $y = (A,t,\alpha)$ as in the statement, and recall that we have $b = 0$ in this situation, by Proposition \ref{CMproposition2} and Theorem \ref{keycorollary}. From Lemma (\ref{limitlemma}), we have
\begin{equation}\label{integralqdR}(y_{\mathrm{dR}}^2f)_y \in \hat{\mathcal{O}}_{\mathcal{Y},y}^+\llbracket q_{\mathrm{dR}} - 1\rrbracket [1/p].
\end{equation}
Localizing (\ref{qdR}), we have
$$(y_{\mathrm{dR}}^2f)_y \in \mathcal{O}_{\mathcal{Y}_x,y} = \varinjlim_{\mathcal{V} \rightarrow \mathcal{U} \ni y}\hat{\mathcal{O}}_{\mathcal{Y}_x}(\mathcal{V})\llbracket q_{\mathrm{dR}}-1 \rrbracket$$
where $\mathcal{V} \rightarrow \mathcal{U} \ni y$ runs over pro\'{e}tale open neighborhoods of $y$, where $\mathcal{U} \subset \mathcal{Y}_x$ is a rational open set. In particular, by (\ref{integralqdR}) there is some pro\'{e}tale open neighborhood of $y$ $\mathcal{V} \rightarrow \mathcal{U} \ni y$, where $\mathcal{U} \subset \mathcal{Y}_x$ is a rational open set, such that 
\begin{equation}\label{integralsection}(y_{\mathrm{dR}}^2f)|_{\mathcal{V}} \in \hat{\mathcal{O}}_{\mathcal{Y}}^+(\mathcal{V})\llbracket q_{\mathrm{dR}}-1\rrbracket [1/p].
\end{equation}
In fact, since $y_{\mathrm{dR}}^2f \in \mathcal{O}_{\Delta,\mathcal{Y}_x}$, and in particular is a section defined on $\mathcal{Y}_x$, we have that (\ref{integralsection}) descends down to $\mathcal{U} \subset \mathcal{Y}_x$, i.e.
\begin{equation}\label{integralsectiondescent}(y_{\mathrm{dR}}^2f)|_{\mathcal{U}} \in \hat{\mathcal{O}}_{\mathcal{Y}}^+(\mathcal{U})\llbracket q_{\mathrm{dR}}-1\rrbracket [1/p].
\end{equation}
Now consider $\mathcal{U}^{\mathrm{ss}} := \mathcal{U} \cap \mathcal{Y}^{\mathrm{ss}}$, which is affinoid since it is the intersection of two affinoids (recall $\mathcal{Y}_x = \{\mathbf{z} \neq 0\}$, which is evidently affinoid, and $\mathcal{U}$ is a rational open and so is affinoid); note that $y \in \mathcal{U}^{\mathrm{ss}}(\mathbb{C}_p,\mathcal{O}_{\mathbb{C}_p})$ still, since $y \in \mathcal{Y}^{\mathrm{ss}}(\mathbb{C}_p,\mathcal{O}_{\mathbb{C}_p})$.  
Now consider the open subdomain
$$\mathcal{U}' := \mathcal{U}^{\mathrm{ss}} \cap \{|\theta(y_{\mathrm{dR}})/\mathbf{z}| < p^{1/p-1}\}$$
and we note that $y \in \mathcal{U}'$ since $p > 2$ by Proposition \ref{CMtorussameperiods}. Then we can take the restriction
$$(y_{\mathrm{dR}}^2f)|_{\mathcal{U}'} \in \hat{\mathcal{O}}_{\mathcal{Y}}^+(\mathcal{U}')\llbracket q_{\mathrm{dR}}-1\rrbracket [1/p]$$
from (\ref{integralsectiondescent}). In fact, we will show
$$\theta_k^j(y_{\mathrm{dR}}^2f)|_{\mathcal{U}'} \in \hat{\mathcal{O}}_{\mathcal{Y}}^+(\mathcal{U}')\llbracket q_{\mathrm{dR}}-1\rrbracket [1/p]$$
for any $j \in \mathbb{Z}_{\ge 0}$, and moreover
$$\theta_k^j(y_{\mathrm{dR}}^2f)^{\flat}(q_{\mathrm{dR}})|_{\mathcal{U}'} \in \hat{\mathcal{O}}(\mathcal{U}')$$
for any $j \in \mathbb{Z}/(p-1)\times\mathbb{Z}_p$ as follows. Let $r \in \mathbb{Q}$ such that 
$$p^r (y_{\mathrm{dR}}^2f)|_{\mathcal{U}'} \in \hat{\mathcal{O}}_{\mathcal{Y}}^+(\mathcal{U}')\llbracket q_{\mathrm{dR}}-1\rrbracket.$$
Clearly 
$$\left(\frac{q_{\mathrm{dR}}d}{dq_{\mathrm{dR}}}\right)^j(p^ry_{\mathrm{dR}}^2f)|_{\mathcal{U}'} \in \hat{\mathcal{O}}_{\mathcal{Y}}^+(\mathcal{U}')\llbracket q_{\mathrm{dR}}-1\rrbracket$$
by a easy computation. Moreover, on $\mathcal{U}'$ we have
$$\left|c_i(j)\left(\frac{\theta(y_{\mathrm{dR}})}{\mathbf{z}}\right)^i\right| \le 1$$
for all $i \gg 0$, since $i| c_i(j)$ and $|i!| \rightarrow p^{-1/(p-1)}$ as $i \rightarrow \infty$ in $\mathbb{Z}_{\ge 0}$. In particular, from (\ref{MSformulaq}) (recalling $b = 0$ in our situation) we see that for some $s \in \mathbb{Q}$ we have
\begin{equation}\label{integral1}\theta_k^j(p^{r+s}y_{\mathrm{dR}}^2f)|_{\mathcal{U}'} \in \hat{\mathcal{O}}_{\mathcal{Y}}^+(\mathcal{U}')
\end{equation}
for all $j \in \mathbb{Z}_{\ge 0}$. Moreover, one sees that
$$\lim_{m \rightarrow \infty}\theta_k^{j_0 + p^m(p-1)}(p^{r+s}y_{\mathrm{dR}}^2f)|_{\mathcal{U}'} \rightarrow \theta_k^{j_0}(p^{r+s}y_{\mathrm{dR}}^2f)^{\flat}(q_{\mathrm{dR}})|_{\mathcal{U}'}$$
in $\hat{\mathcal{O}}_{\mathcal{Y}}^+(\mathcal{U}')$, which in particular implies that 
\begin{equation}\label{integral2}\theta_k^j(p^{r+s}y_{\mathrm{dR}}^2f)^{\flat}(q_{\mathrm{dR}})|_{\mathcal{U}'} \in \hat{\mathcal{O}}_{\mathcal{Y}}^+(\mathcal{U}')
\end{equation}
for all $j \in \mathbb{Z}/(p-1) \times \mathbb{Z}_p$. 

Now we examine the $GL_2(\mathbb{Z}_p)$ action on the above sections. Let
$$\Gamma_0(p^r) := \left\{\left(\begin{array}{ccc}a & b\\
c & d\\
\end{array}\right) \in GL_2(\mathbb{Z}_p) : c \equiv 0 \pmod{p^r}\right\}$$
consider the open subdomain of $\mathcal{Y}_x$
$$\mathcal{U}'_0(p) = \mathcal{U}'\cdot \Gamma_0(p) = \bigcup_{\gamma \in \Gamma_0(p)}\mathcal{U}'\cdot \gamma.$$
We claim that $\theta_2^j(y_{\mathrm{dR}}^2f)|_{\mathcal{U}'}$ extends to $\mathcal{U}'_0(p)$ for all $j \in \mathbb{Z}/(p-1)\times \mathbb{Z}_p$, which we see as follows. For any 
$\left(\begin{array}{ccc}a & b\\
c & d\\
\end{array}\right) \in \Gamma_0(p)$, by Proposition \ref{weightpropositiontheta}, we have
$$\gamma^*\left(\theta_2^j(y_{\mathrm{dR}}^2f)|_{\mathcal{U}'\cdot \gamma}\right) = (bc-ad)^{-(2+2j)}(c\theta(z_{\mathrm{dR}}) + a)^{2+2j}\theta_k^j(y_{\mathrm{dR}}^2f)|_{\mathcal{U}'}$$
and so 
\begin{equation}\label{weightcharacter}\left|\left((bc-ad)^{-(2+2j)}(c\theta(z_{\mathrm{dR}}) + a)^{2+2j}\right)|_{\mathcal{U}'}\right| = 1
\end{equation}
for all $j \in \mathbb{Z}/(p-1) \times \mathbb{Z}_p$ by the definition of $\mathcal{U}'$ above. Hence, from (\ref{integral1}) and (\ref{integral2}), we have that
$$\theta_k^j(p^{r+s}y_{\mathrm{dR}}^2f)|_{\mathcal{U}'\cdot \gamma} \in \hat{\mathcal{O}}_{\mathcal{Y}}^+(\mathcal{U}'\cdot \gamma)$$
for all $j \in \mathbb{Z}_{\ge 0}$ and
$$\theta_k^j(p^{r+s}y_{\mathrm{dR}}^2f)^{\flat}(q_{\mathrm{dR}})|_{\mathcal{U}'\cdot \gamma} \in \hat{\mathcal{O}}_{\mathcal{Y}}^+(\mathcal{U}'\cdot \gamma)$$
for all $j \in \mathbb{Z}_{\ge 0}$. In fact from (\ref{weightcharacter}), we see that
$$\theta_k^j(p^{r+s}y_{\mathrm{dR}}^2f)^{\flat}(q_{\mathrm{dR}})|_{\mathcal{U}'\cdot \gamma} \in \hat{\mathcal{O}}_{\mathcal{Y}}^+(\mathcal{U}'\cdot \gamma)$$
for all $j \in \mathbb{Z}_p \times \mathbb{Z}/(p-1)$. This implies that
$$\theta_k^j(p^{r+s}y_{\mathrm{dR}}^2f)|_{\mathcal{U}'_0(p)} \in \hat{\mathcal{O}}_{\mathcal{Y}}^+(\mathcal{U}'_0(p))$$
has weight $2+2j$ for $\Gamma_0(p)$ on $\mathcal{U}'_0(p)$ (in the sense of Definition \ref{weightdefinitiontheta}) for any $j \in \mathbb{Z}_{\ge 0}$, and 
$$\theta_k^j(p^{r+s}y_{\mathrm{dR}}^2f)^{\flat}(q_{\mathrm{dR}})|_{\mathcal{U}'_0(p)} \in \hat{\mathcal{O}}_{\mathcal{Y}}^+(\mathcal{U}'_0(p))$$
has weight $2+2j$ for $\Gamma_0(p)$ on $\mathcal{U}'_0(p)$ for any $j \in \mathbb{Z}/(p-1) \times \mathbb{Z}_p$. 

Now we consider the closed subspace $\mathcal{Y}^{\mathrm{Ig}} \subset \mathcal{Y}_x$, recall that $\mathcal{Y}^{\mathrm{Ig}} \rightarrow Y^{\mathrm{ord}}$ is a $\mathbb{Z}_p^{\times} \times \mathbb{Z}_p^{\times}$-pro\'{e}tale covering, and on this subspace the restrictions $\theta_2^j(y_{\mathrm{dR}}^2f)|_{\mathcal{Y}^{\mathrm{Ig}}}$ and $\theta_2^j(y_{\mathrm{dR}}^2f)^{\flat}(q_{\mathrm{dR}})|_{\mathcal{Y}^{\mathrm{Ig}}}$ as defined in Section \ref{closedweights}. By Proposition \ref{weightchangeproposition}, we have that 
$$\theta_2^j(y_{\mathrm{dR}}^2f)|_{\mathcal{Y}^{\mathrm{Ig}}}$$
has weight $2+2j$ for $\mathbb{Z}_p^{\times}\times \mathbb{Z}_p^{\times}$ on $\mathcal{Y}^{\mathrm{Ig}}$ (in the sense of Definition \ref{CMweightdefinition}) for all $j \in \mathbb{Z}_{\ge 0}$, and 
$$\theta_2^j(y_{\mathrm{dR}}^2f)^{\flat}(q_{\mathrm{dR}})|_{\mathcal{Y}^{\mathrm{Ig}}}$$
has weight $2+2j$ for $\mathbb{Z}_p^{\times}\times \mathbb{Z}_p^{\times}$ on $\mathcal{Y}^{\mathrm{Ig}}$ for all $j \in \mathbb{Z}/(p-1) \times \mathbb{Z}_p$.

Now consider the adic subspace $\mathcal{Y}^{\mathrm{Ig}} \sqcup \mathcal{U}'_0(p)$ (where this is a disjoint union by the construction of $\mathcal{U}'$). Then we have the section 
$$\theta_k^{-1}(p^{r+s}y_{\mathrm{dR}}^2f)^{\flat}(q_{\mathrm{dR}})|_{\mathcal{Y}^{\mathrm{Ig}} \sqcup \mathcal{U}'_0(p)} \in \hat{\mathcal{O}}_{\mathcal{Y}^{\mathrm{Ig}}}^+(\mathcal{Y}^{\mathrm{Ig}}) \times \hat{\mathcal{O}}_{\mathcal{Y}^{\mathrm{Ig}}}^+(\mathcal{Y}^{\mathrm{Ig}})$$
which is of weight 0 for $\mathbb{Z}_p^{\times} \times \mathbb{Z}_p^{\times}$ on $\mathcal{Y}^{\mathrm{Ig}}$, and of weight 0 for $\Gamma_0(p)$ on $\mathcal{U}'_0(p)$. Recall the projection 
$$\lambda_{\Gamma_0(p)} : \mathcal{Y} \rightarrow Y(\Gamma_0(p)),$$
and let 
$$U'_0(p) = \lambda_{\Gamma_0(p)}(\mathcal{U}'_0(p)),$$
and let $Y^{\mathrm{Ig}}_0(p) = \lambda_{\Gamma_0(p)}(\mathcal{Y}^{\mathrm{Ig}})$; we note that we have a natural section $Y \rightarrow Y(\Gamma_0(p))$ for the natural projection $\rho : Y(\Gamma_0(p)) \rightarrow Y(GL_2(\mathbb{Z}_p)) = Y$, which is moreover an open immersion, given by the canonical subgroup, which gives a canonical rigid-analytic (and hence adic) identification $Y^{\mathrm{ord}} \cong Y^{\mathrm{Ig}}_0(p)$. Moreover, the Galois group of $\mathcal{Y}^{\mathrm{Ig}} \rightarrow Y^{\mathrm{Ig}}_0(p) \cong Y^{\mathrm{ord}}$ is $(1+p\mathbb{Z}_p) \times \mathbb{Z}_p^{\times}$. Now by piecewise descent along $(1+p\mathbb{Z}_p) \times \mathbb{Z}_p^{\times}$ on $\mathcal{Y}^{\mathrm{Ig}} \rightarrow Y^{\mathrm{Ig}}_0(p) \cong Y^{\mathrm{ord}}$ and along $\Gamma_0(p)$ on $\mathcal{U}'_0(p) \rightarrow U'_0(p)$, we have a (rigid analytic) section 
$$G := \theta_k^{-1}(p^{r+s}y_{\mathrm{dR}}^2f)^{\flat}(q_{\mathrm{dR}})|_{\mathcal{Y}^{\mathrm{Ig}} \sqcup \mathcal{U}'_0(p)} \in \hat{\mathcal{O}}_Y(Y^{\mathrm{Ig}}_0(p) \sqcup U'_0(p))$$
on the \emph{affinoid open} $Y^{\mathrm{Ig}}_0(p) \sqcup U'_0(p)) \subset Y(\Gamma_0(p))$. (We see that this latter domain is affinoid open because $U'_0(p)$ is the image of the open $\mathcal{U}'_0(p)$ under $\lambda$, and $Y^{\mathrm{Ig}}_0(p) \cong Y^{\mathrm{ord}}$, which is affinoid open.)

We know that 
\begin{equation}\label{primitive}G|_{Y^{\mathrm{Ig}}_0(p)} = \log_{w^{\flat}}|_{Y^{\mathrm{Ig}}_0(p)}
\end{equation} by Theorems \ref{recoverAtkinSerre}, \ref{stabilizationtheorem} and \ref{keycorollary}. Hence applying the exterior differential
$$d : \hat{\mathcal{O}}_{Y(\Gamma_0(p))} \rightarrow \hat{\mathcal{O}}_{Y(\Gamma_0(p))} \otimes_{\mathcal{O}_{Y(\Gamma_0(p))}} \Omega_{Y(\Gamma_0(p))}^1,$$
we hence see that 
$$dG \in \hat{\mathcal{O}}_{Y(\Gamma_0(p))} \rightarrow \hat{\mathcal{O}}_{Y(\Gamma_0(p))} \otimes_{\mathcal{O}_{Y(\Gamma_0(p))}} \Omega_{Y(\Gamma_0(p))}^1(Y^{\mathrm{Ig}}_0(p) \sqcup U'_0(p))$$
is rigid, and that 
$$dG|_{Y^{\mathrm{Ig}}_0(p)} = w^{\flat}|_{Y^{\mathrm{Ig}}_0(p)}$$
which implies by rigidity, since $Y^{\mathrm{Ig}}_0(p) \subset Y(\Gamma_0(p))$ is open (admissible) affinoid, that
$$dG = w^{\flat}.$$
Now by uniqueness of Coleman integration (\cite[Proposition A.0.1]{LiuZhangZhang}, or more specifically \cite[Corollary 2.1c]{Coleman}), by (\ref{primitive}) we have 
\begin{equation}\label{primitive2}G = \log_{w^{\flat}}|_{Y^{\mathrm{Ig}}_0(p) \sqcup U'_0(p)}.
\end{equation}

Now we apply the entire above argument to $y = \frak{a}\star (A,t,\alpha)$, for any $\frak{a} \subset \mathcal{O}_c$ prime to $p\frak{N}$. Now by Definition \ref{padicLfunctiondefinition}, we have
\begin{equation}\label{equality1}\mathcal{L}_{p,\alpha}(w,\check{\mathbb{N}}_K\chi) \overset{(\ref{primitive2})}{=} \log_{w^{\flat}}P_{K,\alpha}(\chi).
\end{equation}

Now we claim that for any $\gamma \in \mathcal{O}_{K_p}^{\times} \subset GL_2(\mathbb{Z}
_p)$, letting $(A,\alpha') = (A,\alpha)\cdot \gamma$, we have
\begin{equation}\label{equality2}\mathcal{L}_{p,\alpha}(w,\check{\mathbb{N}}_K\chi) = \mathcal{L}_{p,\alpha'}(w,\check{\mathbb{N}}_K\chi).
\end{equation}
But this follows directly from Proposition \ref{choiceofalpha} and Theorem \ref{approximateinterpolationformula} for $j = -1$.

Applying this to a set of $\alpha'_i$ such that $\{\alpha' \pmod{p^2}\}_{i = 1}^{\#\Gal(K_{p^2}/K)} = \{\alpha^{\sigma}\}_{\sigma \in \Gal(K_{p^2}/K)}$, where one observes 
$$\#\Gal(K_{p^2}/K) = \begin{cases} p^2(p^2-1) & \text{$p$ is inert in $K$}\\
p^3(p-1) & \text{$p$ is ramified in $K$}
\end{cases},
$$
then from (\ref{equality1}) and (\ref{equality2}) we have
\begin{equation}\label{equality3}\mathcal{L}_{p,\alpha}(w,\check{\mathbb{N}}_K\chi) \overset{(\ref{primitive2})}{=} \log_{w^{\flat}}P_{K,\alpha}(\chi) \overset{(\ref{equality2})}{=} \frac{1}{\#\Gal(K_{p^2}/K)}\log_{w^{\flat}}P_K(\chi) 
\end{equation}
which is what we wanted to show.

Now for the final statement of the Theorem, we note that $\omega^{\flat}$, viewed as a form on $Y(\Gamma_0(p^2,\varepsilon))$ is in the same Hecke isotypic component as that of $\omega$, which is a new eigenform on $Y$, outside of the prime $p$. Hence $P_K(\chi) \in J(N,p^2,\varepsilon)(H_c)^{\chi}$ is non-trivial in the old part of the Eichler-Shimura decomposition corresponding to the new eigenform $\omega$ from level $\Gamma_1(N)$. This latter component is isomorphic to the direct product of some number of copies of $A_{\omega}(H_c)^{\chi}$. Hence, given the nonvanishing of (\ref{equality3}), we see that $P_K(\chi)$ projects to a non-trivial point in this latter group, and hence projects non-trivially into some copy of $A_{\omega}(H_c)^{\chi}$. 
\end{proof}

\end{document}